\documentclass{amsart}[11pt]

\usepackage{etex}
\usepackage[lmargin=1in,rmargin=1in,tmargin=1in,bmargin=1in]{geometry}
\usepackage{amsmath,amsthm,amsfonts,amssymb,verbatim}
\usepackage{graphicx}
\usepackage{tabularx}
\usepackage{overpic}
\usepackage[dvipsnames]{xcolor}
\usepackage{tikz}\usetikzlibrary{decorations.markings}
\usetikzlibrary{arrows}
\usepackage{rotating}
\usepackage{hyperref}
\usepackage{framed}
\usepackage{placeins}
\usepackage{booktabs}
\usepackage{pinlabel}
 \usepackage{picins}
\usepackage{color}
\usepackage{multirow}
\usepackage{overpic}
\usepackage{colortbl}
\usepackage{geometry}
\usepackage{slashed}
\usepackage[font=footnotesize,labelfont=bf]{caption}
\usepackage{arydshln}
\usepackage{blkarray}
\usepackage{mathtools}

\usepackage{setspace}

\usepackage{wrapfig}

\DeclareMathAlphabet{\mathpzc}{OT1}{pzc}{m}{it}

\newcommand{\into}{\hookrightarrow}

\newcommand{\Z}{\mathbb{Z}}
\newcommand{\Q}{\mathbb{Q}}
\newcommand{\R}{\mathbb{R}}
\newcommand{\F}{\mathbb{F}}

\newcommand{\sI}{\mathcal{I}}
\newcommand{\sJ}{\mathcal{J}}

\newcommand{\sF}{\mathcal{F}}

\newcommand{\sR}{\mathcal{R}}
\newcommand{\sRhat}{\widehat{\mathcal{R}}}
\newcommand{\sRinfty}{\mathcal{R}^\infty}
\newcommand{\sRminus}{\mathcal{R}^-}

\newcommand{\generators}{\mathcal{G}}
\newcommand{\sH}{\mathcal{H}}

\DeclareMathOperator{\Mor}{Mor}

\DeclareMathOperator{\rot}{rot}
\DeclareMathOperator{\Id}{Id}

\newcommand{\bchain}{{\bf b}}
\newcommand{\bchainhat}{{\bf \widehat{b}}}

\newcommand{\spinc}{\operatorname{Spin}^c}

\newcommand{\Spinc}{\operatorname{Spin}^c}

\newcommand{\Ztwo}{\Z/2\Z}

\newcommand{\gr}{\operatorname{gr}}

\newcommand{\cylinder}{\mathcal{Z}}
\newcommand{\strip}{\mathcal{S}}
\newcommand{\pcylinder}{\mathcal{Z}^*}
\newcommand{\pstrip}{\mathcal{S}^*}
\newcommand{\dmcylinder}{\mathcal{Z}^{z,w}}
\newcommand{\dmstrip}{\mathcal{S}^{z,w}}

\newcommand{\tracks}{{\boldsymbol\vartheta}}

\newcommand{\HFhat}{\widehat{\mathit{HF}}}

\newcommand{\HFminus}{\mathit{HF}^{-}}
\newcommand{\CFminus}{\mathit{CF}^-}

\newcommand{\CFinfty}{\mathit{CF}^\infty}

\newcommand{\CFKminus}{\mathit{CFK}^-}
\newcommand{\CFKinfty}{\mathit{CFK}^\infty}

\newcommand{\CFD}{\widehat{\mathit{CFD}}}

\newcommand{\Alg}{\mathcal{A}}

\newcommand{\Ainfty}{\mathcal{A}_\infty}

\newcommand{\spin}{\mathfrak{s}}

\newcommand{\x}{{\bf x}}
\newcommand{\y}{{\bf y}}

\newtheorem{theorem}{Theorem}[section]

\newtheorem{corollary}[theorem]{Corollary}
\newtheorem{proposition}[theorem]{Proposition}
\newtheorem{lemma}[theorem]{Lemma}
\newtheorem{conjecture}[theorem]{Conjecture}

\newtheorem*{namedtheorem}{\theoremname}
\newcommand{\theoremname}{testing}

\theoremstyle{definition}
\newtheorem{definition}[theorem]{Definition}
\newtheorem{question}[theorem]{Question}

\newtheorem{remark}[theorem]{Remark}
\newtheorem{example}[theorem]{Example} 

\title[Knot Floer homology as immersed curves]{Knot Floer homology as immersed curves}
\author[Jonathan Hanselman]{Jonathan Hanselman}
\address {Department of Mathematics, Princeton University.\newline \it{E-mail address:} \tt{jh66@princeton.edu}}
\thanks{The author was partially supported by NSF grant DMS-2105501 }

\makeatletter
\def\l@subsection{\@tocline{2}{0pt}{2.5pc}{5pc}{}}
\makeatother

\begin{document}
\maketitle

\begin{abstract} To a nullhomologous knot $K$ in a 3-manifold $Y$, knot Floer homology associates a bigraded chain complex over $\F[U,V]$ as well as a collection of flip maps; we show that this data can be interpretted as a collection of decorated immersed curves in the marked torus. This is inspired by earlier work of the author with Rasmussen and Watson, showing that bordered Heegaard Floer invariants $\CFD$ of manifolds with torus boundary can be interpreted in a similar way \cite{HRW, HRW:companion}. Indeed, if we restrict the construction in this paper to the $UV = 0$ truncation of the knot Floer complex for knots in $S^3$ with $\Z/2\Z$ coefficients, which is equivalent to $\CFD$ of the knot complement, we get precisely the curves in \cite{HRW}; this paper then provides an entirely bordered-free treatment of those curves in the case of knot complements, which may appeal to readers unfamiliar with bordered Floer homology. On the other hand, the knot Floer complex is a stronger invariant than $\CFD$ of the complement, capturing ``minus" information while $\CFD$ is only a ``hat" flavor invariant. We show that this extra information is realized by adding an additional decoration, a bounding chain, to the immersed multicurves. We also give geometric surgery formulas, showing that $\HFminus$ of rational surgeries on nullhomologous knots and the knot Floer complex of dual knots in integer surgeries can be computed by taking Floer homology of the appropriate decorated curves in the marked torus. A section of the paper is devoted to a giving a combinatorial construction of Floer homology of Lagrangians with bounding chains in marked surfaces, which may be of independent interest. 
\end{abstract}

\renewcommand{\baselinestretch}{0}\normalsize
\tableofcontents
\renewcommand{\baselinestretch}{1.0}\normalsize

\section{Introduction}\label{sec:intro}% !TEX root = ../CFKcurvesZHS3s.tex
%introduction.tex

Knot Floer homology, defined by Ozsv{\'a}th and Szab{\'o} \cite{OzSz:knots} and Rasmussen \cite{Ras:knot-floer}, is an invariant of a knot $K$ in a closed 3-manifold $Y$. In the decades since its introduction, knot Floer homology has proved to be a tremendously useful invariant, with numerous applications in the study of knots as well as three and four dimensional manifolds.  In its usual formulation it associates an algebraic object to the pair $(Y,K)$, namely a bigraded chain complex $CFK_{\sRminus}(Y,K)$ over the ring $\sRminus = \F[U,V]$ for some coefficients $\F$; we will assume $\F$ is a field throughout, though we briefly remark on the case of $\Z$ coefficients in Section \ref{sec:Z-coefficients}. The complex $CFK_{\sRminus}(Y,K)$ is an invariant of the pair $(Y,K)$ up to graded chain homotopy equivalence. This complex splits over spin$^c$ structures of $Y$:
$$CFK_{\sRminus}(Y,K) = \bigoplus_{\spin\in\Spinc(Y)} CFK_{\sRminus}(Y,K;\spin).$$
In addition to this bigraded complex the knot Floer package also comes with a collection of chain maps
$$\Psi_\spin: CFK_{\sRminus}(Y,K;\spin) \to CFK_{\sRminus}(Y,K;\spin + PD[K]),$$
defined up to chain homotopy, called \emph{flip maps}. 
The goal of this paper is to show that this algebraic data, the knot Floer complex $CFK_{\sRminus}(Y,K)$ together with the collection of flip maps $\{\Psi_\spin\}_{\spin\in\Spinc(Y)}$, admits a geometric representation as an element of the immersed Fukaya category of the marked torus, that is, as a decorated immersed curve in the marked torus. We will also show that this geometric description allows for simplified computations of the Heegaard Floer homology $HF^-$ of Dehn surgeries on $K$.

For simplicity we will restrict our attention to nullhomologous knots $K$. We remark that most of the results can be extended to knots that are only rationally nullhomologous, and in fact the core arguments are unchanged, but there is an added layer of complexity describing the spin$^c$ decomposition and the gradings in this more general setting. To avoid obscuring the main constructions with these details, the case of rationally nullhomologous knots will be addressed in a subsequent paper. 

\subsection{Immersed curve invariants for knots}
Given a nullhomologous knot $K$ in $Y$, let $M$ denote the complement $Y \setminus \nu(K)$. Let $\lambda \in H_1(\partial M; \Z)$ be the homology class of the Seifert longitude. We consider the marked torus $T_M = H_1(\partial M; \R) / H_1(\partial M; \Z)$ with a marked point at $0$; note that $T_M$ can be identified with $\partial_M$ with a marked point $w$. We will also consider the universal cover $\widetilde{T}_M = H_1(\partial M; \R)$ with a set of marked points given by $H_1(\partial M; \Z)$, as well as the intermediate covering space $\overline{T}_M = \widetilde{T}_M / \langle \lambda \rangle$. 
For each spin$^c$ structure $\spin$ in $\spinc(M)$ (which can be identified with $\spinc(Y)$, since $K$ is nullhomologous), we define a decorated immersed multicurve $\HFminus(Y,K;\spin)$ in $\overline{T}_M$;  this is a pair $(\Gamma, \bchain)$ where $\Gamma$ is an oriented, weighted, graded immersed multicurve in $T_M$ and $\bchain$ is a bounding chain, which may be thought of as a linear combination of self-intersection points of $\Gamma$ satisfying certain conditions.

The decorated curve $\HFminus(Y,K; \spin)$ encodes the knot Floer complex $CFK_{\sRminus}(Y,K;\spin)$, which can be recovered by adding additional marked points to $\overline{T}_M$ and taking the Lagrangian Floer complex with (a lift of) a meridian of $K$. It also encodes the flip map
$$\Psi_\spin: CFK_{\sRminus}(Y,K;\spin) \to CFK_{\sRminus}(Y,K;\spin),$$
which can be recovered from the Lagrangian Floer complex of $\HFminus(Y,K; \spin)$ with another particular curve in $\overline{T}_M$. Conversely, the decorated curve is uniquely determined by $CFK_{\sRminus}(Y,K; \spin)$ and the flip map, up to equivalence in the immersed Fukaya category $\overline{T}_M$. Here two decorated curves in the Fukaya category are considered to be equivalent if they have the same Floer homology with any other decorated curve.

Uniqueness up to equivalence in the Fukaya category is a slightly unsatisfying notion: there are many decorated curves representing the same equivalence class, and it is not always apparent when two decorated curves are equivalent. A much stronger claim is that we can always choose the decorated curve $\HFminus(Y,K; \spin)$ to have a particularly nice form, and that with this assumption the underlying immersed multicurve $\Gamma$ is well defined up to homotopy in the marked surface $\overline{T}_M$. The first condition for this nice representative is that the immersed multicurve $\Gamma$ is in almost simple position (see Definition \ref{def:almost-simple-position}), which essentially means it is in minimal position subject to the constraint that it bounds no immersed annuli. The second condition concerns the bounding chain $\bchain$. The self-intersection points of $\Gamma$ all have a degree, and $\bchain$ is a linear combination of the self-intersection points with non-positive degree. Let $\bchainhat$ denote the restriction of this linear combination to self-intersection points of degree zero. We say $\bchainhat$ is of local system type if it contains only a very special subset of degree zero intersection points (see Definition \ref{def:local-system-type}); an immersed curve decorated with such a $\bchainhat$ is equivalent to an immersed curve decorated with local systems.

\begin{theorem}\label{thm:curve-invariant}
For any nullhomologous knot $K$ in a 3-manfiold $Y$ and for any $\spin \in \spinc(Y)$, there is a decorated immersed curve $\HFminus(Y,K; \spin)$ in $\overline{T}_M$ representing the knot Floer complex $CFK_{\sRminus}(Y,K;\spin)$ and the flip map $\Psi_\spin$, such that the underlying curve $\Gamma$ is in simple position and the restriction $\bchainhat$ of the bounding chain $\bchain$ to degree zero intersection points is of local system type. This decorated curve is a well-defined invariant of $(Y,K;\spin)$ up to equivalence in the Fukaya category of $\overline{T}_M$; moreover, the underlying immersed curve is unique up to homotopy in the marked cylinder $\overline{T}_M$  and $\bchainhat$ is unique as a subset of self-intersection points of $\Gamma$ .
\end{theorem}

Note that in claiming $\bchainhat$ is unique as a subset of self-intersection points of $\Gamma$, when $\Gamma$ is only defined up to homotopy, we use the fact that there is an obvious identification of the relevant self-intersection points between any two homotopic curves $\Gamma$ and $\Gamma'$ that are both in simple position.

We remark that we do not have a unique representative for the portion of $\bchain$ coming from strictly negative degree self-intersection points. Thus, while $\Gamma$ is well-defined and the degree zero part $\bchainhat$ of $\bchain$ is well-defined, there may be different choices of $\bchain$ on $\Gamma$ that are equivalent in the Fukaya category and satisfy our conditions for a nice representative. It may be possible to define a normal form for $\bchain$, imposing additional constraints on $\bchain$ so that we can always find a representative satisfying these restraints and so that such a representative is unique as a subset of the self-intersection points of $\Gamma$, and we hope to explore this in future work. However, in practice, once $\Gamma$ and $\bchainhat$ are fixed there are very few valid choices for $\bchain$ and it is generally not difficult to tell which choices are equivalent to each other. On a case by case basis, we can often find a representative that is clearly the simplest possible (see Section \ref{sec:trivial-bounding-chains}).

Theorem \ref{thm:curve-invariant} follows from a structure theorem relating bigraded complexes and flip maps to immersed curves in the infinite marked strip $\strip = [-\tfrac 1 2, \tfrac 1 2]\times \R$ and the infinite marked cylinder $\cylinder = (\R/\Z)\times \R$, each with marked points at points $(0,n+\tfrac 1 2)$ for $n\in\Z$:

\begin{theorem}\label{thm:curve-invariant-for-complex}
Any bigraded complex $C$ over $\sRminus$ can be represented by a decorated immersed curve $(\Gamma, \bchain)$ in the infinite marked strip $\strip$, and a bigraded complex $C$ equipped with a flip map $\Psi$ can be represented by a decorated immersed curve $(\Gamma, \bchain)$ in the infinite marked cylinder $\cylinder$, where in each case $\Gamma$ is in almost simple position and the restriction $\bchainhat$ of $\bchain$ to degree zero intersection points is of local system type. The decorated curves are well-defined as elements of the relevant Fukaya category, and moreover $\Gamma$ is well defined up to homotopy and $\bchainhat$ is well-defined as a subset of self-intersection points of $\Gamma$.
\end{theorem}

There is a simplified curve invariant $\HFhat(Y,K;\spin)$ obtained from $\HFminus(Y,K;\spin)$ by replacing $\bchain$ with $\bchainhat$. The underlying immersed curve $\Gamma$ is unchanged, but we now view the decorated curve $(\Gamma, \bchainhat)$ as an element of the Fukaya category of the punctured cylinder $\overline{T}_M^*$ obtained from $\overline{T}_M$ by removing the marked points. This curve represents the simplified knot Floer complex $CFK_{\sRhat}(Y,K;\spin)$ over the ring $\sRhat = \F[U,V]/(UV=0)$ along with the corresponding flip map. As noted above, since $\bchainhat$ is of local system type the decorated curve $(\Gamma, \bchainhat)$ may also be interpreted as a collection of immersed curves decorated with local systems. It is much easier to work in the $UV = 0$ setting; our strategy for constructing $\HFminus(Y,K;\spin)$ will be to first construct $\HFhat(Y,K;\spin)$ and then systematically modify $\bchain$, starting from $\bchainhat$, to capture any information lost in the $UV = 0$ quotient. 

While the invariants $\HFminus(Y,K;\spin)$ are curves in the marked cylinder $\overline{T}_M$, we will sometimes think of these curves as living in the marked torus $T_M$ by applying the covering map $p: \overline{T}_M \to T_M$. Some information may be lost under this projection, so the curves $p(\HFminus(Y,K;\spin))$ in $T_M$ should be thought of as decorated with additional grading information that amounts to specifying a lift to $\overline{T}_M$. If we do not care about the spin$^c$ decomposition, we can consider the combined curve invariant
$$\HFminus(Y,K) = \bigcup_{\spin\in\Spinc(Y)} p\left( \HFminus(Y,K;\spin) \right).$$
It may seem more natural to define $\HFminus(Y,K)$ as a curve in $\overline{T}_M$, but in the case of rationally nullhomologous knots it may not be possible to combine the curves $\HFminus(Y,K;\spin)$ in a single copy of $\overline{T}_M$ for grading reasons (see, for example, \cite[Example 57]{HRW:companion}). This is not a problem for nullhomologous knots, but with this issue in mind,  and following the convention of \cite{HRW:companion}, we instead combine the projections and define $\HFminus(Y,K)$ to be a collection of decorated curves in the marked torus $T_M$. The simplified invariant $\HFhat(Y,K)$ is defined similarly as the union of the projections of the simplified curves $\HFhat(Y,K;\spin)$.

\subsection{Surgery formulas}
One of the reasons knot Floer homology has been such a valuable tool is its close connection with Heegaard Floer invariants for closed 3-manifolds. Given a closed 3-manifold $Y$, the Heegaard Floer homology $\HFminus(Y)$ is a graded module over $\F[W]$. Note that our convention throughout the paper will be to use $W$ as the formal variable for Heegaard Floer invariants of closed 3-manifolds and for Floer homology in singly marked surfaces, while $U$ and $V$ are used when two marked points are involved; we also use $W$ in the doubly marked setting to represent the product $UV$. Given a nullhomologous knot $K \subset Y$, Ozsv{\'a}th and Szab{\'o} gave a surgery formula describing $\HFminus(Y_{p/q})$ in terms of the knot Floer complex (and flip maps) associated with $K$. This surgery formula admits a particularly nice description in terms of immersed curves: it amounts to taking the Floer homology with a curve of slope $p/q$ in the marked torus.

\begin{theorem}[Surgery Formula]\label{thm:surgery-formula}
For a nullhomologous knot $K \subset Y$ and rational slope $p/q$, there is an isomorphism of relatively graded $\F[W]$-modules
$$\HFminus(Y_{p/q}(K)) \cong \mathcal{HF}(  \HFminus(Y, K), \ell_{p/q} ),$$
where $\mathcal{HF}$ on the right side denotes Lagrangian Floer homology in the marked torus $T_M$ and $\ell_{p/q}$ is a simple closed curve homotopic to $p\mu + q\lambda$. 
\end{theorem}

For integer slopes, the surgery formula was enhanced in \cite{HeddenLevine:surgery} to compute the knot Floer complex of the dual knot $K^*$ in the surgery $Y_n(K)$. This enhancement can also be described in terms of Floer homology of curves by adding an additional marked point. Let $T_M^{z,w}$ denote the doubly marked torus obtained from $T_M$ by adding a second marked point $z$ just next to the existing marked point $w$ and let $\ell^*_{p/q}$ be a simple closed curve of slope $p/q$ that crosses the short arc connecting $z$ and $w$ exactly once. The decorated curve $\HFminus(Y,K)$ in the singly marked torus $T_M$ can also be viewed as a curve in the doubly marked torus $T_M^{z,w}$ that is disjoint from the short arc connecting $z$ to $w$. Floer homology of (decorated) curves in the doubly marked surface $T_M^{z,w}$ gives a bigraded complex over $\sRminus$, and we have the following:

\begin{theorem}[Surgery Formula for Dual Knots]\label{thm:surgery-formula-dual-knots}
For a nullhomologous knot $K \subset Y$ and an integer slope $n$, $\mathcal{HF}(  \HFminus(Y, K), \ell^*_{p/q} )$ is chain homotopy equivalent to $CFK_{\sRminus}(Y_n(K), K^*)$ as relatively bigraded complexes over $\sRminus$.
\end{theorem}

Though it is suppressed from the statements above, the equivalences in Theorem \ref{thm:surgery-formula} and \ref{thm:surgery-formula-dual-knots} also recover the spin$^c$ decomposition of $\HFminus(Y_{p/q}(K))$ and $CFK_{\sRminus}(Y_n(K), K^*)$, respectively, where the spin$^c$ decomposition on the corresponding Floer homology of curves can be defined by considering the spin$^c$ decomposition on $\HFminus(Y,K)$ as well as appropriate lifts to $\overline{T}_M$; for more details, see Section \ref{sec:surgery}.

Each theorem has a weaker form obtained by setting $W = 0$ in Theorem \ref{thm:surgery-formula} and setting $UV=0$ in  Theorem \ref{thm:surgery-formula-dual-knots}. In both cases we can replace the curve invariants $\HFminus(Y,K)$ with the simplified invariants $\HFhat(Y,K)$. In Theorem \ref{thm:surgery-formula} we take Floer homology in the punctured torus $T_M^*$ rather than in the marked torus $T_M$, thus ignoring disks that cover the marked point, and we recover the $\F$-vector space $\HFhat$ rather than the $\F[W]$-vector space $\HFminus$. In Theorem \ref{thm:surgery-formula-dual-knots} we still take Floer homology in the doubly marked torus but we can ignore disks that cover both marked points, and we recover the $UV=0$ knot Floer complex $CFK_{\sRhat}(Y_n(K), K^*)$.

\subsection{Relationship to Bordered Floer homology and related work}\label{sec:intro-bordered}

This paper is inspired by earlier work of the author with Rasmussen and Watson realizing bordered Heegaard Floer invariants as decorated immersed curves \cite{HRW, HRW:companion}. Given a 3-manifold $M$ with torus boundary and a parametrization $(\alpha, \beta)$ of the boundary, bordered Floer homology associates a type D structure $\CFD(M,\alpha, \beta)$ over a particular algebra $\Alg$. In \cite{HRW} a structure theorem was given for these algebraic objects, showing that $\CFD(M,\alpha,\beta)$ is equivalent to a collection of immersed curves decorated with local systems in the punctured torus $\partial M \setminus z$ for some basepoint $z$; this decorated immersed multicurve is denoted $\HFhat(M)$. A pairing theorem also shows that if $M_1$ and $M_2$ are two manifolds with torus boundary, $\HFhat(M_1 \cup M_2)$ can be obtained from the Floer homology of the corresponding curves in the gluing torus $-\partial M_1= \partial M_2$ (with a puncture at $z_1 = z_2$).

When $Y = S^3$ and $\F = \Z/2\Z$, it is known that $\CFD(M, \alpha, \beta)$ is equivalent to the $UV=0$ knot Floer complex $CFK_{\sRhat}(S^3, K)$ by an algorithm of Lipshitz, Ozsv{\'a}th, and Thurston \cite[Chapter 11]{LOT:bordered}. In this case, it was observed in \cite[Section 4]{HRW:companion} that the immersed curve invariants $\HFhat(M)$ representing bordered Floer homology also recover the associated graded of the knot Floer complex (i.e. the $V = 0$ quotient of $CFK_{\sRminus}$) by taking Floer homology with the meridian in a doubly pointed torus, and conversely that the curves can be constructed directly from the knot Floer complex given a suitably nice basis. The results in this paper generalize those observations to more general knots. Using the relationship between $CFK_{\sRhat}(S^3, K)$ and the bordered invariant of the knot complement $M$, it is straightforward to check that the curves $\HFhat(S^3, K)$ defined in this paper are precisely the same as the decorated curves $\HFhat(M)$ constructed in \cite{HRW}. We expect that this is true more generally:

\begin{conjecture}\label{conj:same-as-CFD}
For any knot $K \subset Y$ with complement $M = Y \setminus \nu(K)$, the decorated curves $\HFhat(Y, K)$ in the punctured torus $T_M^*$ agree with the curves $\HFhat(M)$ defined in \cite{HRW}; in particular, they are invariants of the knot complement $M$.
\end{conjecture}

If this conjecture holds, the simplified $W = 0$ version of Theorem \ref{thm:surgery-formula} can be viewed as a special case of the immersed curve pairing theorem for bordered invariants.

Immersed curve invariants for knots in $S^3$ were instrumental in the authors previous work on the cosmetic surgery conjecture \cite{Hanselman:cosmetic}. That work used the invariants defined in \cite{HRW}, making use of the equivalence between bordered Floer homology and knot Floer homology in this case; in particular, the surgery formula used was a consequence of the bordered pairing theorem in \cite{HRW}. However, \cite{Hanselman:cosmetic} required slightly more than the bordered approach to immersed curves could offer, since it was important to understand the $d$-invariants of Dehn surgeries but the bordered pairing theorem does not see either minus information or absolute gradings. In fact, the results in \cite{Hanselman:cosmetic} rely in a small way on the proof of Theorem \ref{thm:surgery-formula} presented in this paper, which relates the Floer homology of the relevant curves to the mapping cone formula rather than invoking the bordered pairing theorem. The minus version of Theorem \ref{thm:surgery-formula} was not needed, but the identification with the mapping cone formula was used to identify the distinguished generator determining the $d$-invariant. A sketch of this proof, in the $UV=0$ setting, was given but the detail have not appeared until now.

The construction of immersed curve invariants for bordered Floer homology in \cite{HRW} was based on an algebraic structure theorem for type D structures over the torus algebra. This is a special case of a more general structure theorem due to Haiden, Katzarkov, and Kontsevich \cite{HKK}. This result for other surfaces can also be recovered using the more constructive proof method from \cite{HRW}; this is worked out for arbitrary surfaces in \cite{KWZ}. We use the same core proof to obtain the $UV=0$ simplifications of the results in this paper. We repeat this main argument, adapted to the setting of bigraded complexes, in an effort to make the paper more self-contained and also to set up the argument with a view toward generalizing to the minus setting. But we point out that the structure theorem for complexes over $\sRhat$ also follows from the more general case. We especially wish to point out a close connection between the constructions in this paper and the work of Kotelskiy, Watson, and Zibrowius in \cite{KWZ:mnemonic}, which specifically applies the structure theorem to type D structures over the algebra $\sRhat$ (such type D structures are equivalent to complexes over $\sRhat$). It is shown that these structures are equivalent to immersed curves with local systems in the doubly marked disk. The $UV = 0$ version of Theorem \ref{thm:curve-invariant-for-complex} for bigraded complexes (ignoring flip maps) is equivalent to Theorem 1 in \cite{KWZ:mnemonic}. To relate these results, we remove the marked points from the infinite strip and project our decorated curves to the quotient by the vector $(0,1)$. This punctured cylinder plays the role of the doubly marked disk, where we have interchanged the roles of boundary components and marked points (note that in \cite{KWZ:mnemonic}, curves avoid the boundary and non-compact curves approach the punctures).

While the hat-type curve invariants for knots defined in this paper are parallel to the curve invariants defined in \cite{HRW} and \cite{KWZ:mnemonic}, the minus-type curve invariants are fundamentally new. This is because the curve invariants in \cite{HRW} are constructed using bordered Heegaard Floer homology, and until recently this was defined only as a hat-theory. When this project first began one motivation was to use knot Floer homology to construct immersed curve invariants for manifolds with torus boundary in order to bypass the reliance on bordered Floer homology and access minus information. This allowed us to determine what extra decorations would be needed to enhance the immersed curves from \cite{HRW} without working from a minus bordered invariant. Very recently Lipshitz, Ozsv{\'a}th, and Thurston have extended their construction of bordered Floer homology to a minus-type invariant for manifolds with torus boundary \cite{LOT:minus}; this invariant takes the form of a module over a particular weighted $A_\infty$ algebra. The results of this paper suggest that the algebraic objects defined in \cite{LOT:minus} can also be represented by immersed curves decorated with bounding chains in the marked torus. We hope to explore this and the connection between the curves constructed in this paper and those arising from minus bordered invariants in the future.

Before the minus extension of $\CFD$ appeared in \cite{LOT:minus}, another approach to defining minus Heegaard Floer invariants for manifolds with torus boundary was given by Zemke \cite{Zemke:bordered-minus}. This approach also avoids relying on bordered Floer invariants by using knot (or link) Floer homology along with auxiliary data (the link surgery formula, which contains information about flip maps) to construct an invariant. In this way, the immersed curve invariants in this paper should be closely related to the invariants defined in \cite{Zemke:bordered-minus} (in the case of a knot rather than a link), but those invariants are defined algebraically as a type D module over some algebra. This paper was developed independently of Zemke's work, but it should be the case that the decorated immersed curves described in this paper provide a geometric interpretation for the algebraic invariants defined in \cite{Zemke:bordered-minus}; exploring this connection concretely is another goal for future work.

The curves $\HFminus(Y,K)$ constructed in this paper are invariants of knots, but they provide a possible path to defining minus type bordered Floer invariants for manifolds with torus boundary. Generalizing Conjecture \ref{conj:same-as-CFD}, we expect that the immersed curves for a knot $K$ are in fact an invariant of the knot complement:

\begin{conjecture}\label{conj:invariant-of-complement}
For any knot $K \subset Y$ with complement $M = Y \setminus \nu(K)$, the decorated curves $\HFminus(Y, K)$ in the marked torus $T_M$ are an invariant of $M$.
\end{conjecture}

For any manifold $M$ with torus boundary, we can choose some meridian $\mu$ and view $M$ as the complement of $K_\mu \subset Y_\mu$, where $Y_\mu$ is the Dehn filling of $M$ along $\mu$ and $K_\mu$ is the core of the filling torus. We can then construct the decorated curve $\HFminus(Y_\mu, K_\mu)$, and Conjecture \ref{conj:invariant-of-complement} asserts that the result does not depend on the choice of $\mu$. If this is true, we could denote this curve $\HFminus(M)$.

We note that if Conjecture \ref{conj:invariant-of-complement} is true, it provides a simple way to recover the knot Floer complex and flip maps associated to the dual knot for any Dehn surgery on a knot $K \subset Y$ from the knot Floer complex and flip maps associated with $K$: we simply use the knot Floer data to construct the curve $\HFminus(M) = \HFminus(Y, K)$ and then read off a complex with flip maps from this in the usual way but with the Dehn filling slope in place of the meridian. Conversely, if we knew the knot Floer complex and flip map associated to a dual knot agreed with that predicted by this procedure, then the decorated immersed curve representing the dual knot would be precisely the decorated curve representing the original knot, proving Conjecture \ref{conj:invariant-of-complement}. Theorem \ref{thm:surgery-formula-dual-knots} can be interpreted as saying that for integer surgery, the surgery formula for the dual knot Floer complex predicted by Conjecture \ref{conj:invariant-of-complement} is correct, giving evidence for Conjecture \ref{conj:invariant-of-complement}. To prove the conjecture, we would also need to find a surgery formula for the flip maps associated to a dual knot and check that it agrees with the one predicted by immersed curves.

If the immersed curves defined using knot Floer homology are in fact bordered invariants, we would expect to have a general pairing theorem:
\begin{conjecture}\label{conj:minus-pairing}
If $M_1$ and $M_2$ are manifolds with torus boundary and $\phi:\partial M_1 \to \partial M_2$ is an orientation reversing gluing map, then
$$\HFminus(M_1 \cup_\phi M_2) \cong \mathcal{HF}( \phi(\HFminus(M_1)), \HFminus(M_2)). $$ 
\end{conjecture}
Theorem \ref{thm:surgery-formula} is a special case of this where $M_2$ is a solid torus. To prove this more generally, assuming the curves $\HFminus(Y,K)$ are defined for all rationally nullhomologous knots and assuming Conjecture \ref{conj:invariant-of-complement}, we could choose any slope on the gluing torus in $M_1 \cup_\phi M_2$ and use it as a meridian on either side to view the gluing as a splice of two knot complements. The Floer homology $\mathcal{HF}( \phi(\HFminus(M_1)), \HFminus(M_2))$ can then be identified with the shifted pairing (defined in Section \ref{sec:shifted-pairing}) of the two knot Floer complexes, and this in turn agrees with the mapping cone complex for some integer surgery on the tensor product of the two knot Floer complexes equipped with the tensor product of the two flip maps. Provided flip maps behave in the obvious under connected sums, this is the mapping cone complex for the integer surgery on the connected sum of the two knots, which standard arguments show is equivalent to the splice of the not complement. We aim to carry out this strategy in future work. We remark again that this most likely amounts to a geometric reinterpretation of the pairing theorem for Zemke's algebraic invariants, but it would be enlightening to have a curve based proof of this.

\subsection{An example}\label{sec:intro-example}

We will demonstrate our key results with an example. Let $K \subset S^3$ be the left handed trefoil with meridian $\mu$ and Seifert longitude $\lambda$, and let $M$ denote the complement.

The knot Floer complex $CFK_{\sRminus}(S^3, K)$ has three generators $a$, $b$, and $c$, and differential
$$\partial(a) = -Vb, \quad \partial(b) = 0, \quad \text{ and } \quad \partial(c) = Ub.$$
By the construction in Section \ref{sec:simple-curves} and Section \ref{sec:enhanced-curves} this bigraded complex is represented by the immersed arc in the marked strip $\strip$ shown in Figure \ref{fig:intro-example}(a); the bounding chain is trivial (as it must be since the arc has no self-intersection points). Note that the complex is recovered from this curve by taking Floer homology with the vertical line $\mu$ in the doubly marked strip in which we replace each marked point by a $z$ marked point just to the left of $\mu$ and a $w$ marked point just to the right of $\mu$. In particular there is a bigon on the right side of $\mu$ from $c$ to $b$ covering the right side of a marked point once, contributing $Ub$ to $\partial c$, and there is a bigon from $a$ to $b$ covering the left side of a marked point once and contributing $-Vb$ to $\partial a$ (the sign convention in this case records that the orientation on $\HFminus(S^3, K)$ opposes the boundary orientation of the latter bigon).

The horizontal and vertical homology of $CFK_{\sRminus}(S^3, K)$ are both one dimensional (generated by $a$ and $c$, respectively) and the flip isomorphism associated to $K$ simply takes $a$ to $c$. The decorated immersed curve $\Gamma$ in the cylinder $\cylinder$ representing $CFK_{\sRminus}(S^3, K)$ with this flip map is obtained by gluing the opposite sides of $\strip$ and identifying the endpoints of the immersed arc; the bounding chain is still trivial. After identifying the cylinder $\cylinder$ with $\overline{T}_M$, taking the horizontal direction to $\lambda$ and the vertical direction to $\mu$, the immersed curve $\Gamma$ is the invariant $\HFminus(S^3, K; \spin)$, where $\spin$ is the unique spin$^c$ structure on $S^3$; the projection to the marked torus $T_M$ is denoted $\HFminus(S^3, K)$. The hat version of the curve, which only represents the complex and flip map modulo $UV$, is obtained by restricting the bounding chain to degree zero intersection points; since the bounding chain is trivial, in this case $\HFminus(S^3, K)$ and $\HFhat(S^3, K)$ are the same. Note that $\HFhat(S^3, K)$ agrees with the curve $\HFhat(M)$ defined in \cite{HRW}.

\begin{figure}
\labellist
  \footnotesize
  \pinlabel {{\color{blue} $\mu$}} at 36 120
  \pinlabel {$a$} at 27 113
  \pinlabel {$b$} at 27 61
  \pinlabel {$c$} at 27 15

  \pinlabel {$a$} at 151 86
  \pinlabel {$b$} at 135 71
  \pinlabel {$c$} at 119 42
  
  \pinlabel {$a$} at 364 106
  \pinlabel {$b$} at 339 80
  \pinlabel {$c$} at 321 64
  \pinlabel {$d$} at 304 48
  \pinlabel {$e$} at 274 17
  
  \normalsize
  \pinlabel {$(a)$} at 31 -10
  \pinlabel {$(b)$} at 135 -10
  \pinlabel {$(c)$} at 317 -10
\endlabellist
\includegraphics[scale = .9]{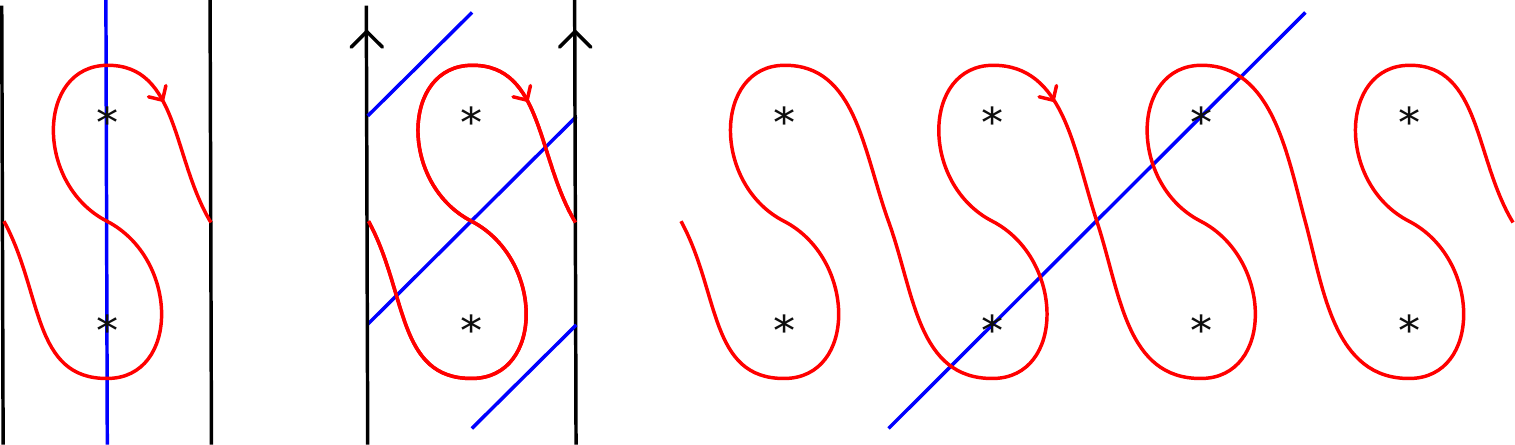}
\vspace{3 mm}
\caption{(a) A curve in $\strip$ representing the knot Floer complex of the left handed trefoil---identifying the edges of the strip also gives a curve in the cylinder $\cylinder$ representing the knot Floer complex and the flip isomorphism; (b) the Floer complex with a curve of slope 1 recovers $HF^-(Y_1(K))$; (c) the Floer complex with a line of slope 1 through the marked points, where we interpret each marked point as a pair marked points $z$ and $w$ on the left and right, respectively, gives the knot Floer complex of the dual knot.}
\label{fig:intro-example}
\end{figure}
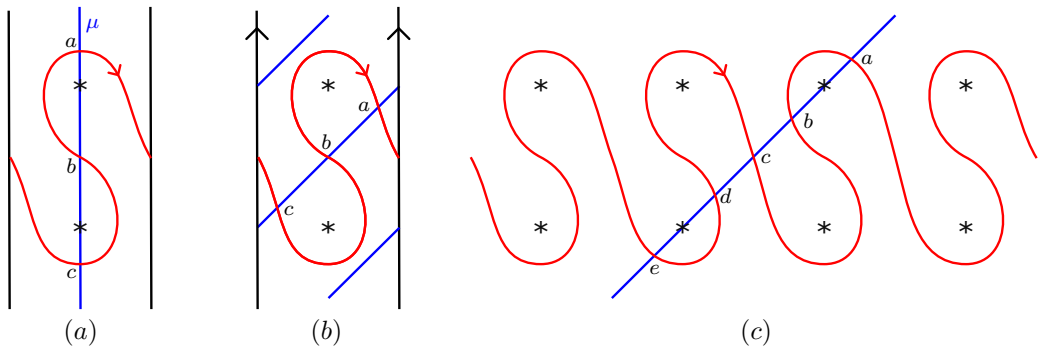

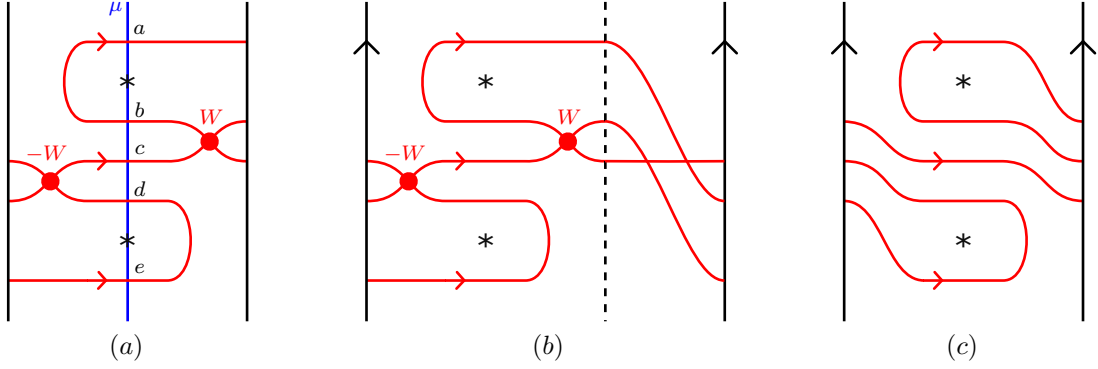
\begin{figure}
\labellist
  \footnotesize
  \pinlabel {{\color{blue} $\mu$}} at 41 118
  \pinlabel {$a$} at 50 110
  \pinlabel {$b$} at 50 80
  \pinlabel {$c$} at 50 65
  \pinlabel {$d$} at 50 50
  \pinlabel {$e$} at 50 20
 
 {\color{red}   
  \pinlabel {$W$} at 77 77
  \pinlabel {$-W$} at 15 63
  \pinlabel {$W$} at 212 77
  \pinlabel {$-W$} at 150 63
  }

  \normalsize
  \pinlabel {$(a)$} at 45 -10
  \pinlabel {$(b)$} at 204 -10
  \pinlabel {$(c)$} at 360 -10
\endlabellist
\includegraphics[scale = 1]{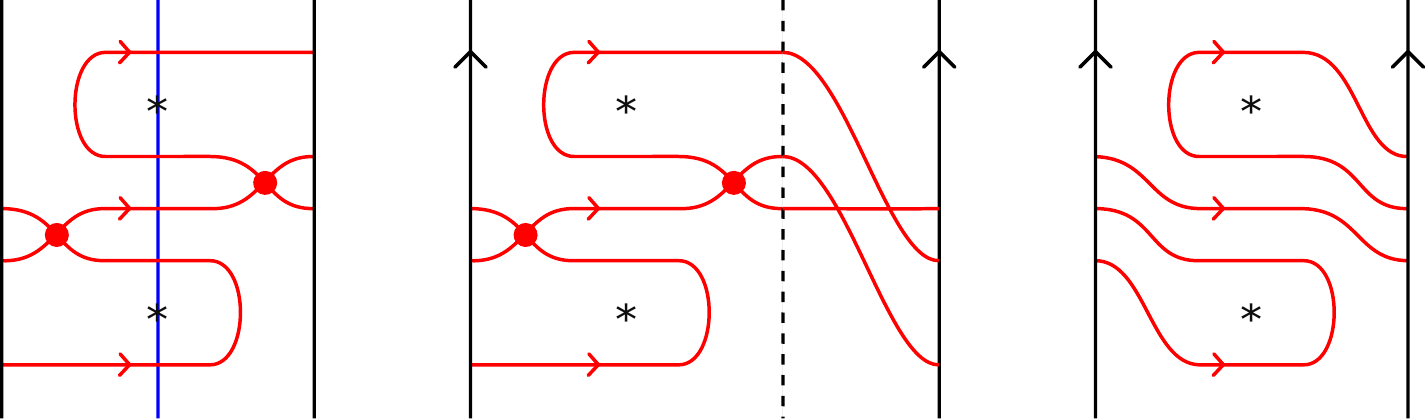}
\vspace{3 mm}
\caption{(a) A decorated immersed curve in $\strip$ representing the knot Floer complex of $K_1 \subset Y_1$; (b) adding a thin strip of arcs encoding the flip isomrophism and then identifying opposite edges of the strip produces a decorated curve in the cylinder $\cylinder$ representing the complex and the flip isomorphism; (c) the curve in $\cylinder$ after a homotopy.}
\label{fig:intro-example2}
\end{figure}

We next consider the manifold $Y = S^3_{1}(K)$ and the dual knot $K^*$ in this surgery. By Theorem \ref{thm:surgery-formula}, $\HFminus(Y)$ is the Floer homology of $\HFminus(S^3, K)$ with a curve of slope 1 in the marked torus $T_M$. This is shown (in the covering space $\overline{T}_M$) in Figure \ref{fig:intro-example}(b). There are 3 generators, $a$, $b$, and $c$, with differential
$$\partial(a) = Wb, \quad \partial(b) = 0, \quad \text{ and } \quad \partial(c) = -Wb,$$
so the homology is isomorphic to $\F[W] \oplus \F$.  
By the refinement of the surgery formula, Theorem \ref{thm:surgery-formula-dual-knots}, the complex $CFK_{\sRminus}(Y, K^*)$ is given by the Floer homology with a line of slope 1 in $T_M$ that passes through the marked point, after we replace the marked point with a $z$ marked point just to the left and a $w$ marked point just to the right. These curves are shown in the covering space $\widetilde{T}_M$ in Figure \ref{fig:intro-example}(c). There are 5 generators, $a$, $b$, $c$, $d$, and $e$, with differential
$$\partial(a) = Ub, \quad \partial(b) = 0, \quad \partial(c) = -UVb + UVd, \quad \partial(d) = 0, \quad \text{ and } \quad \partial(e) = -Vd.$$
Following the construction in Sections \ref{sec:simple-curves} and \ref{sec:enhanced-curves}, we can represent this complex by the immersed multicurve in the strip $\strip$ shown in Figure \ref{fig:intro-example2}(a), decorated with a bounding chain $\bchain$, where $\bchain$ is the linear combination of the two self-interesection points with coefficients as shown in the figure. Note that to recover the complex we count generalized bigons where, which are allowed to make left turns at self-intersection points with nonzero coefficient in $\bchain$ and which are counted according to the weights associated with all such left-turns. For example, there is a generalized bigon from $c$ to $b$ that contributes $Wb$ to $\partial c$.

To turn the immersed curve in the strip $\strip$ from Figure \ref{fig:intro-example2}(a) into an immersed curve in the cylinder $\cylinder$, we need to use the flip isomorphism associated with $K^* \subset Y$, which now carries interesting information because the horizontal and vertical homology both have rank 3. We do not have a surgery formula for the flip isomorphism associated with the dual knot in a surgery, and in this case the flip isomorphism is not uniquely determined by the complex, but we can deduce the correct flip isomorphism using gradings and a surgery argument (for details see Example \ref{ex:1-surgery-on-LHT}). The horizontal complex is generated by $\{a,b,c\}$ while the vertical complex is generated by $\{c,d,e\}$, and (ignoring powers of $U$ and $V$, which are determined by gradings) the flip isomorphism takes $a$ to $c$, $b$ to $d$, and $c$ to $e$. Gluing the sides of the strip $\strip$ after inserting arcs to identify the endpoints according to the flip isomorphism produces the decorated curve in $\cylinder$ in Figure \ref{fig:intro-example2}(b). This can be simplified slightly by a homotopy to give the curve in Figure \ref{fig:intro-example2}(c). We define what we mean by homotopy of immersed curves decorated with bounding chains in Section \ref{sec:invariance-of-Floer-homology}. Note that here we homotope the underlying curve to remove two pairs of intersection points; this is allowed in this case, even though in each pair one intersection point has nontrivial coefficient in $\bchain$ following move $(j)$ in Figure \ref{fig:invariance-moves}. Identifying $\cylinder$ with $\overline{T}_M$, this decorated curve is $\HFminus(Y, K^*;\spin)$, where $\spin$ is the unique spin$^c$ structure on $Y$, and the projection of this curve to $T_M$ is $\HFminus(Y, K^*)$.

We remark that, as curves in $\overline{T}_M$, the decorated curves $\HFminus(S^3, K)$ and $\HFminus(Y, K^*)$ actually agree. They appear differently in the cylinder $\cylinder$ only because we use different parametrizations to identify $\cylinder$ with $\overline{T}_M$. Indeed, starting with the curve in Figure \ref{fig:intro-example}(c) we can apply the lift to $\overline{T}_M$ of a Dehn twist about $\lambda$ in $T_M$ that takes the meridian $\mu^* = \lambda + \mu$ of the dual knot to the vertical direction, and the resulting curve is the one in Figure \ref{fig:intro-example2}(c). This is consistent with Conjecture \ref{conj:invariant-of-complement}.

\subsection{Organization}

We begin by briefly reviewing knot Floer homology and algebraic preliminaries for bigraded complexes over $\sRminus$ in Section \ref{sec:knot-floer}. In Section \ref{sec:floer-theory} we define Floer homology of decorated immersed curves in marked surfaces, and in Section \ref{sec:train-tracks} we discuss an alternate interpretation of these decorated curves in terms of immersed train tracks. The construction of Floer homology is completely combinatorial, and Sections \ref{sec:floer-theory} and \ref{sec:train-tracks} may be of independent interest since this construction is more accessible than other treatments of immersed Lagrangian Floer theory.  In Section \ref{sec:curves-in-strip} we show that a decorated curve in the marked strip $\strip$ encodes a bigraded complex, and a decorated curve in the marked cylinder $\cylinder$ encodes a bigraded complex equipped with a flip map. We also observe that any bigraded complex or any bigraded complex with flip map can be represented by some (not necessarily nice) decorated multicurve in $\strip$ or $\cylinder$. In Section \ref{sec:properties-of-simple-curves} we discuss what it would mean for such a representative to be nice, and prove some properties of the complexes coming from curves in a suitably nice position. In Section \ref{sec:simple-curves} and \ref{sec:flip-maps} we restrict to the $UV = 0$ setting and construct, in Section \ref{sec:simple-curves}, a nice representative in $\strip$ for any bigraded complex over $\sRhat$ and, in Section \ref{sec:flip-maps}, a nice representative in $\cylinder$ for any such complex equipped with a flip map. This proves the existence part of a $UV=0$ version of Theorem \ref{thm:curve-invariant-for-complex}. In Section \ref{sec:enhanced-curves} we show that for a complex over $\sRminus$ and a flip map, the representative of the $UV=0$ quotient can be enhanced, without changing the underlying curve, by modifying the bounding chain decoration on the curve. This completes the existence part of Theorem \ref{thm:curve-invariant-for-complex}. In Section \ref{sec:morphisms} we turn our at attention to the Floer homology of two decorated curves in the marked strip or cylinder, relating this geometric pairing to morphisms of complexes (for curves in $\strip$) or to a more complicated algebraic pairing we define for complexes with flip maps (for curves in $\cylinder$). Using the invariance of this algebraic pairing under homotopy equivalence of the complexes, we prove the uniqueness claim in Theorem \ref{thm:curve-invariant-for-complex}. In Section \ref{sec:surgery} we prove Theorems \ref{thm:surgery-formula} and \ref{thm:surgery-formula-dual-knots} by relating the algebraic pairings from Section \ref{sec:morphisms} to mapping cone formulas known to recover $\HFminus$ of surgeries on knots and $\CFKminus$ of dual knots in surgeries. We end with examples and some discussion about simplifying the bounding chain decoration in Section \ref{sec:examples}.

\subsection*{Acknowledgements} This project has lasted several years and has benefitted from many fruitful conversations over that time. In addition to many others, the author is especially grateful to Liam Watson, Adam Levine, Robert Lipshitz, Artem Kotelskiy, and Wenzhao Chen for helpful conversations, answering questions, and comments on earlier versions of this work.

\section{Knot Floer homology}\label{sec:knot-floer}% !TEX root = ../CFKcurvesZHS3s.tex
%complexes.tex

\subsection{Bigraded complexes over $\sRminus$} The knot Floer complex takes the form of a bigraded chain complex over $\sRminus$ (or, more precisely, a collection of such complexes). We begin by reviewing these algebraic structures and their properties. We will also define a certain notion of filtered maps between bigraded complexes.

Throughout the paper we work with coefficients in an arbitrary field $\F$ and $\sRminus$ denotes the ring $\F[U,V]$. We define a bigrading $\gr = (\gr_w, \gr_z)$ on $\sRminus$ where $\gr(U) = (-2,0)$ and $\gr(V) = (0,-2)$. The two components of the grading are called the \emph{$U$-grading} and the \emph{$V$-grading}, respectively. While the most general knot Floer invariants are defined over $\sRminus$, we can simplify the invariant by passing to certain quotients of $\sRminus$. The most common, which we denote $\sRhat$, is obtained by setting $UV = 0$. More generally, we will consider
$$\sR_n \vcentcolon= \F[U,V] / (UV)^n.$$
Note that in this notation, $\widehat\sR = \sR_1$. Finally, let $\sR^\infty$ denote $\F[U, U^{-1}, V, V^{-1}]$. Since the product $UV$ will appear frequently we will set $W = UV$ throughout the paper.

At times we will need to discuss an object that is nearly a bigraded chain complex but for which $\partial^2$ is not zero; we refer to this as a precomplex.

\begin{definition}
A \emph{bigraded precomplex over $\sR$} is a finitely generated module over $\sRminus$  with an integer bigrading $(\gr_w, \gr_z)$ such that $\gr_w$ and $\gr_z$ agree mod 2 and multiplication by $U$ and $V$ have degree $(-2,0)$ and $(0,-2)$, respectively, equipped with a linear map $\partial$ of degree $(-1,-1)$. A \emph{bigraded complex over $\sRminus$} is a bigraded precomplex over $\sRminus$ which satisfies $\partial^2 = 0$. The grading $\gr_w$ is called the \emph{Maslov grading} and will also be denoted $M$. The \emph{Alexander grading} $A$ is given by $\tfrac{1}{2}(\gr_w - \gr_z)$. Because we assume that $\gr_w$ and $\gr_z$ have the same parity, $A$ is also integral.
\end{definition}

Let $C^-$ be a bigraded complex over $\sRminus$ as described above with differential $\partial$. Let $C^\infty$ denote $C^- \otimes_{\sRminus} \sR^\infty$, the result of localizing both $U$ and $V$; the differential $\partial$ extends to $C^\infty$. We will use the term \emph{bigraded complex over $\sR^\infty$} to mean a complex $C^\infty$ obtained in this way from some $C^-$ over $\sRminus$. Let $\widehat C$ denote the complex over $\sRhat$ obtained from $C^-$ by setting $UV = 0$. For any $s \in \Z$, let $C^-|_{A=s}$ denote the subcomplex of $C^-$ with Alexander grading $s$, so that $C^- = \bigoplus_{s\in\Z} C^-|_{A=s}$, and similarly for $C^\infty|_{A=s}$ and $\widehat C|_{A=s}$. We say a basis $\{x_1, \ldots, x_n\}$ for $C^-$ over $\sRminus$ is \emph{homegeneous} if each $x_i$ lies in an Alexander graded summand $C^-|_{A=A(x_i)}$. Note that given such a basis and any $s \in \Z$, $\{V^{s-A(x_1)}x_1, \ldots, V^{s-A(x_n)}x_n\}$ is a basis for $C^\infty_s$ over $\F[W, W^{-1}]$. Any two homogenous bases $\{x_1, \ldots, x_n\}$ and $\{x'_1, \ldots, x'_n\}$  are related by a homogenous change of basis, where $x'_i = \sum_{j=1}^n c_{i,j} x_j$ for some coefficients $c_{i,j}$ in $\sRminus$ such that each nonzero term $c_{i,j} x_j$ in the sum has the same Alexander grading.  A basis for $C^-$ is \emph{reduced} if $\partial$ is trivial when $U$ and $V$ are both set to zero; it is a standard argument that any bigraded complex $C^-$ is homotopy equivalent to one which admits a reduced basis. Unless otherwise stated, all bases for bigraded complexes will be assumed to be homogeneous and reduced.

Given a basis for $C^-$, we can record the differential $\partial$ with an $n \times n$ matrix with coefficients in $\sR$, with the $(i,j)$ entry specifying the coefficient of $x_j$ in $\partial x_j$. In fact, the powers of $U$ and $V$ in each entry are determined by the bigrading change from $x_i$ to $x_j$, so if the gradings on the generators are specified then $\partial$ can be encoded by a matrix $\{d_{i,j}\}_{1\le i,j \le n}$ with coefficinets in $\F$. More precisely, this means that
$$\partial(x_i) = \sum_{j = 1}^n d_{i,j} U^{a_{i,j}} V^{b_{i,j}} x_j$$
where $a_{i,j}$ and $b_{i,j}$ are defined by 
\begin{equation}\label{eq:UV-exponents}
a_{i,j} = \frac{ \gr_w(x_j) - \gr_w(x_i) + 1 }{2}, \quad \text{ and } \quad b_{i,j} =  \frac{ \gr_z(x_j) - \gr_z(x_i) + 1 }{2}.
\end{equation}
%\begin{equation}\label{eq:UV-exponents}
%a_{i,j} = \frac{ M(x_j) - M(x_i) + 1 }{2}, \quad \text{ and } \quad b_{i,j} = A(x_i) - A(x_j) + \frac{ M(x_j) - M(x_i) + 1 }{2}.
%\end{equation}
Note that $d_{i,j}$ must be zero if $\gr_w(x_i) - \gr_w(x_j)$ is even or if $a_{i,j}$ or $b_{i,j}$ are negative. If the coefficient $d_{i,j}$ is nonzero, we say that there is an arrow from $x_i$ to $x_j$; we will say that this arrow is \emph{vertical} if $a_{i,j} = 0$ and that it is \emph{horizontal} if $b_{i,j} = 0$. This terminology comes from the fact that it is common to represent $C^-$ or $C^\infty$ in the plane with $U^a V^b x_i$ represented by a point at coordinates $(-a,-b)$ and the differential represented by arrows. We will often use the notion of arrows to refer to nonzero terms in the differential; arrows are labeled by the coefficient $d_{i,j} U^{a_{i,j}} V^{b_{i,j}}$ (or just by $d_{i,j}$ when the relevant powers of $U$ and $V$ are understood).

There are several quotient complexes of $C^-$ that will be relevant to us, which we now describe. Assume we have fixed a reduced homogeneous basis $\{x_1, \ldots, x_n\}$ for $C^-$. Let $(C^v, \partial^v)$ denote the complex obtained from $(C^-, \partial)$ by setting $V=1$; we call this the \emph{vertical complex} of $C^-$. The vertical complex is a chain complex over $\F[W]$ (note that setting $V = 1$ also means that $W = U$). It is a singly graded complex: the grading $\gr_w$ on $C^-$ descends to a grading on $C^v$, but the grading $\gr_z$ does not. The \emph{vertical homology} of $C^-$ will refer to the homology of the vertical complex, $H_* C^v$; this is a graded module over $\F[W]$. If we further set $U = 0$, the resulting graded chain complex $(\widehat C^v, \widehat\partial^v)$ is called the \emph{hat vertical complex}. Its homology $H_* \widehat C^v$, a graded vector space over $\F$, is the \emph{hat vertical homology} of $C^-$. Similarly, the \emph{horizontal complex} $(C^h, \partial^h)$ is the complex over $\F[W]$ obtained from $C^-$ by setting $U=1$ and $W=V$, with a grading inherited from $\gr_z$. The \emph{hat horizontal complex} comes from setting $V=1$ and also $U=0$. The \emph{horizontal homology} and \emph{hat horizontal homology} refer to the homologies of the respective complexes.

We are interested in choosing bases for $C^-$ which are well behaved with respect to the horizontal or vertical complexes. We say that a basis $\{x_1, \ldots, x_n\}$ of $C^-$ is \emph{vertically simplified} if for each basis element $x_i$ either $\partial(x_i) \equiv V^{b_i} x_{i+1} \pmod U$ for some $b_i\in\Z$ or $\partial(x_i) \equiv 0 \pmod U$. That is, every generator is an end of at most one vertical arrow; equivalently, every generator in the hat vertical complex has at most one arrow in or out. The generators of the vertical homology are exactly the generators with no vertical arrow in or out. Similarly, a basis for $C^-$ is \emph{horizontally simplified} if for each basis element $x_i$ either $\partial(x_i) \equiv U^{a_i} x_{i+1} \pmod V$ for some $a_i \in \Z$ or $\partial(x_i) \equiv 0 \pmod V$; that is, if each generator is an end of at most one horizontal arrow.

\begin{proposition}\label{prop:reduced-simplified-basis}
Let $C$ be a bigraded chain complex over $\sR = \F[U,V]$, where $\F$ is a field. $C$ is chain homotopy equivalent to a complex $C'$ which is reduced. Moreover, $C'$ admits a homogeneous basis which is vertically simplified. It also admits a (possibly different) homogeneous basis which is horizontally simplified.
\end{proposition}

\begin{proof}
This is essentially Proposition 11.52 of \cite{LOT:bordered}. That result assumes $\Z/2\Z$ coefficients rather than an arbitrary field $\F$ and is stated in terms of filtered complexes over $\F[U]$ rather than complexes over $\F[U,V]$ (see the notational remarks in  Section \ref{sec:old-way} below), but the proof is completely analogous.
\end{proof}

Note that while we can always pick a basis which is either horizontally or vertically simplified, there exist complexes which do not admit a single basis that is both horizontally and vertically simplified (see Example \ref{ex:nontrivial-local-system}).

%If we set $V = 1$ in the horizontal complex $(C^v, \partial^v)$ and take homology, the resulting vector space over $\F$ is called the \emph{hat vertical homology} of $C^-$. The grading $\gr_z$ on $C^-$ does not descend to a grading on the horizontal homology, since multiplying by $V$ changes $\gr_z$, but the vertical homology inherits a single grading from$\gr_w$ on $C^-$. Similarly, the \emph{hat horizontal homology} of $C^-$ is the result of setting $U=1$ in the horizontal complex and taking homology. The horizontal homology inherits a grading from $\gr_z$ in $C^-$. Sometimes when defining the vertical complex from $C^\infty$ rather than setting $U=0$ we will simply disallow negative powers of $U$; equivalently, starting from $C^-$ we localize the variable $V$ but not $U$. This will be called the \emph{minus vertical complex}. If we then set $V = 1$ and take homology, the resulting $\F[U]$-module will be called the \emph{minus vertical homology} of $C^-$.  Similarly, the \emph{minus horizontal complex} of $C$ is obtained from $C^-$ by tensoring with $\F[V,V^{-1}]$ and the \emph{minus horizontal homology} of $C^-$ is the $\F[V]$-module obtained by setting $U = 1$ in $C^-$ and taking homology.

We now consider maps between two bigraded complexes $C_1$ and $C_2$, or more precisely between the localized versions $C^\infty_1$ and $C^\infty_2$. We will be interested in chain maps which interchange the roles of $U$ and $V$. We say that a chain map $\Psi:C^\infty_1 \to C^\infty_2$ is a \emph{skew $\sRminus$-module homomorphism} if $\Psi$ becomes a homomorphism of $\sRminus$-modules if the roles of $U$ and $V$ are exchanged in the action of $\sRminus$ on $C_2$; in particular, $\Psi(Ux) = V\Psi(x)$ and $\Psi(Vx) = U\Psi(x)$ for any $x$ in $C$. We say that such a map has \emph{skew degree} $(a,b)$ if it it interchanges the two gradings and then raises $\gr_w$ by $a$ and $\gr_z$ by $b$, so that $\gr_w(\Psi(x)) = \gr_z(x) + a$ and $\gr_z(\Psi(x)) = \gr_w(x) + b$.

Note that a bigraded complex over $\sR^\infty$ carries two natural filtrations given by the negative exponent of $U$ or of $V$. More precisely, the \emph{$U$-filtration} is defined so that for any generator $x$ of $C_1$ the element $U^a V^b x$ is at filtration level $i = -a$, and the \emph{$V$-filtration} is defined so that $U^a V^b x$ is at filtration level $j = -b$. We say that a skew $\sRminus$-module homomorphism $\Psi:C^\infty_1 \to C^\infty_2$ is \emph{flip-filtered} if it is filtered with respect to the $V$-filtration on $C_1$ and the $U$-filtration on $C_2$; equivalently, $\Psi$ takes each generator $x$ of $C_1$ to a sum of terms of the form $cU^aV^b y$ where $y$ is a generator of $C_2$, $c$ is a nonzero element of $\F$, and $a \ge 0$. Similarly, we say $\Psi$ is \emph{reverse flip-filtered} if it is filtered with respect to the $U$-filtration on $C_1$ and the $V$-filtration on $C_2$. We say that $\Psi$ is a flip-filtered chain homotopy equivalence if it is flip-filtered and there exists a reverse flip-filtered map $\bar\Psi:C_2 \to C_1$ such that $\bar\Psi\circ\Psi$ and $\Psi\circ\bar\Psi$ are both filtered chain homotopic to the respective identity maps (with respect to the $V$-filtration on $C_1$ and the $U$-filtration on $C_2$). Given two flip-filtered maps $\Psi_1$ and $\Psi_2$ from $C_1$ to $C_2$, a \emph{flip-filtered chain homotopy} is a skew-$\sRminus$-module homomorphism $H:C\to C'$ such that $\Psi_1 -\Psi_2 = H\circ\partial + \partial \circ H$ and $H$ is filtered with respect to the $V$-filtration on $C_1$ and the $U$-filtration on $C_2$.

The flip-filtered maps we will consider exchange the gradings; that is, they will have skew degree $(0,0)$. %Note that $s$ must be congruent to $A(C_1) + A(C_2)$ modulo 1, since for any $x$ in $C_1$
%\begin{align*}
%2s &= \gr_w(x) - \gr_z(\Psi(x)) = \left[ 2A(x) + \gr_z(x)\right] - \left[ \gr_w(\Psi(x)) - 2A(\Psi(x))\right] \\
%&= 2A(x) + 2A(\Psi(x)) + \left[\gr_z(x) - \gr_w(\Psi(x))\right] = 2\left[A(x) + A(\Psi(x))\right].
%\end{align*}
In general for each bigraded complex $C_i$ we could fix an Alexander grading shift $s_i$ in $\Z$ and then define a flip map $\Psi_s$ of skew degree $(0,-2s)$ where $s = s_1 + s_2$. However, it is enough to consider one (arbitrary) shift on each complex, since multiplying a flip-filtered map of skew degree $(0,-2s)$ by $V$ gives a flip-filtered map of skew degree $(0, -2s-2)$ and so the maps associated with different choices of shifts carry equivalent information. Thus we will set $s_1 = s_2 = 0$.

Given bases $\{x_1, \ldots, x_m\}$ for $C_1$ and $\{y_1, \ldots, y_n\}$ for $C_2$, a skew $\sRminus$-module homomorphism $\Psi: C_1 \to C_2$ of skew degree $(0,0)$ is specified by a collection of coefficients $c_{i,j}$ for each $1\le i\le m$ and $1\le j\le n$ such that $\gr_w(y_j) - \gr_z(x_i)$ is even (we take $c_{i,j}$ to be 0 for other pairs). In particular,
$$\Psi(x_i) = \sum_j c_{i,j} U^{\frac{\gr_w(y_j) - \gr_z(x_i)}{2}} V^{\frac{\gr_z(y_j)-\gr_w(x_i)}{2}} y_j.$$
If we further assume $\Psi$ is flip-filtered then $c_{i,j}$ is only nonzero if $\gr_w(y_j) \ge \gr_z(x_i)$, since the exponent on $U$ must be nonnegative. If we have a nice basis for $C$, then up to homotopy we can assume $\Psi$ has an even simpler form:
\begin{proposition}\label{prop:flip-map-horiz-basis}
Let $\{x_1, \ldots, x_{2k}, x_{2k+1}, \ldots, x_m\}$ be a horizontally simplified basis for $C_1$ such that $\widehat\partial^h(x_{2i-1}) = x_{2i}$ for $1 \le i \le k$ and $\widehat\partial^h$ is zero on all other generators, where $\widehat\partial^h$ is the differential on the hat horizontal complex. 
If $\Psi:C_1\to C_2$ is a flip-filtered chain map, then $\Psi$ is flip-filtered chain homotopic to another such map $\Psi'$ for which $\Psi'(x_{2i-1}) = 0$ for $1\le i \le k$. Moreover, for each $1\le i\le k$, $\Psi'(x_{2i})$ is trivial mod $U$ and is determined by the values of $\Psi'(x_\ell)$ for $2k+1 \le \ell \le m$.
\end{proposition}
\begin{proof}
We will modify $\Psi = \Psi_0$ to be zero on the generators $x_{2i-1}$ one at a time, defining a sequence of flip-filtered chain homotopies $H_i$ from $\Psi_{i-1}$ to $\Psi_i$ such that $\Psi_i(x_{2i-1}) = 0$ and $\Psi_i(x_{2j-1}) = \Psi_{i-1}(x_{2j-1}) = 0$ for all $j<i$. Then $\Psi_k = \Psi'$ is the desired map. Letting $a_i$ be the length of the horizontal arrow from $x_{2i-1}$ to $x_{2i}$, we define the homotopy $H_i: C_1 \to C_2$ by setting $H(x_{2i}) = -U^{-a_i} \Psi_{i-1}(x_{2i-1})$ and $H_i = 0$ on all other generators. Then
$$(\Psi_i - \Psi_{i-1})(x_{2i-1}) = H_i \circ \partial(x_{2i-1}) + \partial \circ H_i(x_{2i-1}) = -U^{a_i} U^{-a_i} \Psi_{i-1}(x_{2i-1})$$
and so $\Psi_{i}(x_{2i-1}) = 0.$ The final claim follows from the fact that $\Psi$ is a chain map.
\end{proof}
%Similarly let $\{y_1, \ldots, y_{2\ell}, y_{2\ell + 1\}, \ldots, y_n\}$ be a vertically simplified basis for $C_2$ for which $\partial^v(y_{2i-1}) = V^{b_i} y_{2i}$ for $1 \le i \le k$ and $\partial^v$ is zero on all other generators.

When $\Psi$ is a flip-filtered chain homotopy equivalence it induces a homotopy equivalence on each filtration level, using the $V$-filtration on $C^\infty_1$ and the $U$-filtration on $C^\infty_2$. In particular, considering the 0 filtration levels, it gives a chain homotopy equivalence between $C_1 \otimes_{\sRminus} \F[U,U^{-1},V]$ and $C_2 \otimes_{\sRminus} \F[U,V,V^{-1}]$. Setting $U=1$ in $C_1 \otimes_{\sRminus} \F[U,U^{-1},V]$ and $V=1$ in $C_2 \otimes_{\sRminus} \F[U,V,V^{-1}]$ gives a chain homotopy equivalence from the horizontal complex of $C_1$ to the vertical complex of $C_2$, and this induces an isomorphism $\Psi_*$ from the horizontal homology of $C_1$ to the vertical homology of $C_2$. We call such an isomorphism a \emph{flip isomorphism}. When $\Psi$ also has skew-degree $(0,0)$, the induced isomorphism $\Psi^*$ is grading preserving with respect to $\gr_z$ on $C_1$ and $\gr_w$ on $C_2$.  By setting $V=0$ in $C^\infty_1$ and $U=0$ in $C^\infty_2$, we also have a chain homotpy equivalence from the hat horizontal complex of $C_1$ to the hat vertical complex of $C_2$, it induces a grading preserving isomorphism $\widehat\Psi_*$ from the hat horizontal homology of $C_1$ to the hat vertical homology of $C_2$. 

A flip-filtered chain homotopy equivalence $\Psi$ is determined up to chain homotopy by the flip isomorphism $\Psi_*$ it induces on homoloogy. To see this, choose a horizontally simplified basis for $C_1$ and a vertically simplified basis for $C_2$, so that the generators which are not on a horizontal or vertical arrow, respectively, form bases for the horizontal homology of $C_1$ and the vertical homology of $C_2$. By Proposition \ref{prop:flip-map-horiz-basis}, $\Psi$ is determined by its image on the basis for horizontal homology of $C_1$. A similar argument shows that the image is determined by the projection to the basis for vertical homology of $C_2$.

%
%In addition to a bigraded complex for each spin$^c$ structure, the knot Floer homology package includes a collection of chain maps known as \emph{flip maps}. To describe these, we will now introduce some terminology for maps of bigraded complexes which are filtered in a certain sense. First, 
%
%
%Given two bigraded complexes $C$ and $C'$ over $\sR^\infty$, 
%
%
% We will say that a skew $\sR$-module map $\Psi:C \to C'$ is \emph{flip-filtered} if it is filtered with respect to the $j$-filtration on $C$ and the $i$-filtration on $C'$ and if $\Psi$ interchanges the two gradings in the sense that $\gr_z(\Psi(x)) = \gr_w(x)$ and $\gr_w(\Psi(x)) = \gr_z(x)$ for all $x$ in $C$. Similarly, we say $\Psi$ is \emph{reverse flip-filtered} if it is filtered with respect to the $i$-filtration on $C$ and the $j$-filtration on $C'$ and it interchanges the two gradings. We say that $\Psi$ is a flip-filtered chain homotopy equivalence if it is flip-filtered and there exists a reverse flip-filtered map $\bar\Psi:C' \to C)$ such that $\bar\Psi\circ\Psi$ and $\Psi\circ\bar\Psi$ are both filtered chain homotopic to the respective identity maps (with respect to the $j$-filtration on $C$ and the $i$-filtration on $C'$). Given two flip-filtered maps $\Psi_1$ and $\Psi_2$ from $C$ to $C'$, a \emph{flip-filtered chain homotopy} is a skew-$\sR$-module map $H:C\to C'$ such that $\Psi_1 -\Psi_2 = H\circ\partial + \partial \circ H$, $H$ is filtered with respect to the $j$-filtration on $C$ and the $i$-filtration on $C'$, and $H$ interchanges the gradings and increases both by one.

\subsection{The knot Floer chain complex}

We will now describe the knot Floer complex associated to a nullhomologous knot $K$ in a 3-manifold $Y$. We assume the reader is familiar with this invariant as defined by Ozsv{\'a}th and Szab{\'o} \cite{OzSz:knots} and Rasmussen \cite{Ras:knot-floer}; surveys of this material can be found in \cite{Manolescu:HFK} and \cite{Hom:survey}. However, we adopt slightly different conventions; in particular, while the original formulation defines the knot Floer complex as a filtered chain complex over $\F[U]$, we introduce a second formal variable $V$ to keep track of the Alexander filtration and view the invariant as a collection of bigraded chain complexes over $\sRminus$. This notation is becoming more common in the literature; we largely follow \cite[Section 1.5]{Zemke} (see also \cite[Section 2]{DHST}). For the reader's convenience, the relationship between these two notational conventions is explained in Section \ref{sec:old-way}.

Recall that knot Floer homology can be defined in terms of a \emph{doubly pointed Heegaard diagram} for the pair $(Y,K)$, that is, a tuple $\sH = (\Sigma, \boldsymbol{\alpha}, \boldsymbol{\beta}, w, z)$ where $(\Sigma, \boldsymbol{\alpha}, \boldsymbol{\beta})$ is a Heegaard diagram for $Y$ and $z$ and $w$ are points in $\Sigma$ in the complement of $\boldsymbol{\alpha}$ and $\boldsymbol{\beta}$. The pair of basepoints $z$ and $w$ determine the oriented knot $K$ by connecting $z$ to $w$ through the $\boldsymbol{\alpha}$-handlebody, avoiding the $\boldsymbol{\alpha}$ disks, and connecting $w$ to $z$ through the $\boldsymbol{\beta}$-handlebody, avoiding the $\boldsymbol{\beta}$ disks. Given a doubly pointed Heegaard diagram $\sH$, the set $\generators(\sH)$ consists of unordered tuples of points in $\boldsymbol{\alpha} \cap \boldsymbol{\beta}$ such that each alpha curve and each beta curve is occupied exactly once. We construct a chain complex $CFK_{\sRminus}(\sH)$ generated over $\sRminus$ by $\generators(\sH)$ whose differential counts holomorphic disks in an appropriate symmetric product of $\sH$. The differential is given by
$$\partial(\x) = \sum_{\y\in\generators(\sH)}  \sum_{  \substack{ \phi\in\pi_2(\x,\y) \\ \mu(\phi)=1 }} \# \left( \frac{\mathcal{M}(\phi)}{\R}\right) U^{n_w(\phi)} V^{n_z(\phi)} \y,$$
where $\pi_2(\x, \y)$ is the set of homotopy classes of disks connecting $\x$ to $\y$, $\mu$ is the Maslov index of such a class, $\mathcal{M}(\phi)$ is the moduli space of all pseudoholomorphic disks in the homotopy class $\phi$, and $n_w(\phi)$ and $n_z(\phi)$ count the multiplicity with which $\phi$ covers the basepoint $w$ and $z$, respectively. Note that if $\mathcal{M}(\phi)$ is nonempty than $n_w(\phi)$ and $n_z(\phi)$ are both nonnegative. 

To each generator $\x$ in $\generators(\sH)$ we can associate a spin$^c$ structure $\spin_w(\x)$ of $Y$, and generators $\x$ and $\y$ determine the same spin$^c$ structure if and only if $\pi_2(\x, \y)$ is nonempty. It follows that $CFK_{\sRminus}(\sH)$ splits as a direct summand over $\Spinc(Y)$, the set of spin$^c$ structures on $Y$:
$$CFK_{\sRminus}(\sH) = \bigoplus_{\spin \in \Spinc(Y)} CFK_{\sRminus}(\sH;\spin).$$
If $\spin$ is a torsion spin$^c$ structure then $CFK_{\sRminus}(\sH;\spin)$ can be equipped with an absolute bigrading $(\gr_w, \gr_z)$. For generators $\x$ and $\y$ and any $\phi \in \pi_2(\x,\y)$, the grading difference between $\x$ and $\y$ is given by
\begin{equation}\label{eq:relative-gr-w}
\gr_w(\x) - \gr_w(\y) = \mu(\phi) - 2n_w(\phi),
\end{equation}
\begin{equation}\label{eq:relative-gr-z}
\gr_z(\x) - \gr_z(\y) = \mu(\phi) - 2n_z(\phi).
\end{equation}
In particular, the differential has degree $(-1, -1)$. $\gr_w$ and $\gr_z$ have the same parity so $A = \frac{\gr_w-\gr_z}{2}$ is also a $\Z$-grading. Thus $CFK_{\sRminus}(\sH;\spin)$ is a bigraded complex as introduced in the previous section. If $\spin$ is not a torsion spin$^c$ structure then we have only a relative bigrading $(\gr_w, \gr_z)$ defined by Equations \eqref{eq:relative-gr-w} and \eqref{eq:relative-gr-z}.

%The complex $\CFKminus_\sR(\sH)$ splits into direct summands, where generators $\x$ and $\y$ are in the same summand if and only if $\pi_2(\x,\y)$ is nonempty. Equations \eqref{eq:relative-gr-w} and \eqref{eq:relative-gr-z} ensure that on each summand $\gr_w$ and $\gr_z$ give relative $\Z$-gradings, as does $A = \tfrac 1 2 (\gr_w - \gr_z)$. Thus each summand is a bigraded comlex as introduced in the previous section\footnote{The fractional parts of the gradings may be different on different summands, so strictly speaking with our terminology the entire complex is not a bigraded complex. Nevertheless, by slight abuse we will sometimes refer to $\CFKminus_\sR(\sH)$ as a bigraded complex.}. Recall that the grading $\gr_w$ is also called the Maslov grading and $A$ is the Alexander grading.

It turns out that the bigraded complexes $CFK_{\sRminus}(\sH;\spin)$ are invariants, up to filtered chain homotopy equivalence, of the triple $(Y,K; \spin)$ and do not depend on the choice of doubly pointed Heegaard diagram $\sH$. We denote this filtered chain homotopy equivalence class of complexes $CFK_{\sRminus}(Y,K;\spin)$, or $CFK_{\sRminus}(Y,K)$ for the sum over all spin$^c$ structures. In the case that $Y = S^3$ we omit it from the notation and write $CFK_{\sRminus}(K)$. At times it is convenient to allow negative powers of $U$ and $V$; for this we define $\CFKinfty_{\sRminus}(Y,K;\spin)$ to be $CFK_{\sRminus}(Y,K;\spin) \otimes_{\sRminus} \sR^\infty$ (similarly $\CFKinfty_{\sRminus}(Y,K) = CFK_{\sRminus}(Y,K) \otimes_{\sRminus} \sRinfty$). $\CFKinfty_{\sRminus}(Y, K)$ is called the \emph{full knot Floer complex} of the knot $K$. Simpler versions of the invariant can be defined analogously by replacing $\sRminus$ with one of its quotients $\sR_{n}$ defined above. In particular, a frequently used version is $CFK_{\widehat\sR}(Y,K)$, the $UV = 0$ quotient of the knot Floer complex. This complex is considerably easier to compute, since holomorphic discs that cover both basepoints can be ignored. See, for example, \cite{OzSz:bordered-CFK} which gives an effective method for computing the $UV = 0$ knot Floer complex.

In addition to the bigraded complexes above, for each spin$^c$ structure $\spin$ in $\Spinc(Y)$ the knot Floer package defines a flip-filtered chain homotopy equivalence
$$\Psi_\spin: \CFKinfty_{\sRminus}(Y,K;\spin) \to \CFKinfty_{\sRminus}(Y,K;\spin+ \text{PD}[K])$$
known as a \emph{flip map}. Note that since we are restricting to nullhomologous, $\text{PD}[K] = 0$ so $\Psi_\spin$ takes $\CFKinfty_{\sRminus}(Y,K;\spin)$ to itself. The flip map is well-defined up to flip-filtered chain homotopy. 

The flip maps are defined as the composition of three maps. Fix a doubly pointed Heegaard diagram $\sH$ representing $(Y,K)$ and let $\sH_z$ and $\sH_w$ denote the singly pointed Heegaard diagrams for $Y$ obtained by ignoring the $w$ basepoint or the $z$ basepoint, respectively. Both $\sH_z$ and $\sH_w$ are singly pointed Heegaard diagrams for the ambient 3-manifold $Y$, so they can be used to compute $\CFminus(Y;\spin)$. Ignoring the $w$ basepoint corresponds to setting $U = 1$ and $V = W$; this gives a map
$$\Omega_z: \CFKinfty_{\sRminus}(\sH; \spin) \to \CFinfty(\sH_z; \spin).$$
In other words, $\CFinfty(\sH_z; \spin)$ is (the $\infty$ version of) the horizontal complex of $\CFKinfty_{\sRminus}(\sH; \spin)$. Restricting to the Alexander grading zero summand gives an isomorphism
$$\Omega_{z}: \CFKinfty_{\sRminus}(\sH; \spin) |_{A=0} \to \CFinfty(\sH_z; \spin),$$
This map takes $\gr_z$ to the Maslov grading of $\CFinfty(\sH_z; \spin)$. Similarly, ignoring the $z$ basepoint corresponds to setting $V=1$ and $U = W$ and gives an isomorphism
$$\Omega_{w}: \CFKinfty_{\sRminus}(\sH; \spin) |_{A=0} \to \CFinfty(\sH_w; \spin)$$
taking $\gr_w$ to the Maslov grading. Finally, let 
$$\Gamma: \CFinfty(\sH_z;\spin) \to \CFinfty(\sH_w;\spin)$$
be a filtered chain homotopy equivalence arising from a sequence of Heegaard moves taking $z$ to $w$ in $\sH$. We define
$$\Psi_{\spin}: \CFKinfty_{\sRminus}(Y,K;\spin) |_{A=0} \to \CFKinfty_{\sRminus}(Y,K;\spin) |_{A=0}$$
to be the composition $\Omega_{w}^{-1}\circ \Gamma \circ \Omega_{z}$. We can uniquely extend this map to a skew $\sRminus$-module homomorphism $\Psi_\spin$ from $\CFKinfty_{\sRminus}(Y,K;\spin)$ to itself.

\begin{proposition}
The flip map $\Psi_{\spin}$ is flip-filtered and has skew degree $(0,0)$.
\end{proposition}
\begin{proof}
This follows from the fact that $\Gamma$ is filtered and grading preserving. It is enough to check this on $\CFKinfty_{\sRminus}(Y,K;\spin) |_{A=0}$, since both properties remain true when we extend the map as a skew $\sRminus$-module homomorphism to all of $\CFKinfty_{\sRminus}(Y,K;\spin)$. For the second claim, note that $\Omega_{z}$ takes $\gr_z$ to $M$, $\Omega_{w}$ takes $\gr_{w}$ to $M$, and $\Gamma$ preserves $M$, so $\Psi_{\spin}$ takes $\gr_z$ to $\gr_w$. Since $\gr_z = \gr_w$ on both the target and source of $\Psi_{\spin}$ (both being summands with Alexander grading zero), $\Psi_{\spin}$ also takes $\gr_w$ to $\gr_z$. For the first claim, consider an element $U^{A(x)} x$ of $\CFKinfty_{\sRminus}(\sH;\spin) |_{A=0}$. $\Omega_{z}$ takes this to $x$, and $\Gamma$ takes $x$ to a sum of the form $\sum_i c_i U^{a_i} y_i$ where $c_i$ is a constant in $\F$, $y_i$ is a generator and $a_i \ge 0$. It follows that
$$\Psi_{\spin}(U^{A(x)} x) = \sum_i c_i U^{a_i} V^{-A(y_i)+a_i} y_i.$$
and thus the $U$-filtration level of $\Psi_{\spin}(U^{A(x)} x)$ is at most as large as the $V$-filtration level of $U^{A(x)} x$. This relationship is preserved when the input is multiplied by $U$ and $V$, so $\Psi_{\spin}$ is flip filtered.
\end{proof}

\subsection{Notational remarks}\label{sec:old-way}

Though it is becoming more common, some readers may be unfamiliar with the $\F[U,V]$ notation used here for knot Floer complexes. In its original formulation, the knot Floer complex $\CFKinfty(Y,K)$ is defined as a chain complex over $\F[U,U^{-1}]$ equipped with an additional Alexander filtration; we find it convenient to encode this filtration with the second variable $V$. We use the subscript $\sRminus$ in our notation to highlight our different conventions, but the two complexes carry the same information: $\CFKinfty_{\sRminus}(Y,K)$ is isomorphic to infinitely many copies of $\CFKinfty(Y,K)$. More precisely, $\CFKinfty(Y,K)$ is isomorphic to $\CFKinfty_{\sRminus}(Y,K)|_{A=s}$ for any $s\in \Z$. We can view $\CFKinfty_{\sRminus}(Y,K)_{A=s}$ as generated over $\F[W, W^{-1}]$ by generators $\{ V^{-A(\x)}\x \}_{\x \in \generators}$; setting $V =1$ and $U=W$ recovers the familiar complex over $\F[U,U^{-1}]$, and the Alexander filtration is given by negative powers of $V$. For any $s$, multiplication by $V^s$ gives an isomorphism from $\CFKinfty_{\sRminus}(Y,K)|_{A=0}$ to $\CFKinfty_{\sRminus}(Y,K)|_{A=s}$.

In \cite{OzSz:rational-surgeries}, Ozsv{\'a}th and Szab{\'o} in fact define a different copy of $\CFKinfty(Y,K;\spin)$ for each relative spin$^c$ structure $\xi$ in $G_{Y,K}^{-1}(\spin)$, where $G_{Y,K}$ is a map from the set of relative spin$^c$ structures for $(Y,K)$ to the set of spin$^c$ structures for $Y$(for nullhomologous knots $G_{Y,K}^{-1}(\spin)$ is indexed by $s$ in $\Z$). These complexes are described as generated over $\F$ by triples $[x, i, j]$ where $x$ is a generator and $i$ and $j$ are integers satisfying $j-i = A(x)-s$. We identify the triple $[x, i, j]$ with $U^{-i} V^{-j} x$ and note that the Ozsv{\'a}th-Szab{\'o} complex associated to $s$ is precisely the Alexander grading $s$ summand of $\CFKinfty_{\sRminus}(Y,K;\spin)$. In \cite{HeddenLevine:surgery}, which we rely on substantially for the background on flip maps, slightly different notation is used. There a single complex is given for each $\spin$, generated by triples $[x, i, j]$ with $j-i = A(x)$. However, the dependence on a choice of $s$ in $\Z$ arises when defining filtered maps; the relevant filtration on the sources is given by the integer $j - s$ rather than by $j$. For us the $V$ filtration is always the negative power, so $[x, i, j]$ in the notation of \cite{HeddenLevine:surgery} corresponds to $U^{-i} V^{s-j} x$. This distinction is not relevant in the present setting, since we only define the flip maps corresponding to $s=0$, though it is relevant for rationally nullhomologous knots. In general, for any spin$^c$ structure $\spin$ in $\spinc(Y)$ we can define a family of flip maps by choosing relative spin$^c$-structures; these maps are equivalent to each other, differing only by multiplication by a power of $V$, so it suffices to compute any one. For arbitrary knots $G_{Y,K}^{-1}(\spin)$ is indexed by $s$ that is not in $\Z$ but in $\Z + A(\spin)$ for some rational $A(\spin)$. Since our knots are nullhomologous $A(\spin) = 0$, and it makes sense to choose $s=0$.

We remark that $\CFKminus(Y,K; \spin)$ can be identified with the subcomblex of $\CFKinfty_{\sRminus}(Y,K) |_{A=s}$ (for any integer $s$) with nonnegative power of $V$.  This carries the same information as $CFK_{\sRminus}(Y,K) |_{A=s}$ but the two are not quite the same, at least under the identification above, since $CFK_{\sRminus}(Y,K)_s$ also requires nonnegative powers of $U$. The subcomblex of $\CFKinfty_{\sRminus}(Y,K) |_{A=s}$ with nonnegative powers of $V$ can also be described as the Alexander grading $s$ summand of $CFK_{\sRminus}(Y,K) \otimes_{\sRminus} \F[U, U^{-1}, V]$. In fact, the complex $CFK_{\sRminus}(Y,K) |_{A=s}$ is more directly related to the complex $A^-_s$ appearing in minus version of the surgery formulas of Ozsv{\'a}th and Szab{\'o}; this is the subcomplex of $\CFKminus(Y,K)$ consisting of triples $[x, i, j]$ with $\max(i, j-s) \le 0$. Under the identification given above, this corresponds to the subcomplex of $CFK_{\sRminus}(Y,K)|_{A=0}$ generated by terms of the form $U^{-i} V^{-j} x$ with $j - i = A(x)$ and $\max(i, j-s) \le 0$. Multiplying by $V^s$ gives an isomorphism between this and $CFK_{\sRminus}(Y,K)|_{A=s}$.

%\begin{remark}\label{rmk:old-way}
%In its original formulation, the knot Floer complex $\CFKinfty(Y,K)$ is defined as a chain complex over $\F[U,U^{-1}]$ along with an Alexander filtration. We use the subscript $\sR$ in our notation to highlight our different conventions, but the two complexes carry the same information. In fact, $\CFKinfty(Y,K)$ is isomorphic to $\CFKinfty_\sR(Y,K)_0$. We can view the latter as generated over $\F[W, W^{-1}]$ by generators $\{ V^{-A(\x)}\x \}_{\x \in \generators}$. Setting $V =1$ recovers the familiar complex over $\F[U,U^{-1}]$, and the Alexander filtration is given by negative powers of $V$. Said another way, in the notation of \cite{OzSz:knots} $\CFKinfty(Y,K)$ is described as being generated over $\F$ by triples of the form $[\x, i, j]$ where ${A(\x) - j + i = 0}$; the corresponding generators over $\F$ in our notation are the terms $U^{-i} V^{-j} \x$ for $\x \in \generators$ and integers $i$ and $j$ with and $0 = A(U^{-i} V^{-j} \x) = A(x) - j + i$.  Note that $\CFKminus(Y,K)$ corresponds to the elements of $\CFKinfty(Y,K)_0$ with nonnegative power of $V$. This carries the same information as $\CFKminus_\sR(Y,K)_0$ but the two are not quite the same, at least under the identification above, since $\CFKminus_\sR(Y,K)_0$ also requires nonnegative powers of $U$.
%\end{remark}

\subsection{Examples}

To clarify conventions, particularly regarding flip maps, we will describe two examples in detail. We will return to these examples later when we represent bigraded complexes and flip maps in terms of immersed curves.

\begin{example}\label{ex:1-surgery-on-fig8}
Let $Y$ be $+1$-surgery on the figure eight knot, and let $K$ be the dual knot, i.e. the core of the filling torus. $Y$ has a single spin$^c$-structure, which we denote $\spin$. The knot Floer complex of $CFK_{\sRminus}(Y,K)$ can be computed using the surgery formula of Hedden and Levine \cite{HeddenLevine:surgery} (the easiest way to do this is using immersed curves, using Theorem \ref{thm:surgery-formula-dual-knots}). For a particular choice of basis, the resulting complex has five generators, which we denote $a$, $b$, $c$, $d$, and $e$, and the only nonzero differentials are
$$\partial(a) = V b, \quad \text{ and } \quad \partial(e) = Ud.$$
The bigrading and Alexander grading are given in the table below:
$$\def\arraystretch{1.5} \begin{array}{c|ccccc}
  & a & b & c & d &e \\ \hline
(\gr_w, \gr_z) & (1,-1) & (0,0) & (0,0) & (0,0) & (-1,1) \\
A & 1 & 0 & 0 & 0 & -1
\end{array} $$
Although it would be difficult to compute directly from a Heegaard diagram, the map $\Psi_{\spin}$ is uniquely determined by the bigradings up to homotopy and multiplication by a unit in $\F$. Since $\Psi_{\spin}$ interchanges the gradings it must take $a$ to a multiple of $e$, $e$ to a multiple of $a$, and $b$, $c$, and $d$ to linear combinations of $b$, $c$, and $d$. By Proposition \ref{prop:flip-map-horiz-basis} we can assume after applying an appropriate flip-filtered chain homotopy that $\Psi_{\spin}(e) = 0$, since there is a horizontal arrow starting at $e$. Then we must have $\Psi_{\spin}(d) = 0$ since $\Psi_{\spin}$ is a chain map. By applying flip-filtered chain homotopies $H$ that take $b$ or $c$ to appropriate multiples of $a$, we can assume that the coefficients of  $b$ in $\Psi_{\spin}(b)$ and $\Psi_{\spin}(c)$ are zero. We thus have that
\begin{align*}
\Psi_{\spin}(a) &= c_1 e, \\
\Psi_{\spin}(b) &= c_2 c + c_3 d, \\
\Psi_{\spin}(c) &= c_4 c + c_5 d,
\end{align*}
where the $c_i$'s are constants in $\F$. Note that we have reduced the problem to finding the induced map from horizontal homology to vertical homology, which are generated by $\{a,b,c\}$ and $\{c,d,e\}$, respectively. The constant $c_1$ must be nonzero so that the induced map from horizontal homology to vertical homology is an isomorphism; after a change of basis rescaling $e$ by a constant, we can take this multiple to be 1. When we rescale $e$, we will also rescale $d$ by the same amount so that the differential is unchanged. We then must have $c_2 = 0$ and $c_3 = 1$ for $\Psi_{\spin}$ to be a chain map. Up to a change of basis adding a multiple of $b$ to $c$, we can assume $c_5 = 0$. The flip map is then determined, up to homotopy, by the constant $c_4$ (which must be nonzero to have an isomorphism on homology). In particular, in the case that $\F = \Z/2\Z$ the flip map is uniquely determined from the complex. We note that when $\F$ is not $\Z/2\Z$ different choices of $c_4$ give non-equivalent flip maps. In this case we can indirectly deduce that the correct flip map is given by $c_4 = 1$ as follows: given the flip map, the surgery formula allows us to compute $\HFminus$ of rational surgeries on $K$. Changing the value of $c_4$ does not affect the answer for non-zero slopes, but considering the 0-surgery on $K$ (which is the same as 0-surgery on the figure eight knot), we see that only $c_4 = 1$ gives the correct answer.
\end{example}

In the example above, the collection of flip maps is uniquely determined by the bigraded complexes up to a unit in $\F$. This is not always the case, as the next example demonstrates.

\begin{example}\label{ex:1-surgery-on-LHT}
Let $Y$ be $+1$-surgery on the left handed trefoil, and let $K$ be the core of the surgery. Once again $Y$ has a single spin$^c$-structure, denoted $\spin$, and the complex $\CFKinfty_\sR(Y,K;\spin)$ can be computed using the surgery formula. This complex is identical as an ungraded complex to the one in the previous example: the generators are  $a$, $b$, $c$, $d$, and $e$, and the nonzero differentials are
$$\partial(a) = V b, \quad \text{ and } \quad \partial(e) = Ud.$$
The bigradings, however, are different and are given in the table below:
$$\def\arraystretch{1.5} \begin{array}{c|ccccc}
  & a & b & c & d &e \\ \hline
(\gr_w, \gr_z) & (2,0) & (1,1) & (0,0) & (1,1) & (0,2) \\
A & 1 & 0 & 0 & 0 & -1
\end{array} $$
We describe the possible flip maps up to homotopy. As in the previous example, it is enough to describe the induced map from horizontal to vertical homology: gradings force $\Psi_{\spin}(e)$ to be zero, the chain map property implies that $\Psi_{\spin}(d) = 0$, and by by appropriate chain homotopies we can assume that the coefficients of $a$ and $b$ in $\Psi_{\spin}(a)$, $\Psi_{\spin}(b)$, and $\Psi_{\spin}(c)$ are zero. The bigradings now tell us that
\begin{align*}
\Psi_{\spin}(a) &= c_1 V^{-1} c + c_2 e, \\
\Psi_{\spin}(b) &= c_3 d, \\
\Psi_{\spin}(c) &= c_4 c + c_5 V e .
\end{align*}
The constant $c_3$ must be non-zero and can be made to be 1 after a change of basis rescaling $d$ (and also rescaling $e$, so that the differential is unchanged). The fact that $\Psi_{\spin}$ is a chain map implies that $c_5 = 0$ and $c_2 = 1$. The constant $c_4$ must then be nonzero. If $c_1$ is nonzero, we may assume it is $1$ by a change of basis rescaling $c$. We are left with two fundamentally different cases: $c_1 = 0$ or $c_1 = 1$, along with a choice of nonzero constant $c_4$ in each case. None of these remaining choices are equivalent. In particular, even when $\F = \Z/2\Z$ we have two nonequivalent flip maps that could occur for the given bigraded complex.

The surgery formula does not give the flip map on the dual surgery, and computing the flip map directly would be quite difficult, but as in the previous example we can use known surgeries on $K$ to deduce that the correct choice is $c_1 = c_4 = 1$. We check that $c_4 = 1$ by considering the zero surgery on $K$, as before. To see that $c_1 = 1$, we use both possible flip maps in the mapping cone formula to compute $\HFhat(Y_{-4/3}(K))$; using $c_1 = 1$ gives rank 1 while using $c_1 = 0$ gives rank 3. $Y_{-4/3}(K)$ is the same as $-1$-surgery on the left handed trefoil in $S^3$, so the correct rank is 1.
\end{example}

\begin{remark}
There is no algebraic obstruction to some other knot $K' \subset Y'$ giving the exact same complexes as in Example \ref{ex:1-surgery-on-LHT} and a different choice of flip map. However, this does not happen because such a knot complement would still be genus one and fibered. This implies that the complement of $K'$ is either the figure-eight complement or a trefoil complement and no framing on one of these gives the complex above with a different flip map.
\end{remark}

\section{Immersed Floer theory in marked surfaces}\label{sec:floer-theory}% !TEX root = ../CFKcurvesZHS3s.tex
%floer-theory.tex

The algebraic objects described in the previous section, bigraded complexes and flip maps, can be given a geometric interpretation using Floer homology of immersed curves in certain marked surfaces. In this section we define Floer theory for immersed Lagrangians in these surfaces. Immersed Floer theory is defined more generally in \cite{AkahoJoyce}, but in our two-dimensional setting Floer theory can be defined combinatorially. We will also prove, in our setting, a stronger notion of homotopy invariance than is shown in \cite{AkahoJoyce}. Our construction is inspired by but not identical to the combinatorial treatment of Floer cohomology and the Fukaya category for curves in surfaces found in \cite{Abouzaid:surfaces}. One difference is that we use homological conventions rather than cohomological conventions, since we are ultimately interested in representing Heegaard Floer homology. Another difference is that we avoid the use of Novikov coefficients and ignore the area of disks we count; in particular, we do not need to choose a symplectic form on the surface. This is possible because we restrict to non-compact surfaces (in fact, for our purposes it is sufficient to work only in the infinite strip and the infinite cylinder) and impose an assumption that immersed curves are in an admissible configuration (see Definition \ref{def:admissible} below). In this respect, our construction is more similar to the combinatorial treatment of Floer homology of curves in \cite{DRS:combinatorial-floer}, though that work restricts to embedded curves and does not define higher product operations. Another difference is that, unlike in \cite{Abouzaid:surfaces}, we allow curves which bound immersed monogons. This adds significant technical difficulties and requires curves to carry a special decoration, a \emph{bounding chain}, before Floer homology can be defined.

A final difference is that we will consider curves in marked surfaces. In this case we introduce a formal variable $W$ to record interactions of immersed disks with the marked points, and the Floer complex will be a module over $\F[W]$ rather than a vector space over $\F$. In fact, we will ultimately be interested in doubly marked surfaces in which there are two types of marked points; in this setting the relevant coefficient ring is $\sRminus = \F[U,V]$, where the formal variables $U$ and $V$ are associated with the two types of marked points. To simplify notation we will stick to the singly marked setting for most of this section, and we explain how the definitions extend to doubly marked surfaces in Section \ref{sec:doubly-marked}.

\subsection{The space $CF(L_0, L_1)$}

Consider a non-compact surface $\Sigma$ with a collection $\{w_i\}_{i\in I}$ of marked points (where $I$ is any index set---generally we take $I$ to be $\Z$, but we occasionally consider finite collections of marked points); we allow $\Sigma$ to have boundary, but we will require that any compact boundary component is decorated with a basepoint called a \emph{stop}. We mention them for completeness, but we will not need the case of compact boundary components with stops; the two key examples of surfaces for the purposes of this paper are the infinite strip $\strip = [-\tfrac 1 2, \tfrac  1 2] \times \R]$, with marked points $w_i$ at $(0, i - \tfrac 1 2)$ for integers $i$, and the infinite cylinder $\cylinder = \strip / (-\tfrac 1 2, y) \sim (\tfrac 1 2 , y)$.

Let $L$ be an immersed Lagrangian in $\Sigma$ that is disjoint from all of the marked points; more precisely, $L$ is a disjoint union of copies of $S^1$, $[0,1]$ and $\R$ along with an immersion $\iota:  L \to \Sigma$ whose image is disjoint from the marked points. We let $\iota_i$ denote the restriction of $\iota$ to a component $L_i$ of $L$. The image of a component $L_i$ will be referred to as an immersed circle, immersed arc, or immersed line when $L_i$ is $S^1$, $[0,1]$ or $\R$, respectively. We require that the endpoints of an immersed arc lie on the boundary of $\Sigma$. We also require that an immersed line eventually leaves any compact subsurface of $\Sigma$ on both ends. For example, when $\Sigma$ is the infinite strip or cylinder, this means the ends of immersed lines must escape to the infinite ends of the strip or cylinder. Such a collection of immersions $\iota_i: L_i \to \Sigma$ will be called collectively an \emph{immersed multicurve}. By slight abuse of notation, we will sometimes conflate the immersion with its image. The immersed multicurves we consider will be weighted in the following sense: there will be a collection of basepoints on $L$ and a nonzero weight in $\F$ will be associated to each basepoint. We can usually assume that there is one basepoint on each $S^1$ component of $L$ and no basepoints on other components of $L$, but at times it is convenient to allow additional basepoints.

In addition to weights our immersed multicurves will be equipped with grading information, which we now define. We first review how gradings are defined on immersed curves in unmarked surfaces, and then describe a modification of this definition for marked surfaces. To define a grading on an immersed multicurve $\iota: L \to \Sigma$ in an unmarked surface, we must first fix a trivialization of the tangent bundle of $\Sigma$; in the case of the strip $[0,1]\times \R$ or the cylinder $S^1 \times \R$ we will use the obvious trivialization coming from viewing the strip as a subset of $\R^2$ and from cutting the cylinder open to give the strip. Having fixed a trivialization, the tangent slope defines a map $\tau$ from $L$ to $\mathbb{RP}^1$, which we identify with $\R/\Z$. An orientation on $L$ allows us to lift this to a map from $L$ to $S^1 = \R/2\Z$. More generally, a \emph{$\Z/N\Z$-grading} on $L$ is a lift of this map to $\R/N\Z$, and a \emph{$\Z$-grading} on $L$ is a lift $\tilde\tau$ of this map to $\R$. Note that each $S^1$ component of $L$ presents a potential obstruction to the existence of a $\Z$-grading. For such a component, the tangent slope map defines a loop in $\mathbb{RP}^1$, and the winding number of this loop must be zero for this loop to lift to $\R$. The \emph{Maslov class} of $L$, denoted $\mu_L$, is the element of $\text{Hom}(\pi_1(L), \Z)$ that records this obstruction for all closed components. If $\mu_L$ vanishes, we say that $L$ is \emph{$\Z$-gradable}.

We now describe a modified notion of gradings for immersed Lagrangians in marked surfaces, which allows the gradings to capture the interaction of the curves with marked points. For each marked point, we choose a half-open oriented arc starting at that marked point and converging to a puncture or infinite end of $\Sigma$. We will refer to these arcs as \emph{grading arcs}. The grading arcs should be chosen so that any compact curve in $\Sigma$ intersects finitely many arcs. Given such a choice of arcs, we now define a grading on $L$ to be a piecewise continuous map $\tilde\tau: L \to \R$ that lifts the tangent slope map and is continuous except at intersections of $L$ with the arcs from the marked points, at which it has jump discontinuities of magnitude 2. More precisely, any time $L$ crosses a grading arc passing from the left side to the right side of the arc, $\tilde\tau$ increases by 2. The obstruction to such a grading is the \emph{adjusted Maslov class} of $L$, an element of $\text{Hom}(\pi_1(L), \Z)$ which records the change in tangent slope around each $S^1$ component, taking into account the jump discontinuities described above. When the adjusted Maslov class vanishes then $L$ is $\Z$-gradable; otherwise $L$ only admits a $\Z/N\Z$ grading for some $N$. Note that the marked points have no affect on the grading modulo 2, and as before a $\Z/2\Z$-grading is equivalent to an orientation on $L$. All of our immersed multicurves will be oriented, and unless otherwise noted they will all be $\Z$-graded.

Given immersed multicurves $L_0$ and $L_1$ in $\Sigma$ that intersect transversally, we define $CF(L_0, L_1)$ to be the module over $\F[W]$ generated by the intersections of $L_0$ and $L_1$. If $L_0$ and $L_1$ both contain immersed arcs, we also make the requirement that on any given boundary component all endpoints of arcs in $L_1$ occur after all endpoints of arcs in $L_0$, with the order coming from the boundary orientation (in the case of compact boundary components, this is the reason for marking a stop; we interpret the order of endpoints by following the boundary orientation starting and ending at the stop). We remark that the requirement on ordering the endpoints of arcs is necessary given that hope aim to promote $CF(L_0, L_1)$ to a chain complex whose homology is invariant under reasonable homotopies of the curves, since sliding an endpoint of an arc in $L_0$ past an endpoint of an arc $L_1$ would change the parity of the intersection number and thus could not preserve homology.

If $L_0$ and $L_1$ are oriented, then $CF(L_0, L_1)$ has a $\Z/2\Z$ grading given by the sign of intersection points as in Figure \ref{fig:mod2-grading}. This grading can be enhanced if $L_0$ and $L_1$ are $\Z$-graded. Given $\Z$-gradings $\tilde\tau_0: L_0 \to \R$ on $L_0$ and $\tilde\tau_1: L_1\to \R$ on $L_1$, we can define a $\Z$ grading on $CF(L_0, L_1)$ as follows:
\begin{definition}\label{def:grading}
For each intersection point $p$ in $L_0 \cap L_1$, we define the grading $\gr(p)$ to be the greatest integer less than $\tilde\tau_1(p) - \tilde\tau_0(p)$. Equivalently, the grading is $\tilde\tau_1(p) - \tilde\tau_0(p) - \theta_{01}(p)$, where $\theta_{01}(p)$ is $\tfrac 1 \pi$ times the angle covered when turning counterclockwise from $L_0$ to $L_1$.
\end{definition}
Note that if $L_0$ and $L_1$ only carry $\Z/N\Z$ gradings, then the definition above defines a $\Z/N\Z$ grading on $CF(L_0, L_1)$. In particular, given orientations on $L_0$ and $L_1$ this definition determines the $\Z/2\Z$ grading described in Figure \ref{fig:mod2-grading}. We remark that the grading on $CF(L_0, L_1)$ only depends on the homotopy classes of the grading arcs used to define gradings on $L_0$ and $L_1$; if we apply a homotopy to the arc in $\Sigma$, each time the arc passes an intersection point the gradings of $L_0$ and $L_1$ at that point both jump by two in the same direction, so their difference is unchanged. 

\begin{figure}
\labellist

  \pinlabel {$L_0$} at 43 28
  \pinlabel {$L_1$} at 29 45
  
  \pinlabel {$L_0$} at 118 28
  \pinlabel {$L_1$} at 104 45
  
  \pinlabel {$\gr(p) \equiv 0 \mod 2$} at 20 -10
  \pinlabel {$\gr(p) \equiv 1 \mod 2$} at 105 -10

\endlabellist
\includegraphics[scale = 1.2]{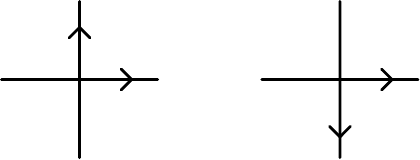}
\vspace{3mm}
\caption{The mod 2 grading on an intersection point $p \in L_0 \cap L_1$}
\label{fig:mod2-grading}
\end{figure}
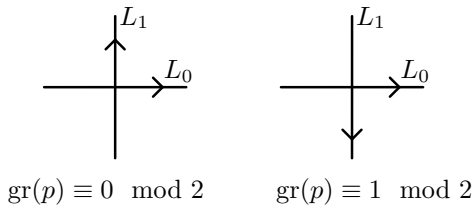

\begin{remark}
Readers may notice that the grading in Definition \ref{def:grading} differs from the usual definition of the grading by one. For instance, in \cite{Auroux:beginner} the degree of an intersection point corresponds to $\tilde\tau_0(p) - \tilde\tau_1(p)$ plus the angle of clockwise rotation from $L_0$ to $L_1$. Likewise, the mod 2 degree defined in \cite{Abouzaid:surfaces} differs from ours by one. This is due to the fact that we are using homological rather than cohomological conventions.
\end{remark}

There is a canonical identification between $CF(L_0, L_1)$ and $CF(L_1, L_0)$, as they are generated by the same intersection points, but the gradings are different. That is, the grading of an intersection point depends on whether it is viewed as a point in $L_0 \cap L_1$ or $L_1 \cap L_0$. It is straightforward to check from Definition \ref{def:grading} that these two gradings for a given intersection point sum to $-1$, so the identification between $CF(L_0, L_1)$ and $CF(L_1, L_0)$ takes the grading $\gr$ to $-1 - \gr$.

\begin{figure}
\labellist

  \pinlabel {$L$} at -9 27
  \pinlabel {$L'$} at -9 12
  
  \pinlabel {$\epsilon$} at 163 18

\endlabellist
\includegraphics[scale = 1.2]{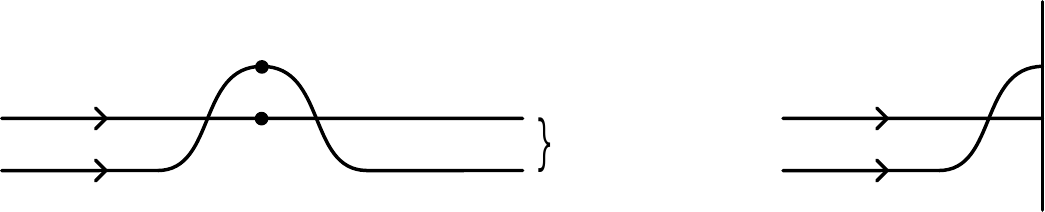}
\caption{$L'$ is obtained by translating $L$ slightly to the right, except in a neighborhood of the basepoints on $L$ or the terminal endpoint of any arc component of $L$. Near each basepoint, $L'$ lies to the left of $L$ instead; this results in two intersection points between $L$ and $L'$ near each basepoint. Near the endpoint of an arc in $L$ at which the arc is oriented toward the boundary $L'$ also lies to the left of $L$, requiring one additional intersection point.}
\label{fig:pushoff}
\end{figure}

Defining $CF(L_0, L_1)$ requires $L_0$ and $L_1$ to intersect transversally. To allow for arbitrary Lagrangians, we can apply a small perturbation to one of them. In particular, for an immersed multicurve $L$ let $L'$ denote an immersed multicurve which agrees with the pushoff of $L$ by some small amount $\epsilon$ to the right (with respect to the orientation on $L$) except in a small neighborhood of the basepoints on $L$ or of the terminal endpoint of an arc component of $L$, near which $L'$ lies to to the left of $L$ as  in Figure \ref{fig:pushoff}. We then define $CF(L_0, L_1)$ to be $CF(L_0, L'_1)$ for a sufficiently small perturbation. Note that if $L_0$ and $L_1$ were already transverse then the perturbation can be chosen small enough to not affect the intersection points. We remark that when perturbing a curve it is important to perturb in the way described above rather than simply pushing the curve to the same side everywhere; this is a combinatorial realization of the usual requirement for Floer homology that isotopies are Hamiltonian.

As a special case, we can define $CF(L)$ to be $CF(L,L)$, which means $CF(L,L')$ for a suitable small perturbation $L'$ of $L$. Note that there are two points in $L \cap L'$ for each self-intersection point of $L$, and two additional points in $L \cap L'$ near each basepoint on $L$ and one additional intersection point for each arc component of $L$. A $\Z$-grading on an immersed multicurve $L$ gives rise to a $\Z$-grading on the pushoff $L'$ and thus defines a $\Z$-grading on $CF(L)$. It is easy to check that the pair of intersection points in $L\cap L'$ associated to any basepoint of $L$ have gradings $0$ and $-1$. Similarly, the pair of intersection points associated to any self intersection point of $L$ have gradings that sum to $-1$. Because of this relationship, it is convenient to encode the grading information for both points associated to a self-intersection of $L$ by keeping track of only the even grading.

\begin{definition}
The \emph{degree} of a self intersection point $p$ of $L$ is the even integer $\deg(p)$ such that the two intersection points of $L\cap L'$ corresponding to $p$ have gradings $\deg(p)$ and $-1-\deg(p)$. 
\end{definition}

\subsection{Polygon counting maps and $\Ainfty$ relations}

For an immersed Lagrangian $L$ in a marked surface, the module $CF(L)$ can be equipped with operations giving it an $\Ainfty$ structure. These operations count immersed polygons bounded by various perturbations of $L$. More generally, given a collection of $k+1$ pairwise-transverse immersed Lagrangians $L_0, \ldots, L_k$ in a marked surface $\Sigma$, we will define a polygon counting operation
$$m_k: CF(L_0, L_1) \otimes \cdots \otimes CF(L_{k-1}, L_k) \to CF(L_0, L_k).$$
We note that if $\Sigma$ has boundary and the $L_i$ have arc components, we will require that all endpoints of arcs in $L_j$ occur after all endpoints in $L_i$ with respect to the boundary orientation for any $j > i$.

\begin{definition} Given Lagrangians $L_0, \ldots, L_k$ in $\Sigma$, intersection points $p_i$ in $L_{i-1} \cap L_{i}$ for $1 \le i \le k$ and an intersection point $q$ in $L_0 \cap L_k$, an \emph{immersed $(k+1)$-gon with corners $p_1, \ldots, p_k$, and $q$} is an orientation preserving map
$$u: (D^2, \partial D^2) \to (\Sigma, L_0 \cup \cdots \cup L_k)$$
with the following properties:
\begin{itemize}
\item $u(\bar q) = q$ and $u(\bar p_{i}) = p_i$ for $1 \le i \le k$, where $\bar q, \bar p_1, \ldots, \bar p_k$ are fixed points on $\partial D^2$ appearing in clockwise order;
\item $u$ maps the segment of $\partial D^2$ between $\bar p_{i}$ and $\bar p_{i+1}$ to $L_i$ for $1 \le i < k$, and the segments between $\bar q$ and $\bar p_1$ and between $\bar p_k$ and $\bar q$ are mapped to $L_0$ and $L_k$, respectively;
\item $u$ is an immersion away from the points $\bar q$ and $\bar p_i$; and
\item the points $q$ and $p_i$ are convex corners of the image of $u$ (that is, that the image of a neighborhood of $\bar q$ or $\bar p_i$ covers only of the four quadrants near the intersection point $q$ or $p_i$).
\end{itemize}
\end{definition}

We can make sense of this definition even when $k = 0$, in which case it describes an \emph{immersed monogon}. These shapes have also been referred to as \emph{teardrops} or \emph{fishtails}.

\begin{definition} Given a self-intersection point $q$ of $L_0$, an \emph{immersed mongon with corner $q$} is an orientation preserving map
$$u: (D^2, \partial D^2) \to (\Sigma, L_0)$$
with the following properties:
\begin{itemize}
\item $u(\bar q) = q$, where $\bar q$ is a fixed point on $\partial D^2$.
\item $u$ is an immersion away from $\bar q$; and
\item the point $q$ is a convex corner of the image of $u$.
\end{itemize}
\end{definition}

\begin{figure}
\labellist

  \pinlabel {$\bar q$} at 43 -2
  \pinlabel {$\bar p_1$} at -2 30
  \pinlabel {$\bar p_2$} at 12 91
  \pinlabel {$\bar p_3$} at 77 91
  \pinlabel {$\bar p_k$} at 91 30
  
  \pinlabel {$ q$} at 205 -2
  \pinlabel {$ p_1$} at 157 35
  \pinlabel {$ p_2$} at 175 93
  \pinlabel {$ p_3$} at 239 92
  \pinlabel {$ p_k$} at 257 33
  
  \pinlabel {$L_0$} at 182 17
  \pinlabel {$L_1$} at 168 63
  \pinlabel {$L_2$} at 206 90
  \pinlabel {$L_3$} at 241 77
  \pinlabel {$L_{k-1}$} at 255 47
  \pinlabel {$L_k$} at 231 17
  
  \Large  
  \pinlabel {$D^2$} at 44 50

\endlabellist
\includegraphics[scale = 1]{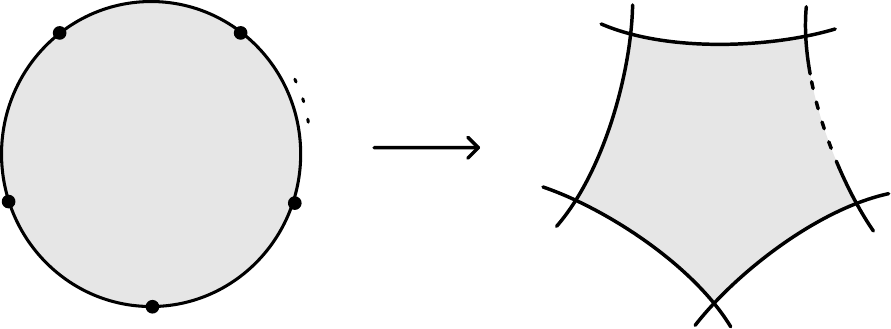}
\caption{An immersed $(k+1)$-gon.}
\label{fig:immersed-polygon}
\end{figure}
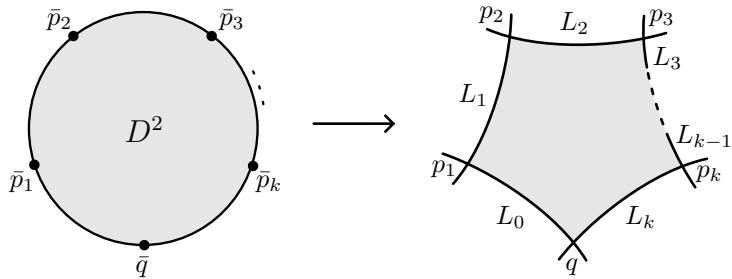

We will consider immersed polygons up to smooth reparametrization of the disk $D^2$, and we define $\mathcal{M}(p_1, \ldots, p_k, q)$ to be the set of equivalence classes of immersed $(k+1)$-gons with corners $p_1, \ldots, p_k$, and $q$. In order to ensure that the set $\mathcal{M}(p_1, \ldots, p_k, q)$ is finite, we will need to impose an admissibility condition on the immersed Lagrangians. As usual, we require the surface $\Sigma$ to be noncompact. We will restrict mainly to compact Lagrangians, but we do allow at most one of $L_0, \ldots, L_k$ to be noncompact provided it has finitely many self intersection points (in most cases the non-compact curves we consider will be embedded).

\begin{definition}\label{def:admissible}
A collection of immersed Lagrangians $L_0, \ldots, L_k$ in a non-compact surface $\Sigma$ is \emph{admissible} if they are pairwise transverse, no Lagrangian bounds an immersed disk in $\Sigma$, and no two Lagrangians bound an immersed annulus in $\Sigma$. We furthermore assume that at most one of $L_0, \ldots, L_k$ is non-compact and that any non-compact Lagrangian has finitely many self-intersection points.
\end{definition}

\begin{proposition}\label{prop:admissible}
If $L_0, \ldots, L_k$ are admissible, then for any intersection points $p_{1}, \ldots, p_{k}$, and $q$ as above the set $\mathcal{M}(p_1, \ldots, p_k, q)$ is finite.
\end{proposition}
\begin{proof}
The collection of Lagrangians define a cell structure on the surface $\Sigma$, where $0$-cells are intersections between Lagrangians, $1$-cells are segments of Lagrangians, and $2$ cells are the connected components of $\Sigma \setminus (L_0 \cup \cdots \cup L_k)$. The image of any immersed polygon determines a 2-chain (with nonnegative coefficient for each region) called the \emph{domain} of the polygon. In addition to the domain, an immersed polygon determines combinatorial gluing data specifying which edges of which regions should be indentified to form the disk. Standard arguments show that to study immersed polygons it is enough to work with this combinatorial data. In particular, a polygon is uniquely determined up to equivalence by its domain and gluing data. Moreover, there are finitely many choices of combinatorial gluing data for each domain, and a domain is uniquely determined by its boundary, a $1$-cycle. Thus we need to show that there are finitely many $1$-cycles which could bound an immersed polygon with the given corners.

The boundary of a polygon with the given corners consists of paths from $p_1$ to $q$ in $L_0$, from $q$ to $p_k$ in $L_k$, and from $p_{i+1}$ to $p_i$ in $L_i$ for each $1 \le i < k$. Although each $L_i$ may have multiple components, only one component will be involved in the boundary of a polygon so we will assume without loss of generality that each $L_i$ has a single component. If $L_i$ is an immersed arc or an immersed line, then there is a unique path up to reparametrization connecting the given corners, but if $L_i$ is an immersed circle there are infinitely many such paths obtained from each other by adding full multiples of the closed curve $L_i$. Thus the difference between any two potential boundaries of an immersed polygon, as 1-cycles, is a collection of full copies of the $L_i$'s (specifically of those that are immersed circles).

We will say that the length of a path in $L_i$ is the number of times the path passes its starting point before reaching its ending point (this is the number of full copies of $L_i$ that the path covers). Suppose there are infinitely many immersed polygons with the desired corners, and thus infinitely many distinct 1-cycles representing their boundaries. Then there must be boundaries which contain arbitrarily long paths on at least one Lagrangian, say $L_\ell$. We will consider a polygon $u_N$ for which the $L_\ell$ part of the boundary has length at least $N$ for some very large $N$. The finiteness of $\mathcal{M}(p_1, \ldots, p_k, q)$ does not depend on the orientation of the Lagrangians, so we will assume without loss of generality that all Lagrangians are oriented such that their orientation agrees with the boundary orientation induced by $u_N$.

We will choose an arc $\gamma$ in $\Sigma$ and consider the preimage of $\gamma$ under $u_N$, noting that $u_N^{-1}(\gamma)$ is a collection of disjoint arcs in $D^2$. For each $i$ let $\bar L_i$ denote the portion of $\partial D^2$ mapping to $L_\ell$ under $u_N$. We will choose $\gamma$ with the property that for large enough $N$ there are arbitrarily many arcs in $u_N^{-1}(\gamma)$ with one endpoint on $\bar L_\ell$ and one endpoint on $\bar L_m$ for some $m \neq \ell$. To see that this is possible, we first consider the case that $L_\ell$ is homotopically nontrivial in $\Sigma$. In this case we can let $\gamma$ be arc in $\Sigma$ (with either closed ends on $\partial \sigma$ or open ends approaching punctures of $\Sigma$) that intersects $L_\ell$ in a homotopically essential way, meaning that any curve homotopic to $L_\ell$ intersects $\gamma$ at least once. In particular, given a decomposition of $\Sigma$ into 0 and 1-handles, we can take $\gamma$ to be the cocore of a 1-handle that $L_\ell$ runs over. Since the path $u_N(\bar L_\ell)$ runs over $L_\ell$ at least $N$ times, $u_N^{-1}(\gamma) \cap \bar L_\ell$ has at least $N$ points and we would like to show that a large number of the arcs in $u_N^{-1}(\gamma)$ starting at these points do not end on $\bar L_\ell$ (since there are finitely many Lagrangians, it will follow that there is some $m$ such that a large number of these arcs end on $\bar L_m$). Consider the path in $\bar L'_\ell$ in $D^2$ from $\bar p_{\ell+1}$ to $\bar p_{\ell}$ that follows $\bar L_\ell$ except that just before hitting any arc in $u_N^{-1}(\gamma)$ that has both endpoints on $\bar L_\ell$ it turns left and follows (a push off of) the arc before continuing along $\bar L_\ell$. This path is clearly homotopic to $\bar L_\ell$ and $u_N(\bar L'_\ell)$, being homotopic to $u_N(\bar L_\ell)$, must intersect $\gamma$ at least $N$ times. For each intersection of $u_N(\bar L'_\ell)$ with $\gamma$ there is an arc in  $u_N^{-1}(\gamma)$ with exactly one end on $\bar L_\ell$, so there are at least $N$ of these as desired. In the case that $L_\ell$ is nullhomotopic we need a different way of choosing $\gamma$. In this case there must be a point $p$ in $\Sigma$ about which $\L_\ell$ has negative winding number, since otherwise it bounds an immersed disk violating admissibility; we choose $\gamma$ to be any arc passing through $p$.  We interpret the winding number as the signed intersection number of an arc going from this region to some fixed boundary or puncture of $\Sigma$, and assume that the ends of $\gamma$ approach this same boundary or puncture. The point $p$ divides $\gamma$ into two halves, and on either half there are more negative intersection points with $L_\ell$ than positive intersection points. The negative intersection points are the ones at which the image of an arc in $u_N^{-1}(\gamma)$ points away from $p$ along $\gamma$, and the positive intersection points are the ones at which the image of such an arc moves toward $p$. It follows that not at least one of the arcs from a negative intersection point moving away from $p$ do not stop at intersection point with $L_\ell$ and must end on some other $L_i$; this gives rise to an arc in $u_N^{-1}(\gamma)$ from $\bar L_\ell$ to $\bar L_i$. Since $u_N(\bar L_\ell)$ runs over $L_\ell$ at least $N$ times there are at least $N$ such arcs.

Assume we have an arc $\gamma$ as described above. Let $\bar \gamma^1, \bar \gamma^2, \ldots, \bar \gamma^N$ denote a set of arcs in $u_N^{-1}(\gamma)$ that have one endpoint on $\bar L_\ell$ and one endpoint on $\bar L_m$, and let $\bar x^i_\ell$ and $\bar x^i_m$ denote the endpoints of $\bar \gamma^i$ on $\bar L_\ell$ and $\bar L_m$, respectively (see Figure \ref{fig:finitely-many-polygons}). Since there are finitely many intersections of $\gamma$ with both $L_\ell$ and $L_m$, for $N$ large enough we can find distinct arcs $\bar \gamma^i$ and $\bar \gamma^j$ so that $u_N(\bar x_\ell^i) = u_N(\bar x_\ell^j)$ and $u_N(\bar x_m^i) = u_N(\bar x_m^j) $. Note that $u_N(\bar \gamma^i)$ and $u_N(\bar \gamma^j)$ are homotopic, since they both follow the unique path in $\gamma$ connecting $u_N(\bar x_\ell^i)$ to $u_N(\bar x_m^i)$.  Restricting $u_N$ to the part of $D^2$ between the two arcs $\bar \gamma^i$ and $\bar \gamma^j$ defines an immersed rectangle with sides on $L_\ell$, $\gamma$, $L_m$, and two identical opposite sides on $\gamma$; identifying the two $\gamma$ sides of this rectangle forms an immersed annulus bounded by $L_\ell$ and $L_m$. This contradicts the assumption that the curves are in admissible position.
\end{proof}

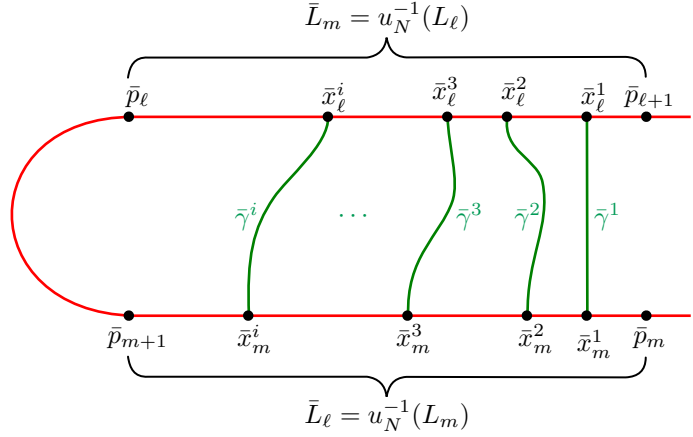
\begin{figure}
\labellist

  \pinlabel {$\bar p_{m+1}$} at 48 22
  \pinlabel {$\bar x^i_m$} at 92 22
  \pinlabel {$\bar x^3_m$} at 152 22
  \pinlabel {$\bar x^2_m$} at 198 22
  \pinlabel {$\bar x^1_m$} at 220 20
  \pinlabel {$\bar p_{m}$} at 241 22
  
  \pinlabel {$\bar p_{\ell}$} at 48 113
  \pinlabel {$\bar x^i_\ell$} at 122 113
  \pinlabel {$\bar x^3_\ell$} at 164 115
  \pinlabel {$\bar x^2_\ell$} at 190 115
  \pinlabel {$\bar x^1_\ell$} at 220 113
  \pinlabel {$\bar p_{\ell + 1}$} at 241 113

  \pinlabel {$\bar L_\ell =u_N^{-1}(L_m)$} at 142 -7
  \pinlabel {$\bar L_m = u_N^{-1}(L_\ell)$} at 142 143
  
  \color{ForestGreen}
  \pinlabel {$\bar \gamma^1$} at 225 68
  \pinlabel {$\bar \gamma^2$} at 195 68
  \pinlabel {$\bar \gamma^3$} at 173 68

  \pinlabel {$\cdots$} at 130 68
  \pinlabel {$\bar\gamma^i$} at 89 68

\endlabellist
\vspace{3mm}
 \includegraphics[scale=1]{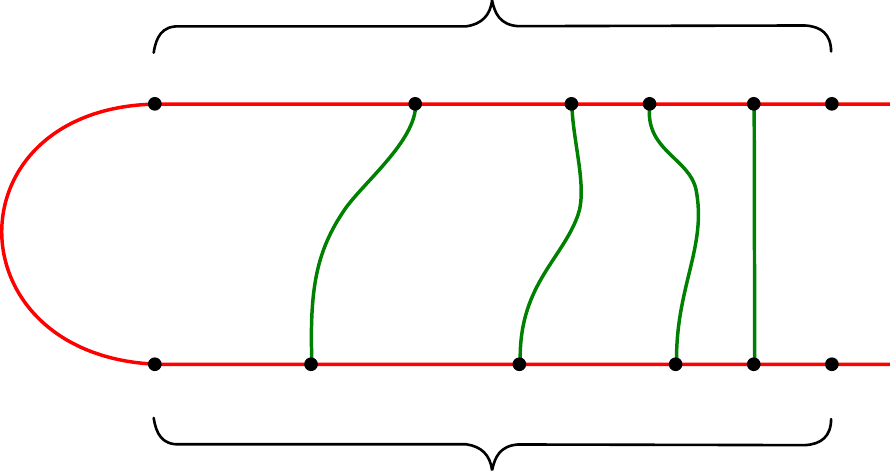}
 \vspace{3mm}
 \caption{A portion of the disk $D^2$ including the parts of $\partial D^2$ mapping to $L_\ell$ and $L_m$ and the arcs $\bar\gamma^i$ from $\bar x^i_\ell$ to $\bar x^j_m$.}
 \label{fig:finitely-many-polygons}
\end{figure}

Two (non-admissible) collections of curves in the infinite cylinder are shown in Figure \ref{fig:finitely-many-polygons2} that each bound an infinite number of 4-gons; local multiplicities of a particular 4-gon $u_N$ are indicated in the Figure. In each example $L_1$ plays the role of $L_\ell$ in the proof of Proposition \ref{prop:admissible}. The boundary of $u_N$ runs over both $L_1$ and $L_3$ $N$ times, forming a long strip; cutting this strip twice along the arc $\gamma$ and regluing shows that there is an immersed annulus bounded by $L_1$ and $L_3$. We remark that $\Sigma$ being non-closed is essential in the proof of Proposition \ref{prop:admissible}; on the right of Figure \ref{fig:finitely-many-polygons2} is a collection of three embedded curves which bound infinitely many triangles but which are admissible as no immersed annuli are present. It is still true that two Lagrangians are covered arbitrarily many times as the triangles grow and we can choose a curve $\gamma$ that intersects both such that for a large triangle there are arbitrarily many arcs in $u_N^{-1}(\gamma)$ connecting the preimages of $L_1$ and $L_0$; the difference is that here $\gamma$ is a closed curve so there is not a unique path between two points and the images of different arcs in $u_N^{-1}(\gamma)$ can not be identified.

\begin{figure}
 \includegraphics[scale=1.3]{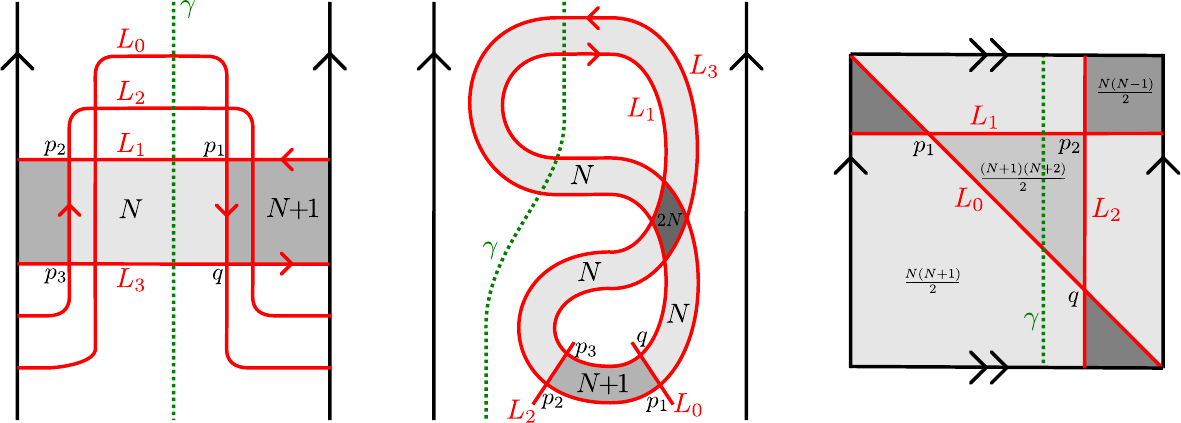}
 \vspace{3mm}
 \caption{(Left and middle) Arrangements of Lagrangians in the infinite strip for which there are infinitely many polygons (the local multiplicities of each region in the $N$th such polygon are indicated). In each case there is an immersed annulus, so the curves are not admissible. (Right) Three Lagrangians in the closed torus bounding infinitely many triangles, even though there are no immersed annuli.}
 \label{fig:finitely-many-polygons2}
\end{figure}

\begin{remark}
This combinatorial proof of Proposition \ref{prop:admissible} is new, to the author's knowledge. But in the symplectic setting there is a simpler argument relying on the fact that Lagrangians in a non-closed surface are exact. Alternatively, one can define Floer homology with coefficients in a Novikov ring and keep track of the symplectic area of polygons; this approach works for closed surfaces as well (see \cite{Abouzaid:surfaces}).
\end{remark}

Given an immersed $(k+1)$-gon $u$ in $\mathcal{M}(p_1, \ldots, p_k, q)$, we define a sign $(-1)^{s(u)}$ by multiplying contributions from each corner. The contribution of a corner $p$ in $L_i \cap L_j$ with $i < j$ is $-1$ if and only if the corner has odd grading and the orientation of $L_i$ opposes the boundary orientation of the polygon. In particular, the corner $p_i$ in $L_i \cap L_j$ contributes $-1$ if and only if the orientations of $L_i$ and $L_j$ both oppose the boundary orientation of the polygon, while the corner $q$ contributes $-1$ if and only if the orientation on $L_0$ opposes the boundary orientation and the orientation on $L_k$ agrees with the boundary orientation.

An immersed polygon is allowed to cover the marked points in $\Sigma$, but we will keep track of how many times this happens. Let $n_w(u)$ denote the total multiplicity with which the polygon $u$ covers the marked points in $\Sigma$. Finally, each immered polygon $u$ has weight $c(u)$ given by the product of the weight or the inverse of the weight associated with any basepoints of $L$ occurring on the boundary, where the inverse of the weight is used if the boundary orientation opposes the orientation of $L$ at the given basepoint. We now define:
\begin{equation}\label{eq:m_k}
m_k(p_1, \ldots, p_k) = \sum_{q \in L_0 \cap L_k} \sum_{u \in \mathcal{M}(p_1, \ldots, p_k, q)} (-1)^{s(u)} c(u) W^{n_w(u)} q.
\end{equation}
If the collection of Lagrangians is admissible then Proposition \ref{prop:admissible} ensures that the second sum above is finite for any $q$ in $L_0\cap L_k$, and the assumption that at most one of the Lagrangians is non-compact ensures that there are finitely many intersection points between any two Lagrangians, and in particular the first sum is finite.

Strictly speaking the $p_i$ and $q$ appearing in \eqref{eq:m_k} are intersection points in $L_{i-1}\cap L_i$ and $L_0 \cap L_k$, respectively, but these are identified with generators of $CF(L_{i-1}, L_i)$ and $CF(L_0, L_k)$, so the same formula defines the desired map
$$m_k: CF(L_0, L_1) \otimes \cdots \otimes CF(L_{k-1}, L_k) \to CF(L_0, L_k).$$
This definition makes sense for any $k > 0$. When $k = 0$ we also get a monogon counting map 
$$m_0: \{1\} \to CF(L_0, L_0) = CF(L_0).$$
We think of $m_0$ as a function with no inputs, so that $m_0()$ is an element of $CF(L_0)$. The map $m_0$ is defined by \eqref{eq:m_k} just like $m_k$, except that we need to be slightly more careful because in this case $q$ is a self-intersection point of $L_0$ and there are two generators of $CF(L_0)$ corresponding to each self-intersection point of $L_0$. These two generators have gradings of opposite parity, and we will declare that in this case the relevant generator is the one with even grading. That is, if we let $q^+$ denote the even grading generator of $CF(L_0)$ associated with the self-intersection point $q$, we define
$$m_0() = \sum_{q} \sum_{u \in \mathcal{M}(q)} c(u) W^{n_w(u)} q^+,$$
where the first sum is over self-intersection points $q$ of $L_0$. Recall that by admissibility there are finitely many self-intersection points. Note that the sign term $(-1)^{s(u)}$ from \eqref{eq:m_k} is always positive since the orientation on $L_0$ can not both agree with and oppose the boundary orientation of the monogon.

\begin{proposition}\label{prop:degree-of-m_k}
For $k \ge 0$, the degree of the map $m_k$ is $k-2$.
\end{proposition}
\begin{proof}
Fix an immersed polygon $u$ contributing a multiple of $q$ to $m_k(p_1, \ldots, p_k)$. Let $\tilde\tau_i: L_i \to \R$ denote the grading on $L_i$ and for $p$ in $L_i\cap L_j$, let $\theta(p_i)$ denote $\tfrac 1 \pi$ times the angle covered when turning counterclockwise from $L_{i-1}$ to $L_i$, and let $\theta(q)$ denote $\tfrac 1 \pi$ times the angle covered when turning counterclockwise from $L_0$ to $L_k$. Recall that $\gr(p_i) = \tilde\tau_{i}(p_i) - \tilde\tau_{i-1}(p_i) - \theta(p_i)$, and $\gr(q) = \tilde\tau_k(q) - \tilde\tau_0(q) - \theta(q)$, so that
$$\sum_{i=1}^k \gr(p_i) - \gr(q) = \sum_{i=0}^k \left[ \tilde\tau_i(p_i) - \tilde\tau_i(p_{i+1}) \right] + \theta(q) - \sum_{i=1}^k \theta(p_i), $$
where in the first sum on the right hand side we adopt the convention that $p_0 = p_{k+1} = q$. Let $\rot_i(u)$ denote the net rotation when traversing the $L_i$ portion of $\partial u$ (in radians, divided by $\pi$), and note that $\tilde\tau_i(p_i) - \tilde\tau_i(p_{i+1})$ is given by $\rot_i(u) - 2n_{w;i}(u)$ where $n_{w;i}(u)$ is the signed number of times the $L_i$ portion of $\partial u$ crosses a grading arc from a $w$ marked point. The sum $\sum_{i=0}^k n_{w;i} (u)$ is simply $n_w(u)$, the number of $w$ marked points enclosed by $u$. The total rotation when traversing $\partial u$, including the corners, which must be 2 because $u$ is a polygon, is given by
$$2 = \sum_{i=0}^k \rot_i(u) + \theta(q) + \sum_{i=1}^k (1 - \theta(p_i)).$$
It follows that
$$\sum_{i=1}^k \gr(p_i) - \gr(q) = \sum_{i=0}^k \left[ \rot_i(u) -2n_{w;i}(u)\right] + \theta(q) +  \sum_{i=1}^k \left[ 1- \theta(p_i) \right]- k = 2 - k - 2n_w(u). $$
Since the coefficient of $q$ in $m_k(p_1, \ldots, p_k)$ contains $W^{n_w(u)}$ and $\gr(W) = -2$, it follows that
$$\gr( m_k(p_1, \ldots, p_k)) - \sum_{i=1}^k \gr(p_i) = k - 2.$$
\end{proof}

An important property of the maps $m_k$ is that they satisfy the \emph{$\Ainfty$-relations}:
\begin{proposition}\label{prop:Ainfty-relations}
For any admissible collection of immersed curves $L_0, \ldots, L_k$ and any collection of intersection points $p_1,\ldots,p_k$ with $p_i$ in $L_{i-1}\cap L_i$.
$$\sum_{\ell =1}^k \sum_{j = 0}^{k-\ell} (-1)^\ast m_{k+1-\ell}(p_1, \ldots, p_j, m_\ell(p_{j+1}, \ldots, p_{j+\ell}), p_{j+\ell+1}, \ldots, p_k) = 0 $$
where $\ast = \sum_{i=j+\ell+1}^k 1 + \gr(p_i)$.
\end{proposition}
\begin{proof}
This is a standard argument (see, for instance, \cite[Lemma 3.6]{Abouzaid:surfaces}). We consider immersed $(k+1)$-gons with one obtuse corner; each of these can be cut in two different ways at the obtuse corner, and each resulting pair of polygons contributes to a term in the relation. Conversely, each pair of polygons contributing to a term in the relation arises in this way, as combining the polygons gives a polygon with one obtuse corner. The most subtle part of the argument is checking that signs and weights work out correctly.

The main difference here compared to \cite[Lemma 3.6]{Abouzaid:surfaces} is that we allow monogons contributing to $m_0$ operations, but this does not substantially affect the argument. Another difference is the presence of marked points, but it is clear that the two pairs of polygons arising from cutting a given polygon with an obtuse corner cover the same collection of marked points, and so the corresponding terms in the relation have the same power of $W$.
\end{proof}

Ultimately, we would like to view the  bigon counting map $m_1: CF(L_0,L_1) \to CF(L_0, L_1)$ as a differential and consider homology of the resulting complex. However, the $\Ainfty$-relations above show that this may not be possible. The relation with $k = 2$ is
$$ m_1(m_1(p_1)) = (-1)^{\gr(p_1)} m_2(m_0(), p_1) - m_2(p_1, m_0()), $$
so $m_0$ gives an obstruction to $(m_1)^2$ being zero. For this reason it is common to simply restrict to Lagrangians which do not bound immersed monogons, so that the map $m_0$ is always zero; this is the approach taken in \cite{Abouzaid:surfaces}. However, this turns out to be too restrictive an assumption for our purposes. Instead, we will see that a broader class of Lagrangians can be decorated, and the polygon counting maps modified, in order to eliminate this obstruction and define a differential.

\subsection{Turning points and bounding chains}

In order to define Floer homology, we will need to add decorations to our immersed Lagrangians; these decorations will take the form of linear combinations of self-intersection points. Given an immersed multicurve $L$, let $\sI$ denote the set of self-intersection points of $L$. We will only need to include self-intersection points with non-positive degree, and at times we will restrict to only points with degree zero; let $\sI_{\le 0}$ and $\sI_0$ denote the subsets of $\sI$ consisting of points with non-positive degree or degree zero, respectively.

\begin{definition}
A \emph{collection of turning points} on $L$ is a linear combination, over $\F$, of points in $\sI_{\le 0}$.
\end{definition}

The objects we will consider are now pairs $(L, \bchain)$ where $L$ is a weighted and graded immersed Lagrangian in $\Sigma$ and $\bchain$ is a collection of turning points. A key observation is that a collection of turning points $\bchain$ may be viewed as an element of $CF(L)$.

\begin{definition}
A collection of turning points $\bchain = \sum_{p\in \sI_{\le 0}} c_p p$ determines an element
$$\overline\bchain = \sum_{p\in \sI_{\le 0}} c_p W^{-\deg(p)/2} p^-$$
 of $CF(L)$, where $p^-$ is the odd grading generator of $CF(L)$ associated with the intersection point $p$.
\end{definition}

The powers of $W$ here are chosen so that $\overline\bchain$ is a homogeneous element of $CF(L)$ with grading $-1$. Conversely, any homogeneous grading $-1$ element of $CF(L)$ determines a collection of turning points by taking the coefficients of all $p^-$ for $p$ in $\sI_{\le 0}$ and ignoring the power of $W$. 

\begin{remark}
For any homogeneous grading $-1$ element of $CF(L)$, the coefficient of $p^{-}$ must be zero for any $p$ in $\sI$ with positive degree; this is because $\gr(p^-) = -1 - \deg(p) < -1$ and we do not allow negative powers of $W$. However, such an element of $CF(L)$ may include the odd grading generators associated with a basepoint on $L$. A homogeneous grading $-1$ element of $CF(L)$ is therefore equivalent to a linear combination (over $\F$) of points in $\sI_{\le 0} \cup \sJ$, where $\sJ$ denotes the set of basepoints on $L$. A linear combination of basepoints in $\sJ$ with nonzero coefficients is the same as a weighting of $L$. Thus a collection of turning points on $L$ along with the weights on $L$ determines a homogeneous grading $-1$ element of $CF(L)$. In view of this, it might make sense to work with unweighted curves and include the weighting along with the collection of turning points as one decoration by an element of $CF(L)$. However, since the generators of $CF(L)$ coming from basepoints behave differently from the other generators in a few places, we have chosen to keep these decorations separate.
\end{remark}

Given a collection of immersed multicurves with collections of turning points $(L_0, \bchain_0), \ldots, (L_k, \bchain_k)$, we can define a modified polygon counting operation
$$m_k^\bchain: CF(L_0, L_1) \otimes \cdots \otimes CF(L_{k-1}, L_k) \to CF(L_0, L_k)$$
by allowing arbitrarily many copies of the turning points to be inserted between the corners. That is, for points $p_1, \ldots, p_k$ with $p_i$ a generator of $CF(L_{i-1}, L_k)$ we seek to define

$$m_k^\bchain(p_1, \ldots, p_k) = \sum_{n_0, \ldots, n_k \ge 0} m_{k+n_0+\cdots +n_k}( \underbrace{\overline\bchain_0, \ldots, \overline\bchain_0}_{n_0}, p_1, \underbrace{\overline\bchain_1, \ldots, \overline\bchain_1}_{n_1}, p_2, \ldots, p_k, \underbrace{\overline\bchain_k, \ldots, \overline\bchain_k}_{n_k}).$$

We can interpret $m_k^\bchain(p_1, \ldots, p_k)$ as counting polygons with at least $k+1$ corners, where $k$ of the corners are the intersection points corresponding to $p_1, \ldots, p_k$, one is an intersection point in ${L_0 \cap L_k}$ corresponding to some output generator $q$ of $CF(L_0,L_k)$, and the remaining corners are self-intersection points of $L_i$ for some $i$ whose coefficient is nonzero in $\bchain_i$. We will refer to corners of the third type as \emph{false corners}. Strictly speaking, a false corner at a self-intersection point $p$ in $L_i$ should be viewed as the odd grading generator $p^-$ of $CF(L_i, L_i')$ associated with $p$. Each such polygon contributes to $m_k^\bchain(p_1, \ldots, p_k)$ with a weight in $\F[W]$ given by the product of the coefficient of $p^-$ in $\overline\bchain_i$ for each false corner at an intersection point $p$. Since each false corner is an odd grading element of $CF(L_i)$, the orientation on $L_i$ is consistent at the corner in the sense that it either agrees with the boundary orientation of the polygon before and after the corner or it opposes the boundary orientation before and after the corner; note that in the latter case the false corner contributes $-1$ to the sign of the polygon.

To formalize the view of counting polygons with false corners, let a \emph{polygonal path} in $L$ be a piecewise smooth path in $L$ in which adjacent smooth sections are connected by a left turn at a self-intersection point of $L$ and for which the path either always follows or always opposes the orientation on $L$. Given a collection of turning points $\bchain$ for $L$, we say that a polygonal path is \emph{consistent with $\bchain$} if it only has left turns at self-intersection points that have nonzero coefficient in $\bchain$. A polygon counted in $m^\bchain_k(p_1, \ldots, p_k)$ can be thought of as the image of a $(k+1)$-gon whose corners map to $p_1, \ldots, p_k$ and $q$ and whose $i$th side maps to a polygonal path in $L_i$ consistent with $\bchain_i$. Keeping with the terminology introduced above, we will refer to the left turns along the polygonal paths as \emph{false corners} of the polygon.

\begin{definition} Given intersection points $p_i$ in $L_{i-1} \cap L_{i}$ for $1 \le i \le k$ and an intersection point $q$ in $L_0 \cap L_k$, a \emph{generalized immersed $(k+1)$-gon with corners $p_1, \ldots, p_k$, and $q$} is an orientation preserving map
$$u: (D^2, \partial D^2) \to (\Sigma, L_0 \cup \cdots \cup L_k)$$
with the following properties:
\begin{itemize}
\item $u(\bar q) = q$ and $u(\bar p_{i}) = p_i$ for $1 \le i \le k$, where $\bar q, \bar p_1, \ldots, \bar p_k$ are fixed points on $\partial D^2$ appearing in clockwise order.
\item $u$ maps the segment of $\partial D^2$ from $\bar p_{i+1}$ to $\bar p_{i}$ to a polygonal path in $L_i$ consistent with $\bchain_i$ from $p_{i+1}$ to $p_i$, for $1 \le i < k$, and the segments from $\bar p_1$ to $\bar q$ and from $\bar q$ to $\bar p_k$ are mapped to polygonal paths in $L_0$ and $L_k$ consistent with $\bchain_0$ and $\bchain_k$, respectively;
\item $u$ is an immersion away from the points $\bar q$, $\bar p_i$, and the preimages of any left turns on the polygonal paths in the boundary; and
\item the points $q$ and $p_i$ are convex corners of the image of $u$.
\end{itemize}
\end{definition}

We define $\mathcal{M}^{\bchain}(p_1, \ldots, p_k, q)$ to be the set of equivalence classes of generalized immersed polygons with corners $p_1, \ldots, p_k$ and $q$. A polygonal path in $L_i$ consistent with $\bchain_i$ has a weight in $\F[W]$ that is the product of a contribution for each left turn along the path and each basepont of $L$ passed. The contribution of a left turn at an intersection point $p$ is the coefficient of $p^-$ in $\overline\bchain_i$ (recall that this is the coefficient of $p$ in $\bchain_i$ times $W^{-\deg(p)/2}$) if the polygonal path follows the orientation on $L_i$ and $-1$ times the coefficient of $p^-$ in $\overline\bchain$ if the polygonal path opposes the orientation on $L_i$. As usual, the contribution of a basepoint is the weight associated to that basepoint or its inverse, depending on whether the path is following or opposing the orientation on $L_i$. We then define a weight $c(u)$ for a polygon $u$ in $\mathcal{M}^{\bchain}(p_1, \ldots, p_k, q)$ to be the product of weights of the $k+1$ polygonal paths making up the boundary of the polygon. With this notation, $m^{\bchain}_k$ can equivalently be defined by
$$m^{\bchain}_k(p_1, \ldots, p_k) = \sum_{q \in L_0 \cap L_k} \sum_{u \in \mathcal{M}^{\bchain}(p_1, \ldots, p_k, q)} (-1)^{s(u)} c(u) U^{n_w(u)} q,$$
where $(-1)^{s(u)}$ is the usual product of signs associated with the corners $p_1, \ldots, p_k$, and $q$.

We still need to enhance our notion of admissibility to ensure that only finitely many polygons in $\mathcal{M}^{\bchain}(p_1, \ldots, p_k, q)$ contribute; this entails prohibiting certain generalized immersed disks and annuli. Just as $m_k^{\bchain}$ counts polygons which may have any number of false corners inserted on each side, we will consider immersed disks and annuli in which false corners may appear on each boundary component. More precisely, a \emph{generalized immersed disk} bounded by $(L_i, \bchain_i)$ is a map $u: D^2 \to \Sigma$ taking $\partial D^2$ to a closed polygonal path in $L_i$ consistent with $\bchain_i$ that is an immersion away from the preimages of the false corners on the boundary. If $A$ is the annulus $1\le z \le 2$ in the complex plane, a \emph{generalized immersed annulus} bounded by $(L_i, \bchain_i)$ and $(L_j, \bchain_j)$ is a map $u:A \to \Sigma$ that takes the outer boundary to a closed polygonal path in $L_i$ consistent with $\bchain_i$ and the inner boundary to a closed polygonal path in $L_j$ consistent with $\bchain_j$, such that $u$ is an immersion except at the preimages of the false corners.

\begin{definition}\label{def:admissible-enhanced}
A collection of immersed Lagrangians with turning points $(L_0, \bchain_0), \ldots, (L_k, \bchain_k)$ in $\Sigma$ is \emph{admissible} if the immersed multicurves $L_i$ are pairwise transverse, no decorated multicurve bounds a generalized immersed disk, and no two decorated multicurves bound a generalized immersed annulus. We also assume all but at most one of $L_0, \ldots, L_k$ are compact, and any non-compact Lagrangians have finitely many self-intersection points.
\end{definition}

\begin{proposition}\label{prop:admissible-enhanced}
If $(L_0, \bchain_0), \ldots, (L_k, \bchain_k)$ are admissible, then for any intersection points $p_{1}, \ldots, p_{k}$, and $q$ as above the set $\mathcal{M}^\bchain(p_1, \ldots, p_k, q)$ is finite. Thus the operation $m^\bchain_k$ is well-defined.
\end{proposition}
\begin{proof}
The proof is essentially the same as Proposition \ref{prop:admissible}. The main difference is that there may be more than one closed path in $L$ and an arbitrarily long path in $L_\ell$ may not cover every point of $L_\ell$ many times. But if we consider infinitely many paths in $L_\ell$ then there must be a path covering some point arbitrarily many times; we choose $\bar p_\ell$ to be a preimage of this point, and the proof is otherwise the same.
\end{proof}

For any choice of turning points $\bchain_i$ on each immersed multicurve $L_i$, the modified polygon maps $m^\bchain_k$ still satisfy the $\Ainfty$ relations.

\begin{proposition}\label{prop:Ainfty-relations-enhanced}
For any admissible collection of immersed multicurves with turning points \\ $(L_0, \bchain_0), \ldots, (L_k, \bchain_k)$ and any collection of intersection points $p_1,\ldots,p_k$ with $p_i$ in $L_{i-1}\cap L_i$.
$$\sum_{\ell =1}^k \sum_{j = 0}^{k-\ell} (-1)^\ast m^\bchain_{k+1-\ell}(p_1, \ldots, p_j, m^\bchain_\ell(p_{j+1}, \ldots, p_{j+\ell}), p_{j+\ell+1}, \ldots, p_k) = 0 $$
where $\ast = \sum_{i=j+\ell+1}^k 1 + \gr(p_i)$.
\end{proposition}
\begin{proof}
This follows from the usual proof of the $\Ainfty$ relations. The sum in the relation counts pairs of polygons such that the first contributes some $q$ to
$$m_{\ell+n}( \vec\bchain_j, p_{j+1}, \vec\bchain_{j+1}, p_{j+2}, \ldots, p_{j+\ell}, \vec\bchain_{j+\ell}),$$
where here $\vec\bchain_i$ represents a sequence of any number of false corners in $\bchain_i$ and $n$ is the total number of such false corners, and the second contributes $q'$ to
$$m_{k+1-\ell+n'}( \vec\bchain_0, p_1, \vec\bchain_1, p_2, \ldots, p_j, \vec\bchain_j, q, \vec\bchain_{j+\ell}, p_{j+\ell}, \ldots, p_k, \vec\bchain_k ),$$
where again the $\vec\bchain_i$'s stand in for a total of $n'$ false corners. The union of these two polygons forms a polygon with one obtuse corner. This polygon has corners at $q$ and $p_1$, \ldots, $p_k$, as well as false corners, and the obtuse corner may be of any type. Cutting in two ways at the obtuse corner gives two polygons contributing to the sum on the left side of the relation, and it can be checked as before that these contribute with opposite weight.
\end{proof}

Although we will consider more general collections of turning points as intermediate objects in our proof, we will ultimately be interested in collections of turning points satisfying an additional constraint.

\begin{definition}\label{def:bounding-chain}
Given an immersed multicurve $L$, a \emph{bounding chain} for $L$ is is a collection of turning points $\bchain$ (or by slight abuse of notation its corresponding element $\overline\bchain \in CF(L)$) such that
$$ m^\bchain_0 = \sum_{k\ge 0} m_k(\underbrace{\overline\bchain, \ldots, \overline\bchain}_{k}) = 0.$$
\end{definition}
This constraint on $\bchain$ is known as the \emph{Maurer-Cartan equation}. The motivation for this constraint should now be clear: the $\Ainfty$ relations imply that $m_0$ is an obstruction to $m_1$ being a differential---since the same relations hold for the modified operations $m^\bchain_k$, $m^\bchain_0$ is an obstruction to $m^\bchain_1$ being a differential.

\begin{proposition}
If $(L_0, \bchain_0)$ and $(L_1, \bchain_1)$ are immersed multicurves decorated with bounding chains, then
$$m^\bchain_1: CF(L_0, L_1) \to CF(L_0, L_1)$$
is a differential.
\end{proposition}
\begin{proof}
This is immediate from the $\Ainfty$ relations and the fact that $m^\bchain_0 = 0$.
\end{proof}

\begin{definition}
Given two immersed multicurves $L_0$ and $L_1$ decorated with bounding chains $\bchain_0$ and $\bchain_1$, the \emph{Floer chain complex $CF\left( (L_0, \bchain_0), (L_1, \bchain_1) \right)$} is defined to be the graded complex $CF(L_0, L_1)$  equipped with the differential $\partial = m^\bchain_1$.
\end{definition}

If we ignore marked points, $CF\left( (L_0, \bchain_0), (L_1, \bchain_1) \right)$ becomes a chain complex over $\F$. In this case it makes sense to take its homology, which we denote $HF\left( (L_0, \bchain_0), (L_1, \bchain_1) \right)$, but in general we will view $CF\left( (L_0, \bchain_0), (L_1, \bchain_1) \right)$ as a chain homotopy equivalence class of graded chain complexes over $\F[W]$.

\subsection{Invariance of Floer homology}\label{sec:invariance-of-Floer-homology}

We will now investigate the effect of applying a homotopy to an immersed curve with bounding chain $(L, \bchain)$. In particular, we will show that the Floer chain complex is invariant, up to chain homotopy equivalence, under such homotopies. We will first explain what we mean by a homotopy of a pair $(L, \bchain)$; in short, this is simply a regular homotopy of the immersed multicurve $L$ with a corresponding effect on the bounding chain $\bchain$.

We consider regular homotopies of $L$ that do not pass through any marked points. We allow any regular homotopy of $L$ that does not remove a self intersection point with nonzero coefficient in $\bchain$. In some cases homotopies which remove self intersection points with nonzero coefficients in $\bchain$ are also allowed, as in Figure \ref{fig:invariance-moves}(j). Note that in this situation admissibility implies that at most one of the two intersection points has nonzero coefficient in $\bchain$, since otherwise the bigon in question would form a generalized immersed disk. We emphasize that removing two intersections as in the reverse of move (i) is only possible when neither intersection point has nonzero coefficient in $\bchain$.

Any regular homotopy of $L$ can be decomposed as the composition of smaller homotopies, each of which can be realized either by an ambient isotopy of $\Sigma$ (fixing the marked points), by a local move sliding a piece of $L$ past a self intersection point or basepoint, as in (f) and (g) of Figure \ref{fig:invariance-moves}, or by a local move adding or removing two intersection points as in (i) or (j) of Figure \ref{fig:invariance-moves}. By a regular homotopy of a decorated immersed multicurve $(L, \bchain)$ we mean the composition of a sequence of these small homotopies of $L$ along with the corresponding change to $\bchain$ as indicated in Figure \ref{fig:invariance-moves} (for ambient isotopies of $\Sigma$ the bounding chain $\bchain$ is unchanged).

Consider the Floer complex associated with a pair of decorated immersed multicurves $(L_0, \bchain_0)$ and $(L_1, \bchain_1)$. We will show invariance under homotopy of $(L_0, \bchain_0)$ while keeping $(L_1, \bchain_1)$ fixed; homotopies of $(L_1, \bchain_1)$ follow similarly. As discussed above, any homotopy of $(L_0, \bchain_0)$ can be broken into small steps which either can be realized by an ambient isotopy of the surface (it is clear that these steps preserve the complex exactly) or have one of the forms shown in Figure \ref{fig:invariance-moves}.

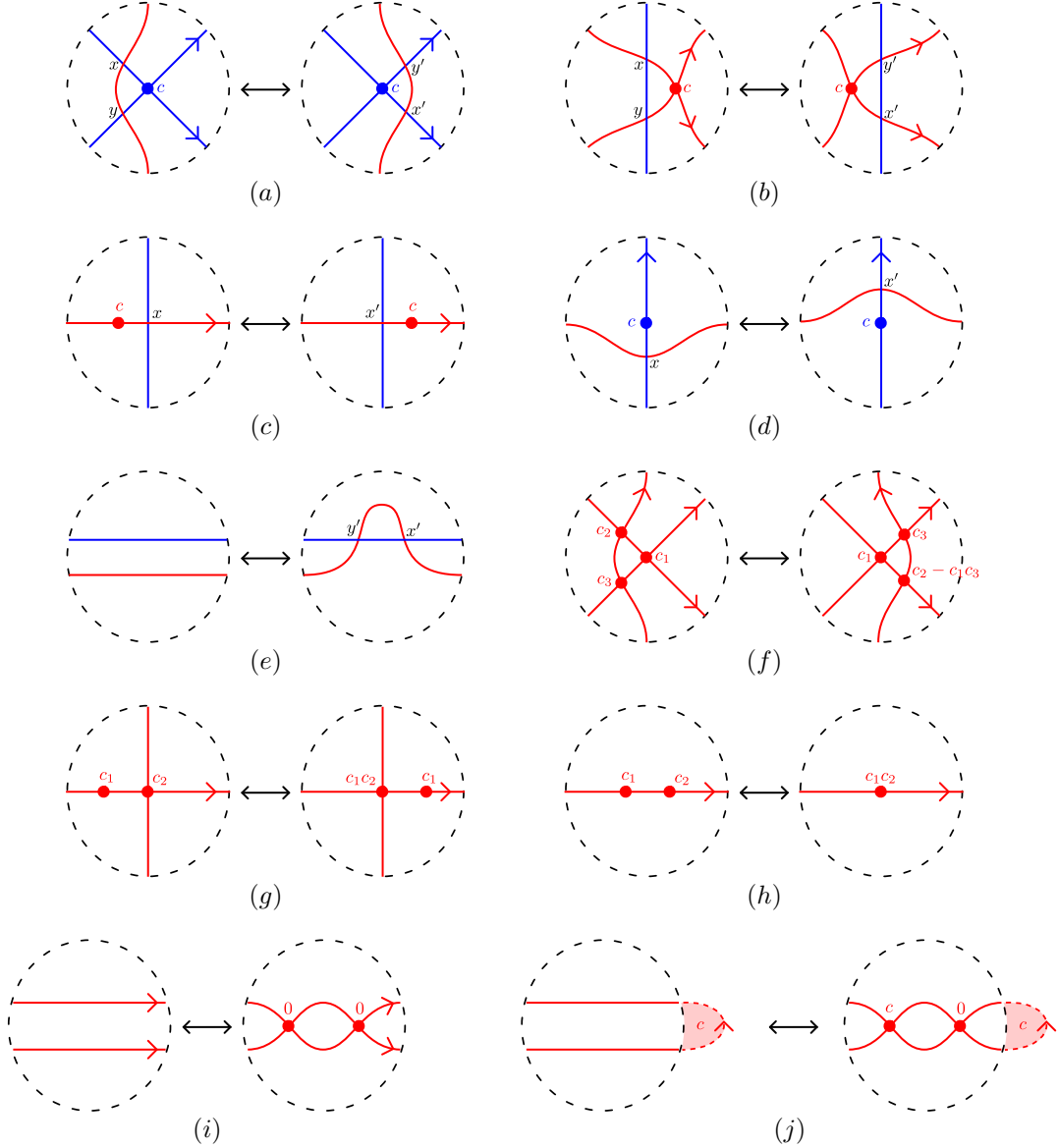
\begin{figure}
\labellist
  \pinlabel {$(a)$} at 132 470
  \pinlabel {$(b)$} at 387 470
  \pinlabel {$(c)$} at 132 350
  \pinlabel {$(d)$} at 387 350
  \pinlabel {$(e)$} at 132 230
  \pinlabel {$(f)$} at 387 230
  \pinlabel {$(g)$} at 132 110
  \pinlabel {$(h)$} at 387 110
  \pinlabel {$(i)$} at 102 -10
  \pinlabel {$(j)$} at 400 -10

\endlabellist
\includegraphics[scale=.75]{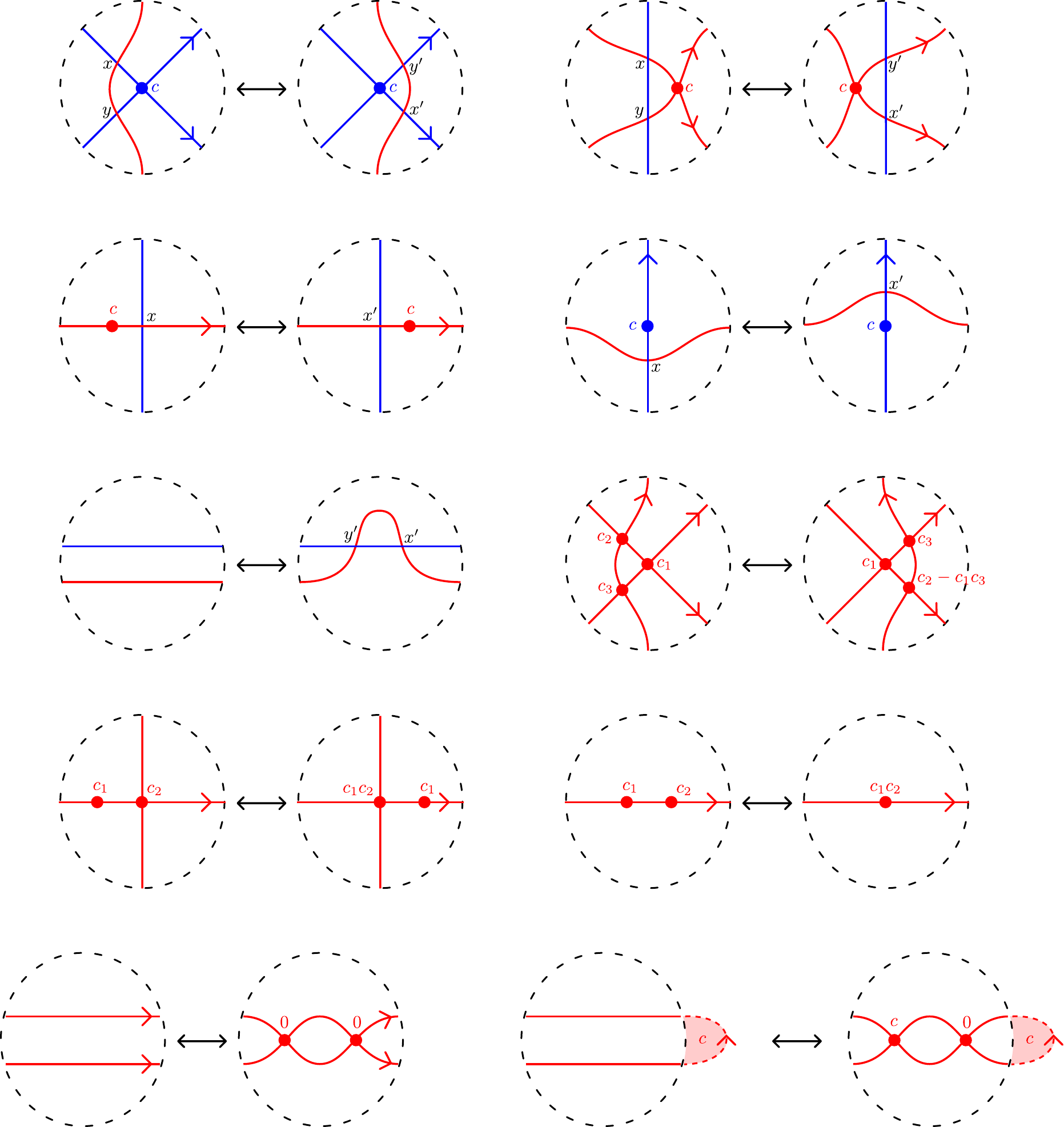}
\vspace{3 mm}
\caption{Given pair of decorated immersed multicurves $(L_0, \bchain_0)$ and $(L_1, \bchain_1)$, a homotopy of $(L_0, \bchain_0)$ can be broken into smaller homotopies such that each is either realized by an ambient isotopy of $\Sigma$ fixing $L_1$ or by one of the local modifications shown here. $L_0$ is red and $L_1$ is blue. Each self-intersection point on $L_0$ or $L_1$ is labeled by the coefficient of the corresponding generator of $\bchain_0$ or $\bchain_1$, and basepoints on $L_i$ are labeled with the corresponding weight. Orientations on arcs are included unless the orientation does not affect the argument for the given move. Move $(j)$ requires non-local information about the immersed curves outside of the specified disk; the red shaded bigon in the figure indicates the weighted sum of all paths in $L_0$ outside the disk that form a bigon with the given portion of the boundary of the disk has weight $c$.}
\label{fig:invariance-moves}
\end{figure}

\begin{proposition}
If $(L_0, \bchain_0)$ and $(L_1, \bchain_1)$ are modified by one of the local changes in Figure \ref{fig:invariance-moves} and the configuration is admissible before and after the change then the complex $CF\left( (L_0, \bchain_0), (L_1, \bchain_1) \right)$ is unchanged up to chain homotopy equivalence.
\end{proposition}
\begin{proof}
For each move there is a clear identification of generators of $CF\left( (L_0, \bchain_0), (L_1, \bchain_1) \right)$ before and after the move, with the exception of move (e) where the complex on the right has two additional generators. Let $D$ denote the pictured disk in which the local modification occurs. To understand how the moves affect the differential $m_1^\bchain$, we must consider immersed polygons which intersect $D$.

For move (a), it is clear that any polygon on the left that does not have $x$ or $y$ as a corner will be effectively unchanged by this move. It is also clear that the differential is unaffected by this move if the coefficient $c$ in $\bchain_1$ of the given self-intersection point of $L_1$ is zero. If the coefficient $c$ is nonzero, we must consider generalized immersed bigons which have a false corner at this self intersection point of $L_1$. Suppose there is a polygon contributing $z$ to $m_1^\bchain(x)$ on the left; this polygon must contain one of the triangles $u_1$ or $u_2$ shown in Figure \ref{fig:isotopy-invariance-proof}. If it contains $u_1$ then replacing $u_1$ with $v_1$ gives a polygon contributing $z$ to $m_1^\bchain(x')$ on the right, and replacing $u_1$ with $v_2$ gives another polygon contributing $c \cdot z$ to $m_1^\bchain(y')$ on the right. If th polygon contains $u_2$, then replacing $u_2$ with $v_3$ yields a polygon contributing $z$ to $m_1^\bchain(x')$ on the right. In this case, there is an additional polygon on the left contributing $-c\cdot z$ to $m_1^\bchain(y)$ obtained by replacing $u_2$ with $u_3$, and this has no analogous polygon on the right. All other polygons with initial corner $x$ or $y$ on the left or $x'$ or $y'$ on the right are in clear one-to-one correspondence. We see that (a) has the effect of adding $c$ times $m_1^\bchain(x)$ to $m_1^\bchain(y)$; that is, $m_1^\bchain(y')$ on the right is identified with $m_1^\bchain(y) + c \cdot m_1^\bchain(x)$ on the left and $m_1^b$ acts the same on all other generators. A similar argument shows that for each polygon contributing a multiple of $y$ to $m_1^\bchain(z)$ for some $z$, there is a corresponding contribution of $y' - cx'$ to $m_1^\bchain(z')$ on the right. It follows that move (a) has the effect of a change of basis; the two complexes are isomorphic, where the isomorphism identifies $y'$ with $y + cx$, $x'$ with $x$, and is the identity on all generators outside of $D$.

\begin{figure}
\labellist

  \scriptsize
  %(a)
  \pinlabel {$x$} at 23 55  
  \pinlabel {$y$} at 23 32  
  \pinlabel {$x$} at 121 55  
  \pinlabel {$y$} at 121 32  
  \pinlabel {$x$} at 218 55  
  \pinlabel {$y$} at 218 32  
  \pinlabel {$y'$} at 392 55  
  \pinlabel {$x'$} at 393 35  
  \pinlabel {$y'$} at 490 55  
  \pinlabel {$x'$} at 490 35  
  \pinlabel {$y'$} at 587 55  
  \pinlabel {$x'$} at 587 35  

  \color{blue}
  %(a)
  \pinlabel {$c$} at 50 45  
  \pinlabel {$c$} at 148 45  
  \pinlabel {$c$} at 245 45  
  \pinlabel {$c$} at 362 45
  \pinlabel {$c$} at 460 45
  \pinlabel {$c$} at 557 45

\endlabellist

\includegraphics[scale=.70]{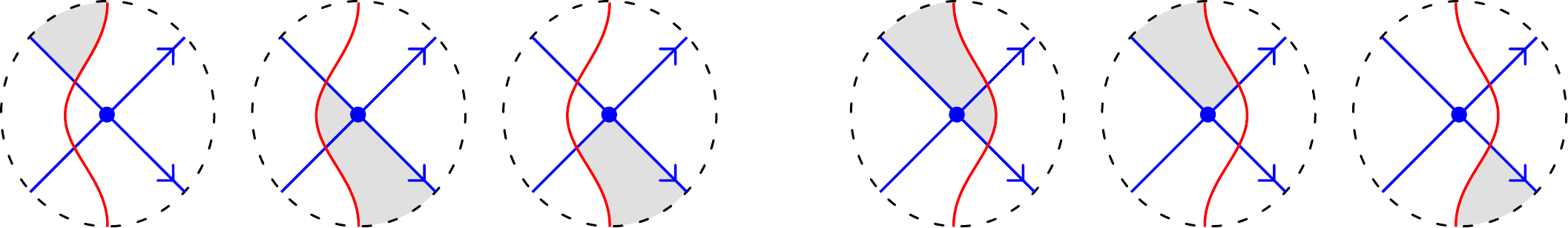}
$u_1$ \hspace{18.5 mm} $u_2$ \hspace{18.5 mm} $u_3$ \hspace{25mm} $v_1$ \hspace{18.5 mm} $v_2$ \hspace{18.5 mm} $v_3$
\caption{Portions of polygons relevant for relating $m_1^\bchain(y')$ to $m_1^\bchain(y)$ and $m_1^\bchain(x)$.}
\label{fig:isotopy-invariance-proof}
\end{figure}

Move (b) is similar to move (a) and we omit the details. Once again the move produces an isomorphic complex with the isomorphism corresponding to a change of basis identifying $x'$ with $x - cy$ and $y'$ with $y$.

In moves (c) and (d) there is a one-to-one correspondence not only of generators but also of polygons contributing to the differential, but the weights of polygons involving the intersection point in $D$ are affected. As a result of the change, any generalized bigon with initial generator $x$ either gains the basepoint on its boundary with positive orientation or loses the marked point with negative orientation; in either case, the weight of the polygon is multiplied by $c$. Any polygon with $x$ as the final generator either gains the basepoint with positive orientation or looses it with negative basepoint, so the weight is multiplied by $c^{-1}$. For either move, the Floer complexes are isomorphic and the isomorphism identifies $cx$ with $x'$.

Move (e) introduces two new intersection points, but there is a clear bigon on the right contributing a differential from $x'$ to $y'$. We claim that the complex on the left is precisely the result of canceling this differential in the complex on the right, and thus the complexes are chain homotopy equivalent. First note that there can be no other polygon contributing $y'$ to $m_1^\bchain(x')$, since the portion of such a polygon outside of $D$ combined with the strip between $L_0$ and $L_1$ in $D$ on the left would produce a generalized immersed annulus, implying that the curves were not in admissible position on the left side. To cancel the differential from $x'$ to $y'$, for each pair of polygons contributing $ay'$ to $m_1^\bchain(w)$ and $bz$ to $m_1^\bchain(x')$ we must add $abz$ to $m_1^\bchain(w)$. Such polygons would approach $y'$ from the left and leave $x'$ on the right, and it is clear that when the the fingermove in $D$ is reversed these two polygons form a single polygon connecting $w$ to $z$ whose weight is the product of the weights of the two polygons. Since all other bigons are unchanged, the claim follows.

Move (f) is similar to moves (a) and (b) and move (g) is similar to moves (c) and (d), but in each case the complex is unchanged as a result of the changes to the bounding chain pictured. Combining two weighted base points into one as in move (h) clearly has no effect on the complex. Move (i) has no effect on the complex; since the two new self intersection points on the right are added to the bounding chain with coefficient 0, we do not consider polygons with corners at these intersection points and there is a one-to-one correspondence of polygons, which preserves weights, before and after the move.

Move (j) is the only move which requires knowledge of how $(L_0, \bchain_0)$ behaves outside of $D$. The move performs a finger move that introduces two new self-intersection points of $L_0$ in $D$, but in some cases we need to add one of these to $\bchain_0$ with nonzero coefficient. In particular, if the two right endpoints are connected by a polygonal path in $(L_0, \bchain_0)$ that forms an immersed polygon with the relevant segment of the boundary of $D$ and if this polygonal path (or, more precisely, the sum of all such paths) has weight $c$, then we must add the leftmost of the two new self-intersection points to $\bchain_0$ with coefficient $c$. This is clearly necessary for $\bchain_0$ to remain a bounding chain, since the path outside of $D$ contributes $c$ times the other new intersection point to $m_0^\bchain$ and this can only be canceled by the newly formed bigon in $D$. 
Having made this change, it is straightforward to check that the complex is unchanged. For instance, if a polygon on the left intersects $D$ in the rectangular strip between the two strands of $L_0$ and also contains the $c$ weighted path outside $C$ to the right, then the finger move destroys this polygon but creates a new truncated one with a false corner at the $c$ weighted self-intersection point in $D$, and the contribution of this new polygon to $m_1^\bchain$ is precisely the same.
\end{proof}

\subsection{Doubly marked surfaces}\label{sec:doubly-marked}

So far we have assumed that all marked points on a surface are treated the same. In fact, we are primarily interested in surfaces with two families of basepoints, a collection $\{w_i\}_{i\in I}$ of \emph{$w$-marked points} and a collection $\{z_i\}_{i\in I}$ of \emph{$z$-marked points}. We call such a surface \emph{doubly marked} (even if there are infinitely many marked points) because there are two distinct types of marked points.

It is straightforward to extend the definitions in the rest of this section to the doubly marked setting. Instead of $W$ we have two formal variables $U$ and $V$, with the variable $U$ associated with the $w$-marked points and the variable $V$ associated with the $z$-marked points. The relevant coefficients are $\sRminus = \F[U, V]$ rather than $\F[W]$. An immersed Lagrangian $L$ will be equipped with a bigrading $(\tilde\tau_w, \tilde\tau_z)$, where $\tilde\tau_w: L \to \R$ is a grading in the usual sense with respect to the $w$-marked points only (that is, it has discontinuities at intersections of $L$ with grading arcs emanating from the $w$-marked points but not at grading arcs coming from $z$-marked points) and $\tilde\tau_z: L \to \R$ is a grading with respect to only the $z$-marked points. Each self-intersection point of $L$ has a bidegree $(\deg_w, \deg_z)$, where each degree is defined as usual with respect to the relevant grading.

Given two Lagrangians $L_0$ and $L_1$, $CF(L_0, L_1)$ is generated over $\sRminus$ by intersections of $L_0$ with (a suitable perturbation of) $L_1$. Bigradings on $L_0$ and $L_1$ induce a bigrading $(\gr_w, \gr_z)$ on the $\sRminus$-module $CF(L_0, L_1)$. The operations $m_k$ count immersed $(k+1)$-gons, where the coefficient in $\sRminus$ records the number of times each type of marked point is covered. More precisely, given Lagrangians $L_0, \ldots, L_k$ and intersection points $p_1, \ldots, p_k$ with $p_i$ in $L_{i-1}\cap L_{i}$, we have
\begin{equation}\label{eq:m_k_doubly_marked}
m_k(p_1, \ldots, p_k) = \sum_{q \in L_0 \cap L_k} \sum_{u \in \mathcal{M}(p_1, \ldots, p_k, q)} (-1)^{s(u)} U^{n_w(u)} V^{n_z(u)} q,
\end{equation}
where $n_w(u)$ and $n_z(u)$ are the numbers of times $u$ covers $w$-marked points and $z$-marked points, respectively. Since ignoring either family of marked points returns us to the singly marked setting, Proposition \ref{prop:degree-of-m_k} implies that the map $m_k$ has bidegree $(k-2, k-2)$. 

A collection of turning points $\bchain$ on an immersed Lagrangian $L$ in a doubly marked surface is once again a linear combination (over $\F$) of self-intersection points in $\sI_{\le 0}$, where $\sI_{\le 0}$ is the subset of self-intersection points for which the degree is non-positive with respect to both gradings. This determines a homogenous element $\overline\bchain$ of bigrading $(-1, -1)$ which can be viewed as a linear combination over $\sRminus$ of points in $\sI_{\le 0}$ whose coefficient for a given self-intersection point $p$ is given by multiplying the corresponding coefficient of $\bchain$ by $U^{-\deg_w(p)/2} V^{-\deg_z(p)/2}$. We can define the modified map $m_k^\bchain$ just as before, except that coefficients include powers of $V$ as well as powers of $U$. The proof of the $\Ainfty$-relations (Proposition \ref{prop:Ainfty-relations}) is unchanged---for pairs of terms that are meant to cancel, the powers of $V$ agree for the same reason the powers of $U$ agree. As usual, $\bchain$ is a bounding chain if $m^\bchain_0 = 0$. In this case, $m^\bchain_1$ is a differential on $CF((L_0, \bchain_0), CF(L_1, \bchain_1))$, making this space into a bigraded chain complex over $\sRminus$.

\section{Train tracks and arrow slides}\label{sec:train-tracks}% !TEX root = ../CFKcurvesZHS3s.tex
%train-tracks.tex

\subsection{Immersed curves with turning points as train tracks}

There is an alternative perspective on immersed curves with collections of turning points that we find convenient to work with: we can think of such a decorated immersed multicurve as an immersed train track (this is the perspective taken in \cite{HRW}). We will now describe these objects and their relation to the curves with turning points discussed so far. As in the previous section we will first describe these objects in the setting of a singly marked surface, but the extension to doubly marked surfaces is straightforward.

A \emph{train track} refers to a graph with the property that at each vertex all incident edges are mutually tangent; vertices of a train track are often called \emph{switches}, in reference to the junctions in railroads. At each switch there are two (opposite) directions from which an edge may approach, and a smooth path in the train track is one which approaches and leaves each switch on opposite sides. In the train tracks we consider, all vertices will have valence at least two and will have at least one incident edge on each side; in other words, no switch will be a dead end for smooth paths (an exception will be switches on the boundary of the surface). We will consider partially directed train tracks in which some edges have a specified direction; when considering smooth paths in a train track, we will only consider paths that follow directed edges in the specified direction. Our train tracks will also be oriented, meaning that every edge has an orientation; the orientations are required to be consistent at each switch, in the sense that all edges on one side are oriented toward the switch and all edges on the other are oriented away from the switch; it follows that a smooth path must always follow the orientation or always oppose the orientation. The orientation should not be confused with the direction on directed edges; these need not agree, and we do not require smooth paths to follow the orientation. Finally, edges in our train tracks will be weighted by (nonzero) homogenous elements of $\F[W]$. For undirected edges the power of $W$ is zero, so the weight is just an element of the field $\F$. For directed edges the power of $W$ is determined by a grading on the train track as described below. In practice, the weights on most undirected edges will be 1, and these weights can be ignored; thus the weights on directed edges will be recorded by marking some number of basepoints on the train track (on the undirected edges with weight other than 1) and labeling these with a weight.

A grading on an immersed train track in a marked surface is nearly the same as a grading on an immersed Lagrangian. Having fixed a trivialization of the tangent bundle of the surface to identify tangent slopes with $\mathbb{RP}^1$, a grading is a map from the train track to $\R$ that lifts the tangent slope map and is piecewise smooth, possibly with jump discontinuities of even magnitude. These discontinuities are of two types: they occur whenever the train track crosses a chosen grading arc emanating from a marked point, as usual, and there is also a discontinuity at the midpoint of each directed edge. At the discontinuity along a directed edge the grading decreases by $2k$ for some integer $k$; we associate $W^k$ to this edge, so that the weight on this directed edge is $c W^k$ for some $c$ in $\F$.

An immersed multicurve $L$ with a collection of turning points $\bchain$ determines an immersed train track of the form described above as follows: First, we view the immersed multicurve as an immersed train track with only undirected edges by placing at least one valence two switch anywhere on each component. We place enough switches so that the weighted basepoints on the curves each lie on their own directed edge; these directed edges inherit the weight associated with the basepoint, and all other directed edges have weight 1. To this we add a pair of directed edges near each self intersection point $p$ that has nonzero coefficient in $\bchain$, as pictured in Figure \ref{fig:turning-point-to-train-track}. These directed edges make the two left turns at $p$ that are consistent with the orientation on $L$. The orientation on each directed edge is chosen to be consistent with the orientation on $L$ at either end; note that for one edge in the pair the orientation agrees with the direction of the edge and for the other edge the orientation opposes the direction. The new directed edges are weighted by $\pm c W^{-\deg(P)/2}$, where $c$ is the coefficient of $p$ in $\bchain$ and the minus sign appears on the edge for which the direction opposes the orientation. These directed edges are added precisely so that the false corners which may appear in generalized bigons contributing to $m_1^\bchain$ are smoothed in the train track. In other words, a polygonal path in $L$ consistent with $\bchain$ corresponds to a smooth path in the train track, and vice versa. Moreover, the weight associated with a polygonal path agrees with the product of the weights of all edges on the corresponding smooth path in the train track (where the weight of an undirected edge is inverted if the edge is traversed in the opposite direction of the orientation).

We will define a Floer complex associated to two immersed train tracks in much the same way we defined it for immersed curves: the complex is generated by intersection points and the differential counts immersed bigons whose sides map to smooth paths in the train track. With this in mind, we can view the Floer complex of two immersed curves with bounding chains as the Floer complex of the corresponding train tracks. Indeed, the additional edges in the train track precisely allow left turns to be made at intersection points in the bounding chains, and the bigons counted between two train tracks are precisely the generalized bigons counted between the two corresponding decorated curves with the false corners smoothed.

\subsection{The Floer complex of train tracks}

We now more precisely define what we mean by the Floer complex of two immersed train tracks. We caution that this notion is not well-defined for arbitrary train tracks, and we will not make an attempt to describe the biggest family of train tracks for which our definitions make sense. Instead, we will restrict our attention to a class of train tracks which can be directly related to immersed curves decorated with bounding chains and observe that in this case the two notions of Floer homology agree. The fact that Floer homology of these train tracks is well defined then follows from the corresponding fact for Floer homology of decorated curves. This means that, strictly speaking, it is not necessary to use train tracks at all since every train track we will consider really represents an immersed curve with a bounding chain; however, we find that train tracks are a very convenient framework to use when manipulating these objects.

Given a pair of train tracks $\tracks_i$ and $\tracks_j$ in a marked surface that intersect transversally, we define $CF(\tracks_i, \tracks_j)$ to be the $\F[U]$-module generated by $\tracks_i \cap \tracks_j$. Given $k+1$ immersed train tracks $\tracks_0, \ldots, \tracks_k$, we define operations
$$m_k^{\tracks}: CF(\tracks_0, \tracks_1) \otimes \cdots \otimes CF(\tracks_{k-1}, \tracks_k) \to CF(\tracks_0, \tracks_k)$$
by counting immersed $(k+1)$-gons with appropriate corners whose sides are smooth paths in the appropriate train tracks. The weight with which each $(k+1)$-gon contributes is the product of the weights of all edges in the boundary (where for undirected edges the inverse of the weight is used if the boundary orientation of the polygon opposes the orientation on the train track), a sign contribution from each corner defined in the usual way, and a factor of $W$ for each time the marked point is covered by the polygon.  The operation $m_0^\tracks$ counts immersed monogons with boundary on a given train track. We say that a train track $\tracks$ is unobstructed if $m_0^\tracks = 0$. The operations $m^\tracks_k$ satisfy $\Ainfty$-relations, and if $\tracks_0$ and $\tracks_1$ are both unobstructed then $\partial^\tracks = m^\tracks_1$ is a differential on $CF(\tracks_0, \tracks_1)$.

We say that a collection of train tracks are admissible if no train tracks bounds an immersed disk and no two train tracks bound an immersed polygon. We expect that the counts of immersed polygons used to define $m^\tracks_k$ are finite as long as the train tracks are admissible, though we will not prove this fact. Instead, as already mentioned, we will restrict our attention to particularly nice train tracks for which the finiteness of these polygon counts can be established indirectly. For instance, suppose for each $i$ that $\tracks_i$ is the train track constructed from an immersed multicurve with bounding chain $(\Gamma_i, \bchain_i)$ as described in the previous section, so that all directed edges appear in pairs as in the middle of Figure \ref{fig:turning-point-to-train-track}. Generically we may assume that $\Gamma_i \cap \Gamma_j$ is disjoint from the self intersection points of $\Gamma_i$ and $\Gamma_j$, and thus that $\tracks_i$ and $\tracks_j$ only intersect on their undirected edges. In this case there is an obvious identification between $\Gamma_i \cap \Gamma_j$ and $\tracks_i \cap \tracks_j$. For any collection of such curves and train tracks there is also a clear bijection between the $(k+1)$-gons contributing to $m^\tracks_k$ and the generalized $(k+1)$-gons contributing to $m^\bchain_k$. It follows that in this case the operations $m^\tracks_k$ are well defined and are in fact identical to the operations $m^\bchain_k$.

\begin{figure}
\labellist
  \pinlabel {$c$} at 51 51
  
  \scriptsize
  \pinlabel {$-cW^m$} at 199 30
  \pinlabel {$cW^m$} at 165 60
  \pinlabel {$-cW^m$} at 344 72
  \pinlabel {$cW^m$} at 327 57
  
  \pinlabel {degree} at 75 21
  \pinlabel {$-2m$} at 75 13

\endlabellist
\includegraphics[scale=1]{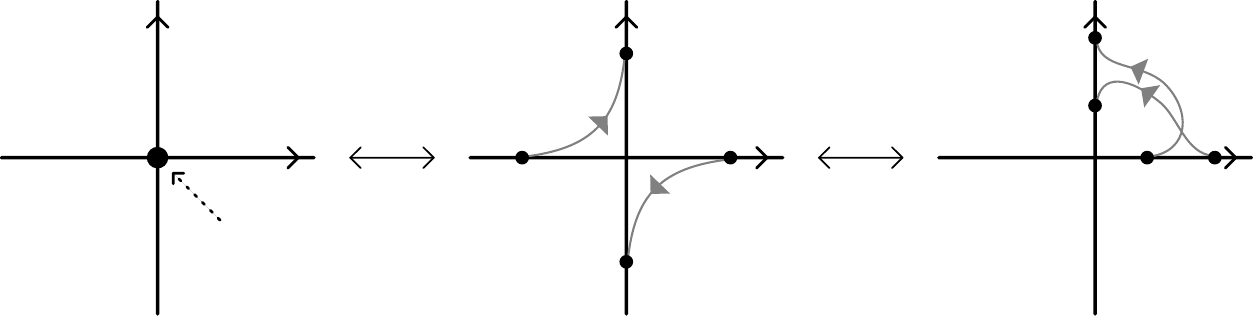}
\caption{Converting an immersed curve with a collection of turning points into an immersed train track. For each intersection point with nonzero coefficient $c$ and degree $-2m$ in the collection of turning points (left) we add two directed edges to the train track (middle) with the weights $\pm c W^m$. These new edges realize the false corners which may occur at the given intersection as smooth paths in the train track. We could equivalently place both directed edges in the train track in the same quadrant near the self-intersection point (right).}
\label{fig:turning-point-to-train-track}
\end{figure}
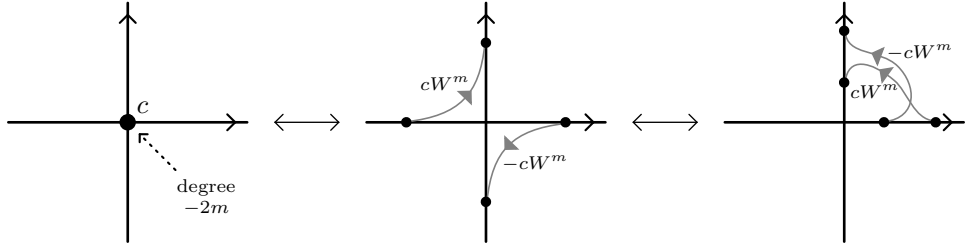

\begin{figure}
\labellist
  \scriptsize
  \pinlabel {$1$} at 38 53
  \pinlabel {$1$} at 38 -2

  \pinlabel {$1$} at 233 53
  \pinlabel {$1$} at 233 -2
  
  \color{gray}
  \pinlabel {$cW^m$} at 51 31
  \pinlabel {$-cW^m$} at 21 31
  \pinlabel {$cW^m$} at 135 22
  
  \pinlabel {$-cW^m$} at 249 31
  \pinlabel {$cW^m$} at 219 31
  \pinlabel {$cW^m$} at 330 22
 
\endlabellist
\includegraphics[scale=1]{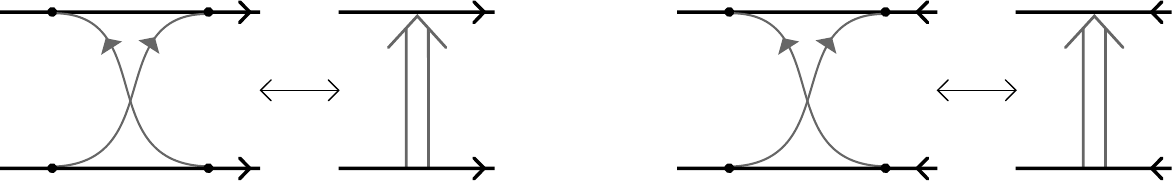}
\caption{Shorthand for crossover arrows. Note that the sign convention for crossover arrow labels is that the minus sign is on the weight of the directed edge whose direction opposes the orientation on the train track.}
\label{fig:crossover-arrow}
\end{figure}
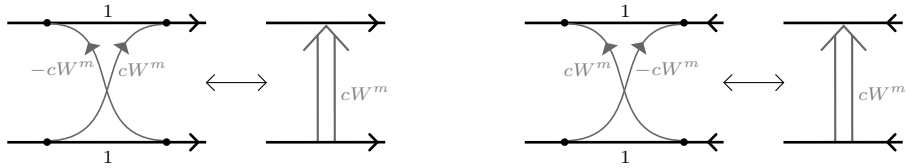

We will be interested in a slightly larger class of train tracks than those immediately constructed from immersed curves with turning points. In particular, we would like to be able to apply homotopies to train tracks of this form that slide the directed edges away from the self-intersection points. Invariance of the Floer complex quickly fails if we only slide one directed edge, but it turns out we can make more sense of homotopies that slide a pair of directed edges as a unit. In a train track representing an immersed curve with turning points, if we move both edges near a self-intersection point to one side of the self-intersection point, as on the right of Figure \ref{fig:turning-point-to-train-track}, they cross each other in an X-shaped pattern. We will refer to a pair of edges in this arrangement as a \emph{crossover arrow}, and as a diagrammatic shorthand we denote these by a bold arrow as in Figure \ref{fig:crossover-arrow}. The train tracks we will consider will have the property that the collection of undirected edges forms an immersed multicurve and the directed edges come in pairs forming crossover arrows. We will assume, as usual, that all self-intersection points of the train track are transverse; we will also assume that crossover arrows do not cross each other.

If a crossover arrow lies in a neighborhood of an intersection point of two undirected edges in a train track and moves counter-clockwise in one quadrant, as in the rightmost train track in Figure \ref{fig:turning-point-to-train-track}, we say that it is a \emph{left turn crossover arrow}. It is clear that an immersed curve with a collection of left turn crossover arrows is equivalent to the same immersed curves with a bounding chain. Train tracks which allow arbitrary crossover arrows are seemingly more general objects, but in fact any unobstructed train track that consists of an immersed multicurve with crossover arrows attached can be modified in a neighborhood of each crossover arrow to produce an equivalent train track which has only left turn crossover arrows. To accomplish this, we apply a finger move homotopy pushing the section of immersed curve at the tail of the the arrow forward along the crossover arrow through the section of curve at the head of the arrow. The crossover arrow then becomes a left turn crossover arrow at one of the new intersection points as in Figure \ref{fig:any-crossover-arrow-left-turn}.

Figure \ref{fig:any-crossover-arrow-left-turn} shows the simplest case in which the crossover arrow has no intersections with the rest of the train track apart form its endpoints. In general, a neighborhood of the crossover arrow will consist of the arrow, the segments it connects, and some number of other immersed curve segments transverse to the crossover arrow between the initial and final segments. For concreteness, suppose the crossover arrow moves upward vertically connecting two rightward oriented horizontal segments and all immersed curve segments through a rectangular neighborhood of the arrow are horizontal. In this case we must be conscious of paths in the train track outside the rectangular neighborhood of the crossover arrow that bound immersed bigons with the left or right side of the neighborhood. We first observe that for an unobstructed train track, there can be no such paths on the left side of the crossover arrow for which one end is the top segment or bottom segment (or, more precisely, the weighted count of such paths for any pair of endpoints is zero). If such a path did exist, there would be an immersed monogon with a corner on the crossover arrow, as shown in Figure \ref{fig:obstructed-arrows}, and there is no way for this to cancel with any other monogons. We say that a crossover arrow is \emph{unobstructed} if there are no bigons on the left side of the arrow as in Figure \ref{fig:obstructed-arrows}, and we need not consider train tracks with obstructed arrows. Similar paths on the right side of the arrow that bound bigons with the right side of the rectangular neighborhood are less problematic. However, any such paths starting at the bottom segment must be taken into account when performing the finger move. Following the train track analogue of move $(j)$ from Figure \ref{fig:invariance-moves}, pushing the bottom segment past the middle segments may require adding left turn crossover arrows at some of the new intersection points. The finger move transformation in this more general case is shown in Figure \ref{fig:unobstructed-arrow}.

\begin{figure}
\includegraphics[scale=1]{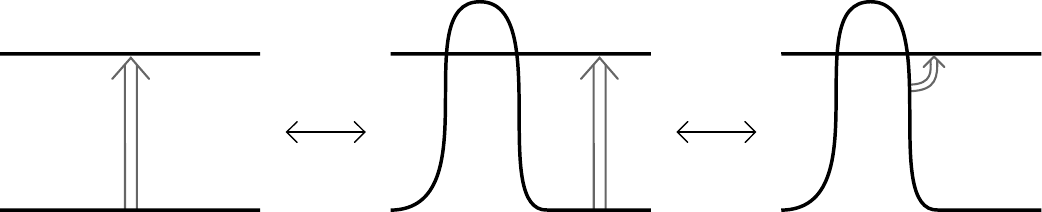}
\caption{}
\label{fig:any-crossover-arrow-left-turn}
\end{figure}

\begin{figure}
\includegraphics[scale=1]{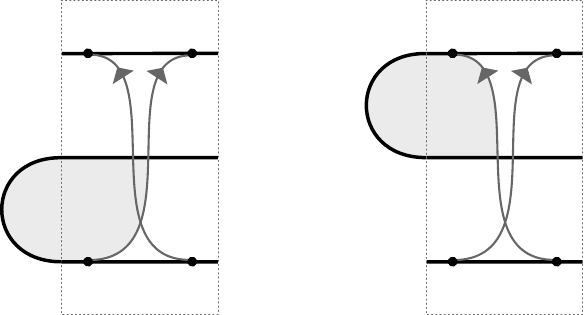}
\caption{Obstructed crossover arrows.}
\label{fig:obstructed-arrows}
\end{figure}

\begin{figure}
\labellist
  \scriptsize
  \pinlabel {$a_k$} at  55 70
  \pinlabel {$a_2$} at  53 48
  \pinlabel {$a_1$} at  51 33
  
  \pinlabel {$a_k$} at  190 70
  \pinlabel {$a_2$} at  188 48
  \pinlabel {$a_1$} at  186 33
  
  \pinlabel {$a_k$} at  325 70
  \pinlabel {$a_2$} at  323 48
  \pinlabel {$a_1$} at  321 33
      
  \color{gray}
  \pinlabel {$c$} at 28 55
  
  \pinlabel {$c$} at 168 55
  \pinlabel {$a_k$} at  140 77
  \pinlabel {$a_2$} at  140 54
  \pinlabel {$a_1$} at  140 39

  \pinlabel {$c$} at 297 76
  \pinlabel {$a_k$} at  275 77
  \pinlabel {$a_2$} at  275 54
  \pinlabel {$a_1$} at  275 39
 
\endlabellist
\includegraphics[scale=1.1]{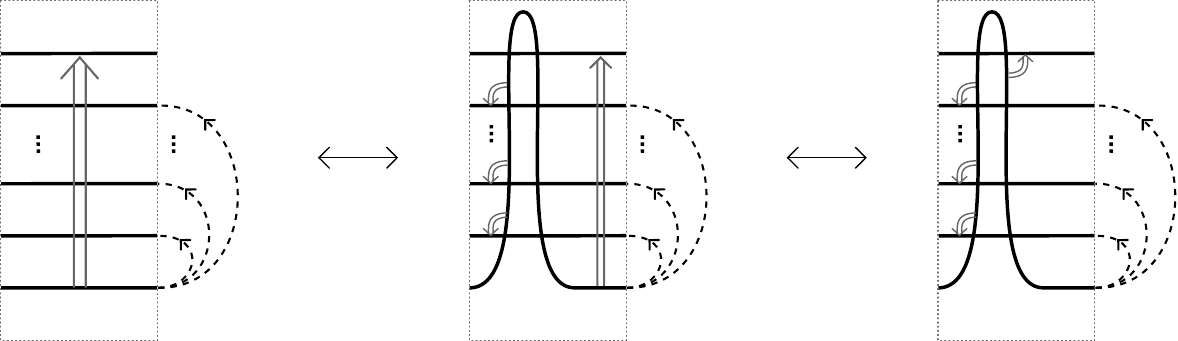}
\caption{An unobstructed arrow can be transformed to a left turn crossover arrow by applying a finger move just as in Figure \ref{fig:any-crossover-arrow-left-turn}, but when the arrow crosses other segments this may require adding additional left turn crossover arrows as indicated. A rectangular neighborhood of the original crossover arrow is shown. The coefficients $a_i \in \F[W]$ on the dotted arcs represent the weighted count of smooth paths in the train track with the given ends that bound bigons with the right side of the rectangular neighborhood of the arrow; note that some or all of these coefficients may be zero.}
\label{fig:unobstructed-arrow}
\end{figure}

\begin{proposition}
Transforming a train track in the neighborhood of a crossover arrow as in Figure \ref{fig:unobstructed-arrow} results in an equivalent train track.
\end{proposition}
\begin{proof}
To show that the train tracks are equivalent, we need to check that any counts of immersed polygons involving the train tracks agree before and after the transformation. We consider the two steps of the transformation shown in Figure \ref{fig:unobstructed-arrow} separately. The first can be viewed as a transformation in the left third of the neighborhood of the crossover arrow, which does not involve the arrow. This is just repeated application of move (j) from Figure \ref{fig:invariance-moves}; we have left-turn crossover arrows in place of a bounding chain, but we have already established that these are equivalent and the same proof applies. The second part of the transformation is simply translating the crossover arrow along the curve. It is clear that any polygon passing through this neighborhood is not meaningfully altered by this homotopy.
\end{proof}

By applying this transformation for each arrow, we see that any unobstructed train track that has the form of an immersed multicurve with crossover arrows is equivalent to an immersed multicurve with left-turn crossover arrows. In particular:

\begin{proposition}
For unobstructed train tracks that consist of immersed curves with crossover arrows, the map $\partial^\tracks$ and the Floer complex are well defined. \qed
\end{proposition}

We will always require that crossover arrows are unobstructed, but if we work over a quotient of $\F[W]$ obtained by setting $W^m = 0$ for some $m$ then we can weaken the unobstructed assumption accordingly. We will say that a crossover arrow is \emph{unobstructed modulo $W^m$} if the weight on the arrow is a multiple of $W^k$ for some $k$ and if any bigon bounded by $\tracks$ and the left side of the rectangular neighborhood of the crossover arrow has a weight which is a multiple of $W^{m-k}$. In this case we use the transformation in Figure \ref{fig:unobstructed-arrow} but note that additional left turn crossover arrows may need to be added corresponding to bigons on the left side of the crossover arrow ending at the bottom segment. Figure \ref{fig:unobstructed-mod-Um} shows an example of a crossover arrow with weight $W^{m-1}$ which is unobstructed only modulo $W^m$. Note that the finger move transformation does not produce an equivalent train track working over $\F[W]$, since the shaded bigon on the right does not have an analogous bigon on the left. However, this problematic bigon involves two crossover arrows whose weights multiply to $W^m$, so this bigon does not contribute modulo $W^m$.

\begin{figure}
\labellist
  \scriptsize
  \pinlabel {$W$} at  0 45
  \pinlabel {$\pmod{W^m}$} at  100 37
  \pinlabel {$W$} at  136 45
  
  \tiny
  \color{gray}
  \pinlabel {$W^{m-1}$} at  55 60
  \pinlabel {$W^{m-1}$} at  198 70
  \pinlabel {$W^{m}$} at  194 52

\endlabellist
\includegraphics[scale=1]{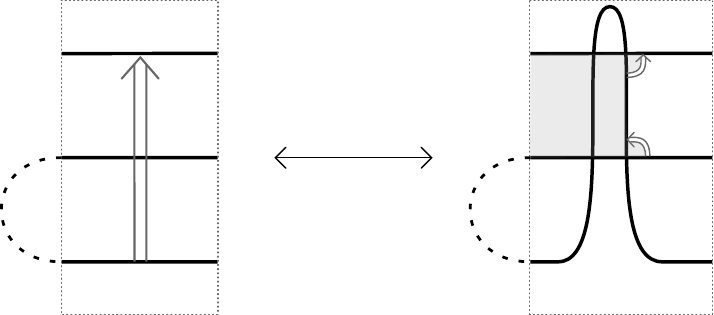}
\caption{A crossover arrow that is obstructed, but unobstructed modulo $W^m$. Applying the finger move as in Figure \ref{fig:unobstructed-arrow} and sliding the crossover arrow to an intersection point produces a train track that is equivalent modulo $W^m$. There is an unwanted bigon on the right involving both crossover arrows, but this does not contribute modulo $W^m$.}
\label{fig:unobstructed-mod-Um}
\end{figure}
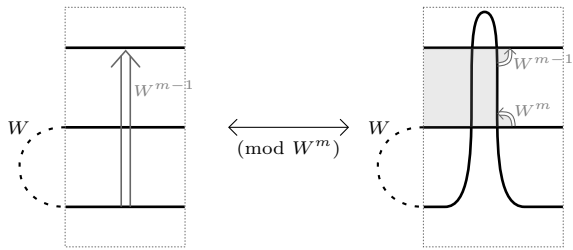

\subsection{Sliding crossover arrows}

The main advantage of working with train tracks of the form described above is that in many cases sliding crossover arrows along the immersed curves produces an equivalent train track. Since any (unobstructed) crossover arrow can be replaced with a left turn crossover arrow after an isotopy of the immersed curves, the arrow slide moves described in this section could alternatively be described as modifications of immersed curves with turning points involving isotopies of the immersed curve as well as modification to the collection of turning points. However, we find arrow sliding more convenient because it allows us to (for the most part) keep the underlying immersed curves fixed throughout the process.

By arrow sliding we mean translating the points at which a crossover arrow meets the immersed curve part of the train track along the relevant curves. The simplest example of such a move is sliding a crossover arrow between two parallel sections of curve, as in Figure \ref{fig:arrow-slide-simple}. Note that in this simple move the arrow does not interact with any other features of the train track (such as self-intersection points, weighted basepoints, or other crossover arrows). We do allow there to be other parallel segments of immersed curve between the two segments connected by the curve. To see that applying this move in a train track $\tracks_0$ results in an equivalent train track, we need to show that the complex $C(\tracks_0, \tracks_1)$ is unchanged for any other train track $\tracks_1$. This is immediately clear if $\tracks_1$ is disjoint from the strip through which the crossover arrow slides. The slide is more interesting if it passes some portion of $\tracks_1$. By breaking an arrow slide into pieces, it is sufficient to consider the case of sliding past one segment of $\tracks_1$ as in Figure \ref{fig:arrow-slide-basis-change}.

\begin{proposition}\label{prop:arrow-slide-basis-change}
If $\tracks_0$ and $\tracks_1$ are unobstructed train tracks (or unobstructed modulo $W^m$), the effect on the $\F[W]$ (or $\F[W]/W^m$) complex $C(\tracks_0, \tracks_1)$ of sliding a crossover arrow on $\tracks_0$ through a segment of $\tracks_1$ as in Figure \ref{fig:arrow-slide-basis-change} is that of a change of basis replacing $x$ with $x' = x + cy$.
\end{proposition}
\begin{proof}
While we could check this directly by considering the effect on immersed bigons that involve the crossover arrow (c.f. \cite[Theorem 5]{HRW}), this also follows from the invariance of the Floer complex for curves with turning points once we realize the curves with crossover arrows as representing curves with bounding cochains and apply the move (b) from Figure \ref{fig:invariance-moves} twice, as in the bottom of Figure \ref{fig:arrow-slide-basis-change}.
\end{proof}

\begin{figure}
\includegraphics[scale=.8]{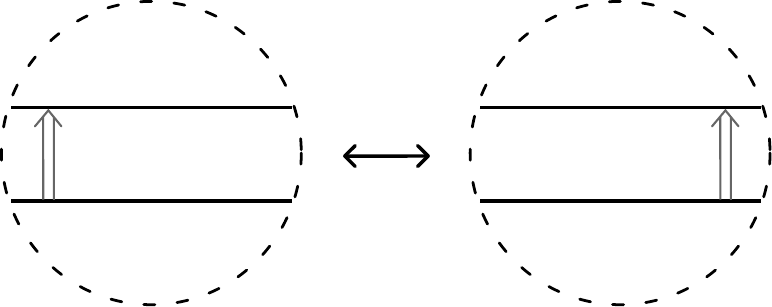}
\caption{Sliding a crossover arrow between two parallel curve segments.}
\label{fig:arrow-slide-simple}
\end{figure}

\begin{figure}
\labellist
  \scriptsize
  \pinlabel {$x$} at  48 145
  \pinlabel {$y$} at  48 182
  \pinlabel {$x$} at  48 25
  \pinlabel {$y$} at  48 62
  \pinlabel {$x'$} at  184 73
  \pinlabel {$y$} at  174 52
  \pinlabel {$x'$} at  308 36
  \pinlabel {$y$} at  310 62
  \pinlabel {$x'$} at  308 156
  \pinlabel {$y$} at  310 182
  
  \color{red}
  \pinlabel {$c$} at  20 163
  \pinlabel {$c$} at  31 62
  \pinlabel {$c$} at  199 62
  \pinlabel {$c$} at  344 62
  \pinlabel {$c$} at  338 163

\endlabellist
\includegraphics[scale=.8]{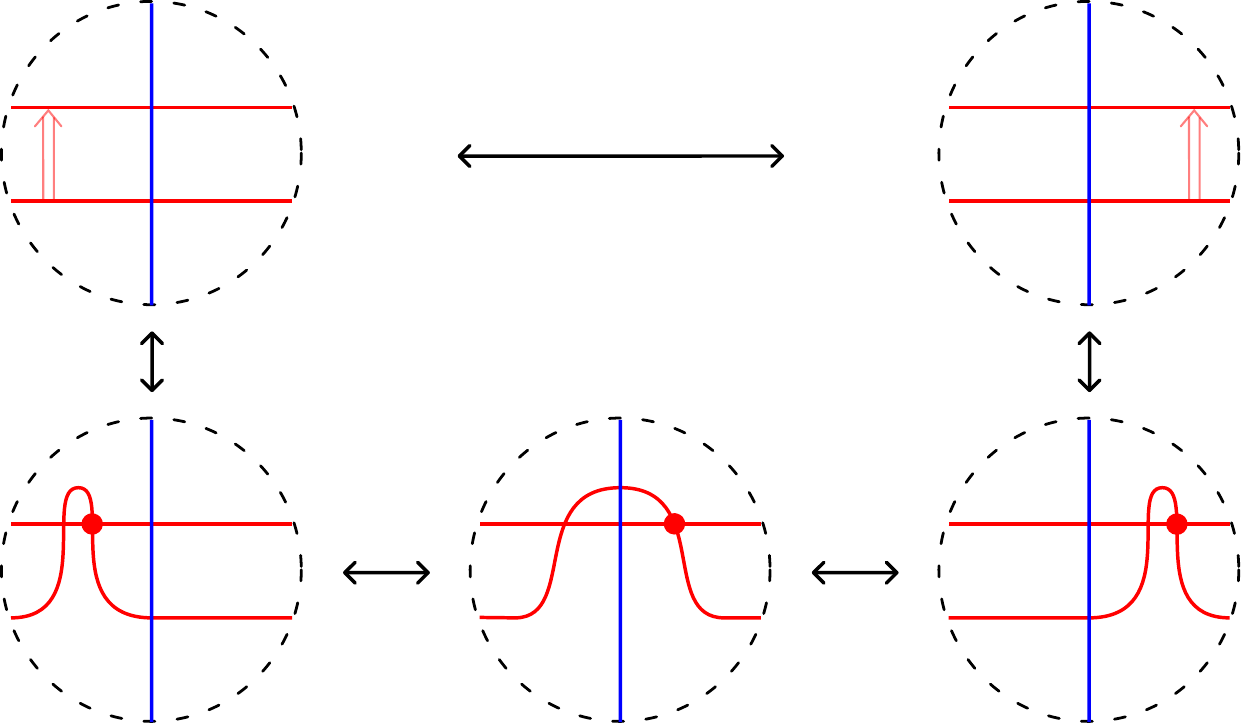}
\caption{Sliding a crossover arrow between parallel segments of $\tracks_0$ past a crossing with $\tracks_1$ has the effect of a change of basis on the complex $C(\tracks_0, \tracks_1)$ that replaces $x$ with $x'  = x + cy$. This can also be interpreted as applying the move $(b)$ from Figure \ref{fig:invariance-moves} twice.}
\label{fig:arrow-slide-basis-change}
\end{figure}

Sliding a crossover arrow can have interesting side-effects if the arrow interacts with self-intersection points or other crossover arrows. 
We will describe a number of local arrow sliding moves as replacements of certain arrow configurations. Consider a rectangular region in $\Sigma$ in which $\tracks_0$ consists of $n$ immersed curve segments (or strands) moving from one side of the rectangle to the opposite side, oriented in the same way, with some number of weighted basepoints on strands and crossover arrows between strands. We will call a portion of train track of this from an \emph{$n$-strand arrow configuration}. The side of the rectangle on which the oriented strands start will be called the initial side and the side on which the strands end will be called the terminal side. If we fix an ordering of the $n$ strand endpoints on the initial side and of the $n$ endpoints on the terminal side, any $n$-strand arrow configuration determines an $n \times n$ matrix over $\F[W]$ where the $(i,j)$ entry is the weighted count of paths in the train track from the $i$th endpoint on the initial side to the $j$th endpoints on the terminal sides. Since the powers of $W$ are forced by the gradings, we will omit them and view this matrix as having coefficients in $\F$. We remark that the matrix corresponding to $n$ parallel strands with a crossing between the $i$th and $j$th strands is a transposition matrix $T_{ij}$ and the matrix for $n$ parallel strands with a single crossover arrow of weight $cW^m$ from the $i$th strand to the $j$th strand is the elementary $A_{ij}^c$ with 1's on the diagonal, $c$ in the $(i,j)$ entry, and 0's elsewhere.

The matrix determined by an $n$-strand arrow configuration is invertible, with the inverse matrix given by counting paths from the terminal side to the initial side. To see this, first observe that if two $n$-strand arrow configurations are concatenated, identifying the terminal end of one with the initial end of the other, the matrix counting paths in the combined configuration is the product of the matrices for the two original configurations. Next, observe that any $n$-strand arrow configuration can be realized as a collection of parallel strands with a sequence of crossover arrows and crossing inserted; this corresponds to decomposing a matrix as a product of elementary matrices. Finally, we observe that reading the configuration backwards corresponds to the same sequence of elementary matrices in the opposite order, with a minus sign on the non-diagonal entry of each $A^c_{ij}$, which produces the inverse matrix.

We will say that two $n$-strand arrow configurations are equivalent if they determine the same matrix. It turns out that the matrix contains all the information needed from this region to construct the Floer complex with any other train track, so we can freely replace $n$-strand arrow configurations with equivalent ones.
\begin{proposition}
If $\tracks_0$ is an unobstructed train track (modulo $U^m$) and $\tracks_0'$ is obtained from $\tracks_0$ by replacing an $n$-strand arrow configuration with another $n$-strand arrow configuration determining the same matrix, then $\tracks'_0$ is equivalent (modulo $U^m$) to $\tracks_0$.
\end{proposition}
\begin{proof}
We must show that the complexes $CF(\tracks_0, \tracks_1)$ and $CF(\tracks_0', \tracks_1)$ are homotopy equivalent for any other train track $\tracks_1$. It suffices to consider the case that $\tracks_1$ is disjoint from the rectangle containing the $n$-strand arrow configuration, since $\tracks_1$ can be isotoped off of the rectangle, possibly using a combination of the basis change move in Proposition \ref{prop:arrow-slide-basis-change} and moves from Figure \ref{fig:invariance-moves}.

In the remaining case, it is easy to see that the complexes $CF(\tracks_0, \tracks_1)$ and $CF(\tracks_0', \tracks_1)$ are identical. If there is a bigon contributing to the first complex whose boundary passes through the region containing the $n$-strand arrow configuration, it does not matter what path the boundary takes from the point it enters the region to the point it leaves the region, only that a path with the given weight exists. Replacing the path through the $n$-strand arrow configuration of $\tracks_0$ with a path with the same endpoints through the new $n$-strand arrow configuration in $\tracks_0'$ produces a bigon contributing to the second complex connecting the same two generators. Because the weighted count of paths between endpoints in the arrow configurations agree, the counts of bigons defining the differential in the complexes agree.
\end{proof}

Some useful replacements of $n$-strand configurations are shown in Figure \ref{fig:local-moves}. For example, crossover arrows may slide past each other, with the provision that if the head of one arrow passes the tail of another, a new arrow corresponding to the composition of the two must be added (see Figure \ref{fig:local-moves}, third row). A crossing between strands can also be replaced by a sequence of crossover arrows. These moves can be interpreted as pictorial representations of familiar relations on elementary matrices.

\begin{remark}
Note that the last move in Figure \ref{fig:local-moves} requires introducing new weights that are the inverse of the weight on the crossover arrow. This is the main issue with extending the techniques described in this paper to non-field coefficients. All the other arrow slides make sense with $\Z$-coefficients, but this move requires weights to be invertible.
\end{remark}

\begin{figure}
\labellist

  \pinlabel {$\left( \begin{matrix}1 & 0 \\ a+b & 1\end{matrix} \right)$} at -70 323 
  \pinlabel {$\left( \begin{matrix}1 & 0 \\ 0 & 1\end{matrix} \right)$} at 510 323 
  \pinlabel {$\left( \begin{matrix}1 & b & 0 \\  0 & 1 & 0 \\ 0 & a & 1\end{matrix} \right)$} at -70 228 
  \pinlabel {$\left( \begin{matrix}1 & 0 & 0 \\  0 & 1 & 0 \\ b & a & 1\end{matrix} \right)$} at 510 228 
  \pinlabel {$\left( \begin{matrix}1 & 0 & 0 \\  b & 1 & 0 \\ ab & a & 1\end{matrix} \right)$} at -70 118 
  \pinlabel {$\left( \begin{matrix}1 & b & 0 \\  0 & 1 & 0 \\ a & ab & 1\end{matrix} \right)$} at 510 118 
  \pinlabel {$\left( \begin{matrix}a & 1 \\  1 & 0 \end{matrix} \right)$} at 0 23 

 \Large
  \pinlabel {$\sim$} at 80 323
  \pinlabel {$\sim$} at 368 323
  \pinlabel {$\sim$} at 80 228
  \pinlabel {$\sim$} at 368 228
  \pinlabel {$\sim$} at 80 118
  \pinlabel {$\sim$} at 368 118
  \pinlabel {$\sim$} at 140 21
  \pinlabel {$\sim$} at 295 21

\small
  \color{gray}
  \pinlabel {$a$} at 5 323
  \pinlabel {$b$} at 58 323
  \pinlabel {$a+b$} at 150 323
  \pinlabel {$0$} at 310 326
  
  \pinlabel {$a$} at 12 211
  \pinlabel {$b$} at 51 245 
  \pinlabel {$a$} at 110 245
  \pinlabel {$b$} at 148 211 
  
  \pinlabel {$a$} at 300 211
  \pinlabel {$b$} at 339 245 
  \pinlabel {$a$} at 398 245
  \pinlabel {$b$} at 436 211 
  
  \pinlabel {$a$} at 12 99
  \pinlabel {$b$} at 51 133 
  \pinlabel {$b$} at 102 133
  \pinlabel {$a$} at 117 99
  \pinlabel {$ab$} at 157 128
  
  \pinlabel {$a$} at 300 99
  \pinlabel {$b$} at 339 133 
  \pinlabel {$b$} at 390 133
  \pinlabel {$ab$} at 445 99
  \pinlabel {$a$} at 406 99
  
  \pinlabel {$a$} at 55 21
  \pinlabel {$a^{-1}$} at 177 18
  \pinlabel {$a^{-1}$} at 260 22
  \pinlabel {$a$} at 377 21
  
  \pinlabel {$a$} at 215 43
  \pinlabel {$-a^{-1}$} at 215 -7

         \endlabellist
\includegraphics[scale=.6]{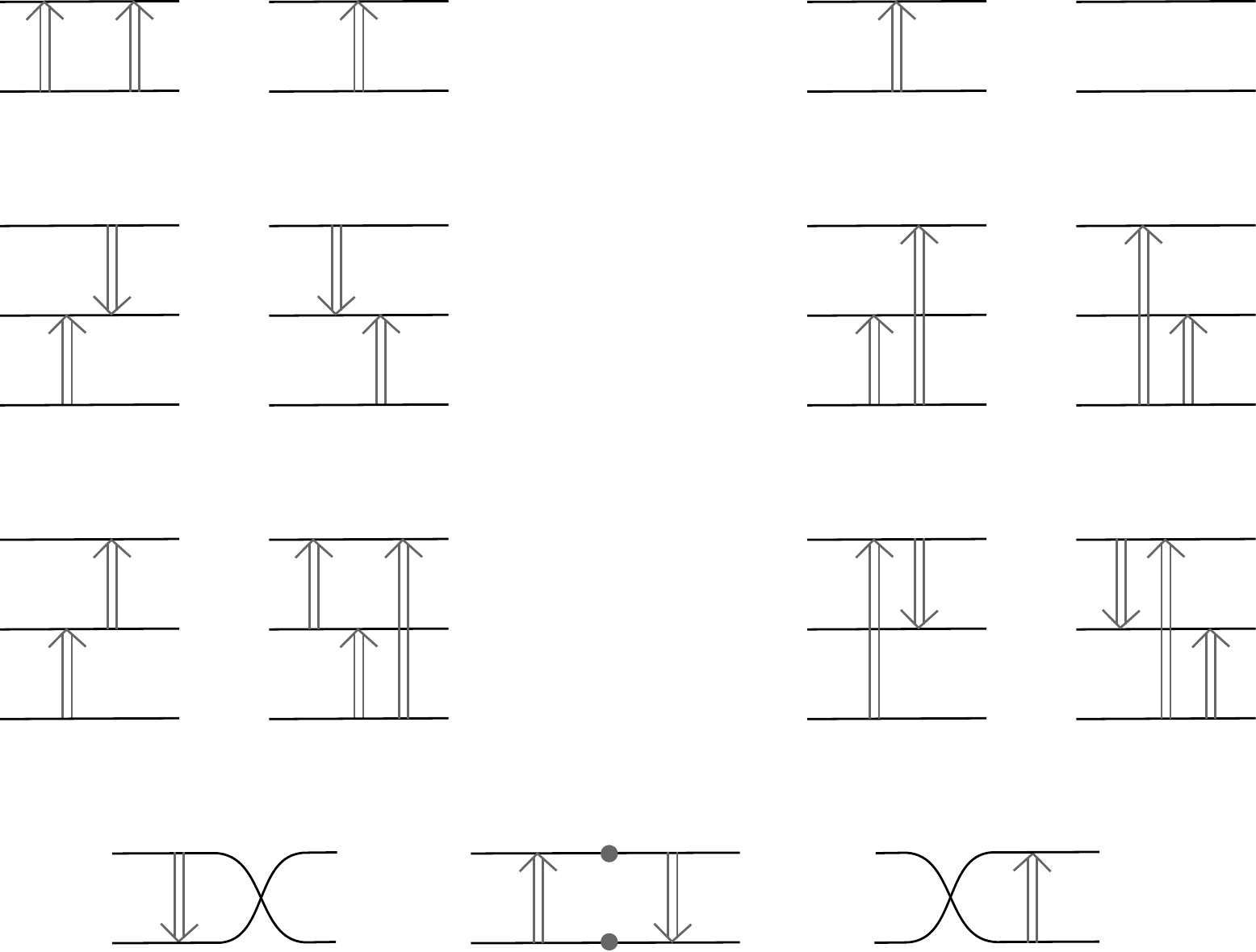}
\caption{ Some common pairs of exchangeable $n$-strand arrow configurations, with the corresponding matrices. All horizontal strands are oriented rightward. These exchanges can be thought of as local moves for crossover arrows. Two parallel crossover arrows next to each other can be combined into one, and 0 weighted arrows can be deleted (first row). The endpoints of crossover arrows can slide along the immersed curve track until they meet another crossover arrow. Arrows commute if their endpoints do not meet or if only their tails or only their heads meet (second row). If the head of one arrow passes the tail of another, a new arrow coming from the composition of the two must be added (third row). Finally, a pair of opposing arrows can be replaced with an arrow and a crossing (fourth row). In the last row, weighted basepoints are added to the train tracks between the two crossover arrows on each strand.}
\label{fig:local-moves}
\end{figure}
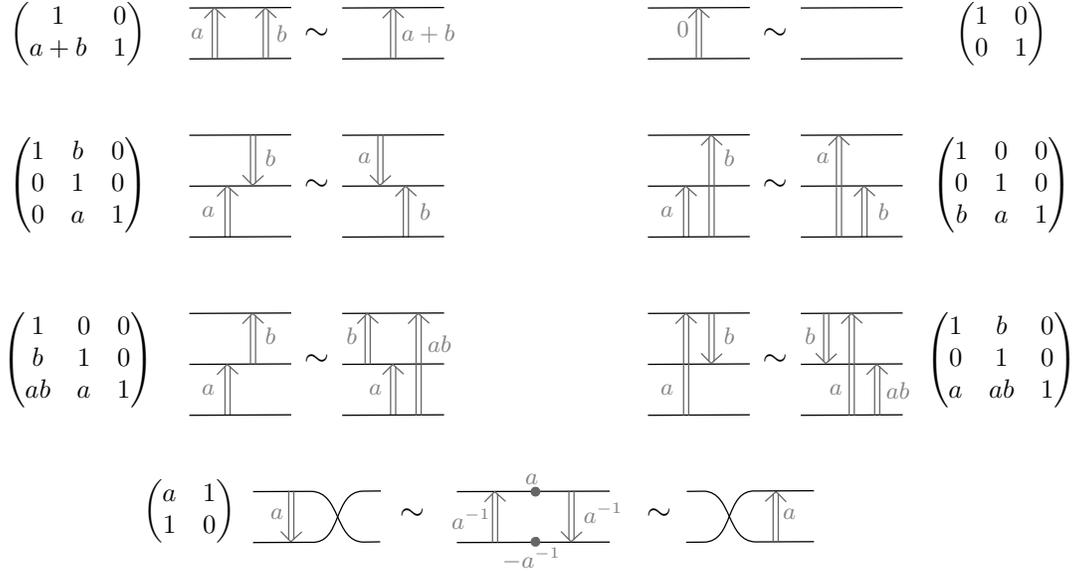
 
In Section \ref{sec:simple-curves} we will give an algorithm for simplifying train tracks representing complexes over $\sRhat$ using arrow slide moves. The following $n$-strand arrow configuration replacement will be used repeatedly in this arrow sliding algorithm; this lemma is essentially \cite[Lemma 31]{HRW}, but with the matrix language introduced above we give a much simpler proof.
\begin{lemma}\label{lem:sort-arrows}
Given an $n$-strand arrow configuration, fix any orderings $<_L$ and $<_R$ on the left endpoints and right endpoints, respectively. The $n$-strand configuration is equivalent to an $n$-strand configuration which can be divided into three regions, with all crossings between strands and all weighted basepoints in the middle region, and all crossover arrows in the left or right region moving in increasing order with respect to the relevant ordering on endpoints of strands $<_L$ or $<_R$.
\end{lemma}
\begin{proof}
When encoding the $n$-strand arrow configuration by a matrix, we will index the right endpoints with respect to the ordering $<_R$ and the left endpoints with respect to the opposite of the ordering $<_L$. That is, we number the right endpoints $1, \ldots, n$ such that $1 <_R \cdots <_R n$ and we number the left endpoints $1, \ldots, n$ so that $n <_L \cdots <_L 1$. Then the desired $n$-strand configuration replacement corresponds to the $LDPU$ decomposition of the matrix. It is a standard exercise in linear algebra that an invertible matrix $M$ with coefficients in a field can be decomposed as the product $LDPU$, where $L$ is lower triangular with 1's on the diagonal, $D$ is diagonal, $P$ is a permutation matrix, and $U$ is upper triangular with 1's on the diagonal. $L$ can be decomposed as a product of lower triangular elementary matrices, each of which correspond to an $n$-strand arrow configuration with a single arrow from a strand $i$ to a strand $j$ with $j < i$ and thus $i <_L j$ . Similarly, $U$ is a product of upper triangular elementary matrices which each correspond to a single crossover arrow from a strand $i$ to a strand $j$ with $i < j$, and thus $i <_R j$. $D$ corresponds to $n$ parallel strands, possibly containing weighted basepoints, and $P$ corresponds to a region of strands with crossings but no crossover arrows.
\end{proof}

In addition to replacing $n$-strand arrow configurations, which lie in neighborhoods disjoint from the marked points, another important arrow move is to slide a crossover arrow past a marked point. An arrow can always be moved over a marked point on its left side at the expense of multiplying the weight of the arrow by $W$. It is clear that the resulting train track is equivalent to the initial one: there is a one-to-one correspondence between bigons, and bigons involving the crossover arrow cover one less marked point but have an additional factor of $W$ from the weight of the crossover arrow, so the overall weight is unchanged.

Sliding over a marked point is particularly valuable when working over a quotient ring $F[W]/W^m$, since if the weight on a crossover arrow increases enough we can delete the arrow entirely. In particular, we say that a crossover arrow in a train track $\tracks$ with weight $cW^{m-1}$ is \emph{removable modulo $W^m$} if there is an arc in the complement of $\tracks$ from the left side of the crossover arrow to a marked point. Such an arrow can be deleted with no effect modulo $\tracks$ since sliding the arrow over the marked point gives an arrow with weight $cW^m$. On a similar note, we say any crossover arrow is \emph{removable} if there is an arc in the complement of $\tracks$ from the left side of the crossover arrow to the boundary of $\Sigma$ or to a puncture or infinite end of $\Sigma$. To see that deleting a removable crossover arrow preserves the equivalence type of the train track, note that any other train track being paired with may be homotoped off of the arc and then any bigon involving the crossover arrow would have to contain the entire arc, which is not possible.

Similar to sliding crossover arrows, we may also slide basepoints along the immersed curves in $\tracks$. If the the basepoint slides past an end of a crossover arrow, we need to modify the weight on the crossover arrow as in Figure \ref{fig:slide-weights} to ensure that weighted counts of polygons are unchanged. To see that this produces an equivalent train track, note that if a basepoint in $\tracks_0$ slides through a segment of $\tracks_1$, the complex $CF(\tracks_0, \tracks_1)$ changes by a change of basis as in move (c) from Figure \ref{fig:invariance-moves}. Moreover, two basepoints can be combined, multiplying their weights, and a basepoint with weight 1 can be deleted. Note that by applying these moves we can alway arrange that there is at most one basepoint on each $S^1$ component of the immersed multicurve in $\tracks$. We can eliminate all basepoints on immersed arcs or lines by sliding them to the ends until they cannot lie on any immersed polygons.

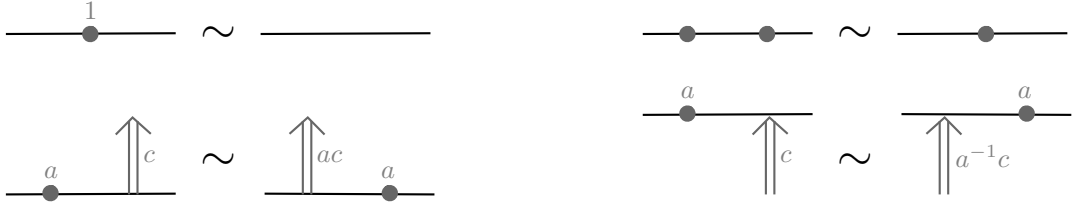
\begin{figure}
\labellist

  \huge
  \pinlabel {$\sim$} at 80 63 
  \pinlabel {$\sim$} at 80 15 
  \pinlabel {$\sim$} at 320 63 
  \pinlabel {$\sim$} at 320 15

  \normalsize
  \color{gray}
  \pinlabel {$1$} at  32 72
  \pinlabel {$c$} at  54 17
  \pinlabel {$a$} at  17 11
  \pinlabel {$ac$} at  122 17
  \pinlabel {$a$} at  144 11
  
  \pinlabel {$c$} at  294 17
  \pinlabel {$a$} at  257 41
  \pinlabel {$a^{-1}c$} at  368 17
  \pinlabel {$a$} at  384 41

\endlabellist
\includegraphics[scale=1]{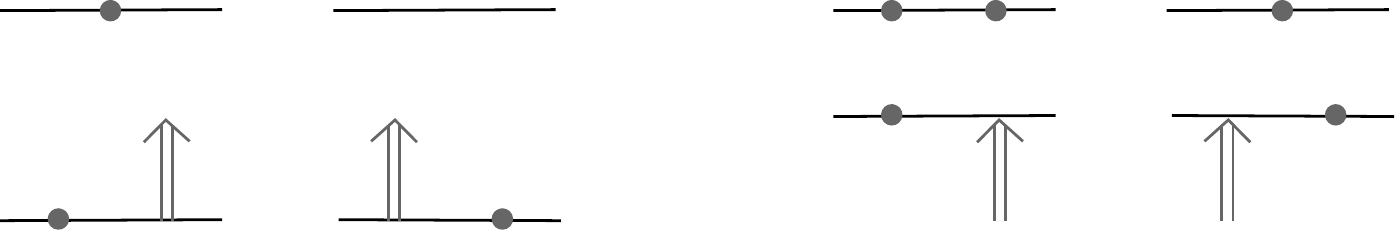}
\vspace{1 mm}
\caption{Moves used to slide weighted basepoints along immersed curves in $\tracks$. All horizontal segments are oriented rightward. Basepoints weighted by 1 can be freely added or removed, nearby basepoints can be replaced with a single basepoint weighted by the product of the weights, and basepoints can slide past the end of the arrows with an appropriate change to the arrow weights.}
\label{fig:slide-weights}
\end{figure}

\begin{remark}[Train tracks in doubly marked surfaces] Just as with immersed curves with bounding chains, for simplicity we have introduced train tracks in the setting of singly marked surfaces. However, these constructions all generalize immediately to doubly marked surfaces. In the doubly marked setting undirected edges still have weights in $\F$, but directed edges are weighted by elements of $\sRminus$. Train tracks can be equipped with a bigrading, and the powers of $U$ and $V$ associated with an undirected edge are determined by these gradings. In the polygon counting maps, we use the variables $U$ and $V$ to record how a polygon covers the $w$ and $z$ marked points, respectively. Setting either $U=1$ or $V=1$, and ignoring the corresponding grading, corresponds to ignoring one type of marked point.
\end{remark}

\section{Complexes from curves in $\strip$ or $\cylinder$}\label{sec:curves-in-strip}% !TEX root = ../CFKcurvesZHS3s.tex
%curves-in-strip.tex

In the Section \ref{sec:floer-theory} we defined Floer theory for immersed curves with bounding chains in arbitrary non-compact (doubly) marked surfaces. We now focus our attention on curves in particular marked surfaces, namely the infinite strip and infinite cylinder. By taking the Floer complex with a certain fixed curve, we will see that any decorated immersed curve in the marked strip determines a bigraded complex over $\sRminus$ and a decorated immersed curve in the infinite cylinder determines a bigraded complex with a flip map.

\subsection{Infinite marked strips and cylinders}

We first define and set notation for the relevant surfaces.  By the infinite marked strip, we mean the surface $\strip = \left[-\tfrac 1 2, \tfrac 1 2\right]\times\R$ equipped with infinitely many marked points occurring at $\left( 0, n + \tfrac 1 2 \right)$ for each integer $n$. Note that $\strip$ has two boundary components, $\partial_L \strip = \{-\tfrac 1 2 \}\times \R$ and $\partial_R \strip = \{\tfrac 1 2 \}\times \R$. At times we will consider the punctured surface $\pstrip$ obtained by removing the marked points in $\strip$. Another important variation we will consider is the doubly marked surface $\dmstrip$ obtained from $\strip$ by replacing each marked point with a pair of marked points: a $w$-marked point just to the right and a $z$-marked point just to the left of each marked point of $\strip$. Thus $\dmstrip$ has a family of $z$-marked points at $\left( -\epsilon, n+\tfrac 1 2 \right)$ and a second family of $w$-marked points at $\left( +\epsilon, n+\tfrac 1 2 \right)$ for integers $n$ and some very small $\epsilon$. The vertical line $\{0\}\times\R$ will play an important role; this line (in either $\strip$ or $\dmstrip$) will be denoted $\mu$. Note that $\mu$ passes through the marked points on $\strip$ and passes in between the $z$ and $w$ marked points for each pair of nearby marked points on $\dmstrip$. More generally, we will consider the vertical lines $\mu_{a} = \{a\}\times \R$, which are translations of $\mu$. In particular we will consider the lines $\mu_{2\epsilon} = \{2\epsilon\} \times \R$ and $\mu_{-2\epsilon} = \{-2\epsilon\} \times \R$, which pass just to the right or left, respectively, of all marked points in either $\strip$ or $\dmstrip$. The infinite cylinder $\cylinder$ is obtained by identifying the opposite edges of $\strip$; that is, $\cylinder$ is $(\R/\Z)\times \R$ with marked points at $(0, n+\tfrac 1 2)$ for $n$ in $\Z$. As with the strip, the doubly marked cylinder $\dmcylinder$ is obtained from $\cylinder$ by replacing each marked point with a $w$-marked points $\epsilon$ units to the right and a $z$-marked points $\epsilon$ units to the left. Removing each marked point in $\cylinder$ results in a punctured cylinder $\pcylinder$.

We will assume that the grading arcs starting from each marked point in $\dmstrip$ are disjoint from the line segment $[-\tfrac 1 2, \tfrac 1 2] \times \{0\}$, so that grading arcs approach the positive or negative end of the strip depending on whether the marked point has positive or negative height. We also assume the grading arcs are disjoint from $\mu$, $\mu_{2\epsilon}$, $\mu_{-2\epsilon}$; that is they lie in a neighborhood of $\mu$ and lie to the right of $\mu$ for arcs coming from $w$ marked points and to the left of $\mu$ for arcs coming from $z$ marked points. The grading arcs from marked points in the singly marked strip $\strip$ are the same as the arcs from $w$-marked points in $\dmstrip$ (preceded by the length $\epsilon$ horizontal arc from the marked point to the corresponding $w$ marked point). The grading arcs on $\cylinder$ and $\dmcylinder$ come from those in $\strip$ or $\dmstrip$ after identifying opposite edges.

\subsection{Bigraded complexes from curves in $\strip$} 
Consider a compact decorated immersed multicurve $(\Gamma, \bchain)$ in $\strip$, where $\Gamma$ is a weighted and graded immersed multicurve in $\strip$ disjoint from the marked points and $\bchain$ is a bounding chain. We will assume every component of $\Gamma$ intersects $\mu$. Given such a decorated curve, we will define a bigraded complex $C(\Gamma, \bchain)$ over $\sRminus$. To do so we first observe that, since $\Gamma$ avoids a sufficiently small neighborhood around every marked point, $(\Gamma, \bchain)$  can also be viewed as a decorated multicurve in the doubly marked surface $\dmstrip$. The bounding chain $\bchain$, when viewed as a linear combination of points in $\sI_{\le 0}$, is unchanged; note, however, that in the corresponding element $\overline\bchain$ of $CF(L)$ we must add appropriate powers of $V$ when passing to the doubly marked surface. By slight abuse of notation we will let $(\Gamma, \bchain)$ denote both the decorated curve in $\strip$ and the corresponding decorated curve in $\dmstrip$.

The curve $(\Gamma, \bchain)$ in $\strip$ comes equipped with a grading function $\tilde\tau: \Gamma \to \R$. This will give rise to two grading functions $\tilde\tau_w$ and $\tilde\tau_z$ on the curve in $\dmstrip$. We define $\tilde\tau_w$ to be identically equal to $\tilde\tau$; note that this is possible since the relevant sets of grading arcs agree. Recall that since the mod 2 grading is determined by the orientation on $\Gamma$, $\tilde\tau_z$ and $\tilde\tau_w$ must differ by an even integer at any point. We will choose the grading function $\tilde\tau_z$ so that at any point $p$ in $\Gamma \cap \mu$ that falls between the marked points at $\left(0, n - \tfrac 1 2 \right)$ and $\left(0, n + \tfrac 1 2 \right)$, $\tilde\tau_z(p) = \tilde\tau(p) + 2n$. This rule defines the grading function $\tilde\tau_z$ on all of $\Gamma$, since by assumption every component of $\Gamma$ intersects $\mu$. To see that this rule is consistent note that if two points $p_1$ and $p_2$ in $\Gamma \cap \mu$ are connected by a path in $\Gamma$ lying on the right side of $\mu$, at heights $n_1$ and $n_2 > n_1$, the gradings $\tilde\tau_w$ and $\tilde\tau_z$ change in the same way along the path from $p_1$ to $p_2$ except that $\tilde\tau_w$ jumps down by an additional $2(n_2-n_1)$ since the path crosses $n_2-n_1$ grading arcs from $w$-marked points. Similarly, if there is a path from $p_1$ to $p_2$ on the left side of $\mu$ the signed number of times this path crosses grading arcs from $z$-marked points is $n_1 - n_2$, and so $\tilde\tau_z$ increases $2(n_2-n_1)$ more than $\tilde\tau_w$ does along this path. In either case, we have that
$$\tilde\tau_z(p_2) = \tilde\tau_z(p_1) + (\tilde\tau_w(p_2) - \tilde\tau_w(p_1)) + 2(n_2 - n_1) = (\tilde\tau_z(p_1) - \tilde\tau_w(p_1) - 2n_1) + \tilde\tau_w(p_2) + 2n_2 = \tilde\tau_w(p_2) + 2n_2.$$

Each self-intersection point $p$ of $\Gamma$ has a degree $\deg(p)$ which is the even one of the gradings of the two intersection points of $\Gamma \cap \Gamma'$ associated with $p$. When viewed as a curve in $\dmstrip$, each intersection point $p$ has a bidegree $(\deg_w(p), \deg_z(p))$. It is clear that $\deg_w(p) = \deg(p)$, since $\tilde\tau_w$ agrees with $\tilde\tau$ by definition. It can also be checked that $\deg_z(p) = \deg(p)$ for any self-intersection point $p$; this is because the curve never passes in between a pair of $z$ and $w$ marked points. Thus we will simply speak of the degree of $p$ rather than the bidegree, even when working in the doubly marked surface.

We now define $C(\Gamma, \bchain)$ to be the Floer complex of $(\Gamma, \bchain)$ with $\mu$ in $\dmstrip$. We understand $\mu$ to be equipped with the trivial bounding chain (as $\mu$ has no self-intersection points) and exclude this from the notation. We also equip $\mu$ with the constant bigrading $\left(\tfrac 1 2, \tfrac 1 2\right)$; taken mod 2, this corresponds to orienting $\mu$ upwards. Recall that the Floer complex is defined when at most one of the input curves is non-compact; here $\mu$ is non-compact, but $\Gamma$ is compact. We will always assume that $\Gamma$ has transverse self-intersection with no triple points and that $\Gamma$ is transverse to $\mu$; in fact, we will assume that $\Gamma$ is perpendicular to $\mu$ at every intersection point.  We will also assume $\Gamma$ bounds no immersed disks. Since $\mu$ does not contain an immersed circle there can be no (generalized) immersed annuli, so $(\Gamma, \bchain)$ and $\mu$ are in admissible position.

We will sometimes weaken the assumption that $\bchain$ is a bounding chain and consider arbitrary collections of turning points. Recall that the Floer complex can be constructed in the same way except that $m_1^\bchain$ is no longer guaranteed to be a differential, so we obtain a precomplex rather than a complex.

\begin{definition}
For an immersed multicurve $\Gamma$ with a collection of turning points $\bchain$ in $\strip$, which we also view as a decorated curve in $\dmstrip$, $C(\Gamma, \bchain)$ is the bigraded precomplex $CF((\Gamma,\bchain), \mu)$ over $\sRminus$. If $\bchain$ is a bounding chain then $C(\Gamma, \bchain)$ is a bigraded complex.
\end{definition}

We will adopt a similar definition when using the language of train tracks rather than decorated curves. That is, for an immersed train track $\tracks$ that consists of an immersed curve with crossover arrows, $C(\tracks)$ will denote the bigraded precomplex $CF(\tracks, \mu)$ over $\sRminus$ coming from the Floer complex of train tracks in $\dmstrip$, and if $\tracks$ is unobstructed then this is a complex. The differential on $C(\tracks)$ will be denoted $\partial^\tracks$.

The construction of $C(\Gamma, \bchain)$ is a special case of the more general definition of Floer complexes from Section \ref{sec:floer-theory}, but we recall the key features here. There is one generator of $C(\Gamma, \bchain)$ for each intersection of $\Gamma$ with the vertical line $\mu$. Each of these generators has a bigrading $(\gr_w, \gr_z)$; it follows from Definiton \ref{def:grading} and the fact that the grading on $\mu$ is $\left(\tfrac 1 2, \tfrac 1 2\right)$ that $\gr_w(x) = -\tilde\tau_w(p_x)$ and $\gr_z(x) = -\tilde\tau_z(p_x)$, where $p_x$ is the intersection point of $\Gamma \cap \mu$ corresponding to the generator $x$.

Each generator $x$ has an Alexander grading
$$A(x) = \frac{ \gr_w(x) - \gr_z(x) }{2} = \frac{ \tilde\tau_z(p_x) - \tilde\tau_w(p_x) }{2}.$$
By the definition of $\tilde\tau_z$, we have that $A(x) = n$ if $p_x$ falls on the segment of $\mu$ between $(0, n - \tfrac 1 2)$ and $(0, n + \tfrac 1 2)$, so the Alexander grading records the discrete height of the corresponding intersection point. Note that the Alexander grading defines a partial ordering on the generators of $C(\Gamma, \bchain)$. The actual height of intersection points refines this to a total ordering on generators that will be useful in future arguments; we will say that $x_1 < x_2$ if $p_{x_1}$ occurs below $p_{x_2}$. Thus $(\Gamma, \bchain)$ defines not just the precomplex $C(\Gamma, \bchain)$ but also a preferred choice of ordered basis for that precomplex.

The map $m_1^\bchain$ counts generalized immersed bigons with left boundary on $\mu$ and right boundary a polygonal path in $\Gamma$ consistent with $\bchain$. Each generalized bigon $u$ contributes with a coefficient given by the product of the following;
\begin{itemize}
\item $U^{n_w(u)} V^{n_z(u)}$ where $n_w(u)$ and $n_z(u)$ are the multiplicities with which $u$ covers the $w$ and $z$ marked points;
\item the coefficient in $\overline\bchain$ for each false corner in the boundary, or the opposite of the coefficient if the boundary orientation on $u$ opposes the orientation on $\Gamma$;
\item the weight $c$ of each basepoint of $\Gamma$ passed along $\partial u$, or $c^{-1}$ if the boundary orientation of $u$ opposes the orientation on $\Gamma$; and 
\item an additional factor of $(-1)$ if the the orientation on $\partial u$ opposes the orientation on $\Gamma$.
\end{itemize}

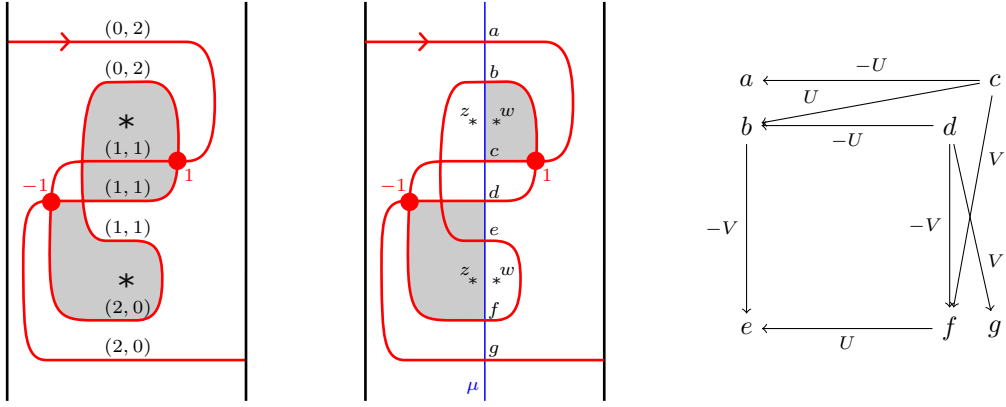
\begin{figure}
\labellist
  \scriptsize
  \pinlabel {$a$} at  184 139
  \pinlabel {$b$} at  184 124
  \pinlabel {$c$} at  184 94
  \pinlabel {$d$} at  184 79
  \pinlabel {$e$} at  184 64
  \pinlabel {$f$} at  184 34
  \pinlabel {$g$} at  184 19
  
  \color{blue}
  \pinlabel {$\mu$} at  176 5
  
  \color{black}
  \scriptsize
  \pinlabel {$z$} at 173 109
  \pinlabel {$w$} at 189 109
  \pinlabel {$z$} at 173 49
  \pinlabel {$w$} at 189 49

  \color{red}
  \pinlabel {$-1$} at  11 81
  \pinlabel {$1$} at  69 85
  \pinlabel {$-1$} at  146 81
  \pinlabel {$1$} at  204 85
  
  \color{black}
  \pinlabel {$(0,2)$} at  46 140
  \pinlabel {$(0,2)$} at  46 125
  \pinlabel {$(1,1)$} at  46 95
  \pinlabel {$(1,1)$} at  46 80
  \pinlabel {$(1,1)$} at  46 65
  \pinlabel {$(2,0)$} at  46 35
  \pinlabel {$(2,0)$} at  46 20
  
\endlabellist
 \includegraphics[scale=1]{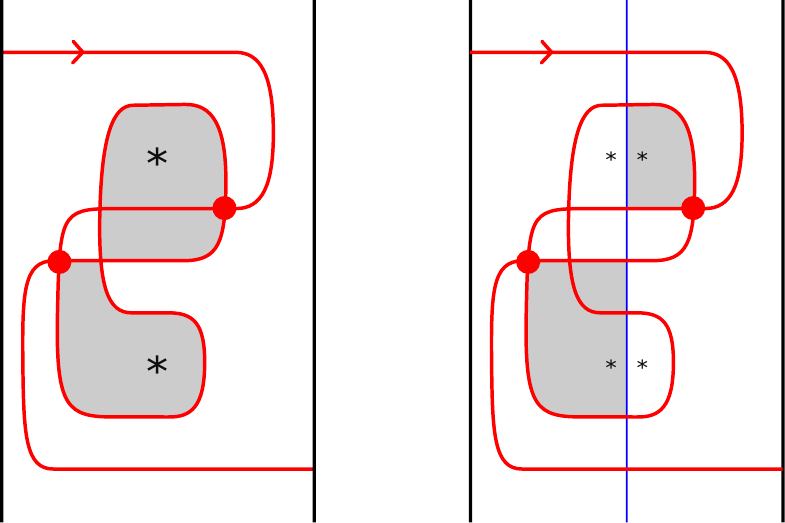}
\hspace{10 mm}
\raisebox{6mm}{\begin{tikzpicture}[scale = 1.5]
\node (a) at (0,2.2) {$a$};
\node (b) at (0,1.8) {$b$};
\node (c) at (2.2,2.2) {$c$};
\node (d) at (1.8,1.8) {$d$};
\node (e) at (0,0) {$e$};
\node (f) at (1.8,0) {$f$};
\node (g) at (2.2,0) {$g$};

\scriptsize
\draw[->] (c) to node[above] {$-U$} (a);
\draw[->] (c) to node[above left, pos = .7] {$U$} (b);
\draw[->] (d) to node[below] {$-U$} (b);
\draw[->] (f) to node[below] {$U$} (e);

\draw[->] (d) to node[left] {$-V$}( f);
\draw[->] (d) to node[right, pos = .7] {$V$} (g);
\draw[->] (c) to node[right, pos = .3] {$V$} (f);
\draw[->] (b) to node[left] {$-V$} (e);

\end{tikzpicture}}
 \caption{Left: A decorated immersed curve $(\Gamma, \bchain)$ in $\strip$. $\Gamma$ has a single component and $\bchain$ is the sum of the two self-intersection points marked, with coefficients $1$ and $-1$ as indicated. To see that $\bchain$ is a bounding chain, note that the two shaded monogons (and another similar pair of monogons) make canceling contributions to $m^\bchain_0$. The value of the bigrading function $(\tilde\tau_w, \tilde\tau_z)$ is $(0,2)$ at the endpoint on the left boundary of $\strip$; this determines the bigrading function on all of $\Gamma$, but for convenience the value of $(\tilde\tau_w, \tilde\tau_z)$ is indicated near each intersection of $\Gamma$ with the vertical line through the marked points. Center: $(\Gamma,\bchain)$ and $\mu$ in the doubly marked strip $\dmstrip$, with two contributions to the Floer complex shaded. Right: The complex $C(\Gamma, \bchain)$.}
 \label{fig:complex-from-curve-example}
\end{figure}

\begin{example}\label{ex:complex-from-curves}
Consider the decorated immersed curve in Figure \ref{fig:complex-from-curve-example}. The complex $C(\Gamma, \mu)$ has  seven generators $a, b, c, d, e, f$, and $g$. Generators $a$ and $b$ both have bigrading $(0,-2)$ and Alexander grading 1, generators $c$, $d$, and $e$ have bigrading $(-1,-1)$ and Alexander grading 0, and generators $f$ and $g$ have bigrading $(-2,0)$ and Alexander grading $-1$. There are eight generalized bigons contributing to $\partial = m^\bchain_1$; six of these are true bigons and two are triangles with one false corner (the two triangles are shaded in Figure \ref{fig:complex-from-curve-example}). This gives
$$\partial a = \partial e = \partial g = 0, \quad \partial b = -Ve, \quad \partial c = -Ua + Ub + Vf, \quad \partial d = -Ub - Vf + Vg, \quad \text{and} \quad \partial f = Ue.$$
\end{example}

Just as $C = C(\Gamma, \bchain)$ is the Floer complex of $(\Gamma, \bchain)$ with $\mu$, we will define $C_{\pm 2\epsilon} = C_{\pm 2\epsilon}(\Gamma, \bchain)$ to be the Floer complex $CF( (\Gamma, \bchain), \mu_{\pm 2\epsilon})$. These complexes are closely related to $C$: the localized versions $C^\infty_{\pm 2\epsilon}$ are isomorphic to $C^\infty$ but have a different natural choice of basis. In particular, if $\{x_i\}_{i=1}^n$ is the basis for $C$ (and thus for $C^\infty$) then the generators of $C_{2\epsilon}$ can be identified with $\{ U^{A(x_i)} x_i \}_{i=1}^n$ and the generators of $C_{-2\epsilon}$ can be identified with $\{ V^{-A(x_i)} x_i \}_{i=1}^n$. To see this, observe that essentially the same bigons contribute to the differential in all three complexes, but a bigon from $x_i$ to $x_j$ covers $A(x_j) - A(x_i)$ fewer $w$ marked points when considered in $C_{2\epsilon}$ than in $C$ and covers $A(x_i) - A(x_j)$ fewer $z$ marked points when considered in $C_{-2\epsilon}$ than in $C$. Note that all generators for $C_{\pm 2\epsilon}$ have Alexander grading zero, and the given basis generates the Alexander grading zero summand of $C_{\pm 2\epsilon}$ as an $\F[W]$-module. When viewed as a complex over $\F[W]$, $C_{\pm 2\epsilon} |_{A=0}$ is the same as the Floer complex $CF((\Gamma, \bchain), \mu_{\pm 2\epsilon})$ taken in the singly marked strip $\strip$. 

Note that $C_{2\epsilon}|_{A=0}$, viewed as a complex over $\F[W]$, agrees with is the horizontal complex $C^h$ of $C$. Similarly $C_{-2\epsilon}|_{A=0}$ is the vertical complex of $C^v$ of $C$. By the invariance of the Floer complex, we can slide line $\mu_{2\epsilon}$ rightward or we can slide the line $\mu_{-2\epsilon}$ leftward without changing the Floer complex with $(\Gamma,\bchain)$ up to homotopy equivalence. In particular the Floer complex in $\strip$ with the right boundary of $\strip$, $CF((\Gamma,\bchain), \mu_{\tfrac 1 2})$, is homotopy equivalent to the horizontal complex of $C$ and $CF((\Gamma,\bchain), \mu_{-\tfrac 1 2})$ is homotopy equivalent to the vertical complex. In many cases the intersection number with $\Gamma$ decreases as $\mu$ is slid to the boundaries of the strip to the point that $CF((\Gamma,\bchain), \mu_{\tfrac 1 2})$ and $CF((\Gamma,\bchain), \mu_{-\tfrac 1 2})$ are reduced. In this case the generators of these complexes are the generators of the horizontal and vertical homology of $C$. For instance, in Example \ref{ex:complex-from-curves} the single intersection of $\Gamma$ with the right or left boundary of $\strip$ correspond $g$ or $a$, respectively, the generators of horizontal or vertical homology of $C$. We caution that in some cases the process of sliding $\mu_{\pm 2\epsilon}$ out to $\partial\strip$ involves basis changes to the corresponding complex, so in general the generators remaining at the boundary may not be a subset of the original generators.

\subsection{Complexes and flip maps from curves in $\cylinder$}\label{sec:flip-maps-from-curves}

A decorated immersed curve $(\Gamma, \bchain)$ in the marked cylinder $\cylinder$ determines a bigraded complex $C$ as well as a flip map $\Psi: C^\infty \to C^\infty$. The complex $C$ is defined to be the Floer complex in $\dmcylinder$ of $(\Gamma, \bchain)$ with $\mu$. The flip map is defined using a curve $\mu_{\Psi}$ constructed from $\mu_{2\epsilon}$ and $-\mu_{-2\epsilon}$; fixing a sufficiently large height that the compact curve $\Gamma$ does not reach, we truncate $\mu_{2\epsilon}$ and $-\mu_{-2\epsilon}$ by cutting off the portion above that height and then connect the two loose ends by a cap. We now define $C_{\Psi}$ to be the Floer complex of $(\Gamma, \bchain)$ and $\mu_{\Psi}$ in the singly marked cylinder $\cylinder$. As an $\F[W]$-module, we have
$$C_{\Psi} \cong C_{2\epsilon} |_{A=0} \oplus C_{-2\epsilon} |_{A=0} [-1].$$
Bigons contributing to the differential on $C_{\Psi}$ are of three types: the $\mu_{\Psi}$ boundary of the bigon either is contained in the $\mu_{2\epsilon}$ portion of $\mu_{\Psi}$, is contained in the $-\mu_{-2\epsilon}$ portion of $\mu_{\Psi}$, or involves the cap in $\mu_{\Psi}$. Bigons of the first two types recover the differential on $C_{2\epsilon} |_{A=0}$ and $C_{-2\epsilon} |_{A=0} [-1]$, so $C_{\Psi}$ is the mapping cone of a map defined by counting bigons of the third type; we define $\Psi: C_{2\epsilon}|_{A=0} \to C_{-2\epsilon}|_{A=0}$ to be this map. More precisely, counting bigons from intersection points on the $\mu_{2\epsilon}$ part of $\mu_{\Psi}$ to intersection points on the $\mu_{-2\epsilon}$ part of $\mu_{\Psi}$ defines a degree $-1$ map $\widetilde\Psi$ from $C_{2\epsilon} |_{A=0}$ to $C^-_{-2\epsilon} |_{A=0} [-1]$, which may also view as a degree 0 map from $\widetilde\Psi: C_{2\epsilon} |_{A=0} \to C^-_{-2\epsilon} |_{A=0}$. We define $\Psi$ in terms of $\widetilde\Psi$ so that
$$\widetilde\Psi(x)  = (-1)^{\gr_w(x)} \Psi(x).$$
That is, $\Psi$ counts bigons with the usual conventions except with an extra minus sign for bigons whose boundary orientation opposes the orientation of $\Gamma$. We remark that the distinction between $\Psi$ and $\widetilde\Psi$ can be ignored when working with $\Z/2\Z$ coefficients; in general the extra signs are needed so that the differential on the Floer complex of $(\Gamma, \bchain)$ and $\mu_{\Psi}$ is given by
$$\left( \begin{array}{cc}
\partial_{C_{2\epsilon}|_{A=0}} & 0 \\
\widetilde\Psi & \partial_{C_{-2\epsilon}|_{A=0}} \end{array} \right)$$
which after a change of basis replacing $x$ with $-x$ for $x$ in $C_{-2\epsilon}|_{A=0}$ with odd grading becomes
$$\left( \begin{array}{cc}
\partial_{C_{2\epsilon}|_{A=0}} & 0 \\
\Psi & -\partial_{C_{-2\epsilon}|_{A=0}} \end{array} \right),$$
the differential on $\text{Cone}(\Psi)$.  The map $\Psi$ defined in this way on $C_{2\epsilon} |_{A=0}$ can be uniquely extended as a skew $\sRminus$-module homomorphism to a map $\Psi: C^\infty_{2\epsilon} \to C^\infty_{-2\epsilon}$. Finally, using the identification above we can view $\Psi$ as a map from $C^\infty$ to $C^\infty$.

\begin{proposition}
$\Psi$ is a flip-filtered chain homotopy equivalence and has skew-degree $(0,0)$.
\end{proposition}
\begin{proof}
$\Psi$ is a skew $\sRminus$-module homomorphism by construction and it is a chain map because $\partial^2 = 0$ on $C_{\Psi} = \text{Cone}(\Psi)$. It is clear that $\Psi$ is flip-flip filtered since it takes each generator of $C_{+2\epsilon}$ (which has $V$-filtration level zero as an element of $C^\infty$) to a sum of terms that are each a nonnegative power of $W$ times a generator $C_{-2\epsilon}$ (which has $U$-filtration level zero). The complex $C_{\Psi}$ clearly has trivial homology, since $\mu_{\Psi}$ can be homotoped off of $\Gamma$; it follows that $\Psi$ is a chain homotopy equivalence. Since $\mu_{\Psi}$ has bigrading $(\tfrac 1 2, \tfrac 1 2)$ on the portion coming from $\mu_{2\epsilon}$ and bigrading $(-\tfrac 1 2, -\tfrac 1 2)$ on the portion coming from $-\mu_{-2\epsilon}$, the map $\Psi$ on $C_{+2\epsilon} |_{A=0}$ preserves the bi-grading. Since the Alexander grading is zero interchanging the gradings has no effect and $\Psi$ has skew-degree $(0,0)$. This property is preserved when $\Psi$ is extended as a skew-module homomorphism.
\end{proof}

\begin{example}\label{ex:flip-map-Z1}
Consider the curve $(\Gamma, \bchain)$ in $\cylinder$ shown in Figure \ref{fig:flip-map-Z1}; this is obtained from the decorated immersed arc in $\strip$ in Figure \ref{fig:complex-from-curve-example} by identifying the sides of $\strip$ such that the endpoints of the immersed arc are identified. The complex $C = C(\Gamma, \bchain)$ is the complex $C$ computed in Example \ref{ex:complex-from-curves} (it is easy to see that, in this case, identifying the sides of the strip does not affect the differential). The flip map $\Psi: C^\infty \to C^\infty$ counts generalized bigons between $(\Gamma, \bchain)$ and $\mu_{\Psi}$ involving the cap portion of $\mu_{\Psi}$; there are three such bigons, shown in Figure \ref{fig:flip-map-Z1}. The first contributes $V^{-1}a$ to $\Psi(U^{-1}g)$, the second contributes $W(V^{-1}a) = Ua$ to $\Psi(Ub)$, and the third contributes $W(V^{-1}a) = Ua$ to $\Psi(Ua)$. Using the skew-module homomorphism property of $\Psi$, we have $\Psi(a) = \Psi(b) = UV^{-1} a$, $\Psi(g) = a$, and $\Psi(c) = \Psi(d) = \Psi(e) = \Psi(f) = 0$. 

We remark that we can more easily find the induced isomorphism $\Psi_*$ from the horizontal homology of $C$ to the vertical homology of $C$ by perturbing $\mu_{\Psi}$ so that the vertical portions are closer to $\mu_{\tfrac 1 2}$. When we do this, the single intersection of the left vertical piece of $\mu_{\Psi}$ with $\Gamma$ corresponds to $g$ (the generator of horizontal homology of $C$), the single intersection of the right vertical piece with $\Gamma$ corresponds to $a$ (the generator of vertical homology), and the single bigon indicates that $\Psi_*(g) = a$.
\end{example}

\begin{figure}
\labellist
  \tiny
  \pinlabel {$V^{\scalebox{.75}{-1} } \! a$} at  51 140
  \pinlabel {$V^{\scalebox{.75}{-1} } \! a$} at  186 140
  \pinlabel {$V^{\scalebox{.75}{-1} } \! a$} at  321 140
  
  \pinlabel {$U^{\scalebox{.75}{-1} } \! g$} at  76 11
  \pinlabel {$U b$} at  209 125
  \pinlabel {$U a$} at  345 131

  \scriptsize
  \color{blue}
  \pinlabel {$\mu_\Psi$} at  70 148
  \pinlabel {$\mu_\Psi$} at  205 148
  \pinlabel {$\mu_\Psi$} at  340 148

  \color{red}
  \pinlabel {$-1$} at  17 81
  \pinlabel {$1$} at  85 85
  \pinlabel {$-1$} at  152 81
  \pinlabel {$1$} at  220 85
  \pinlabel {$-1$} at  287 81
  \pinlabel {$1$} at  355 85

\endlabellist
\includegraphics[scale = 1]{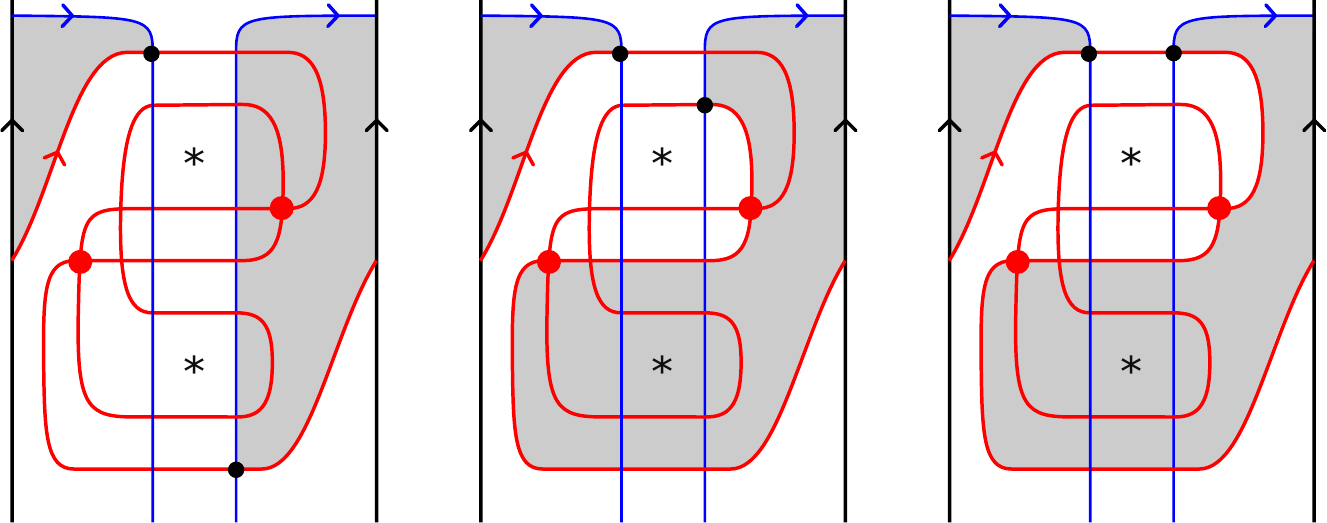}
\caption{The three bigons contributing to the flip map $\Psi$ associated to the pictured curve $(\Gamma, \bchain)$ in $\cylinder$.}
\label{fig:flip-map-Z1}
\end{figure}

\begin{example}\label{ex:flip-maps-from-curve}
Consider the curve $\Gamma$ in $\cylinder$ shown in Figure \ref{fig:flip-maps-from-curve}. The values of the bigrading function on $\Gamma$ at intersections with $\mu$ are indicated in the figure. We equip this curve with the trivial bounding chain $\bchain$ and omit it from the notation. The complex $C = C(\Gamma)$ is the Floer complex (in $\dmcylinder$) of $\Gamma$ with the vertical line $\mu$. There are five generators, $a'$, $b'$, $c'$, $d'$, and $e'$, with bigradings $(2,0)$, $(1,1)$, $(0,0)$, $(1,1)$, and $(0,2)$, respectively. There are four bigons giving the differential
$$\partial(a') = -Vb', \quad \partial(b') = 0, \quad \partial(c') = UVb' - UVd', \quad \partial(d') = 0, \quad \text{ and } \quad \partial(e') = Ud'.$$
To determine the flip map we consider the Floer complex with $\mu_{\Psi}$, specifically the terms coming from bigons involving the cap portion of $\mu_{\Psi}$. There are three obvious bigons arising from the three arcs from the left side of $\mu_{\Psi}$ to the right side of $\mu_{\Psi}$. Note that because of how $\Psi$ is defined from a curve, all three of these bigons contribute with positive sign even though one of the arcs is oriented leftward. There is also a fourth bigon contributing to $\Psi$, which is shown in the figure. From these four bigons we see that
$$\Psi(a') = UV^{-1} a' + V^{-1}c', \quad \Psi(b') = d', \quad \Psi(c') = V e', \quad \Psi(d') = 0, \quad \text{ and } \quad \Psi(e') = 0.$$
We observe that this complex and flip map agrees with the one in Example \ref{ex:1-surgery-on-LHT} coming from the dual knot of $+1$-surgery on the left handed trefoil after the change of basis given by
$$a' = -a, \quad b' = b, \quad c' = -c + Ua - Ve, \quad d' = d, \quad \text{ and } \quad e' = e.$$
\end{example}

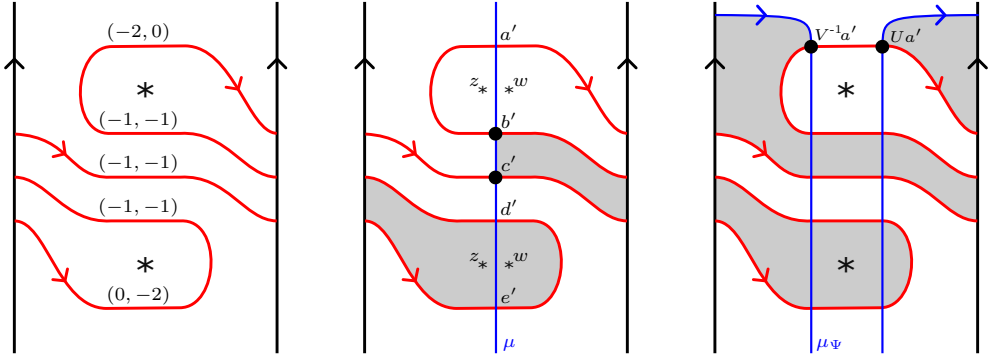
\begin{figure}
\labellist
  \scriptsize
  \pinlabel {$a'$} at  173 110
  \pinlabel {$b'$} at  173 80
  \pinlabel {$c'$} at  173 65
  \pinlabel {$d'$} at  173 50
  \pinlabel {$e'$} at  173 20
  
  \pinlabel {$z$} at 161 93
  \pinlabel {$w$} at 177 93
  \pinlabel {$z$} at 161 33
  \pinlabel {$w$} at 177 33
  
  \tiny
  \pinlabel {$V^{\scalebox{.75}{-1} } \! a'$} at  285 110
  \pinlabel {$U a'$} at  309 109
  
  \scriptsize
  \color{blue}
  \pinlabel {$\mu$} at  173 3
  \pinlabel {$\mu_{\Psi}$} at  283 3

  \color{black}

  \pinlabel {$(-2,0)$} at  46 110
  \pinlabel {$(-1,-1)$} at  46 80
  \pinlabel {$(-1,-1)$} at  46 65
  \pinlabel {$(-1,-1)$} at  46 50
  \pinlabel {$(0,-2)$} at  46 20
  
\endlabellist
\includegraphics[scale = 1.1]{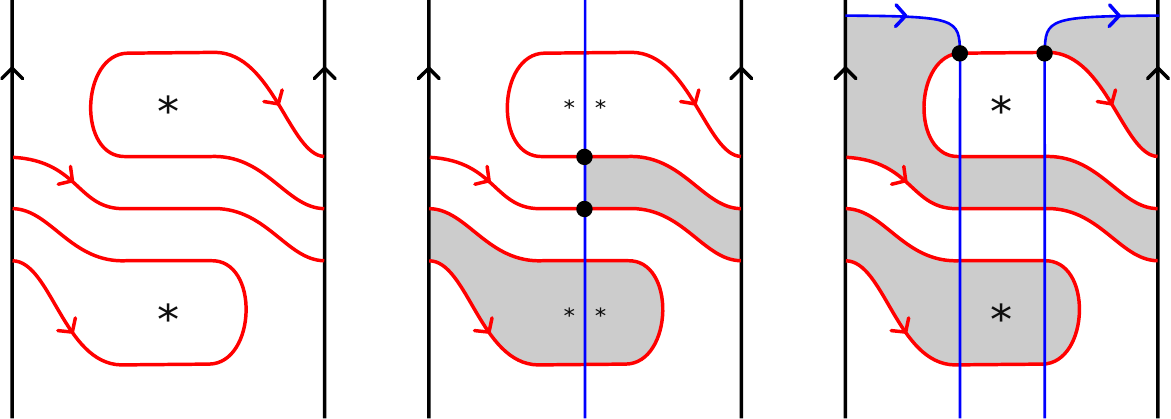}
\caption{Left: A curve $\Gamma$ in $\cylinder$; the value of the bigrading function $(\tilde\tau_w, \tilde\tau_z)$ at intersections with $\mu$ is indicated. Middle: The curve in $\dmcylinder$ paired with $\mu$ to compute the complex $C(\Gamma)$; the shaded bigon gives the term $UVb'$ in $\partial(c')$. Right: The curve paired with $\mu_{\Psi}$ to compute the flip map; the shaded bigon contributes $W(V^{-1} a') = Ua'$ to $\Psi(Ua')$.}
\label{fig:flip-maps-from-curve}
\end{figure}

\begin{figure}
\labellist
  \pinlabel {$\dmstrip$} at 53 -5
  \pinlabel {$\sF$} at 115 -5
  
  \scriptsize
  \pinlabel {$a'$} at 53 126
  \pinlabel {$b'$} at  53 96
  \pinlabel {$c'$} at  53 81
  \pinlabel {$d'$} at  53 66
  \pinlabel {$e'$} at  53 36
  
  \pinlabel {$z$} at 41 109
  \pinlabel {$w$} at 57 109
  \pinlabel {$z$} at 41 49
  \pinlabel {$w$} at 57 49
  
  \tiny
  \pinlabel {$U a'$} at  284 96
  \pinlabel {$c'$} at  281 81
  \pinlabel {$b'$} at  281 60
  \pinlabel {$d'$} at  324 96
  \pinlabel {$c'$} at  324 81
  \pinlabel {$V e'$} at  327 60

  \scriptsize
  \color{blue}
  \pinlabel {$\mu$} at  53 17
  \pinlabel {$\mu_{\Psi}$} at  285 17
  
  \color{red}
  \pinlabel {$-W$} at  23 82
  \pinlabel {$W$} at  85 82

\endlabellist
\includegraphics[scale = 1.1]{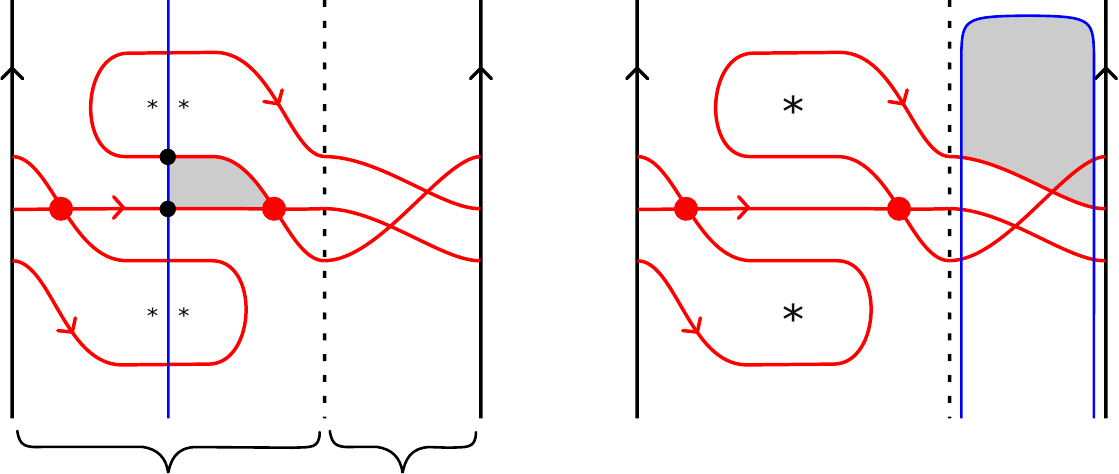}
\vspace{3mm}
\caption{A decorated curve $(\Gamma, \bchain)$ in $\dmcylinder = \dmstrip \cup \sF$ obtained from the curve in \ref{fig:flip-maps-from-curve} by a homotopy, with the property that bigons contributing to the differential on $C(\Gamma, \bchain)$ are contained in $\strip$. The bounding chain is now nontrivial. The shaded generalized bigon on the left contributes the term $UVb$ in $\partial(c)$. Perturbing $\mu_{\Psi}$ to lie in $\sF$ as on the right and taking Floer homology determines the flip isomorphism $\Psi_*$.}
\label{fig:flip-maps-from-curve-perturbed}
\end{figure}

When we computed $C$ in the last example, some of the bigons wrapped completely around the cylinder. As a result, if we cut $\cylinder$ open along $\mu_{\tfrac 1 2}$ to get $\strip$ and computed the complex associated with the resulting curve, we would not get the same result. However, it is always possible to perturb a decorated curve $(\Gamma, \bchain)$ in $\cylinder$ so that bigons contributing to $C(\Gamma, \bchain)$ do not intersect $\mu_{\tfrac 1 2}$. Such a perturbation for Example \ref{ex:flip-maps-from-curve} is shown in Figure \ref{fig:flip-maps-from-curve-perturbed}. Note that the homotopy performed to the curve introduced immersed monogons and therefore requires including some of the new self-intersection points in the bounding chain, following move $(j)$ in Figure \ref{fig:invariance-moves}.

When curves in $\cylinder$ are perturbed as above, we can cut the the cylinder into two strips, a marked strip $\strip$ and an unmarked strip $\sF$, so that the $C$ is determined by the restriction of $(\Gamma, \bchain)$ to $\strip$ and the flip isomorphism $\Psi_*$ is determined by the restriction to $\sF$. We simply let $\sF$ be a neighborhood of $\mu_{\tfrac 1 2}$ and let $\strip$ be the complement of $\sF$ in $\cylinder$. We choose this neighborhood small enough so that bigons contributing to the differential in $C$ (which are assumed to be disjoint from $\mu_{\tfrac 1 2}$) are disjoint from $\sF$. We also choose $\sF$ to be small enough so that the restriction of $\Gamma$ to $\sF$ consists of arcs that move from one side of $\sF$ to the other (i.e. $\Gamma$ has no vertical tangencies in $\sF$). Under this assumption, we can perturb the curve $\mu_{\Psi}$ by sliding the vertical portions from $\mu_{2\epsilon}$ to $\partial_L \sF$ and from $\mu_{-2\epsilon}$ to $\partial_R \sF$ and all bigons contributing to $\Psi_*$ are containing in $\sF$. Counting bigons contributing to $\Psi_*$ is then the same as counting paths from $\partial_L \sF$ to $\partial_R \sF$ (these may be polygonal paths, if the bounding chain includes self-intersection points in $\sF$). We caution that the intersections of $\Gamma$ with $\partial_L \sF$ and $\partial_R \sF$ correspond to a basis of the horizontal and vertical complexes of $C$, respectively, but these may not be the same as the bases coming from intersecting $\Gamma$ with $\mu_{2\epsilon}$ and $\mu_{-2\epsilon}$.

\subsection{Naive curve representatives of bigraded complexes}\label{sec:naive-curves} We have seen that any immersed multicurve $\Gamma$ with a bounding chain $\bchain$ in $\strip$ determines a bigraded complex over $\sRminus$ along with a choice of (ordered) basis corresponding to the intersection points of $\Gamma$ with $\mu$. In this case we say that $(\Gamma, \bchain)$ represents the complex with respect to this basis. We will now observe that any bigraded complex over $\sRminus$ can be represented in this way.

\begin{proposition}\label{prop:naive-curves}
For any bigraded complex $C$ over $\sRminus$ and any choice of basis $\{x_1, \ldots, x_n\}$ for $C$, there exists an immersed multicurve $\Gamma$ in $\strip$ with a collection of turning points $\bchain$ such that $C(\Gamma, \bchain)$ is isomorphic to $C$, with the isomorphism taking the preferred basis of $C(\Gamma, \bchain)$ to $\{x_1, \ldots, x_n\}$. The same is true for complexes over any quotient of $\sRminus$.
\end{proposition}

\begin{proof}
The curve $\Gamma$ consists of a collection of $n$ immersed arcs, one for each generator of $C$, connecting the left and right boundaries of $\strip$. These arcs intersect $\mu$ at appropriate heights determined by the Alexander grading of the corresponding generator, but the endpoints on each side are rearranged so that every arc crosses every other arc once on each side of $\mu$. The grading function $\tilde\tau$ on $\Gamma$ is defined to agree with $-\gr_w $ at intersection points corresponding to generators; in particular, the arc corresponding to $x_i$ is oriented rightward if $\gr_w(x_i)$ is even and leftward if $\gr_w(x_i)$ is odd.

To define the collection of turning points $\bchain$, for each arrow from $x_i$ to $x_j$ in the complex we include the intersection point $p_{ij}$ between the arcs corresponding to $x_i$ and $x_j$ that is left of $\mu$ if $x_i$ is above $x_j$ or right of $\mu$ if $x_i$ is below $x_j$. If the arrow from $x_i$ to $x_j$ has weight $cU^a V^b$, $p_{ij}$ appears in $\bchain$ with coefficient $c$ if the polygonal path in $\Gamma$ from $x_i$ to $p_{ij}$ to $x_j$ follows the orientation on $\Gamma$, and with coefficient $-c$ otherwise. To compute $\partial$ on $C(\Gamma, \bchain)$, we count generalized immersed bigons whose boundary is a polygonal path in $\Gamma$ consistent with $\bchain$ along with a segment of $\mu$. For any such bigon, the net rotation along the polygonal path in $\Gamma$ must be $\pi$. On the other hand, the net rotation along any polygonal path in $\Gamma$ consistent with $\bchain$ between any two points on $\mu$ is given by $\pi$ times the number of left turns in the path. This is because the immersed arcs move monotonically rightward or leftward, while at left turns the polygonal path must go from rightward moving to leftward moving or from leftward moving to rightward moving (since arrows in $C$ connect generators with gradings of opposite parity, so only intersections between oppositely oriented arcs appear in $\bchain$). Thus to compute $C(\Gamma, \bchain)$ we need only consider the triangles with corners $x_i$, $x_j$, and $p_{ij}$ for each pair $1 \le i, j \le n$. The coefficients on each $p_{ij}$ in $\bchain$ were chosen so that each triangle precisely recovers an arrow in $C$.

Note that the collection of turning points $\bchain$ constructed here is not necessarily a bounding chain. However, by construction the precomplex $C(\Gamma, \bchain)$ is a complex so it still makes sense to say that $(\Gamma, \bchain)$ represents $C$.
\end{proof}

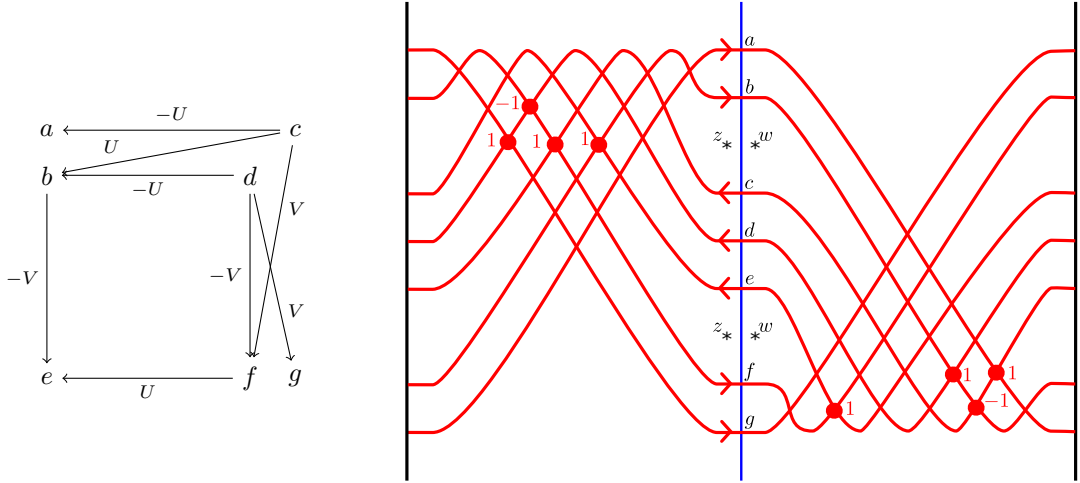
\begin{figure}
\raisebox{10 mm}{\begin{tikzpicture}[scale = 1.5]
\node (a) at (0,2.2) {$a$};
\node (b) at (0,1.8) {$b$};
\node (c) at (2.2,2.2) {$c$};
\node (d) at (1.8,1.8) {$d$};
\node (e) at (0,0) {$e$};
\node (f) at (1.8,0) {$f$};
\node (g) at (2.2,0) {$g$};

\scriptsize
\draw[->] (c) to node[above] {$-U$} (a);
\draw[->] (c) to node[above left, pos = .7] {$U$} (b);
\draw[->] (d) to node[below] {$-U$} (b);
\draw[->] (f) to node[below] {$U$} (e);

\draw[->] (d) to node[left] {$-V$}( f);
\draw[->] (d) to node[right, pos = .7] {$V$} (g);
\draw[->] (c) to node[right, pos = .3] {$V$} (f);
\draw[->] (b) to node[left] {$-V$} (e);

\end{tikzpicture}}
\hspace{10 mm}
 \includegraphics[scale=1.2]{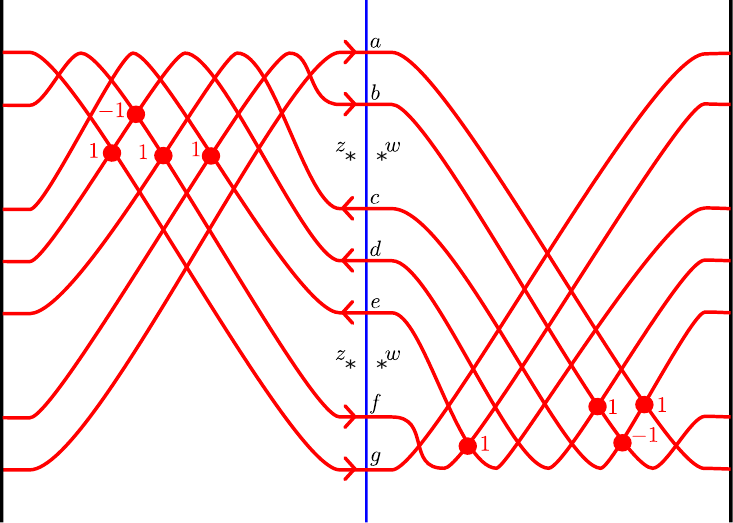}
 \caption{A naive immersed curve representative for a complex.}
 \label{fig:naive-curves-example}
\end{figure}

A pair $(\Gamma, \bchain)$ constructed as above will be called a \emph{naive immersed curve representative} for $C$. As an example, Figure \ref{fig:naive-curves-example} shows the naive immersed curve representative for the complex in Example \ref{ex:complex-from-curves}.

By similar reasoning, we could construct a naive representative for a complex $C$ with a flip map $\Psi$ in the cylinder $\cylinder$. Viewing $\cylinder$ as the union of two strips $\strip$ and $\sF$, the decorated curve restricted to $\strip$ is the naive representative of $C$. In $\sF$ we place another copy of the decorated curve in $\strip$ representing $C$ (or the image of this under the obvious map $\strip \to \sF$ that ignores the marked points), and then add additional appropriately weighted turning points in $\sF$ from the segment corresponding to $x_i$ to the segment corresponding to $x_j$ for each $x_j$ term in $\Psi(x_i)$. One can check that the resulting curve $(\Gamma, \bchain)$ represents $C$ and $\Psi$, and that $\bchain$ is a bounding chain.

While it is good to know that any complex can be represented by \emph{some} decorated curve, the naive immersed curve representative is not particularly easy to work with. Moreover, it is highly dependent on the representative of the complex $C$ used to construct it, so a homotopy equivalence class of complexes gives rise to many different decorated curves. Our challenge moving forward will be to find a simpler decorated curve $(\Gamma, \bchain)$ that represents $C$ and that is uniquely determined (up to an appropriate sense of equivalence) for any chain homotopy equivalence class of complexes.

\section{Simple immersed curve representatives of bigraded complexes}\label{sec:properties-of-simple-curves}% !TEX root = ../CFKcurvesZHS3s.tex
%properties-of-simple-curves.tex

\subsection{Simple position for curves} We have established that any complex can be represented by a decorated curve in $\strip$, but these representatives are not particularly nice. With the goal in mind of defining better immersed curve representatives for a complex, we now discuss some constraints we wish to impose on our decorated immersed curves $(\Gamma, \bchain)$. Some assumptions have already been mentioned: we will assume transverse self-intersection with no triple points, and we assume that $\Gamma$ is perpendicular to $\mu$. We will also assume that $\Gamma$ has minimal intersection with $\mu$ (among curves in its homotopy class, where homotopies do not cross the marked points). We assume that no component of $\Gamma$ is disjoint from $\mu$, since such a component would not contribute to $C(\Gamma, \bchain)$ and can be safely ignored. We will often assume that $\Gamma$ has minimial self-intersection in its homotopy class, though this assumption will need to be relaxed at times.

\begin{figure}
\includegraphics[scale=.9]{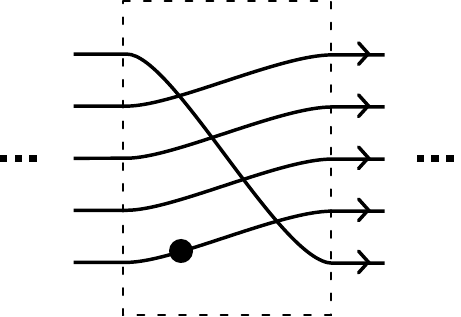}
\vspace{3 mm}
\caption{The crossing region for a non-primitive component of $\Gamma$. The basepoint for the component lies within the crossing region on the strand which crosses all other strands, as shown, and the strands run parallel outside this region.}
\label{fig:crossing-region}
\end{figure}

In addition to requiring curves to be in minimal position, we will also assume that non-primitive components of $\Gamma$ have a particular form. A closed component $\gamma$ of $\Gamma$ is non-primitive if it is homotopic to $k$ copies of some simpler closed curve. If $\gamma'$ is primitive and $[\gamma] = k[\gamma']$, we say $\gamma$ has multiplicity $k$. In this case, we will assume that $\gamma$ lies in a neighborhood of $\gamma'$ and that $\gamma$ looks like $k$ parallel copies of $\gamma'$ outside of a small crossing region of the form depicted in Figure \ref{fig:crossing-region}. There are $k-1$ self intersection points of $\gamma$ in this region, which all have degree 0. We will also assume that there is a single basepoint on $\gamma$ carrying the weight of the curve and that this basepoint lies in the crossing region as shown. Note that for each intersection of $\gamma'$ with $\mu$, there are $k$ intersections of $\gamma$ with $\mu$. We will refer to this set of intersection points, and the corresponding generators of $C(\Gamma, \bchain)$, as a \emph{grouping} of generators. We will assume that $\gamma$ lies in a small enough neighborhood of $\gamma'$ that no other intersections of $\Gamma$ with $\mu$ fall between those in the same grouping.

\begin{definition}\label{def:simple-position}
An immersed multicurve $\Gamma$ in $\strip$ is in \emph{simple position} if 
\begin{itemize}
\item $\Gamma$ has transverse self-intersection points and no triple points;
\item $\Gamma$ is horizontal at intersections with the vertical line $\mu$ as well as intersections with $\partial \strip$;
\item Every component of $\Gamma$ intersects $\mu$, and does so minimally;
\item $\Gamma$ has minimal self-intersection number; and
\item each non-primitive closed component of $\Gamma$ has the form of parallel strands outside of a crossing region, as described above.
\end{itemize}
\end{definition}

At times, we will need to relax the minimal self-intersection condition in two ways. First, we will require that no two components of $\Gamma$ bound an immersed annulus; starting from minimal position, this requires adding a pair of intersection points between any two parallel closed components. Second, we will impose a certain ordering on the endpoints of immersed arc components appearing on $\partial_L \strip$ and $\partial_R \strip$; from minimal position this can be achieved by sliding endpoints of immersed arcs along the relevant boundary of $\strip$. The ordering is specified using the grading functions $\tilde\tau_w$ and $\tilde\tau_z$ on $\Gamma$: endpoints on $\partial_L \strip$ are ordered so that $\tilde\tau_w$ is non-decreasing moving downward along $\partial_L \strip$, and endpoints on $\partial_R \strip$ are ordered so $\tilde\tau_z$ is non-decreasing moving upward along $\partial_R \strip$.

Note that each endpoint $x$ of $\Gamma$ on $\partial \strip$ has a corresponding intersection $x'$ of $\Gamma$ with $\mu$, the first such intersection point reached when traveling along $\Gamma$ from $x$. Moreover, if $x$ is on $\partial_L \strip$ then $\tilde\tau_w(x) = \tilde\tau_w(x')$ since $\Gamma$ is horizontal at both $x$ and $x'$ and there are no $w$-marked points or corresponding grading arcs on the left side of $\mu$ in $\dmstrip$. Similarly, if $x$ is on $\partial_R \strip$ then $\tilde\tau_z(x) = \tilde\tau_z(x')$. Since the bigrading $(\gr_w, \gr_z)$ of the generator of $C(\Gamma, \bchain)$ is given by $(-\tilde\tau_w(x), -\tilde\tau_z(x))$, we can equivalently understand the ordering of endpoints as requiring that endpoints on $\partial_L \strip$ with higher $\gr_w$ appear higher and endpoints on  $\partial_R \strip$ with higher $\gr_z$ appear lower, where here $\gr_w$ and $\gr_z$ refer to the gradings on the corresponding generator of $C(\Gamma, \bchain)$.

A \emph{segment} of $\Gamma$ will refer to a connected component of $\Gamma\setminus(\Gamma\cap\mu)$. Although our curves may not be in minimal position, we will require that any two segments of $\Gamma$ intersect at most once. To ensure this holds, when we add a pair of extra intersection points between parallel curves we should add them on opposite sides of $\mu$, as in Figure \ref{fig:almost-simple-position}.

\begin{figure}
\includegraphics[scale=1.2]{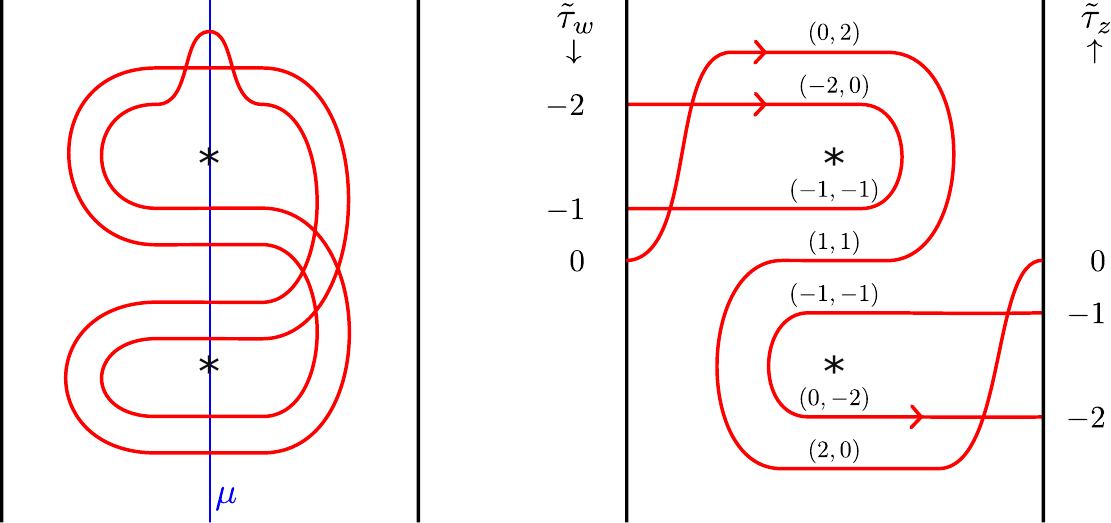}
\vspace{3 mm}
\caption{Multicurves in almost simple position. On the left, starting from minimal position two intersection points were added (on opposite sides of $\mu$) to break up an immersed annulus. On the right, the value of the bigrading $(\tilde\tau_w, \tilde\tau_z)$ is shown for each intersection of with $\mu$, and the endpoints of immersed arcs are arranged in the required order.}
\label{fig:almost-simple-position}
\end{figure}

\begin{definition}\label{def:almost-simple-position}
An immersed multicurve $\Gamma$ in $\strip$ is in \emph{almost simple position} if 
\begin{itemize}
\item $\Gamma$ has transverse self-intersection points and no triple points;
\item $\Gamma$ is horizontal at intersections with $\mu$ or $\partial \strip$;
\item Every component of $\Gamma$ intersects $\mu$, and does so minimally;
\item each non-primitive closed component has the form of parallel strands outside of a crossing region, as described above;
\item The endpoints of $\Gamma$ on $\partial_L \strip$ (resp. $\partial_R \strip$) appear in order of increasing $\tilde\tau_w$ (resp. $\tilde\tau_z$) moving upward (resp. downward); 
\item $\Gamma$ does not bound an immersed annulus; and
\item $\Gamma$ has minimal self-intersection number subject to the above constraints, and any two segments of $\Gamma$ intersect at most once.
\end{itemize}
\end{definition}

We can also define simple position and almost simple position for immersed multicurves in the marked cylinder $\cylinder$: the definitions are the same except that we can ignore the conditions concerning endpoints of immersed arcs in $\partial \strip$.

\subsection{Local systems}\label{sec:local-systems}
The $k-1$ self-intersection points appearing in the crossing region for a non-primitive curve $\gamma$ play a special role, we will call these \emph{local system intersection points}. These self-intersection points always have degree zero. Given an immersed multicurve $\Gamma$ in simple or almost simple position, recall that $\sI$ is the collection of self-intersection points and $\sI_0$ is the subset of these points with degree zero. Let $\sI_0^* \subset \sI_0$ denote the set of local system intersection points for all non-primitive components of $\Gamma$. For a collection of turning points $\bchain$, let $\widehat\bchain$ denote the corresponding linear combination of $\sI_0$ obtained by forgetting intersection points with negative degree, and similarly let $\widehat\bchain^*$ denote the corresponding linear combination of points in $\sI^*_0$. For each non-primitive component $\gamma$ of $\Gamma$, let $\widehat\bchain_\gamma^*$ denote the corresponding linear combination of any local system intersection points on $\gamma$.

A local system is an automorphism of $\F^k$, recorded as a similarity class of invertible $k\times k$ matrices. If $\gamma$ has multiplicity $k$ and the underlying primitive curve is $\gamma'$, then $\widehat\bchain_\gamma^*$ along with the weight associated with the basepoint on $\gamma$ determines a $k$-dimensional local system on $\gamma'$. We think of each generator of $\F^k$ as indexing a parallel copy of $\gamma'$ and the matrix as recording how paths jump between these copies of $\gamma'$ when making one full loop, where the $(i,j)$-entry is a weighted count of paths through the crossing region from the $i$th copy of $\gamma'$ on the left side of the crossing region to the $j$th copy on the right side of the crossing region. We will index the copies $1, \ldots, k$ moving upward on each side, and we let $c_i$ be the coefficient in $\widehat\bchain_\gamma$ of the intersection between the strand starting at 1 and the strand starting at $i+1$ and let $c_0$ be the weight associated with the basepoint on $\gamma$. Then the matrix defining the local system has the following form:

$$\left[ \begin{matrix}
0 & 0 & \cdots & 0 & c_0 \\
1 & 0 & \cdots & 0 & c_1 \\
0 & 1 & \cdots & 0 & c_2 \\
\vdots & \vdots & \ddots & \ddots & \vdots \\
0 & 0 & \cdots & 1 & c_{k-1}
\end{matrix} \right ].$$
This matrix is invertible since we require $c_0$ to be nonzero. Note that this matrix is in rational canonical form. Conversely, a $k$ dimensional local system on $\gamma'$ (that is not a direct sum of smaller local systems) uniquely determines a $k\times k$ matrix in rational canonical form, and the coefficients of this matrix determine the coefficients of the relevant intersection points in $\widehat\bchain_\gamma^*$ and the weight on the basepoint of $\gamma$. In this way the linear combination $\widehat\bchain_\gamma^*$ along with the weight on $\gamma$ is equivalent to a choice of $k$-dimensional local system on the primitive curve underlying $k$. For primitive curves $\gamma$ there are no local system intersection points so $\widehat\bchain_\gamma^*$ is trivial, leaving only the (nonzero) weight $c_0$ of the curve; this can also be interpreted as a $1$-dimensional local system represented by the invertible $1\times 1$ matrix $\left[ c_0 \right]$. Considering all components of the weighted curve $\Gamma$, $\widehat\bchain^*$ and the weights on $\Gamma$ encodes a choice of local system on the primitive curves underlying each component.

\begin{definition}\label{def:local-system-type}
We will say that a linear combination $\widehat\bchain$ of points in $\sI_0$ has \emph{local system type} if the coefficient of any point in $\sI_0 \setminus \sI_0^*$ is zero, so that $\widehat\bchain^* = \widehat\bchain$.
\end{definition}

In Section \ref{sec:simple-curves} we will show that any bigraded complex over $\sRhat$ can be represented up to homotopy by a decorated immersed multicurve $(\Gamma, \widehat\bchain)$ such that $\widehat\bchain$ is of local system type. This is equivalent to saying that every complex over $\sRhat$ can be be represented by an immersed multicurve decorated with local systems (compare \cite[Theorem 1]{HRW} and \cite[Theorem 1]{KWZ:mnemonic}).

\subsection{Useful properties of precomplexes from curves in simple position}

In Section \ref{sec:enhanced-curves}, we will consider immersed multicurves $\Gamma$ with collections of turning points $\bchain$ with the additional hypotheses that $\Gamma$ is in almost simple position and that the restriction $\bchainhat$ of $\bchain$ to degree 0 self-intersection points is of local system type. Decorated curves of this form are relatively well behaved, and we now collect some properties of the corresponding precomplexes that will be useful in Section \ref{sec:enhanced-curves}. In particular, although $\partial^2$ need not be zero on this precomplex, we can show that certain terms of $\partial^2$ must be zero. We will also see that if the precomplex is in fact a complex (i.e. if $\partial^2$ is zero), then the collection of turning points must be a bounding chain.

Some of the properties below depend on a more general observation about monogons bounded by arcs on the boundary of an immersed disk. The following proposition is a restating of Proposition A.2 in \cite{HHHK}.

\begin{proposition}\label{lem:HHHK-prop}
Let $f: D^2 \to S^2$ be an immersion of a disk into $S^2$, and suppose there is an arc $A \subset \partial D^2$ such that $f(A)$ is the counterclockwise oriented boundary of an immersed monogon in $S^2$ (where the orientation on $f(A)$ comes from the boundary orientation of $D^2$). Then there is another arc $B \subset \partial D^2$ such that $f(B)$ is the clockwise oriented boundary of an immersed monogon in $S^2$, and $f(B)$ lies in the interior of the monogon bounded by $f(A)$.
\qed
\end{proposition}

Note that this proposition applies to immersed disks in $S^2$, but it also applies to immersed disks in the strip $\strip$, since $\strip$ is a subset of $\R^2$ which itself is equivalent to $S^2 \setminus \{\infty\}$. It also applies in the infinite cylinder, which is topologically a twice punctured sphere.

Recall that a polygonal path in $\Gamma$ is a piecewise smooth path in $\Gamma$ in which adjacent smooth sections are connected by left-turns at self-intersection points of $\Gamma$. Such paths and immersed polygons they bound can be related to the immersed disk in Lemma \ref{lem:HHHK-prop} by smoothing the corners. When $\Gamma$ is in simple or almost simple position, we have the following restrictions on immersed polygons bounded by polygonal paths in $\Gamma$:

\begin{lemma}\label{lem:clockwise-monogons}
For an immersed multicurve $\Gamma$ in simple or almost simple position, let $P$ be a polygonal path in $\Gamma$ starting and ending at a self-intersection point $q$ of $\Gamma$. If $P$ does not intersect $\mu$, then the smoothing of $P$ is not the clockwise oriented boundary of an immersed monogon.
\end{lemma}
\begin{figure}
 \includegraphics[scale=1]{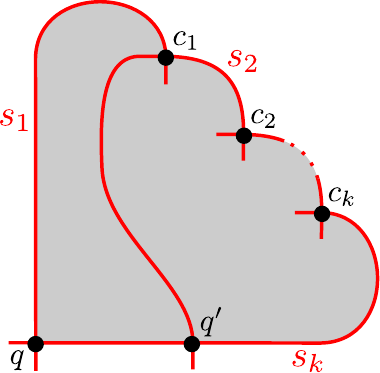}
 \caption{Monogons with corner at $q$ and clockwise oriented boundary containing $k$ train track edges which make left turns at intersection points $c_i$. There is another monogon involving at most $k-1$ of these train track edges with corner at $q'$).}
 \label{fig:clockwise-monogon}
\end{figure}
\begin{proof}
Suppose $P$ is disjoint from $\mu$. The path $P$ consists of some number of arcs $P_0 \ldots P_k$, with each $P_i$ contained in a segment $s_i$ of $\Gamma$, connected by left turns at intersection points $c_1, \ldots, c_k$, where $c_i \in s_{i-1} \cap s_i$ and $q \in s_0 \cap s_k$. If the smoothing of $P$ is the clockwise boundary of an immersed monogon then $P$ is the boundary an immersed polygon with a convex corner at $q$ and all other corners non-convex.

It is clear that there must be at least one obtuse corner $c_i$, since the segments $s_i$ do not intersect themselves. On the other hand, we will show that if such a polygon exists we can construct one with strictly fewer corners; repeating until there are no corners left gives a contradiction. At the corner $c_1$, the segment $s_2$ extends into the interior of the polygon and must leave the polygon again at some point $q' \in s_2 \cap P$. Since any two segments of $\Gamma$ intersect at most once, $q'$ does not lie on $s_1$. We construct a path $P'$ by concatenating the segment of $s_2$ from $q'$ to $c_1$ with the portion of $P$ from $c_1$ to $q'$. This new path $P'$ is the clockwise oriented boundary of an immersed polygon with convex corner at $q'$. This polygon is contained in the one bounded by $P$ and so is still disjoint from $\mu$, and it has strictly fewer crossings.
\end{proof}

\begin{lemma}\label{lem:bigons-have-no-monogons}
For an immersed multicurve $\Gamma$ in simple or almost simple position, let $P$ be a polygonal path in $\Gamma$ from $x$ to $y$ in $\Gamma \cap \mu$ such that $P$ along with the segment of $\mu$ from $y$ to $x$ bounds an immersed polygon. If $P'$ is a portion of $P$ starting and ending at a self-intersection point $q$ and $P'$ does not intersect $\mu$, then $P$ is not the boundary of an immersed polygon.
\end{lemma}
\begin{proof} If we smooth the corners of the immersed polygon bounded by $P$ and part of $\mu$, we get an immersed disk $D$. If $P'$ is the boundary of an immersed polygon, then after smoothing the corners $P'$ becomes the counterclockwise boundary of an immersed monogon with corner at $q$. By Lemma \ref{lem:HHHK-prop}, there must be a portion $P''$ of the smoothing of $P$ which is contained in the interior of this monogon and which is the clockwise boundary of an immersed monogon. This smooth path $P''$ is the smoothing of some polygonal path $P''' \subset P$. This $P'''$ is the clockwise boundary of an immersed polygon with an acute corner at $q$. Since the monogon bounded by the smoothing of $P'$ lies entirely on one side of $\mu$, the path $P''$ and thus $P'''$ is also disjoint from $\mu$, but this is impossible by Lemma \ref{lem:clockwise-monogons}.
\end{proof}

We now consider a multicurve $\Gamma$ in almost simple position equipped with any collection of turning points $\bchain$. We do not assume that $\bchain$ is a bounding chain.

\begin{lemma}\label{lem:some-d-squared-zero-horizontal}
Let $x$ and $x' > x$ be generators of $C(\Gamma, \bchain)$ that are connected by a segment $s_{x,x'}$ of $\Gamma$ on the right side of $\mu$.
\begin{itemize}
\item[(i)] If $y > x'$, then the coefficient of $y$ in $\partial^2 (x)$ is zero.
\item[(ii)] If $y < x$ then the coefficient of $x'$ in $\partial^2 (y)$ is zero.
\end{itemize}
\end{lemma}
\begin{proof}
We will prove (i); (ii) is similar. The $\Ainfty$-relations (Proposition \ref{prop:Ainfty-relations}) imply that $\partial^2(x) = -m^\bchain_2(m^\bchain_0(), x )$. In particular, the $y$ term in $\partial^2(x)$ counts generalized triangles with corners at $x$, $y$, and $p$ for any self-intersection point $p$ of $\Gamma$, weighted with an appropriate coefficient depending on $p$. We will show that there are no such triangles, implying that the $y$ term in $\partial^2(x)$ is zero.

Suppose there is such a generalized immersed triangle $D$. Let $P$ be the polygonal path in $\Gamma$ from $x$ to $y$ that along with the segment from $y$ to $x$ in $\mu$ makes up the boundary of $D$. Note that $P$ consists of a polygonal path consistent with $\bchain$ from $x$ to $p$ along with a polygonal path consistent with $\bchain$ from $p$ to $y$. 
Let $z$ denote the first intersection of $P$ with $\mu$ (not counting the initial point $x$), and let $P'$ denote the portion of $P$ from $x$ until $z$. We first argue, using an induction on paths which agree with $P'$ up to a point, that $x < z \le x'$. Suppose the path $P'$ makes $n$ left turns before reaching $z$, and that the $i$th left turn is from a segment $s_{i-1}$ of $\Gamma$ to a segment $s_i$ of $\Gamma$. For $0\le i \le n$, let $P_i$ denote the piecewise smooth path in $\Gamma$ that starts at $x$ and agrees with $P'$ through the first $i$ corners but then continues along $s_i$ without making any more left-turns until reaching $\mu$ at some point $z_i$. In particular, $P_0$ is just the segment $s_0 = s_{x,x'}$ so that $z_0 = x'$, and $P_n$ agrees with the path $P'$ so that $z_n = z$. Note that $P_0$ and the segment in $\mu$ from $z_0$ to $x$ bounds a bigon $D_0$. We will now show that $x < z_i < z_{i-1}$ and that $P_i$ along with the segment in $\mu$ from $z_i$ to $x$ bounds an immersed polygon $D_i$, for each $0 < i \le n$.

The path $P_i$ diverges from the path $P_{i-1}$ at some intersection point $p_i$. Since $P_i$ makes a left turn at $p_i$ it turns into the interior of $D_{i-1}$ along the segment $s_i$ of $\Gamma$, while the path $P_{i-1}$ continues along the segment $s_{i-1}$ to until it reaches $z_{i-1}$. We consider the point $q$ at which the path $P_i$ first leaves the polygon $D_{i-1}$. This point can not be on the segment of $s_{i-1}$ between $p_i$ and $z_{i-1}$, because $p_i$ is the only intersection between the segments $s_{i-1}$ and $s_i$. It also can not be on the part of $P_i$ before the turn at $p_i$. If it were, the portion of $P_i$ from $q$ to $q$ would be the counterclockwise oriented boundary of an immersed polygon lying entirely on the right side of $\mu$. Moreover, since the path $P$ can only deviate from $P_i$ by turning leftward (into the immersed polygon) and would then need to leave the immersed polygon somewhere, it is clear in this case that a portion of $P$ also bounds a counterclockwise immersed polygon disjoint from $\mu$. But this is impossible by Lemma \ref{lem:bigons-have-no-monogons}. Thus $P_i$ must leave the polygon $D_{i-1}$ at some point $z_i$ on the segment of $\mu$ between $x$ and $z_{i-1}$. It follows that $x < z_i < z_{i-1}$. Removing the triangle formed by $\mu$, $s_i$, and $s_{i-1}$ from the polygon $D_{i-1}$ gives a new polygon $D_i$ whose boundary is the polygonal path $P_i$ followed by the segment of $\mu$ from $z_i$ to $x$.

Since $z \le x'$, we have in particular that $z < y$. We now consider the portion of the boundary of $D$ given by the segment of $\mu$ from $z$ to $x$ followed by the polygonal path $P'$ from $x$ to $z$. This is the counterclockwise boundary of an immersed polygon with a convex corner at $z$, namely the polygon $D_n$. However, this is not possible by the same reasoning used in the proof of Lemma \ref{lem:bigons-have-no-monogons}: if we smooth the boundary of $D$ we get an immersed disk and smoothing the boundary of $D_n$ except the corner at $z$ gives an arc in the boundary of this disc bounding a counterclockwise monogon. By Lemma \ref{lem:HHHK-prop}, there must be a portion of the boundary of the smoothed disk lying strictly inside $D_n$ that bounds a clockwise monogon. This must come from smoothing a portion of the polygonal path $P$ that is the clockwise oriented boundary of an immersed polygon in the interior of $D_n$; since the interior of $D_n$ is disjoint from $\mu$, such a path does not exist by Lemma \ref{lem:clockwise-monogons}.

\end{proof}

\begin{figure}
\includegraphics[scale=1]{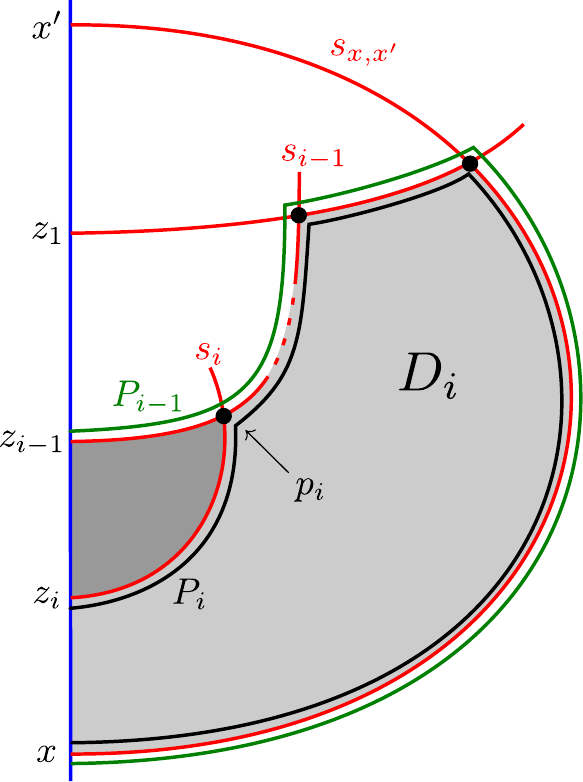}
 \caption{The path $P_i$ and bigon $D_i$ described in the proof of Lemma \ref{lem:some-d-squared-zero-horizontal}. The path $P_i$ differs from $P_{i-1}$ by making one additional turn, and $D_i$ is obtained from $D_{i-1}$ by removing the darkly shaded triangular region.}
 \label{fig:left-turn-bounds}
\end{figure}

The following Lemma is analogous to Lemma \ref{lem:some-d-squared-zero-horizontal} considering segments and bigons on the left side of $\mu$. The proof is identical, with all pictures rotated by $\pi$.

\begin{lemma}\label{lem:some-d-squared-zero-vertical}
Consider an immersed multicurve $\Gamma$ in $\strip$ with a collection of turning points $\bchain$. Let $x$ and $x' < x$ be generators of the precomplex $C(\Gamma, \bchain)$ that are connected by a segment $s_{x,x'}$ of $\Gamma$ on the left side of $\mu$.
\begin{itemize}
\item[(i)] If $y < x'$, then the coefficient of $y$ in $\partial^2 (x)$ is zero;
\item[(ii)] If $y > x$ then the coefficient of $x'$ in $\partial^2 (y)$ is zero. \qed
\end{itemize}
\end{lemma}

Our last observation about $\partial^2 = 0$ in the precomplex $C(\Gamma, \bchain)$ concerns when it is identically zero. As noted earlier, if the collection of turning points $\bchain$ is a bounding chain then $\partial^2 = 0$. The converse to this does not hold in general, but it does when $\Gamma$ in in almost simple position and the restriction $\bchainhat$ of $\bchain$ to degree zero intersection points is of local system type.

\begin{proposition}
Consider an immersed multicurve $\Gamma$ in $\strip$ in almost simple position decorated with a collection of turning points $\bchain$ such that $\bchainhat$ is of local system type. In the precomplex $C(\Gamma, \bchain)$, $\partial^2 = 0$ if and only if $\bchain$ is a bounding chain.
\end{proposition}
\begin{proof}
By the $\Ainfty$-relations, we have that
$$\partial^2(x) = m_1^\bchain( m_1^\bchain (x))  = (-1)^{\gr_z(x)} m_2^\bchain(m_0^\bchain(), x ) - m_2^\bchain (x, m_0^\bchain()). $$
The first term on the right is zero, since $m_0^\bchain$ is zero on $\mu$. Clearly if $\bchain$ is a bounding chain on $\Gamma$, meaning that $m_0^\bchain = 0$, then $\partial^2 = 0$. We need to show the converse, that if $m_0^\bchain$ is not zero then $\partial^2 \neq 0$.

We have that $m_0^\bchain()$ is a linear combination over $\sRminus$ of self-intersection points of $\Gamma$. In fact, since $\deg_w = \deg_z$, the powers of $U$ and $V$ agree in every coefficient of $m_0^\bchain()$. That is, we can think of the coefficients as being in $\F[W]$, where $W$ denotes the product $UV$. For any $p$ in $\sI$ the power of $W$ appearing in the coefficient of $p$ must be $1 + \deg(p)/2$. Thus we have
$$m_0^\bchain() = \sum_{p \in \sI} c_p W^{1+\deg(p)/2} p.$$
Because $\Gamma$ is in almost simple position, it is clear that there can be no generalized immersed monogons that do not either enclose a marked point of $\strip$ or have a false corner at an intersection point in $\bchain$. Moreover, a false corner at a local system intersection point alone does not allow a path to bound a monogon, since the path lies in a neighborhood of a similar path with no false corners. Thus, since $\widehat\bchain$ is of local system type, any monogon either encloses a marked point or contains a false corner at an intersection point with strictly negative degree; in either case, the weight of the monogon is a multiple of $W$. It follows that if the coefficient $c_p$ in $m_0^\bchain()$ is nonzero then $p$ has non-negative degree.

Suppose that $m_0^\bchain()$ is nonzero. We choose an intersection point $p$ with maximal degree such that $c_p$ is nonzero. Let $s_x$ and $s_y$ be the two segments of $\Gamma$ that intersect at $p$. Note that at most one of $s_x$ and $s_y$ has an endpoint on $\partial \strip$; to see this, observe that the ordering on endpoints of immersed arcs assumed when $\Gamma$ is in almost simple position implies that any crossing between two segments approaching the boundary has strictly negative degree. Since either $s_x$ or $s_y$ has two endpoints on $\mu$, we may choose an endpoint $x$ of $s_x$ and an endpoint $y$ of $s_y$ such that the segments are either oriented from $x$ and $y$ to $p$ or they are oriented from $p$ to $x$ and $y$. Up to relabelling $s_x$ and $s_y$, we can assume that $x < y$ if $p$ lies on the right side of $\mu$ and $x > y$ if $p$ lies on the left side of $\mu$.

We will show that the $y$ term in $\partial^2(x)$ in nonzero. By the $\Ainfty$-relations $\partial^2(x) = -m^\bchain_2(x, m^\bchain_0() )$, and the coefficient of $y$ in this is $\sum_{q \subset \sI} c_q n_q$, where $c_q$ is the coefficient of $q$ in $m^\bchain_0()$ and $n_q$ is the weighted count of generalized immersed triangles with corners at $x$, $y$, and $q$. The triangle bounded by $\mu$, $s_x$, and $s_y$ has corners at $x$, $y$, and $p$; since $c_p$ is nonzero by assumption, this gives a nontrivial contribution to the $y$ coefficient in $\partial^2(x)$. Consider any other triangle with corners at $x$, $y$, and $q$. The $\Gamma$ part of the boundary of this triangle is a polygonal path from $x$ to $y$, which must start by moving along $s_x$ and end by moving along $s_y$. Because this polygonal path can only make left turns, similar reasoning to that in Lemma \ref{lem:some-d-squared-zero-horizontal} (see Figure \ref{fig:left-turn-bounds}) implies that the path can not leave the triangle bounded by $s_x$, $s_y$, and $\mu$. The polygonal path must contain at least one false corner at a point in $\bchain$. If the degree of this false corner is negative then $\deg(q) > \deg(p)$, since the sum of the degrees of all corners on a polygonal path from $x$ to $y$ is fixed, being determined by the gradings on $x$ and $y$. But this implies that $c_q$ is zero, since $p$ was chosen to have maximal degree with nonzero $c_p$, and so the triangle does not contribute to the $y$ coefficient of $\partial^2(x)$.

If the false corner has degree zero, then it must be at a local system intersection point since $\widehat\bchain$ is of local system type. It follows that at least one of $s_x$ or $s_y$ lie on non-primitive components of $\Gamma$ (say these have multiplicities $k_x$ and $k_y$, respectively, and that $q$ is the intersection of $s_{x'}$ and $s_{y'}$ where $s_{x'}$ is in the bundle of $k_x$ nearby segments containing $s_x$ and $s_{y'}$ is in the bundle of $k_y$ nearby segments containing $s_{y'}$. Note that there is a partial order on the $k_x k_y$ intersection points between these two bundles of segments where $q < p$ means that $q$ is contained in the triangle formed by the segments intersecting at $p$. In addition to choosing $p$ to have maximal degree among intersection points with $c_p = 0$, we should also choose it to be minimal with respect to this partial ordering. It then follows that if there is another triangle with corner at $q$, $c_q = 0$ and this triangle does not contribute to the $y$ term in $\partial^2(x)$.\end{proof}

\section{Immersed curves representing $UV = 0$ complexes}\label{sec:simple-curves}% !TEX root = ../CFKcurvesZHS3s.tex
%simple-curves.tex

We have seen that an immersed curve $\Gamma$ in the strip $\strip$ decorated with a bounding chain determines a bigraded complex $C(\Gamma, \bchain)$ over $\sRminus$, and that any complex can be represented in this way by some decorated immersed curve in $\strip$. Our next goal is to show that any complex can be represented by a decorated immersed curve of a particularly nice form. In this section, we do this in the substantially easier setting of complexes over $\sRhat$. That is, we will prove the following:

\begin{proposition} \label{prop:simple-curves-existence}
For any bigraded complex $C$ over $\sRhat$, there is a pair $(\Gamma, \bchainhat)$ where $\Gamma$ is an immersed multicurve in $\strip$ in simple position and $\bchainhat$ is a bounding chain of local system type so that $C(\Gamma, \bchainhat)$ is homotopy equivalent to $C$. 
\end{proposition}

Note that since we are working over $\sRhat$, $\bchainhat$ is a linear combination of points in $\sI_0$ rather than $\sI_{\le 0}$; any point with strictly negative degree would contribute a positive power of $UV$ any time it appeared in an immersed polygon contributing to $m^\bchain_k$ and these points can be ignored. Moreover, we only require $\bchainhat$ to be a bounding chain over $\sRhat$; that is, $m^\bchain_0$ vanishes modulo $UV$.

In light of Proposition \ref{prop:naive-curves}, the interesting part of Proposition \ref{prop:simple-curves-existence} is the fact that $\Gamma$ is in simple position and $\bchainhat$ is of local system type. In Section \ref{sec:uniqueness} we will see that for a homotopy equivalence class of complexes the representative $(\Gamma, \bchainhat)$ of this form is unique in the sense that $\Gamma$ is well defined up to homotopy in the punctured strip $\strip^*$ and $\bchainhat$ is well defined as a linear combination of self-intersection points (up to the obvious identification of local system self-intersection points between two homotopic curves in simple position).

\begin{remark}
Proposition \ref{prop:simple-curves-existence} (and the corresponding uniqueness statement) is equivalent to Theorem 5.14 in \cite{KWZ}, and when $\F = \Z/2$ it also follows from the proof of Theorem 5 in \cite{HRW}. We include a proof in order to make this paper self-contained, and since we will build on this crucial construction in later sections when we remove the $UV=0$ simplification. The proof here, while ultimately equivalent to the proofs in \cite{KWZ} and \cite{HRW}, avoids the language of type D structures and uses only the language of bigraded complexes over $\sRhat$, which may be more comfortable for some readers.
\end{remark}

\subsection{Naive train tracks}\label{sec:curves-with-crossover-arrows}

We fix a complex $C$ over $\sRhat$ that we wish to represent by a decorated immersed curve. By Proposition \ref{prop:reduced-simplified-basis}, we may assume that $C$ is reduced. Proposition \ref{prop:naive-curves} ensures that there is some decorated curve representing $C$ over $\sRhat$, so our strategy will be to systematically simplify this representative into the form predicted by Proposition \ref{prop:simple-curves-existence}. To describe this simplification, we will work with immersed train tracks so that we can use the arrow sliding moves developed in Section \ref{sec:train-tracks}.

A naive immersed curve representative for $C$ determines a train track representing $C$, which we call a \emph{naive train track} representing $C$. One approach would be to start applying arrow slide moves to the crossover arrows in this train track (this would involve resolving crossings and would likely result in splitting off immersed arcs with both boundaries on $\partial \strip$, which could then be deleted). However, we can get a significant head start to simplifying the train track if we pick nice bases for the complex $C$. 

\begin{lemma}\label{lem:curves-with-arrows}
Any bigraded complex $C$ over $\sRhat$ is represented by an immersed train track $\tracks$ in $\strip$ of the following form:
\begin{itemize}
\item The restriction of $\tracks$ to the regions $[-\tfrac 1 2, -\tfrac 1 4]\times \R$ and $[0, \tfrac 1 2]\times \R$ consists of a collection of arcs;
\item The restriction of $\tracks$ to the region $[-\tfrac 1 4, 0]\times \R$ consists of a collection of horizontal segments (one for each generator of $C$) connected by crossover arrows; and
\item the crossover arrows either point downward or connect segments corresponding to generators of the same Alexander grading;
\end{itemize}
\end{lemma}

A train track of the form described in Lemma \ref{lem:curves-with-arrows} will be called a \emph{curve-with-arrows train track} representing $C$.

Before proving Lemma \ref{lem:curves-with-arrows}, we first need to review how arrow sliding in a train track relates to changes of basis in the corresponding complex.

\begin{proposition}\label{prop:basis-change}
Suppose $\tracks$ is a train track in $\strip$ representing a reduced complex $C$ over $\sRhat$ with respect to some basis $\{x_1, \ldots, x_n\}$. Let $\tracks'$ be the train track obtained from $\tracks$ by adding crossover arrows in a neighborhood of $\mu$ from $x_i$ to $x_j$, as in Figure \ref{fig:basis-change} (note that we add oppositely weighted arrows on either side of $\mu$ if $A(x_i) = A(x_j)$, an arrow to the right of $\mu$ if $A(x_i) < A(x_j)$, or an arrow to the left of $\mu$ if $A(x_i)>A(x_j)$). Then $\tracks'$ represents $C$ with respect to a different basis in which $x_i$ is replaced with $x_i + c x_j$  if $A(x_i) = A(x_j)$, with $x_i + c U^k x_j$ if $A(x_j) - A(x_i) = k > 0$, or with $x_i + cV^\ell x_j$ if $A(x_i) - A(x_j) = \ell > 0$.
\end{proposition}
\begin{proof}
We first observe that any of the crossover arrows being added are unobstructed modulo $UV$. To see this, note that for any bigon bounded by $\tracks$ and the left side (that is, left when facing in the direction of the arrow) of a rectangular neighborhood of the crossover arrow there is a corresponding bigon bounded by $\tracks$ and $\mu$. Since $C$ is reduced, the weight of any such bigon should include either $U$ or $V$. If it contains both $U$ and $V$, then this bigon may be ignored. If not, then the bigon formed with $\mu$ must have corners at different Alexander gradings. The corresponding bigon with the left side of a neighborhood of the crossover arrow then cover at least one additional marked point of the type not covered by the bigon with $\mu$, so again this bigon has weight a multiple of $UV$.

In the case that $A(x_i) = A(x_j)$ we can introduce two oppositely weighted arrows on the left side of $\mu$ without affecting the corresponding complex (c.f. the $n$-strand arrow replacements in the top row of Figure \ref{fig:local-moves}). Then sliding one of these arrows across $\mu$ has the effect of the given change of basis, by Proposition \ref{prop:arrow-slide-basis-change}. In the case that $A(x_j) - A(x_i) = k > 0$, we can similarly add two oppositely weighted crossover arrows on the right side of $\mu$ and then slide the left arrow (with weight $-c$) leftward. The arrow first slides past $k$ different $w$ marked points, changing its weight to $-cV^k$. It then crosses $\mu$, which has the effect of the given change of basis, by Proposition \ref{prop:arrow-slide-basis-change}. Finally, it slides past $k$ different $z$ marked points, adding a factor of $U^k$ to its weight. At this point the arrow can be deleted since we are working modulo $UV$. The case that $A(x_i) - A(x_j) = \ell > 0$ is similar.
\end{proof}

\begin{figure}
\includegraphics[scale=1]{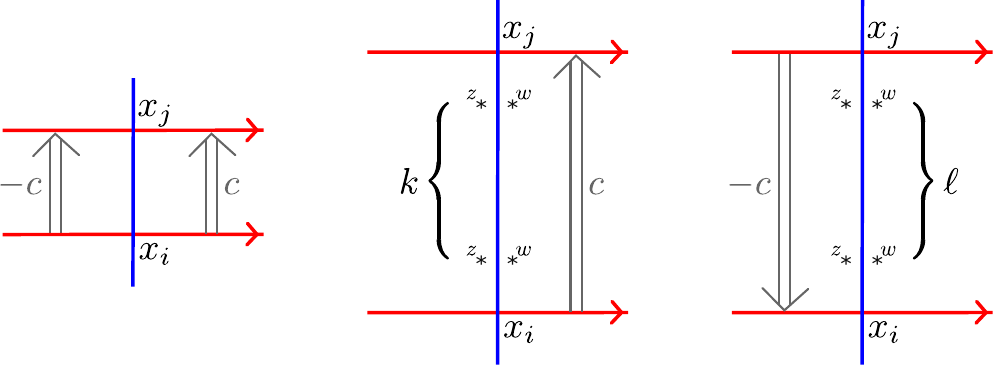}
\caption{Adding the crossover arrows pictured corresponds to a change of basis replacing $x_i$ with $x_i + cx_j$, $x_i + cU^k x_j$, or $x_i + cV^\ell x_j$, respectively. If the horizontal segments are oriented leftward, we multiply the weights on the crossover arrows by $-1$.}
\label{fig:basis-change}
\end{figure}

Another key observation in the proof of Lemma \ref{lem:curves-with-arrows} is that, for any train track $\tracks$ representing $C$ over $\sRhat$, every bigon contributing to $\partial^\tracks$ lies on one side of the $\mu$ or the other (or, more precisely, for any pair of generators the weighted sum of bigons that have portions on both sides of $\mu$ is zero). For any bigon with portions on both sides of $\mu$, consider a connected component of the bigon minus $\mu$ that is a bigon. If this bigon encloses a marked point, then the bigger bigon encloses a marked point of each type and does not contribute modulo $UV$. If the smaller bigpon has a positive power of $W$ in the weight of its boundary, then so does the larger bigon which again does not contribute $UV$. If neither of these two things happen, then the small bigon contributes to $\partial^\tracks$ with no power of $U$ or $V$. Since we assume that $C$ is reduced and $C^\tracks$ agrees with $C$, this bigon must cancel with another bigon and the same is true for the original larger bigon. Because we only need to consider bigons that lie on one side of $\mu$, we can split $\tracks$ into two pieces $\tracks_L$ and $\tracks_R$ on the left and right side of $\mu$, respectivey, and work with these two sides independently. The train track $\tracks_L$ in the left half of $\strip$ represents the hat vertical complex of $C$ (that is, the $U=0$ complex) and the train track $\tracks_R$ in the right half of $\strip$ represents the hat horizontal complex of $C$ (that is, the $V=0$ complex).

\begin{proof}[Proof of Lemma \ref{lem:curves-with-arrows}]
Consider a horizontally simplified basis $\{x^h_1, \ldots, x^h_n\}$ for $C$. There is a particularly nice train track $\tracks^h_R$ on the right side of $\mu$ that represents the horizontal complex with respect to this basis; $\tracks^h_R$ has a point on $\mu$ for each generator $x^h_i$ and an appropriately weighted immersed curve segment connecting the points corresponding to $x^h_i$ to $x^h_j$ for each horizontal arrow from $x^h_i$ to $x^h_j$. Because the basis is horizontally simplified, there is at most one such segment attached  to each point. For any point $x^h_i$ with no segment attached, we attach a horizontal segment connecting $x^h_i$ to $\partial_R \strip$. We let $\tracks^h$ be the union of $\tracks^h_R$ with any train track $\tracks^h_L$ on the left side of $\mu$ representing the vertical complex of $C$ with respect to the basis $\{x^h_1, \ldots, x^h_n\}$ (for example, this could be the left side of a naive train track). Thus $\tracks^h$ is simple on the right side of $\mu$ but (potentially) complicated on the left side of $\mu$. We can similarly define a train track $\tracks^v = \tracks^v_L \cup \tracks^v_R$ representing $C$ with respect to a vertically simplified basis $\{x^v_1, \ldots, x^v_n\}$ such that $\tracks^v_L$ is a collection of arcs and $\tracks^v_R$ is potentially complicated.

We will modify the train track $\tracks^v$ by compressing $\tracks^v_L$ into the strip $[-\tfrac 1 2, -\tfrac 1 4]\times \R$, compressing $\tracks^v_R$ into the strip $[\tfrac 1 4, \tfrac 1 2]\times \R$, and replacing each intersection with $\mu$ with a horizontal line segment across the strip $[-\tfrac 1 4, \tfrac 1 4]\times \R$; it is clear that this homotopy does not affect the associated complex. We then modify the train track further by adding crossover arrows in the strip $[-\tfrac 1 4, \tfrac 1 4]\times \R$, which realizes a change of basis in the corresponding complex. In particular, consider an elementary change of basis that replaces $x_i$ with $x_i + c U^k V^\ell x_j$, where $A(U^k V^\ell x_j) \le A(x_i)$. If $k$ and $\ell$ are both positive, then the corresponding basis change has no effect modulo $UV$ and we do not modify the train track. If $k = l = 0$ then $x_i$ and $x_j$ have the same Alexander grading; we realize a basis change of this form by inserting a pair of crossover arrows on either side of $\mu$ connecting the $i$th horizontal segment to the $j$th horizontal segment. These arrows have weights $c$ and $-c$ as in Figure \ref{fig:basis-change}. If $k = 0$ and $\ell > 0$ we add a crossover arrow with weight $c$ from the $i$th segment to the $j$th segment on the left side of $\mu$ (this arrow moves downward since $A(x_j) < A(x_i)$ in this case). If $\ell = 0$ and $k \ge 0$ we add an upward moving crossover arrow from the $i$th segment to the $j$th segment on the right side of $\mu$. An elementary basis change that replaces $x_i$ with $c x_i$ can be realized by adding basepoints of weights $c$ and $c^{-1}$ on the $i$th segment on either side of $\mu$, and a change of (ordered) basis switching two generators can be realized by introducing a crossing between the corresponding horizontal strands on either side of $\mu$.

The horizontally simplified basis $\{x^h_i\}$ can be obtained from the vertically simplified basis $\{x^v_i\}$ by some sequence of elementary basis changes. Adding arrows to $\tracks^v$ as above to realize this sequence of elementary basis changes results in a train track $\tracks' = \tracks'_L \cup \tracks'_R$ that represents $C$ with respect to the horizontally simplified basis. We now have two train tracks, $\tracks^h$ and $\tracks'$, that represent $C$ with respect to the basis $\{x^h_i\}$. The first is simple on the right side of $\mu$ and the second is simple on the left side of $\mu$. We define $\tracks$ to be $\tracks'_L \cup \tracks^h_R$ and observe that it has the desired form.
\end{proof}

For example, consider the complex from Example \ref{ex:complex-from-curves}. Figure \ref{fig:train-track-example2} shows the construction of a curve-with-arrows train track representing this complex, using the vertically simplified basis $\{a, b, d, c, e, f, g-f\}$ and the horizontally simplified basis $\{a-b,b,d,c,e,f,g\}$. This horizontally simplified basis is obtained from the vertically simplified basis by two elementary basis changes, with one replacing $a$ with $a-b$ and the other replacing $g-f$ with $(g-f) + f = g$.

\begin{figure}
\labellist
  \pinlabel {$\tracks^v$} at 45 -10
  \pinlabel {$\tracks'$} at 210 -10
  \pinlabel {$\tracks^h$} at 385 -10
  \pinlabel {$\tracks'_L \cup \tracks^h_R$} at 535 -10
\tiny
  \pinlabel {$1$} at 194 119
  \pinlabel {$-1$} at 232 119
  \pinlabel {$-1$} at 190 30
  \pinlabel {$1$} at 230 30
  \pinlabel {$1$} at 530 119
  \pinlabel {$-1$} at 526 30

         \endlabellist
\includegraphics[scale=.7]{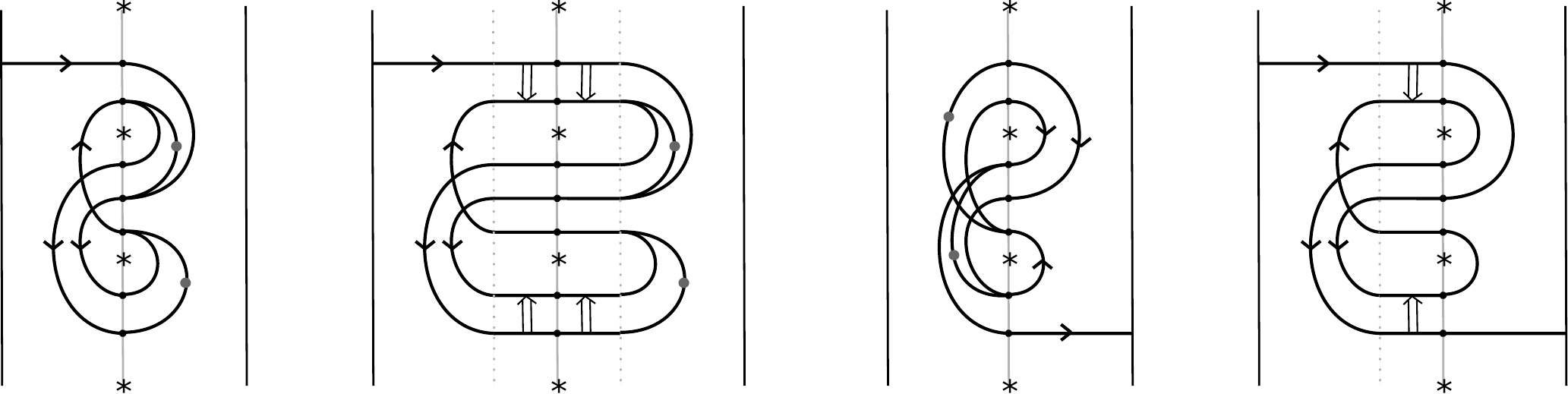} \hspace{2 cm}
 
\caption{The train tracks described in the proof of Lemma \ref{lem:curves-with-arrows} for the example complex shown in Example \ref{ex:complex-from-curves}. The gray dots represent basepoints with weight $-1$, all others edges have weight 1.}

\label{fig:train-track-example2}
\end{figure}
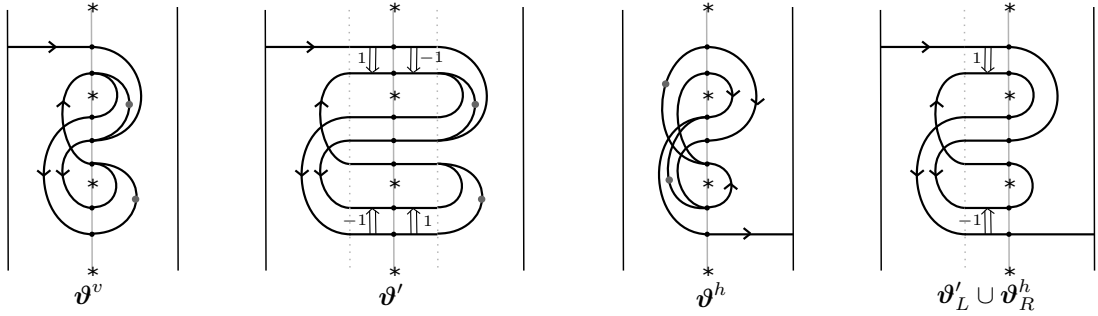

\subsection{Removing crossover arrows}\label{sec:removing-crossover-arrows}

To get from Lemma \ref{lem:curves-with-arrows} to Proposition \ref{prop:simple-curves-existence}, we need to show that essentially all crossover arrows can be removed without changing the complex determined by the train track up to homotopy. This is done using an arrow sliding algorithm that was first described in \cite[Section 3.7]{HRW}. We will now describe the arrow sliding algorithm needed for the train tracks in $\strip$ relevant to bifiltered complexes. The proof presented here is self contained, but some proofs that also appear in \cite{HRW} are repeated more tersely. The reader may wish to compare to the more detailed explanation in \cite{HRW}, which is adapted to train tracks in arbitrary surfaces. Keep in mind that \cite{HRW} assumes $\Ztwo$ coefficents, while we work with an arbitrary field $\F$ (see also \cite{KWZ} for a general treatment of the algorithm with field coefficients). We remark that for train tracks in $\strip$, the algorithm simplifies slightly from the general case.

Let $\tracks$ be a curve-with-arrows train track of the form predicted by Lemma \ref{lem:curves-with-arrows}. $\tracks$ consists of three regions: the middle region, whose right side is the line $\mu$, consists of a horizontal segment for each generator of $C(\tracks)$ and crossover arrows connecting these segments, while the left and right regions consists of arcs connecting pairs of the endpoints of these horizontal segments. We will refer to the Alexander grading of a horizontal segment, by which we mean the Alexander grading of the corresponding generator; this is determined by the height of the segment. We say that the crossover arrows in the middle region are \emph{short} if they connect horizontal segments with the same Alexander grading, and \emph{long} otherwise. By the definition of curve-with-arrows train tracks, all long arrows point downwards. The first step of the simplification is to remove all long arrows by sliding them rightward; once they pass all other arrows and reach the line $\mu$ at the right edge of the middle region, they can be removed by Proposition \ref{prop:basis-change} (this corresponds to an elementary basis change in the corresponding complex $C(\tracks)$). In the process of sliding the long arrow rightward past some number of short arrows, we apply the local moves pictured in Figure \ref{fig:local-moves}. This may at some point introduce a new long arrow, the composition of the long arrow being moved and a short arrow it passes, but this new arrow can be removed in the same way. We can use induction to ensure this process terminates. Suppose the rightmost long arrow has $k$ short arrows to its right. If $k = 0$, we can remove the arrow, and if $k >0$, we slide it past the short arrow immediately to its right. If this short arrow commutes with the long one, the result is a long arrow with $k-1$ short arrows to its right, and if not the result is two long arrows with $k-1$ short arrows to their right. In either case, by induction on $k$, these arrows can be removed, and thus all long arrows can be removed in finite time.

We now turn to removing short arrows. If there is a single crossover arrow, this is easy. The idea is to push the crossover arrow along parallel strands as far as possible until those strands diverge; note that this may involve pushing the arrow through $\mu$, applying a basis change according to Proposition \ref{prop:basis-change}. We say that a crossover arrow is \emph{removable} if there is a path, disjoint from $\tracks$ or $\mu$, from the left side of the arrow either to a puncture or to the boundary of the strip $\strip$ (see Figure \ref{fig:removable-arrows}). In the latter case, there are clearly no bigons whose boundary involves the crossover arrow so it can be deleted with no effect on the complex $C^\tracks$. In the former case, any bigon involving the crossover arrow must meet a puncture. If the crossover arrow is pushed as far as possible it becomes a long arrow, either pointing down on the left side of $\mu$ or up on the right side of $\mu$, and can be removed by Proposition \ref{prop:basis-change}. When two strands diverge, there is a left strand and a right strand; a crossover arrow connecting the two strands will be removable, once pushed to where the strands diverge, if it moves from the left strand to the right strand. Thus to remove a single crossover arrow connecting two strands which eventually diverge, we push the arrow one direction until the strands diverge and remove the arrow if it moves left-to-right. If it moves right-to-left, we slide it the other direction and again remove it if it moves left-to-right. If the arrow is not removable on either end then the strands must have crossed, and we can apply the local move in the last row of Figure \ref{fig:local-moves} to resolve the crossing at the expense of adding a second arrow. The two resulting crossover arrows will be removable when pushed in opposite directions. In this way, any single crossover arrow can be removed unless it connects closed immersed curves that never diverge.

\begin{figure}
\labellist

         \endlabellist
\includegraphics[scale=.7]{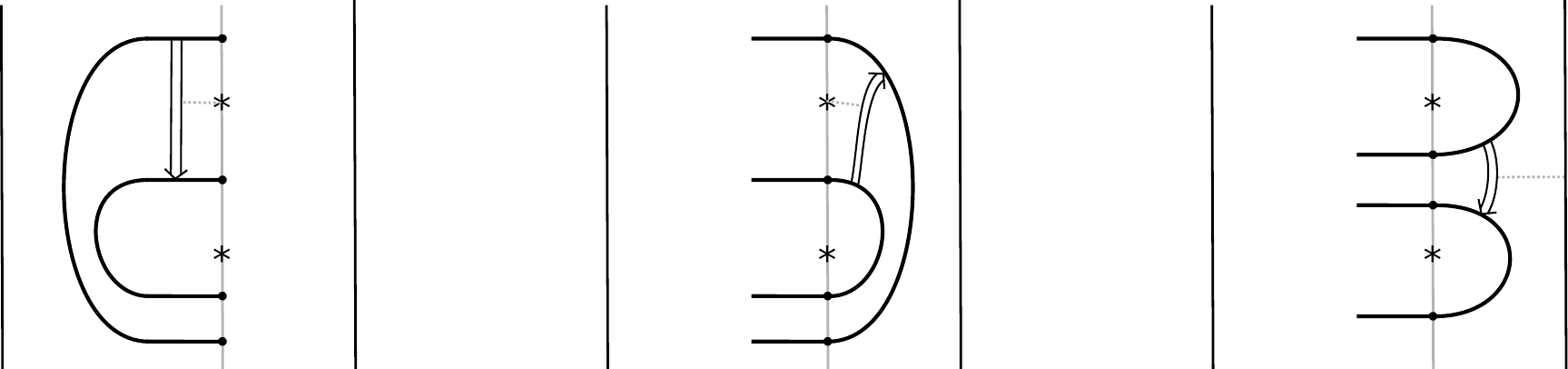}
 
\caption{Examples of removable crossover arrows. If an arrow is parallel to $\mu$ and has a puncture to its left, then it can be removed by a change of basis. If the arrow has an unbounded region of $\strip \setminus (\tracks\cup\mu)$ to its left, then it can not contribute to any bigons and thus can be deleted.}
\label{fig:removable-arrows}
\end{figure}

We need to extend the basic strategy to train tracks $\tracks$ starting with any number of short crossover arrows, showing the $\tracks$ can be reduced to a new train track which consists only of immersed curves and crossover arrows that connect parallel closed curves. The obvious strategy is to remove one arrow at a time as above. The only potential problem is that sliding an arrow past other arrows often introduces new arrows; we need to ensure that the original arrows and any new ones created can all be removed in finite time. To do this, we introduce a notion of complexity for arrows and remove them in order of increasing complexity. We show that arrows can be removed while only creating new arrows of higher complexity. There is an upper bound for complexity of arrows that do not connect parallel closed curves, so eventually these are the only arrows left.

\noindent \emph{Colors and complexity}. Recall that if crossover arrows are ignored, $\tracks$ consists of a collection of immersed curves and arcs. We will label each horizontal segment in the middle region with a \emph{left depth 1 color} and a \emph{right depth 1 color}, which describe how the path in $\tracks$ starting from the segment moving in the given direction behaves. Suppose the horizontal segment has Alexander grading $k$. If the path leaving the segment to the left goes to the left boundary of $\strip$ without returning to the middle region, then the left depth 1 color of the segment is $s$ (this stands for ``straight"). Otherwise, the path returns to the middle region at a horizontal segment with some different Alexander grading $k'$. If the path turned rightward (that is, if $k' > k$), then the left depth 1 color of the starting segment is $r_{k' - k}$. If the path turned leftward (that is, if $k > k'$), then the depth 1 left color of the starting segment is $\ell_{k - k'}$. The depth 1 right color is defined similarly with paths leaving the segment to the right: if the path returns to the middle region at Alexander grading $k'$ the right depth 1 color is $r_{k-k'}$ if $k > k'$ and $\ell_{k'-k}$ if $k' > k$, otherwise the path ends on the right boundary of $\strip$ and the right depth 1 color is $s$. Right and left depth $n$ colors can be be defined inductively to give more information about the paths starting at a horizontal segment $x$. These are length at most $n$ sequences of the letters $\ell_i$, $r_i$ and $s$. If the right (resp. left) depth $1$ color of $x$ is $s$, then the right (resp. left) depth $n$ color of $x$ is also $s$; otherwise, the path leaving $x$ on the right (resp. left) returns to the middle region at a horizontal segment $y$, and the right (resp. left) depth $n$ color of $x$ is the concatenation of the right (resp. left) depth 1 color of $x$ and the left (resp. right) depth $n-1$ color of $y$. See Figure \ref{fig:three-regions} for the depth 2 colorings in an example train track. 

The depth $n$ coloring on the left or right side of a horizontal segment is a sequence of letters in $\{\ell_i\}_{i=1}^\infty \cup \{s\} \cup \{r_i\}_{i=1}^\infty$. These letters are ordered so that $\ell_i < s < r_j$ for any $i,j$, $\ell_i < \ell_j$ if $i < j$, and $r_i < r_j$ if $i > j$. In other words, if paths representing each letter are drawn from a segment without crossing, they are ordered from sharpest left turn to sharpest right turn. The depth $n$ colors among horizontal segments of the same Alexander grading are ordered lexicographically. Colors on endpoints of horizontal segments give rise to labels on crossover arrows, which we call the \emph{left and right complexity} of the arrow. Suppose a crossover arrow connects a horizontal segment $x$ to a horizontal segment $y$, where $x$ and $y$ have the same Alexander grading. The left complexity $w_\ell$ of the crossover arrow is defined as follows. 
\begin{itemize}
\item If $x$ and $y$ have the same left depth $n-1$ color and different depth $n$ color, then $w_\ell$ is $\pm n$, where the sign is positive if the left depth $n$ color of $x$ is less than the left depth $n$ color of $y$ and negative otherwise.
\item If $x$ and $y$ have the same left depth $n$ color for all $n$, and this color never ends in an $s$, then $w_\ell$ is $\infty$.
\item If $x$ and $y$ have the same left color, which ends with an $s$ for sufficiently large depth, and if $n$ is the first depth for which the left depth $n$ color ends in $s$, then $w_\ell$ is $n$. 
\end{itemize}
In simpler terms, if we push the crossover arrow to the left, the left complexity $w_\ell$ counts how many times the arrow leaves the middle region before the strands it connects diverge or go to $\partial \strip$, and the sign is positive if the arrow is removable when pushed to this point. The right complexity $w_r$ of a crossover arrow is defined analogously. The complexity of the arrow is then defined to be the minimum of $|w_r|$ and $|w_\ell |$. Note that if one complexity is $\infty$ then both are; this can only occur when the horizontal segments $x$ and $y$ lie on the same non-primitive immersed curve or they lie on two closed immersed curves that are multiples of the same primitive curve.

\begin{remark}
The coloring used here by the letters $\ell_i$, $r_i$ and $s$ are analogous to the colorings by $\{n, e, s, w\}$ used in \cite{HRW}. The left and right complexity are analogous to the weights $\hat w$ and $\check w$ used in \cite{HRW}; we use the term ``complexity" here to avoid confusion with the weights on crossover arrows and train track edges, which were not present in \cite{HRW} since only $\Ztwo$ coefficients were considered.
\end{remark}

\noindent \emph{Removing lowest complexity arrows.} Let $m \ge 1$ denote the minimum complexity of all crossover arrows in $\tracks$. We will now show that $\tracks$ can be modified so that $C(\tracks)$ is unchanged modulo $UV$, up to basis changes, and all crossover arrows in the resulting train track have complexity at least $m + 1$. This simplification will be done one Alexander grading at a time. For each Alexander grading $k$, the simplification is done in three steps:

\noindent \emph{Step 1: Sort crossover arrows.} The horizontal segments in the middle region of $\tracks$ with Alexander grading $k$ and the crossover arrows between them form an $n$-strand arrow configuration. We apply Lemma \ref{lem:sort-arrows} to this configuration with respect to any ordering of the left (resp. right) endpoints which is consistent with the partial ordering determined by their left (resp. right) depth $m+1$ colors. This replaces the $n$-strand arrow configuration with a new configuration which has crossings in the middle and crossover arrows on either side which are non-decreasing with respect to the depth $m+1$ coloring on that side. It follows that the arrows on the left have $w_\ell = m$, $w_\ell = m + 1$, or $|w_\ell | \ge m + 2$, and that the arrows on the right have $w_r = m$, $w_r = m + 1$, or $|w_r | \ge m + 2$.

Note that applying Lemma \ref{lem:sort-arrows} to modify the arrow configuration at grading $k$ may introduce crossings and change the immersed curves in $\tracks$, which potentially alters the left and right coloring of horizontal segments. However, this can only affect colors of depth greater than $m$. This is because crossings are only added between horizontal segments which have the same depth $m-1$ coloring on each side; to see this, we could have grouped the horizontal strands of Alexander grading $k$ in bundles having the same left depth $m-1$ coloring and the same right depth $m-1$ coloring, and applied Lemma \ref{lem:sort-arrows} to each of these bundles instead of the whole configuration. Note that all crossover arrows are contained in one of these bundles, since there are no arrows with complexity less than $m$. Since a path in $\tracks$ that is modified by a crossing change is unchanged for the first $m-1$ turns after the crossing, and since the path from any endpoint of any horizontal segment must leave the middle region at least once before reaching a modified crossing, the depth $m$ coloring of the given endpoint is unchanged. Moreover, for any endpoint of a segment with Alexander grading $k$, a path from this endpoint must leave the middle region at least twice before returning to a segment at Alexander grading $k$. It follows that the depth $m+1$ colors are unaffected for segments with Alexander grading $k$.

\noindent \emph{Step 2: Remove arrows with outer complexity $m$ or $m+1$.} Once the arrow configuration at Alexander grading $k$ is sorted, it is straightforward to remove all crossover arrows on the left with $w_\ell \in \{m, m+1\}$ and all crossover arrows on the right with $w_r \in \{m, m+1\}$. The following Lemma will be useful.

\begin{lemma}\label{lem:remove-m-1-arrow}
Suppose $\tracks$ consists of immersed curves with crossover arrows. If a given crossover arrow has either $w_\ell$ or $w_r$ equal to $m-1$, while all other crossover arrows have complexity at least $m$, then removing the arrow does not change $C(\tracks)$ modulo $UV$, up to change of basis.
\end{lemma}
\begin{proof}
Suppose without loss of generality that $w_r = m-1$. Since $w_r$ is positive and finite, the arrow will be removable if it is pushed to the right as far as possible. The only potential concern is new arrows that are formed along the way. But note that if two arrows have right complexity $w_r$ and $w'_r$ and these arrows passing each other forms a new composition arrow, then the right complexity of this new arrow is either $w_r$ or $w'_r$ (whichever has the smaller absolute value) unless $w_r = -w'_r$, in which case the new right complexity can be anything with absolute value at least $|w_r|$. In particular, if an arrow with $w_r = m-1$ slides past an arrow with $|w_r| \ge m$ and a new arrow is introduced, this new arrow also has $w_r = m-1$, and both $w_r = m-1$ arrows are now one step closer to being removed. By induction on the number of arrows a $w_r = m-1$ must slide past before it is removable, all $w_r = m-1$ arrows can be removed by sliding them to the point where the strands they connect diverge. The argument is the same for left complexities $w_\ell$.
\end{proof}

Consider the crossover arrows on to the left of the crossings in the new $n$-strand arrow configuration at Alexander grading $k$. We can take the leftmost of these arrows for which $w_\ell = m$ and slide it leftward to the end of the configuration. Note that this may introduce new arrows, but these arrows will also have $w_\ell = m$ and we will push these leftward as well. In this way, we can arrange that all crossover arrows with $w_\ell = m$ are at the far left of the middle region. We then take the leftmost arrow and continue pushing leftward until it leaves the middle region of $\tracks$ and returns at a new Alexander grading. This changes the complexity labels on the crossover arrow: $w_\ell - 1$ becomes the new right complexity, while $w_r + 1$ becomes the new left complexity. In particular, the arrow now has $w_r = m-1$ and thus can be removed by by Lemme \ref{lem:remove-m-1-arrow}. This can be repeated until all the $w_\ell = m$ arrows on the left side of the configuration at Alexander grading $k$ are removed. We now do the same thing for arrows on the left side of the configuration at grading $k$ which have $w_\ell = m+1$. We push them all, and any new arrows formed in the process, to the far left and then one by one push them out of the middle region and back to a configuration at a new Alexander grading. This last slide results in an arrow with $w_r = m$. If this happens at an Alexander grading for which complexity $m$ arrows have not yet been removed, we may stop. If we have pushed the arrow to a grading where complexity $m$ arrows have already been removed, then we can continue pushing the arrow all the way across the middle region, which can only introduce new $w_r = m$ arrows which can be pushed along too, until the arrow leaves the middle region again. When the arrow is pushed back into the middle region once more it will have $w_\ell = m-1$ and can be removed.

The same procedure can be done to arrows with $w_r \in \{m, m+1\}$ on the right side of the arrow configuration. In this way, we arrive at new configuration at Alexander grading $k$ for which all arrows on the left side have $| w_\ell | \ge m+2$ and all arrows on the right side have $|w_r| \ge m+2$, and we have not introduced any new arrows of complexity less than $m$, or of complexity $m$ at gradings for which complexity $m$ arrows have already been removed.

\noindent \emph{Step 3: Remove remaining complexity $m$ arrows.} Consider the crossover arrows on to the left side of the $n$-strand arrow configuration at Alexander grading $k$. These now have $|w_\ell | \ge m+2$ and $|w_r| \ge m$. We now wish to deal with arrows with $|w_r| = m$. Take the leftmost of these arrows and push it to the leftmost end of the configuration, and then further until it returns to the middle region at a different Alexander grading. Since $w_\ell - 1$ becomes the new right complexity and $w_r + 1$ becomes the new left complexity, the arrow now has overall complexity of $m+1$. Again, sliding the arrow to the left of the configuration may produce new arrows, but these new arrows can be pushed leftward in the same way to produce a complexity $m+1$ arrow at a different Alexander grading. The same can be done to arrows on the right side with $|w_\ell | = m$, pushing the arrow rightward until it becomes a complexity $m+1$ arrow at a different Alexander grading.

At the conclusion of these three steps, the minimum complexity of crossover arrows between segments at Alexander grading $k$ is $m+1$. Repeating this for all Alexander gradings, we arrive at a new train track with no crossover arrows of complexity less than $m+1$. We can now prove Proposition \ref{prop:simple-curves-existence} using induction on $m$.

\begin{proof}[Proof of Proposition \ref{prop:simple-curves-existence}]
The argument above shows that any train track with minimum arrow complexity $m$ can be replaced by a train track with minimum arrow complexity $m+1$. By induction, it is clear that we can make the minimum complexity of crossover arrows arbitrarily high. Finally, we observe that since there are finitely many horizontal segments, any path must eventually stop or repeat. It follows that there exists an integer $N$ such that if two horizontal segments have the same left depth $N$ color or the same right depth $N$ color then they have the same color for arbitrary depth, and thus any arrow of complexity at least $N$ actually has $w_\ell = w_r = \infty$.

We have constructed a weighted immersed multicurve $\Gamma$ in strip along with a collection of infinite weight crossover arrows which represents a given bigraded complex $C$.  It only remains to interpret these infinite weight crossover arrows as a collection of left turn crossover arrows at local system intersection points.  For a given homotopy class of primitive curve $\gamma$ in $\strip$, consider all the components of $\Gamma$ which are homologous to some multiple of $\gamma$, along with any crossover arrows with endpoints on these curves. By sliding crossings and crossover arrows, we can realize this as some number $n$ of parallel copies of $\gamma$ with an $n$-strand arrow configuration inserted in one place. We may assume the matrix associated to this $n$-strand configuration is in rational canonical form, since we conjugate the matrix by sliding crossings or crossover arrows around the curve. Suppose the matrix decomposes into $k$ blocks with the $i$th block of dimension $n_i$. We replace the collection of curves in question with $k$ immersed curves, where the $i$th curve is homologous to $k_i$ times $\gamma$, and we assume this curve is in simple position. The coefficients of each block in the rational canonical form then specify coefficients for left turn crossover arrows at the local system intersection points on these curves, as discussed in Section \ref{sec:local-systems}, giving rise to an equivalent train track of the desired form.
\end{proof}

We end this section by demonstrating the above construction in an example. The left side of Figure \ref{fig:three-regions} shows the curve-with-arrows train track from Figure \ref{fig:train-track-example2} that represents the complex, with depth 2 colors labeled; both crossover arrows have complexity 1. Applying one step of the inductive process requires sliding both arrows toward the middle from the bundles of horizontal segments at height $1$ and $-1$ to the bundle of segments at height 0, as shown in the middle of the figure. These two arrows have complexity 2. The next step is to sort the arrows in the middle bundle of horizontal strands with respect to depth 3 colors, which requires replacing the two crossover arrows with one crossover arrow and a crossing, along with a basepoint weighted by $-1$, as shown on the right side of the figure. The resulting arrow is removable when pushed in either direction, so the result is an immersed multicurve with two components (one closed component and one arc component) and trivial bounding chain.

\begin{figure}
\labellist
  \pinlabel {$r_1 \ell_1$} at 49 28
  \pinlabel {$s$} at 86 28
  \pinlabel {$r_1 \ell_1$} at 49 43
  \pinlabel {$\ell_1 r_1$} at 86 43
  \pinlabel {$r_1 r_1$} at 49 67
  \pinlabel {$r_1 r_1$} at 86 67
  \pinlabel {$\ell_1 \ell_1$} at 49 80
  \pinlabel {$\ell_1 s$} at 86 80
  \pinlabel {$ \ell_1 s$} at 49 93
  \pinlabel {$\ell_1 \ell_1$} at 86 93
  \pinlabel {$ \ell_1 r_1$} at 49 116
  \pinlabel {$r_1 \ell_1$} at 86 116
  \pinlabel {$ s$} at 49 131
  \pinlabel {$r_1 \ell_1$} at 86 131
    
 \pinlabel {\tiny $-1$} at 76 30
  \pinlabel {\tiny $1$} at 64 120

         \endlabellist
\includegraphics[scale=1]{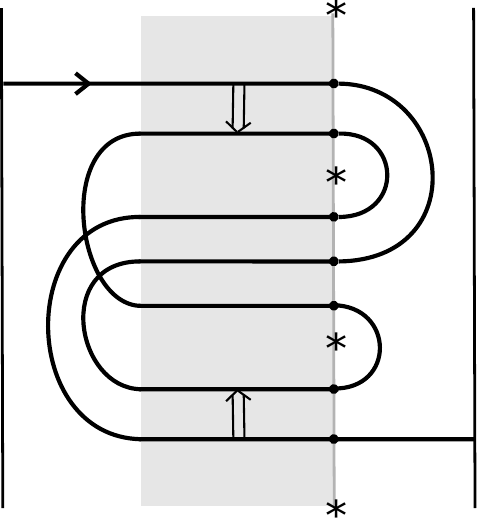} 
\labellist
 \pinlabel {\tiny $-1$} at 52 80
 \pinlabel {\tiny $1$} at 82 80
\endlabellist
\includegraphics[scale=1]{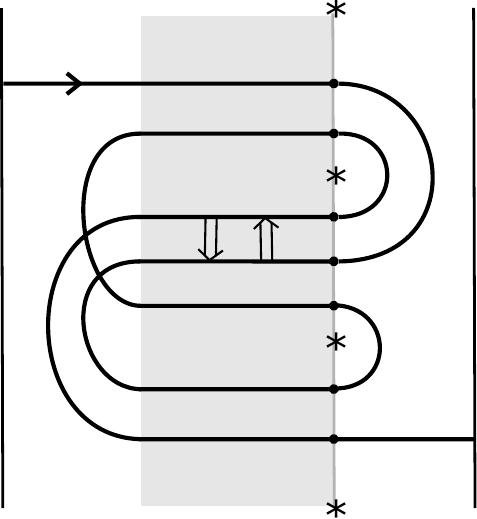} 
\labellist
 \pinlabel {\tiny $-1$} at 44 80
\endlabellist
\includegraphics[scale=1]{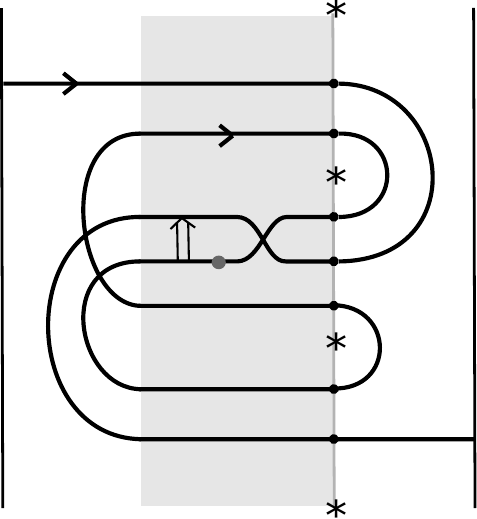} 
 
\caption{Left: The train track from Figure \ref{fig:train-track-example2}, an example of the form generated by Lemma \ref{lem:curves-with-arrows}. The middle region is shaded, and the depth 2 colors of the left and right endpoints of each horizontal segment are shown. The top crossover arrow has $w_\ell = -1$ and $w_r = 3$, while the bottom arrow has $w_\ell = 3$ and $w_r = -1$. Middle: The result of applying one step of the inductive simplification, to increase the minimum arrow complexity from 1 to 2; the left arrow now has $w_\ell = -2, w_r = 2$ and the right arrow has $w_\ell = 2, w_r = -2$. Right: The result of replacing the 3-strand arrow configuration at the middle Alexander grading, as in Lemma \ref{lem:sort-arrows}. The gray dot represents a basepoint with weight $-1$. The remaining crossover arrow has $w_\ell = 2$ and can be removed by pushing it leftward.}
\label{fig:three-regions}
\end{figure}
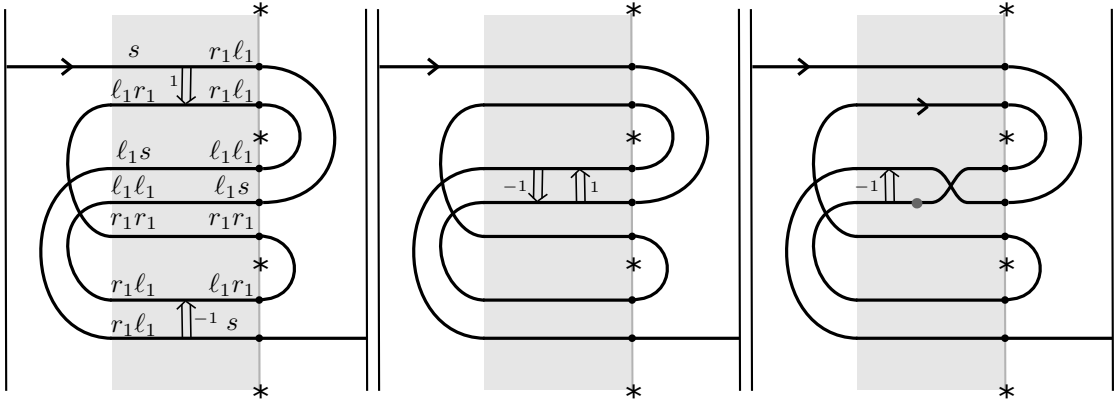

\section{Immersed curves for $UV=0$ complexes with flip maps}\label{sec:flip-maps}% !TEX root = ../CFKcurvesZHS3s.tex
%flip-maps.tex

\subsection{Nice curves in $\cylinder$}

Having established that any bigraded complex over $\sRhat$ can be represented by a (suitably nice) decorated immersed multicurve in the infinite strip $\strip$, we now wish to extend this construction to include flip maps. Our aim is to show the following:

\begin{proposition}\label{prop:curves-from-flip-maps-UVzero}
A bigraded complex $C$ over $\sRhat$ equipped with a flip isomorphism $\widehat{\Psi}_*: H_* \widehat C^h \to H_* \widehat C^v$ can be represented by a decorated curve $(\Gamma, \bchainhat)$ in the marked cylinder $\cylinder$, where $\Gamma$ is in simple position and $\bchainhat$ is a bounding chain consisting of only local system self-intersection points. 
\end{proposition}

The first step in proving Proposition \ref{prop:curves-from-flip-maps-UVzero} is to construct nearly simplified curves with crossover arrows in the marked cylinder representing the given data. This essentially follows from reversing the process for extracting a complex and flip maps from a decorated curve in a cylinder described in Section \ref{sec:flip-maps-from-curves} and uses the curves in $\strip$ representing the complex constructed in Section \ref{sec:simple-curves}. The main difficulty, as with constructing decorated curves in $\strip$, is to then remove crossover arrows to ensure the resulting bounding chain is of local system type.

\begin{proof}[Proof of Proposition \ref{prop:curves-from-flip-maps-UVzero}] 
Our first task is to construct a reasonably nice curve representative in $\cylinder$ for $C$ and $\widehat\Psi_*$. We do this by starting with the decorated curve in $\strip$ representing $C$ predicted by Proposition \ref{prop:simple-curves-existence} and gluing the opposite sides of the strip after first inserting a piece of curve matching up the endpoints according to $\widehat\Psi_*$. More precisely, we will cut the marked cylinder $\cylinder$ into two pieces, a marked strip $[-\tfrac 1 4, \tfrac 1 4]\times \R$ and an unmarked strip $[\tfrac 1 4, \tfrac 3 4]\times \R$. We view the marked strip as a copy of $\strip$ (though it is scaled to be half as wide) and we call the unmarked strip $\sF$. The cylinder is formed by gluing $\partial_R \strip$ to $\partial_L \sF$ and gluing $\partial_R \sF$ to $\partial_L \strip$. We construct a decorated curve in $\cylinder$ by constructing decorated curves $(\Gamma_\strip, \bchainhat_\strip)$ in $\strip$ and $(\Gamma_\sF, \bchainhat_\sF)$ in $\sF$ gluing these together.

The decorated curve $(\Gamma_\strip, \bchainhat_\strip)$ in $\strip$ is the decorated curve representing the complex $C$ as in Proposition \ref{prop:simple-curves-existence}. This means that $\Gamma_\strip$ is in simple position and $\bchainhat_\strip$ consists of only local system intersection points. The decorated curve of $(\Gamma_\sF, \bchainhat_\sF)$ in $\sF$ will consist of arcs from $\partial_L \sF$ to $\partial_R \sF$ connecting the endpoints of $\Gamma_\strip$, along with a collection of turning points. Recall that the endpoints of $\Gamma_\strip$ on $\partial_R \strip$ correspond to generators of the hat horizontal homology of $C$ while endpoints on $\partial_L \strip$ correspond to generators of the hat vertical homology of $C$. The flip isomorphism $\widehat\Psi_*$ is a graded isomorphism from $H_* \widehat C^h$ to $H_* \widehat C^v$ taking $\gr_z$ to $\gr_w$. In particular, for any integer $k$ the isomorphism $\widehat\Psi_*$ restricts to an isomorphism between the spans of generators with grading $k$. Suppose there are $n_k$ generators with grading $k$; we connect the corresponding collections of $n_k$ endpoints on $\partial_R \strip$ and $n_k$ endpoints on $\partial_L \strip$ by a bundle of $n_k$ strands with an $n_k$-strand arrow configuration realizing the flip isomorphism restricted to grading $k$. 

The decorated curve $(\Gamma_\strip \cup \Gamma_\sF, \bchainhat_\strip + \bchainhat_\sF)$ represents the pair $(C, \widehat\Psi_*)$ over $\sRhat$. To see this, observe that adding $\sF$ and gluing the sides of the strip does not affect the Floer homology with $\mu$ modulo $UV$. This is because, since $C$ is reduced, any bigon that is not contained in $\strip$ must enclose a pair of marked points. Similarly, if we consider Floer homology with $\mu_\Psi$ perturbed to lie in $\sF$ (as in Figure \ref{fig:flip-maps-from-curve-perturbed}) the obvious bigons in $\sF$ encode $\widehat\Psi_*$ by construction, and because $C$ is reduced any bigon that is not contained in $\sF$ must cover a marked point and does not contribute mod $UV$.

The curve constructed so far represents $(C, \widehat\Psi_*)$ but it does not have the desired form. We do have that $\bchainhat_\strip$ contains only local system intersection points, but the same may not be true for $\bchainhat_\sF$. To complete the proof, we need to remove any turning points in $\bchainhat_\sF$ that are not at local system intersection points. We will use the language of train tracks, interpreting each turning point as a left-turn crossover arrow, and we repeat the the arrow sliding algorithm from Section \ref{sec:removing-crossover-arrows}, with some modifications.

First, we will ignore components of $\Gamma_\strip \cup \Gamma_\sF$ that are contained in the strips $\strip$, since arrows sliding from $\sF$ will never interact with these. Note that this means we can ignore all crossover arrows coming from the $\bchainhat_\strip$, since these are by assumption at local system intersection points which can only be on closed components of $\Gamma_\strip$. The strip $\sF$ will play the role of the central strip $[-\tfrac 1 4, 0]$ in $\strip$ from Section \ref{sec:removing-crossover-arrows}. Just as the central strip in $\strip$ contained a bundle of horizontal segments at each Alexander grading with a collection of crossover arrows within each bundle, so $\sF$ contains a bundle of arcs for each value of $\gr_w$ with any crossover arrows contained in a bundle; note that now arcs from different bundles may cross each other, but this does not affect the argument.

We again label the arcs in $\sF$ by left and right colors indicating the path the curve takes leaving $\sF$ to the left or right from that arc before returning to $\sF$. While previously the depth 1 colors took values in $\{\ell_k\}_{k=1}^\infty \cup \{s\} \cup \{r_k\}_{k = 1}^\infty$, there are now more possible depth 1 colors with each color representing a homotopy class of path starting and ending on the boundary of a marked strip $\strip$. These colors can be expressed as an integer $n$ followed by a sequence of letters in $\{\ell_k\}_{k = 1}^\infty \cup \{r_k\}_{k=1}^\infty$ and then an $s$; the initial integer represents the Alexander grading of the first crossing of the path in $\Gamma_\strip$ with $\mu$, multiplied by $-1$ for right colors in which the relevant path starts on the left side of $\strip$, and the remaining letters represent the turns the path makes each time it returns to $\mu$ before finally leaving $\strip$. These colors are ordered lexicographically using the usual order on $\Z$ and the order on $\{\ell_k\}_{k=1}^\infty \cup \{s\} \cup \{r_k\}_{k = 1}^\infty$ defined in Section \ref{sec:removing-crossover-arrows}; equivalently, if $c$ and $c'$ are colors representing two paths then $c < c'$ if and only if the paths are not parallel and the path with color $c$ is to the left of the path with color $c'$ when they first diverge. As before, depth $n$ colors are length $n$ words whose letters are the colors described above that record how a path following $\Gamma$ behaves after leaving $\sF$ until the $n$th time it returns to $\sF$. A crossover arrow in $\sF$ can be given a complexity $(w_\ell, w_r)$ in $\Z_{>0}\times \Z_{>0} \cup \{(\infty, \infty)\}$ defined in terms of the colors as before. The absolute values $|w_\ell|$ and $|w_r|$ tell us how many times an arrow needs to be pushed out of the strip $\sF$ before the curves it connects diverge, and $w_\ell$ or $w_r$ is positive if the arrow will be left to right moving (and thus removable) when the curves first diverge.

We now proceed inductively as in Section \ref{sec:removing-crossover-arrows}, assuming that all arrows have some minimal complexity $m$ and removing all arrows with complexity $m$. We do this on one bundle of like-graded strands in $\sF$ at a time. For each bundle we follow a version of the numbered steps from the algorithm in Section \ref{sec:removing-crossover-arrows}, although the steps are slightly more complex. In Section \ref{sec:removing-crossover-arrows} we were able to make a simplifying assumption, but we now require the full generality of the algorithm as it was introduced in \cite{HRW}. The spirit of the algorithm is the same as the argument in Section \ref{sec:removing-crossover-arrows}, but more care is needed to control side effects of sliding arrows; we briefly explain the necessary modifications here.

The reason the algorithm simplifies for curves $\strip$ is that when a path leaves a bundle of segments at a given Alexander grading, it can not return to the same bundle of strands until it leaves and returns to the middle region at least twice. This ensures that when the crossover arrows are sorted in Step 1 the depth $m+1$ colors are not affected. However, it is now possible for a path to leave a bundle of arcs in $\sF$, cross to the other side of $\strip$, and return to the same bundle of arcs the next time it returns to $\sF$ (note that the path still can not return to the same side of the bundle it left from, since arcs with the same grading must be oriented the same way). When applying arrow moves to a collection of arrows with minimum depth $m$, we may now only assume that depth $m$ colors are preserved and not depth $m+1$ colors.

In the following steps, we have fixed a bundle of arcs in $\sF$ with the same grading $k$. Suppose there are $n$ arcs in this bundle, so that the arcs and any crossover arrows between them form an $n$-strand arrow configuration. We have assumed all arrows have complexity at least $m$.

\noindent \emph{Step 1: Sort crossover arrows to remove outer complexity $-m$.}  As before, we apply Lemma \ref{lem:sort-arrows} to the given $n$-strand arrow configuration. The difference is that we do this with respect to any ordering on the endpoints of the arcs consistent with the depth $m$ colorings (rather than the depth $m+1$ colorings as before, since this information will not be preserved). This replaces the $n$-strand arrow configuration by a new one with crossings in the middle, crossover arrows on the left of the bundle which have $w_\ell = m$ or $|w_\ell| \ge m+1$, and crossover arrows on the right of the bundle that have have $w_r = m$ or $|w_r| \ge m+1$. Depth $m$ colors are unchanged by this sorting, so all arrows still have complexity at least $m$.

\noindent \emph{Step 2: Remove arrows with outer complexity $m$.} Any arrows on the left with $w_\ell = m$ can be pushed leftward into the strip $\strip$. If $m=1$ then the curves the arrow connects will diverge at some point before leaving $\strip$ and the arrow will come up against a marked point on its left side, so the arrow can be removed. If $m > 1$, then the arrow will slide between parallel strands until it returns to $\sF$, at which point it will have $w_r = m-1$ or $w_\ell =  m-1$ and can be removed by Lemma \ref{lem:remove-m-1-arrow}. At this point all arrows on the left have $|w_\ell| \ge m+1$ and $|w_r| \ge m$, while all arrows on the right have $|w_r| \ge m+1$ and $|w_\ell| \ge m$.

\noindent \emph{Step 3: Remove $|w_r| = m$ arrows on the left.} After the previous steps it makes sense to split our bundle into two $n$-strand arrow configuration, one containing the arrows on the left side and one containing the arrows on the right side. We now apply Lemma \ref{lem:sort-arrows} again to the left of these configurations with respect to orderings of the endpoints that are consistent with depth $m+1$ coloring on the left and depth $m$ coloring on the right. As a result, the left configuration is replaced with a new configuration such that arrows on the left have $w_\ell \neq -(m+1)$ and arrows on the right have $w_r \neq -m$ (the conditions that $|w_\ell|\ge m+1$ and $|w_r|\ge m$ from Step 2 are also preserved). Any arrows on the right side of the new left configuration with $w_r = m$ can be slid rightward, starting with the rightmost of these arrows. Since the right $n$-strand arrow configuration has $|w_r| \ge m+1$, these arrows with $w_r$ can be slid rightward through that configuration without issue until they leave $\sF$. When they return to $\sF$ they will have either $w_\ell = m-1$ or $w_r = m-1$ and they can be removed by Lemma \ref{lem:remove-m-1-arrow}. We then slide all arrows from the left side of the new left configuration leftward out of $\sF$. When each of these arrows returns to $\sF$ one complexity will have magnitude $|w_r| + 1 \ge m+1$ and the other with maginitude $|w_\ell |-1$. If $|w_\ell | \ge m+2 $ then the resulting arrow has complexity at least $m+1$. Otherwise $w_\ell = m+1$ and the resulting arrow has weight $+m$ on one side when it returns to a bundle of arcs. If it returns to a different bundle from which complexity $m$ arrows have not yet been removed we can now ignore it. If it returns to a bundle from which complexity $m$ arrows have already been removed, it can be slid across this bundle and eventually removed. Finally, if it returns to the opposite side of the bundle currently being simplified, it will have $w_\ell = m$ and $|w_r| \ge m+1$ and can simply be included in right $n$-strand arrow configuration. After this step, all arrows in the left configuration have $|w_\ell| \ge m+1$ and $|w_r \ge m+1|$, while arrows in the right configuration still have $|w_\ell|\ge m$ and $|w_r|\ge m+1$.

\noindent \emph{Step 4: Remove $|w_\ell| = m$ arrows on the right.} This step is analogous to Step 3, but we apply Lemm \ref{lem:sort-arrows} to the right configuration with respect to a depth $m$ ordering on the left and a depth $m+1$ ordering on the right. Arrows on the left of the new right configuration that have $w_\ell = m$ can be slid leftward and removed, leaving only complexity $m+1$ arrows on the left of the right configuration. All arrows from the right of the new right configuration can be slid rightward, and when they return to $\sF$ they will have complexity $m+1$ unless the arrow had $w_r = m+1$. In this case the new arrow will have complexity $+m$ on one side and can be either removed or ignored until a later step depending on which bundle it returns to.

Following the above steps removes all complexity $m$ arrows from a given bundle of arcs in $\sF$ without introducing any new arrows of complexity less than $m$ or of complexity $m$ in a bundle from which these have already been removed. Repeating over all bundles of arcs, we can remove all arrows with complexity less than $m+1$. The rest of the proof is the same as before: by induction we can make the minimum complexity of crossover arrows arbitrarily large. Since there is an upper bound on the complexity of arrows that are not between parallel curves, eventually these are all that remain. Parallel curves with crossover arrows between them can be replaced by non-primitive curves with left-turn crossover arrows at local system intersection points. If necessary we can homotope the curves to put them in simple position, and we have a decorated curve of the desired form.
\end{proof}

We remark that when apply the arrow sliding algorithm, the restriction of the curves to $\strip$ never changes; this is because crossings are only ever resolved within the strip $\sF$. However, when we slide crossover arrows through $\strip$ in order to remove them, this generally involves changing the basis of the complex $C$ corresponding to the curve. Since the immersed curve in $\strip$ is unchanged the curve still represents the complex $C$, but we should understand it as representing the complex with respect to a different basis. This will be relevant later when we add minus information to the curves, since the flip maps should be expressed in terms of this new basis.

\subsection{Examples from knot Floer homology}

We are mainly interested in applying Proposition \ref{prop:curves-from-flip-maps-UVzero} to represent the data from knot Floer homology associated to a nullhomologous knot $K$ in a 3-manifold $Y$. Given such a knot, let $M$ denote the the knot complement $Y \setminus \nu(K)$. Recall that $T_M$ denotes the torus $\partial M$, which is naturally identified with $H_1(\partial M; \R)/H_1(\partial M;\Z)$, with a marked point at $\{0\}$. We consider the covering space $\widetilde{T}_M = H_1(\partial M; \R)$, with a set of marked points identified with $H_1(\partial M; \Z)$, as well as the intermediate covering space $\overline{T}_M = \widetilde{T}_M / \langle \lambda \rangle$, where the homological longitude $\lambda \in H_1(\partial M;\Z)$ generates the kernel of the inclusion $i_*:H_1(\partial M; \Z) \to H_1(M; \Z)$. Because $K$ is nullhomologous, we can take $\lambda$ to be the Seifert longitude; note that $\lambda$ is a primitive element of $H_1(\partial M;\Z)$ and is dual to the meridian $\mu$. In this case we can identify $\overline{T}_M$ with the infinite cylinder $\cylinder$, where $\lambda$ is identified with the horizontal direction and the meridian $\mu$ is identified with the vertical direction. The set of spin$^c$ structures $\Spinc(M)$ can be identified with $H^2(M) \cong H_1(M,\partial M)$. Since $K$ is nullhomologous, the same is true of $\Spinc(Y)$; we will abuse notation and not distinguish between a spin$^c$ structure in $\Spinc(M)$ and the corresponding spin$^c$ structure in $\Spinc(Y)$. For each spin$^c$ structure $\spin \in \Spinc(M)$ we define $\HFhat(Y,K;\spin)$ to be the decorated curve $\Gamma(\widehat C_{\spin}, \widehat\Psi_{\spin,*})$ in $\overline{T}_M \cong \cylinder$ representing the complex $\widehat C_\spin = CFK_{\sRhat}(Y,K;\spin)$ equipped with the flip ismorphism $\widehat\Psi_{\spin,*}$ as constructed in Proposition \ref{prop:curves-from-flip-maps-UVzero}.

The simplest case is that of knots in $S^3$. In this case $M$ has a single spin$^c$ structure $\spin$, and the horizontal and vertical homology of $\widehat C = CFK_{\sRhat}(S^3, K)$ are one dimensional so the flip isomorphism is simply multiplication by a nonzero constant $c$ in $\F$. To construct the curve $\HFhat(S^3,K;\spin)$ we construct the immersed curve in $\strip$ representing the complex $\widehat C$ by following Section \ref{sec:simple-curves} and then glue the sides of $\strip$ together, identifying the endpoints of the curve and inserting a basepoint with weight $c$. In particular, no crossover arrows are introduced when the flip map data is added, so the second application of the arrow sliding algorithm is never needed for knots in $S^3$.

\begin{example}
The immersed multicurves in $\cylinder$ representing the knot Floer complex of the left-handed trefoil and the figure eight knot in $S^3$ are shown in Figure \ref{fig:LHT-and-fig8-curves}.
\end{example}

\begin{figure}
\includegraphics[scale = 1]{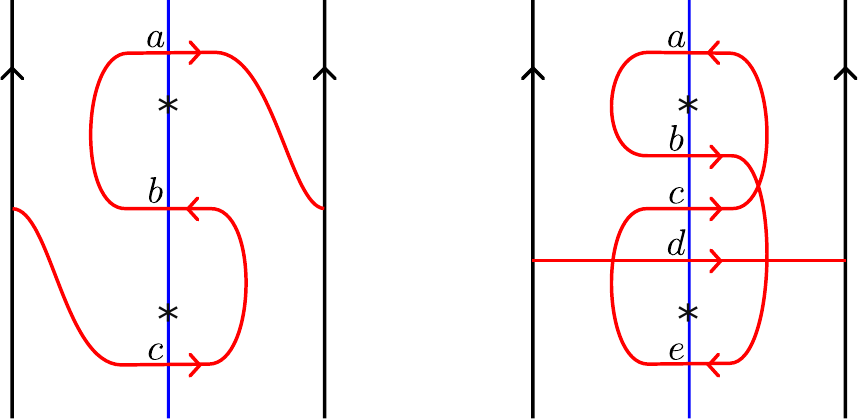}
\caption{The immersed curves associated with the left-handed-trefoil (left) and the figure eight knot (right).}
\label{fig:LHT-and-fig8-curves}
\end{figure}

If $Y$ is not an L-space, the horizontal and vertical homology of the knot Floer complexes are more complicated, and the flip maps can carry interesting information. To illustrate this, we construct decorated curves from the complexes and flip maps given in Examples \ref{ex:1-surgery-on-fig8} and \ref{ex:1-surgery-on-LHT}.

\begin{example}\label{ex:1-surgery-on-fig8-curves}
Consider the the knot Floer data associated with the dual knot in $+1$-surgery on on the figure eight knot given in Example \ref{ex:1-surgery-on-fig8}, viewed as a complex over $\sRhat$. Because the complexes have horizontally and vertically simplified bases, it is straightforward to construct immersed multicurve in the marked strip $\strip$ representing the complex $CFK_{\sRhat}(Y,K)$ (in particular, there are no crossover arrows to remove in this construction); this multicurve is shown on the left of Figure \ref{fig:1-surgery-on-fig8-curves}. Note that a basepoint of weight $-1$ is required to get the correct sign on the $d$ term of $\partial(e)$. The right endpoints of this multicurve correspond to the generators of hat horizontal homology, $\{a,b,c\}$, while the left endpoints correspond to the generators $\{c,d,e\}$ of vertical homology. Recall that up to equivalence the possible flip isomorphisms on this complex are indexed by a nonzero constant $c_4$ in $\F$, with the flip map taking $a$ to $e$, $b$ to $d$, and $c$ to $c_4 \cdot c$; though we know the flip isomorphism associated to $K \subset Y$ corresponds to $c_4 = 1$, we will describe the construction for any of these flip isomorphisms. We add an unmarked strip $\sF$ containing arcs that match up the endpoints in the appropriate way; that is, we add arcs connecting the $a$, $b$, and $c$ endpoints on left of $\sF$ to the $e$, $d$, and $c$ endpoints on the right of $\sF$, respectively, and we place a weighted basepoint with weight $c_4$ on the arc from $c$ to $c$. Gluing the strips $\strip$ and $\sF$ along with the curves they contain produces the decorated immersed multicurve in the marked cylinder $\cylinder$ representing $CFK_{\sRhat}(Y,K)$ with the given choice of flip map; see the right side of Figure \ref{fig:1-surgery-on-fig8-curves}. When $c_4 = 1$, this decorated curve is $\HFhat(Y,K;\spin)$; note that in this case the basepoint in $\sF$ can be omitted.
\end{example}

\begin{figure}
\includegraphics[scale = .85]{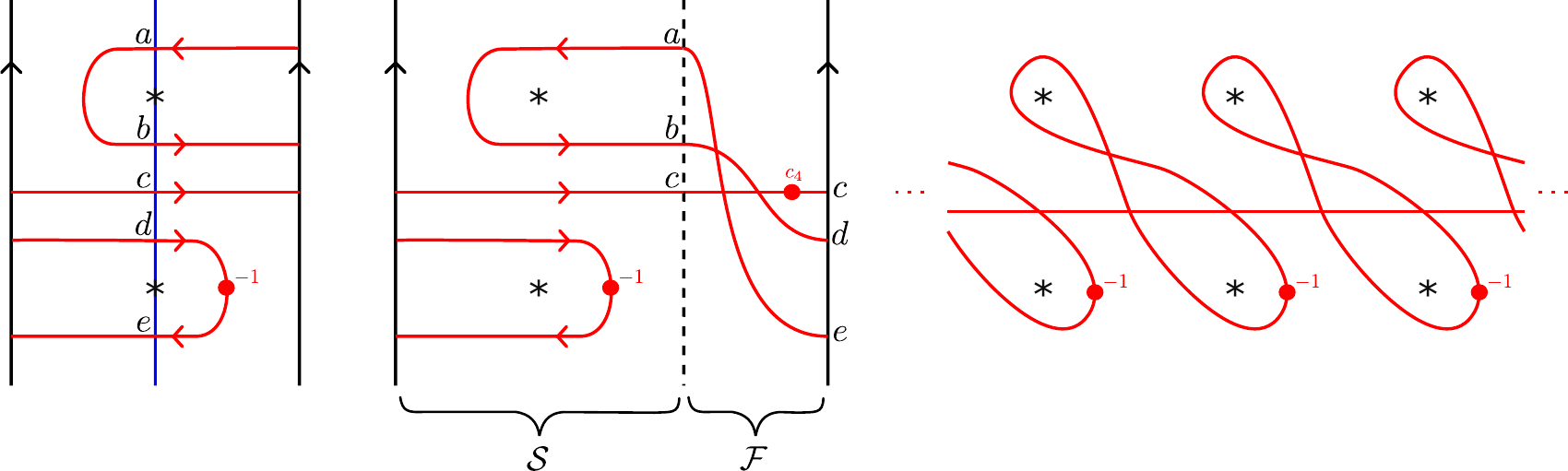}

(a) \hspace{35mm} (b) \hspace{45mm} (c) \hspace{20mm}
\caption{The immersed curves constructed in Example \ref{ex:1-surgery-on-fig8-curves}:(a) the immersed multicurve in $\strip$ representing the complex with respect to the given basis; (b) the immersed curve in $\cylinder$ representing the complex with a flip isomorphism---when $c_4=1$ this gives the curve associated to the dual knot of $+1$-surgery on the figure eight knot; (c) the curve (for $c_4 = 1$) lifted to the covering space $\widetilde{T}_M$.}
\label{fig:1-surgery-on-fig8-curves}
\end{figure}

\begin{example}\label{ex:1-surgery-on-LHT-curves}
Consider the knot Floer data associated with the dual knot in $+1$-surgery on on the left-handed trefoil given in Example \ref{ex:1-surgery-on-LHT}, viewed as a complex over $\sRhat$. As in Example \ref{ex:1-surgery-on-fig8-curves}, it is straightforward to compute the immersed curves in $\strip$ representing the complex over $\sRhat$; in fact, the complex is identical to the complex in Example \ref{ex:1-surgery-on-fig8-curves} so the immersed curves in $\strip$ are the same apart from orientations and weights; see Figure \ref{fig:1-surgery-on-LHT-curves}(a). Recall that for this complex there are two fundamentally different families of non-equivalent flip isomorphisms: in the first family $\widehat\Psi_\spin$ takes $a$ to $e$ while in the second family it $\widehat\Psi_\spin$ takes $a$ to $c + e$, and in both cases $\widehat\Psi_\spin$ takes $b$ to $d$ and $c$ to $c_4 \cdot c$ for some nonzero $c_4$ in $\F$. For the first family of flip isomorphisms, the resulting curves are the same as those in Example \ref{ex:1-surgery-on-fig8-curves} (apart from orientations); see Figure \ref{fig:1-surgery-on-LHT-curves}(b).

For the second family of isomorphisms, we add the same arcs in $\sF$ as before but also add a left-turn crossover arrow (with weight 1) from the segment from $a$ to $e$ to the segment from $c$ to $c$, as shown in Figure \ref{fig:1-surgery-on-LHT-curves}(c). Note that since $a$ and $c$ have the same value of $\gr_z$, the segments starting at $a$ and $c$ in $\sF$ form a bundle, and the $2$-strand arrow configuration consisting of the arcs starting at $a$ and $c$ along with the crossover arrow encodes the flip map restricted to the appropriate grading. We now need to perform the arrow sliding algorithm to remove the crossover arrow. The first step is to replace the 2-strand arrow configuration connecting $a$ and $c$ to $c$ and $e$ so that it is sorted with respect to depth 1 colors; this amounts to resolving the crossing at which the left-turn crossover arrow appears, using the local move in the last line of Figure \ref{fig:local-moves}, resulting in a new curve with two crossover arrows (both with weight 1) as shown on the right of Figure \ref{fig:1-surgery-on-LHT-curves}(d). These arrows can be removed since they both immediately encounter a marked point when slid toward their left side into $\strip$, though we note that removing these arrows involves a change of basis. After removing the crossover arrows, the two basepoints weighted by $-1$ can be slid together and canceled (this may also involve basis changes), leaving a single basepoint with weight $c_4$. When $c_4 = 1$, the resulting curve is $\HFhat(Y,K;\spin)$; in this case the remaining basepoint has weight 1 and can be ignored.
\end{example}

\begin{figure}
\includegraphics[scale = .9]{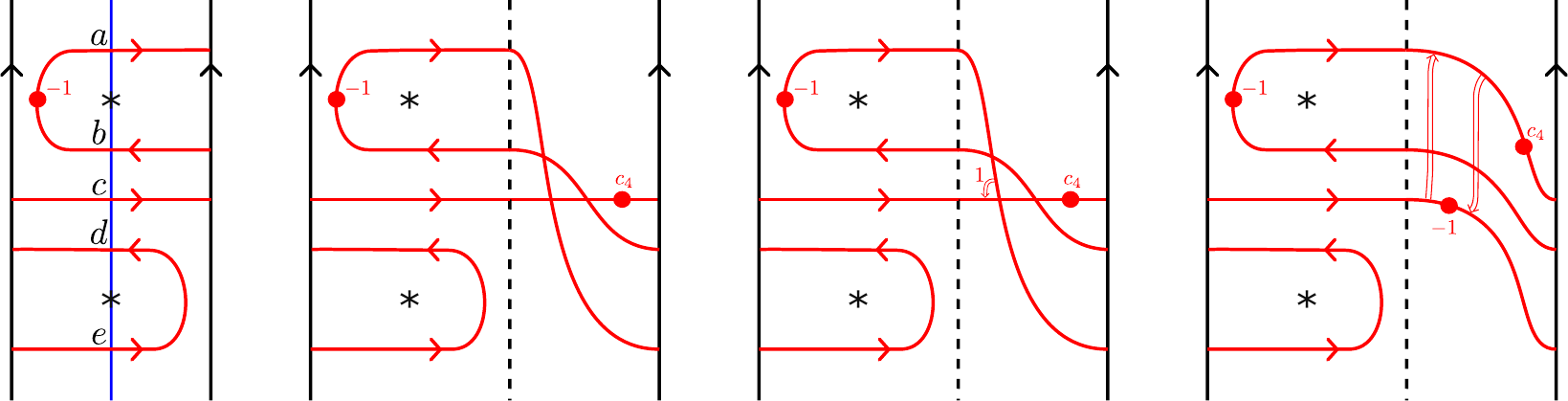}

\hspace{1 mm} (a) \hspace{32mm} (b) \hspace{36mm} (c) \hspace{37mm} (d) \hspace{5 mm}

\caption{The immersed curves constructed in Example \ref{ex:1-surgery-on-LHT-curves}. (a) the immersed multicurve in $\strip$ representing the complex; (b) a curve in $\cylinder$ representing the complex with the first family of flip isomorphisms; (c) a multicurve with crossover arrow in $\cylinder$ representing the complex with the second family of flip isomorphisms. (d) the first step of removing the crossover arrow.}
\label{fig:1-surgery-on-LHT-curves}
\end{figure}

The curve constructed in Example \ref{ex:1-surgery-on-LHT-curves} is the same as the curve considered in Example \ref{ex:flip-maps-from-curve}; it was observed there that the complex extracted from this curve is the complex from Example \ref{ex:1-surgery-on-LHT} and the flip isomorphism extracted from this curve is equivalent to the flip isomorphism from Example \ref{ex:1-surgery-on-LHT} after a change of basis. We remark that the change of basis needed is precisely the one arising from the arrow removal process.

\section{Enhancing the curves for complexes over $\sRminus$}\label{sec:enhanced-curves}% !TEX root = ../CFKcurvesZHS3s.tex
%enhanced-curves.tex

In Sections \ref{sec:simple-curves} and \ref{sec:flip-maps} we worked in the simpler $UV=0$ setting and showed that a bigraded complex over $\sRhat$ can be represented by a decorated curve in a marked strip $\strip$ and that a complex over $\sRhat$ equipped with a flip isomorphism can be represented by a decorated curve in a marked cylinder $\cylinder$. We now aim to extend these results to complexes over $\sRminus$, proving the existence part of Theorem \ref{thm:curve-invariant-for-complex}. We first show that for a bigraded complex over $\sRminus$, the decorated curve in $\strip$ representing the $UV=0$ complex can be enhanced to recover the diagonal arrows as well; in fact, the immersed curve will not change (up to homotopy) in this process, we only need to add to the bounding chain. Similarly, given a curve in $\cylinder$ representing complexes and flip maps over $\sRhat$ we can capture the extra information in the flip isomorphism over $\sRminus$ by adding self intersection points to the bounding chain.

\subsection{Enhanced curves in $\strip$: two examples}\label{sec:two-examples}
Before starting the proof, we discuss two illustrative examples; for simplicity we consider both examples with coefficients in $\Ztwo$. Consider the knot Floer complex for the $(2,-1)$-cable of the left-handed trefoil. The bigraded complex $C_1$ is shown in Figure \ref{fig:cable-of-trefoil}, along with a immersed curve $\Gamma_1$ in the strip $\strip$ which results from applying the algorithm in Section \ref{sec:simple-curves}; note that there are no self intersection points of $\Gamma_1$ so the bounding chain coming from the algorithm is necessarily trivial and we omit it from the notation. The construction in Section \ref{sec:simple-curves} guarantees that the complex $C(\Gamma_1)$ agrees with $C_1$ as complexes over $\sRhat$, with the bigons contained fully on the left (respectively right) side of $\mu$ corresponding to the vertical (respectively horizontal) arrows in $C_1$. In this case, it turns out that $C(\Gamma_1)$ is in fact a bigraded complex over $\sRminus$ and it agrees with the full complex $C_1$; in addition to the bigons contributing to the complex over $\sRhat$, the differential on $C(\Gamma_1)$ over $\sRminus$ counts the two bigons shaded in the figure, which exactly recover the two diagonal arrows in $C_1$.

The good fortune of the last example does not always hold. Figure \ref{fig:example-with-crossing} shows another bigraded complex $C_2$ over $\sRminus$. This complex is one summand of the knot Floer complex for $T_{2,9} \# -T_{2,3;2,5}$.  On the left is the immersed curve $\Gamma_2$ which represents the complex over $\sRhat$. Once again, the bounding chain determined by the construction in Section \ref{sec:simple-curves} is trivial. By construction, the complex $C(\Gamma_2)$ agrees with $C_2$ as a complex over $\sRhat$. However, if we work over $\sRminus$ then $C(\Gamma_2)$ does not agree with $C_2$; two of the four diagonal arrows in $C_2$ are missing (these two arrows are gray in the figure). In fact, this means that $C(\Gamma_2)$ is not even a complex over $\sRminus$, as $\partial^2$ is not zero. The situation can be salvaged if we decorate $\Gamma_2$ with a bounding chain $\bchain_2$. We take $\bchain_2$ to be the linear combination of the two self-intersection points of $\Gamma_2$, each with weight $W$, as shown on the right side of the figure. Working over $\sRminus$, the complex $C(\Gamma_2, \bchain_2)$ has the same generators as the precomplex $C(\Gamma_2)$ and all the same terms in the differential, in addition to two new terms coming from the shaded generalized bigons in the figure; these correspond precisely to the two missing arrows, so $(\Gamma_2, \bchain_2)$ represents $C_2$ over $\sRminus$.

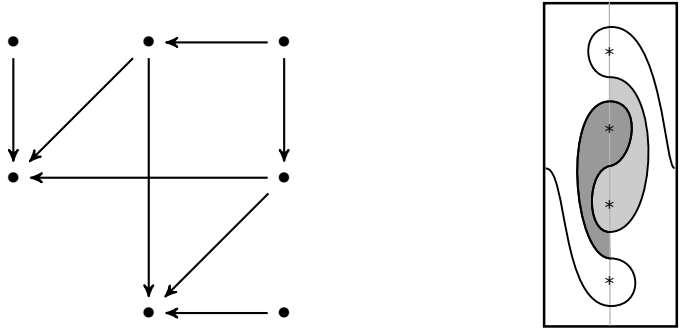
\begin{figure}
 \begin{tikzpicture}[scale=1.8,>=stealth', thick] 
 
 \node (a) at (0,2) {$\bullet$};
 \node (b) at (0,1) {$\bullet$};
 \node (c) at (2,1) {$\bullet$};
 \node (d) at (2,2) {$\bullet$};
 \node (e) at (1,2) {$\bullet$};
 \node (f) at (1,0) {$\bullet$};
 \node (g) at (2,0) {$\bullet$};
 
 \draw[->] (a) to (b);
 \draw[->] (c) to (b);
 \draw[->] (d) to (c);
 \draw[->] (d) to (e);
 \draw[->] (e) to (f);
 \draw[->] (g) to (f);
 \draw[->] (e) to (b);
 \draw[->] (c) to (f);

\end{tikzpicture} \hspace{3 cm}
\includegraphics[scale=.6]{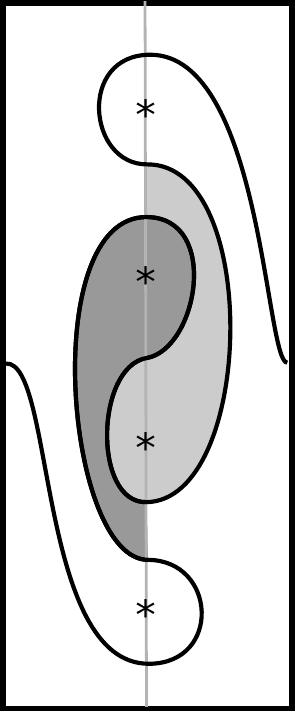} 

\caption{$\CFKinfty$ of the the $(2,-1)$ cable of the left-handed trefoil. On the left is a representation of the bifiltered complex; on the right is the corresponding immersed curve. The two shaded bigons, which each cover one puncture, correspond to the diagonal arrows in the complex.}\label{fig:cable-of-trefoil}
\end{figure}

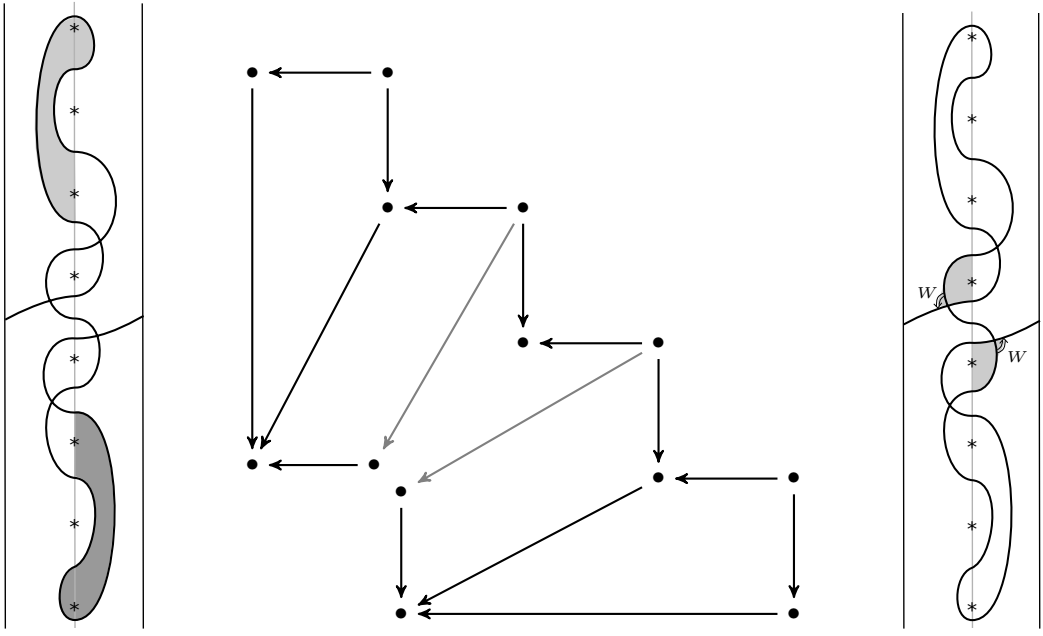
\begin{figure}
\includegraphics[scale=.65]{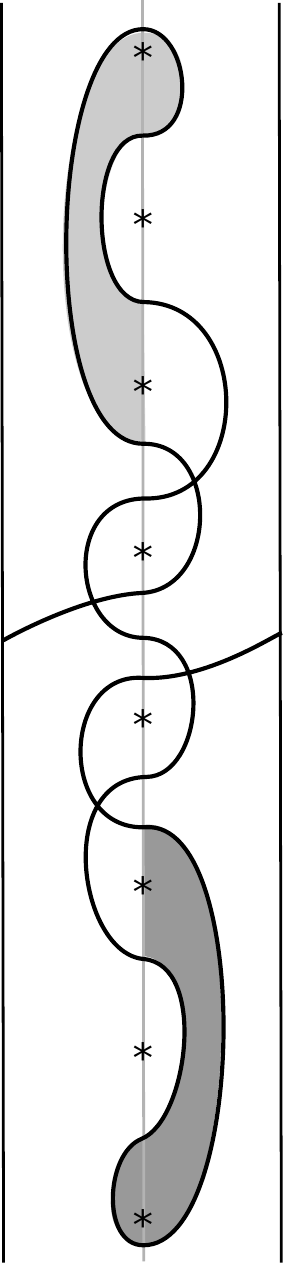} \hspace{1 cm}
 \begin{tikzpicture}[scale=1.8,>=stealth', thick] 
 
 \node (a) at (.9,1.1) {$\bullet$};
 \node (b) at (0,1.1) {$\bullet$};
 \node (c) at (0,4) {$\bullet$};
 \node (d) at (1,4) {$\bullet$};
 \node (e) at (1,3) {$\bullet$};
 \node (f) at (2,3) {$\bullet$};
 \node (g) at (2,2) {$\bullet$};
 \node (h) at (3,2) {$\bullet$};
 \node (i) at (3,1) {$\bullet$};
 \node (j) at (4,1) {$\bullet$};
 \node (k) at (4,0) {$\bullet$};
 \node (l) at (1.1,0) {$\bullet$};
 \node (m) at (1.1,.9) {$\bullet$};

 \draw[->] (a) to (b);
 \draw[->] (c) to (b);
 \draw[->] (d) to (c);
 \draw[->] (d) to (e);
 \draw[->] (f) to (e);
 \draw[->] (f) to (g);
 \draw[->] (h) to (g);
 \draw[->] (h) to (i);
 \draw[->] (j) to (i);
 \draw[->] (j) to (k);
 \draw[->] (k) to (l);
 \draw[->] (m) to (l);

 \draw[->] (e) to (b);
 \draw[->, color = gray] (f) to (a);
 \draw[->] (i) to (l);
 \draw[->, color = gray] (h) to (m);

\end{tikzpicture} \hspace{1 cm}
\labellist
\tiny
  \pinlabel {$\tiny {W}$} at 68 160
  \pinlabel {$\tiny {W}$} at 15 198
\endlabellist
\includegraphics[scale=.65]{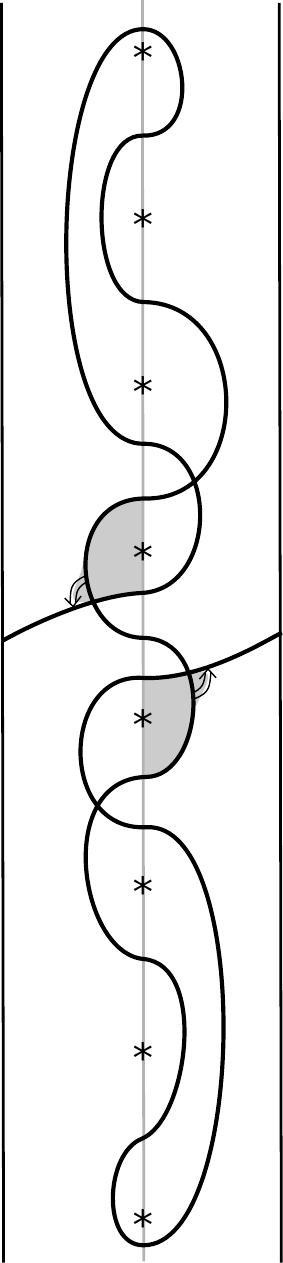}

\caption{For the complex shown, on the left is an immersed curve representing the $UV = 0$ quotient of the complex. Note that even if we consider bigons between which fully cover the puncture, the pre-complex determined by the immersed curve does not agree with the full complex; the two diagonal arrows in the middle of the complex correspond to the shaded bigons on the left, but the remaining two diagonal arrows are missing. On the right is the curve decorated with two crossover arrows, each of which is weighted by $UV$. The complex determined by this train track contains two additional arrows coming from the shaded bigons and determines the correct complex.}\label{fig:example-with-crossing}
\end{figure}

\subsection{Enhanced curves in $\strip$: General case}\label{sec:enhanced-general-case}

In general, for any bigraded complex $C$ over $\sRminus$ let ($\Gamma, \bchainhat)$ be an immersed multicurve in $\strip$ and bounding chain that represents $C$ over $\sRhat$, as constructed in Section \ref{sec:simple-curves}. Our strategy will be to incrementally improve $(\Gamma, \bchainhat)$ by adding more turning points to $\bchainhat$ until we arrive at a bounding chain $\bchain$ such that $(\Gamma, \bchain)$ represents $C$ over $\sRminus$. The immersed multicurve $\Gamma$ will not change during this process, though in order to have enough intersection points we must assume that $\Gamma$ begins in almost simple position rather than simple position. The construction in Section \ref{sec:simple-curves} gives curves in simple position, but these can easily be put in almost simple position by sliding endpoints along $\partial_L \strip$ and $\partial_R \strip$ to put them in the correct order and by applying finger moves to remove any immersed annuli; these changes do not change the fact that $(\Gamma, \bchainhat)$ represents $C$ over $\sRhat$.

A bigraded (pre)complex over $\sRminus$ is determined by a set of $n$ generators with associated Alexander and Maslov gradings and an $n\times n$ matrix with coefficients in $\F$. We will order the $n^2$ entries of this matrix and inductively construct collections of turning points $\bchain_N$ such that the precomplex $C(\Gamma, \bchain_N)$ agrees with the complex $C$ as a vector space and such that the differentials agree in the first $k$ entries of this matrix. When $N = n^2$ the matrices agree entirely, and so if we take $\bchain = \bchain_{N^2}$ then $(\Gamma, \bchain)$ represents $C$ over $\sRminus$. It will be convenient to work with train tracks so that we can slide crossover arrows in the construction, so let $\tracks_N$ denote the immersed train track corresponding to the pair $(\Gamma, \bchain_N)$ (recall that this consists of the multicurve $\Gamma$ along with left turn crossover arrows determined by $\bchain_N$). As usual we let $C(\tracks_N)$ denote the (pre)complex represented by this train track, which is equipped with the map $\partial^{\tracks_N}$. As the base case of the induction, we define $\bchain_0$ to be $\bchainhat$. By assumption, the train track $\tracks_0$ corresponding to $(\Gamma, \bchainhat)$ represents $C$ over $\sRhat$. Note that $\bchainhat$ is a bounding chain over $\sRhat$ but not necessarily over $\sRminus$, so the $C(\tracks_0)$ may be only a precomplex over $\sRminus$.

The intersections of $\Gamma$ with $\mu$ specify an ordered basis $\{x_1, \ldots, x_n\}$ for $C(\tracks_0)$, where the ordering is given by height. We will use $<$ to denote this ordering; that is $x_i < x_j$ if the point $x_i$ occurs below $x_j$. Note that this is a refinement of the partial ordering on generators given by the Alexander grading.  Over $\sRhat$ the complex $C(\tracks_0)$ is isomorphic to $C$, by assumption; by slight abuse of notation we will use $\{x_1, \ldots, x_n\}$ to refer to the corresponding basis of the complex $C|_{\sRhat}$, which can also be taken as a basis of $C$. Let $\{d_{i,j}\}_{1\le i,j \le n}$ be the matrix of coefficients (in $\F$) of the differential $\partial$ on $C$ with respect to this basis, so that
$$\partial(x_i) = \sum_{j = 1}^n d_{i,j} U^{a_{i,j}} V^{b_{i,j}} x_j$$
where $a_{i,j}$ and $b_{i,j}$ are defined by Equation \eqref{eq:UV-exponents}. Note that $d_{i,j} = 0$ if $M(x_i)$ and $M(x_j)$ have the same parity or if $a_{i,j}$ or $b_{i,j}$ are negative.
Similarly, let $\{d^{\tracks_k}_{i,j}\}_{1\le i,j \le n}$ be the coefficients of the map $\partial^{\tracks_k}$ for $C(\tracks_k)$.

Let $\mathcal{P}$ denote the set of pairs $(i,j)$ with $1\le i, j \le n$; we now define an ordering $\lessdot$ on $\mathcal{P}$ given the gradings on the ordered basis $\{x_1, \ldots, x_n\}$. We first order by the parity of $M(x_i) - M(x_j)$, with the pairs for which this parity is even coming first. The pairs for which $M(x_i) - M(x_j)$ is even can be given any arbitrary order. For pairs with $M(x_i)-M(x_j)$ odd, we next order by $\min(a_{i,j}, b_{i,j})$ and then by $\max(a_{i,j}, b_{i,j})$. That is, we pick the ordering $\lessdot$ such that $(i,j) \lessdot (i',j')$ if $\min(a_{i,j}, b_{i,j}) < \min(a_{i',j'}, b_{i',j'})$ or if $\min(a_{i,j}, b_{i,j}) = \min(a_{i',j'}, b_{i',j'})$ and $\max(a_{i,j}, b_{i,j}) < \max(a_{i',j'}, b_{i',j'})$. To refine the ordering further, we define a complexity on pairs $(i,j)$ based on how long the paths following $\Gamma$ starting from $x_i$ and $x_j$ take to cross or diverge, where the paths begin moving rightward if $i < j$ and they begin moving leftward if $i > j$. More precisely, the complexity is zero if either path does not return to $\mu$, if the two paths cross before returning to $\mu$, or if the two paths return to $\mu$ at different heights; otherwise, if the paths starting at $x_i$ and $x_j$ first return to $\mu$ at points $x_{i'}$ and $x_{j'}$, respectively, than the complexity of $(i,j)$ is one more than the complexity of $(i', j')$. Among pairs with the same $\min(a_{i,j}, b_{i,j})$ and $\max(a_{i',j'}, b_{i',j'})$, we choose $\lessdot$ so that pairs are ordered by increasing complexity. Finally, pairs with the same $\min(a_{i,j}, b_{i,j})$, $\max(a_{i',j'}, b_{i',j'})$, and complexity can be ordered arbitrarily, subject to the following constraint: If $x_i$ is part of a grouping of generators with indices $\{i_0, \ldots, i_{r-1}\}$ on a non-primitive curve component of order $r$ and $x_j$ is part of a grouping of generators with indices $\{j_0, \ldots, j_{s-1}\}$ on a non-primitive curve of order $s$, then the relative ordering on the $rs$ pairs coming from an index in the first grouping and an index in the second grouping satisfies the following:
$$(i_0, j_{\ell \neq 0}) \lessdot (i_0, j_0) \lessdot (i_{k \neq 0}, j_{\ell \neq 0}) \lessdot (i_{k \neq 0}, j_0) .$$

If a train track $\tracks$ represents the complex $C$ over $\sRhat$ with respect to some basis of $C$, we will say that it \emph{represents $C$ over $\sRminus$ to $N$ entries} if $d_{i,j} = d^\tracks_{i,j}$ for the first $N$ pairs $(i,j)$ in $\mathcal{P}$. We will inductively construct the collections of turning points $\bchain_N$ so that the train tracks $\tracks_N$ represent $C$ over $\sRminus$ to at least $N$ entries. Note that for any train track $\tracks$ representing $C$ over $\sRhat$, $d_{i,j}$ and $d^\tracks_{i,j}$ are both zero if $M(x_i) - M(x_j)$ is even or if $\min(a_{i,j}, b_{i,j}) < 0$. The entries with $\min(a_{i,j}, b_{i,j}) = 0$ record the horizontal and vertical arrows in the complex, so the fact that $\tracks$ represents $C$ over $\sRhat$ implies that $d_{i,j}$ and $d^\tracks_{i,j}$ also agree for these entries. Thus, if $N_0$ is the number of pairs with $M(x_i) - M(x_j)$ even or with  $\min(a_{i,j}, b_{i,j}) \le 0$ and we can define $\bchain_N$ to be $\bchainhat$ for all $N \le N_0$, and the induction really begins at $N = N_0$. To show that there is some $\bchain = \bchain_{n^2}$ so that $\tracks_{n^2}$ represents $C$ over $\sRminus$, we need the following inductive step.

\begin{proposition}[Main inductive step]\label{prop:inductive-step} Let $C$ be a bigraded complex over $\sRminus$. Suppose $\Gamma$ is a weighted and graded immersed multicurve in $\strip$ in almost simple position and $\bchain_N$ is a collection of turning points such that $(\Gamma, \bchain_N)$ represents $C$ over $\sRhat$ and represents $C$ over $\sRminus$ to $N$ entries (with respect to some given basis of $C$) and suppose that the restriction $\bchainhat_N$ of $\bchain_N$ to degree zero intersection points is of local system type. There exists a new collection of turning points $\bchain_{N+1}$, also with the property that $\bchainhat_{N+1}$ is of local system type, such that $(\Gamma, \bchain_{N+1})$ represents $C$ over $\sRhat$ and over $\sRminus$ to $N+1$ entries (possibly with respect to a different basis of $C$). 
\end{proposition}

\begin{proof}[Proof of Proposition \ref{prop:inductive-step}]
Let $(i,j)$ be the $(N+1)$st pair of indices in $\mathcal{P}$ with respect to the ordering $\lessdot$ defined above. Our aim is to show that $d_{i,j}^{\tracks_N} = d_{i,j}$ already or that this can be made true by modifying $\tracks_N$ through the addition or removal of points in $\bchain_N$, where this modification does not affect the fact that $d_{i',j'}^{\tracks_N} = d_{i',j'}$ for all $(i',j') \lessdot (i,j)$. For simplicity we will assume that $i < j$, i.e. that the potential arrow from $x_i$ to $x_j$ is horizontal type. The opposite case, with $i > j$, is identical after rotating all pictures in the proof by 180 degrees.

Let $a = a_{i,j}$ and $b = b_{i,j}$ so that a hypothetical arrow from $x_i$ to $x_j$ has coefficient $c U^a V^b$ for $c \in \F$. Since $i < j$, we have that $a - b = A(x_j) - A(x_i) \ge 0$. Note that since the coefficients $d_{i',j'}$ and $d_{i',j'}^{\tracks_N}$ agree for pairs $(i', j')$ with $(i',j') \lessdot (i,j)$, the inductive hypothesis implies that $(\Gamma, \bchain_N)$ represents $C$ over $\sR_{b} = \sRminus / W^{b}$. On the other hand, since we do not care about pairs $(i', j')$ with $(i,j) \lessdot (i',j')$, to verify the conclusion of the proposition it is sufficient to do so over $\sR_{b+1} = \sRminus / W^{b+1}$.

Recall that the generators $x_i$ and $x_j$ correspond to intersection points of $\Gamma$ with the vertical line $\mu$. We can define an arc $s_i$ by following $\Gamma$, initially moving rightward from the point corresponding to $x_i$, until it either returns to $\mu$ or hits the right boundary of $\strip$; if the terminal endpoint of $s_i$ lies on $\mu$ let $x_{i'}$ be the corresponding generator of $C(\tracks_N)$ and otherwise we say $x_{i'} = \emptyset$. Define $s_j$ and $x_{j'}$ similarly. Note that $s_i$ and $s_j$ can not be the same segment---that is, $x_{i'}$ can not be $x_j$---since then there would be a horizontal arrow from $x_i$ to $x_j$. We consider cases based on the relative values of $x_i$, $x_j$, $x_{i'}$, and $x_{j'}$. There are nineteen cases, which are depicted in Figure \ref{fig:19cases}. In each case, we will do one of three things: (1) we show that there is an intersection point at which we can add a left turn crossover arrow to $\tracks_N$ which changes $d^{\tracks_N}_{i,j}$ without changing any earlier coefficients, (2) we show that $d^{\tracks_N}_{i,j}$ can be adjusted, without changing any earlier coefficients, by adding a crossover arrow to $\tracks_N$ which is not a left-turn crossover arrow, and we show that that using arrow slide moves we can remove this arrow or it can become a left-turn crossover arrow, or (3) we show that $d^{\tracks_N}_{i,j}$ must already agree with $d_{i,j}$. In cases (1) and (2) we define $\bchain_{N+1}$ from $\bchain_N$ by adding the intersection point corresponding to the left turn crossover arrow, and in case (3) we simply let $\bchain_{N+1} = \bchain_N$.

\begin{figure}
\labellist
  \pinlabel {$x_{j'}$} at -5 207
  \pinlabel {$x_{i'}$} at -5 194
  \pinlabel {$x_j$} at -5 181
  \pinlabel {$x_i$} at -5 168

  \pinlabel {$x_j$} at 43 207
  \pinlabel {$x_{i'}$} at 43 194
  \pinlabel {$x_{j'}$} at 43 181
  \pinlabel {$x_i$} at 43 168
  
  \pinlabel {$x_{i'}$} at 91 207
  \pinlabel {$x_j$} at 91 194
  \pinlabel {$x_i$} at 91 181
  \pinlabel {$x_{j'}$} at 91 168
  
  \pinlabel {$x_j$} at 139 207
  \pinlabel {$x_i$} at 139 194
  \pinlabel {$x_{j'}$} at 139 181
  \pinlabel {$x_{i'}$} at 139 168
  
  \pinlabel {$x_{i'}$} at 187 201
  \pinlabel {$x_j$} at 187 188
  \pinlabel {$x_i$} at 187 175
  \pinlabel {$x_{j'} = \emptyset$} at 212 167
  
  \pinlabel {$x_j$} at 235 201
  \pinlabel {$x_i$} at 235 188
  \pinlabel {$x_{j'}$} at 235 175
  \pinlabel {$x_{i'} = \emptyset$} at 260 167
  
  \pinlabel {$x_j$} at 283 194
  \pinlabel {$x_i$} at 283 181
  \pinlabel {$x_{i'} = \emptyset$} at 312 203
  \pinlabel {$x_{j'} = \emptyset$} at 312 173
 
  \pinlabel {$(a)$} at 10 152
  \pinlabel {$(b)$} at 58 152
  \pinlabel {$(c)$} at 106 152
  \pinlabel {$(d)$} at 154 152
  \pinlabel {$(e)$} at 202 152
  \pinlabel {$(f)$} at 250 152
  \pinlabel {$(g)$} at 298 152
  
  \pinlabel {$x_{j'}$} at -5 127
  \pinlabel {$x_j$} at -5 114
  \pinlabel {$x_{i'}$} at -5 101
  \pinlabel {$x_i$} at -5 88

  \pinlabel {$x_j$} at 52 127
  \pinlabel {$x_{j'}$} at 52 114
  \pinlabel {$x_{i'}$} at 52 101
  \pinlabel {$x_i$} at 52 88
  
  \pinlabel {$x_{j'}$} at 110 127
  \pinlabel {$x_j$} at 110 114
  \pinlabel {$x_i$} at 110 101
  \pinlabel {$x_{i'}$} at 110 88
  
  \pinlabel {$x_j$} at 167 127
  \pinlabel {$x_{j'}$} at 167 114
  \pinlabel {$x_i$} at 167 101
  \pinlabel {$x_{i'}$} at 167 88
  
  \pinlabel {$x_{j'}$} at 225 121
  \pinlabel {$x_j$} at 225 108
  \pinlabel {$x_i$} at 225 95
  \pinlabel {$x_{i'} = \emptyset$} at 252 87
  
  \pinlabel {$x_j$} at 283 121
  \pinlabel {$x_{j'}$} at 283 108
  \pinlabel {$x_i$} at 283 95
  \pinlabel {$x_{i'} = \emptyset$} at 310 87
  
  \pinlabel {$(h)$} at 10 72
  \pinlabel {$(i)$} at 68 72
  \pinlabel {$(j)$} at 126 72
  \pinlabel {$(k)$} at 184 72
  \pinlabel {$(l)$} at 242 72
  \pinlabel {$(m)$} at 300 72
  
  \pinlabel {$x_{i'}$} at -5 47
  \pinlabel {$x_{j'}$} at -5 34
  \pinlabel {$x_j$} at -5 21
  \pinlabel {$x_i$} at -5 8

  \pinlabel {$x_{i'}$} at 52 47
  \pinlabel {$x_j$} at 52 34
  \pinlabel {$x_{j'}$} at 52 21
  \pinlabel {$x_i$} at 52 8
  
  \pinlabel {$x_j$} at 110 47
  \pinlabel {$x_{i'}$} at 110 34
  \pinlabel {$x_i$} at 110 21
  \pinlabel {$x_{j'}$} at 110 8
  
  \pinlabel {$x_j$} at 167 47
  \pinlabel {$x_i$} at 167 34
  \pinlabel {$x_{i'}$} at 167 21
  \pinlabel {$x_{j'}$} at 167 8
  
  \pinlabel {$x_j$} at 225 41
  \pinlabel {$x_{i'}$} at 225 28
  \pinlabel {$x_i$} at 225 15
  \pinlabel {$x_{j'} = \emptyset$} at 252 7
  
  \pinlabel {$x_j$} at 283 41
  \pinlabel {$x_i$} at 283 28
  \pinlabel {$x_{i'}$} at 283 15
  \pinlabel {$x_{j'} = \emptyset$} at 310 7
  
  \pinlabel {$(n)$} at 10 -8
  \pinlabel {$(o)$} at 68 -8
  \pinlabel {$(p)$} at 126 -8
  \pinlabel {$(q)$} at 184 -8
  \pinlabel {$(r)$} at 242 -8
  \pinlabel {$(s)$} at 300 -8

\endlabellist
\includegraphics[scale=1.2]{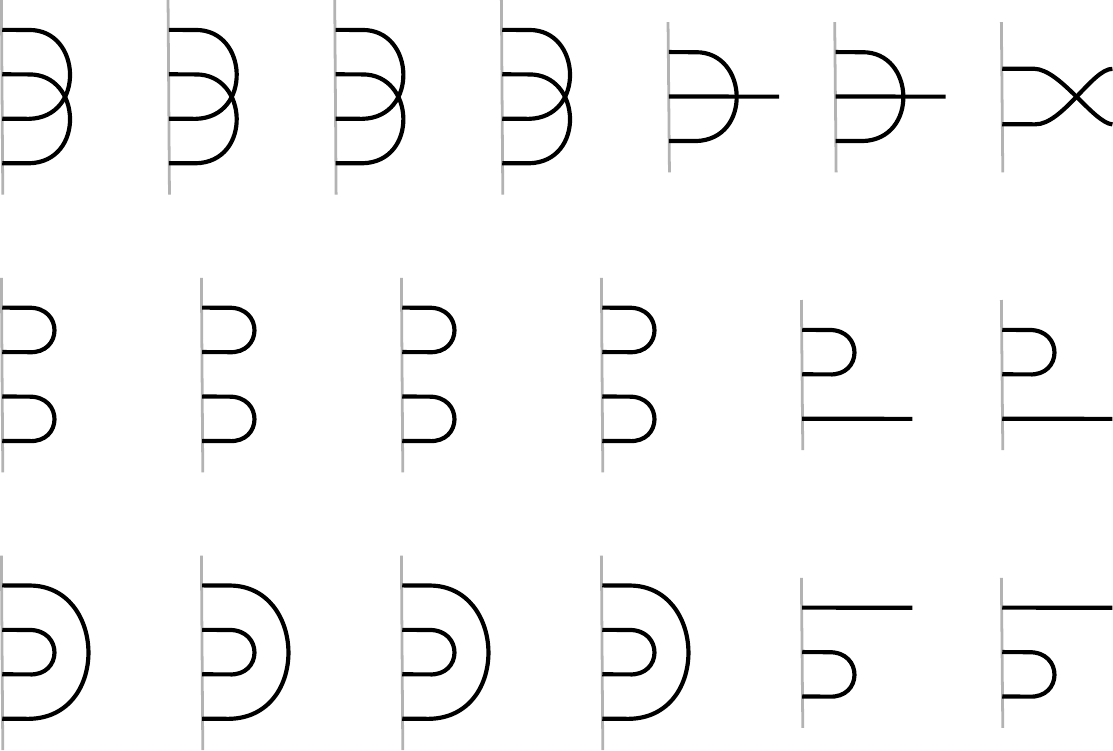}
\vspace{3 mm}
\caption{The nineteen cases for in the proof of  Proposition \ref{prop:inductive-step}.}
\label{fig:19cases}
\end{figure}
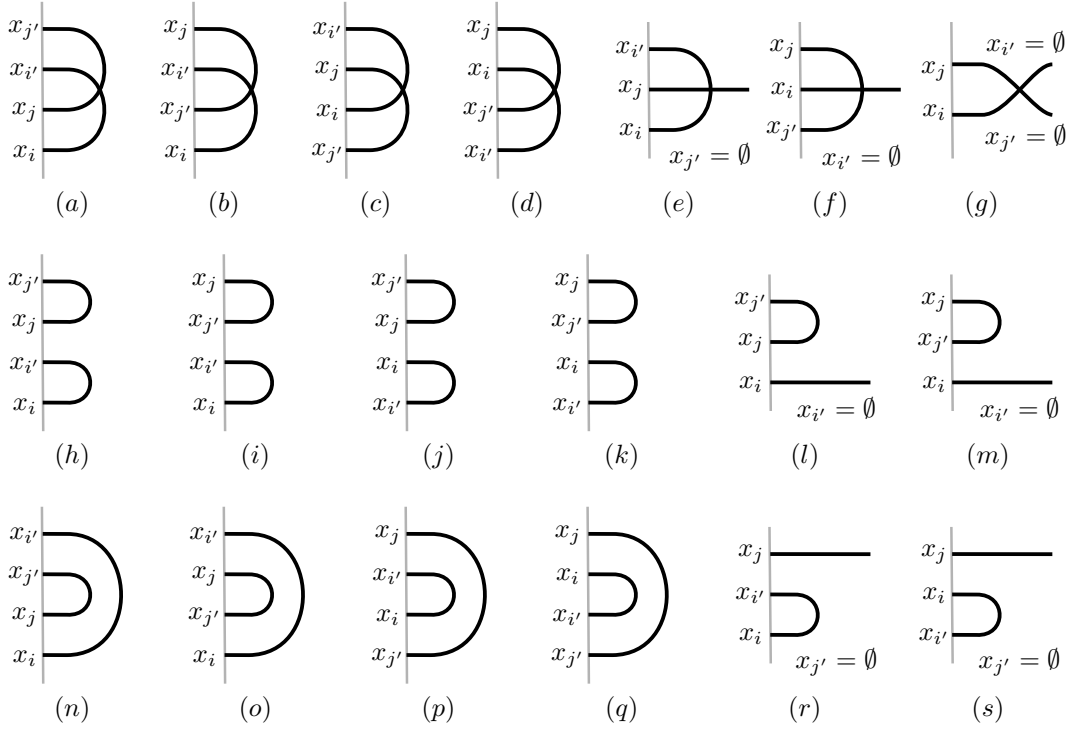

{\bf Cases (c), (e), (f), and (g):} In each of these cases the segments $s_i$ and $s_j$ intersect. We will call the intersection point $p$; it is straightforward to check that $p$ has degree $-2b$, so that the power of $W$ in the weight of any left turn arrow at $p$ must be $b$. If $d^{\tracks_N}_{i,j} = d_{i,j}$ we are done, otherwise let $c = d_{i,j} - d^{\tracks_N}_{i,j}$. We obtain $\tracks_{N+1}$ from $\tracks_N$ by adding a left turn crossover arrow weighted by $cW^b$ at the point $p$ as shown in Figure \ref{fig:case1-crossing}. The newly created bigon shaded in the figure contributes $c(UV)^b U^{a-b} x_j$ to $\partial^{\tracks_{N+1}}(x_i)$. This is precisely the change to $d^{\tracks_N}_{i,j}$ needed to make it agree with $d_{i,j}$; we only need to check that adding this crossover arrow does not have any unwanted side effects on other terms in the differential.

We can restrict our attention to bigons which lie entirely on the right side of $\mu$, since if a bigon crosses to the left side of $\mu$ and comes back, it must either enclose a puncture or make a left turn at an intersection point with negative degree, either of which would contribute at least one $W$ to the weight; if this bigon also involves the crossover arrow at point $p$, this contributes an additional $W^b$ and the bigon can be ignored modulo $W^{b+1}$. Similarly, we can ignore any bigons which involve another crossover arrow at a self intersection point of strictly negative degree. For simplicity, first assume that $\bchain_N$ contains no degree zero intersection points. In this case it is clear that there are no other relevant bigons formed when the new crossover arrow is added. Such a bigon would have to cover the opposite quadrant near $p$ and the $\tracks_{N+1}$ part of its boundary would be the path from $x_{i'}$ to $p$ to $x_{j'}$; since $x_{i'}$ is above $x_{j'}$ in case (c) and at least one of $x_{i'}$ or $x_{j'}$ is $\emptyset$ in cases (e), (f), and (g), this can not be the right side of a bigon formed with a segment of $\mu$.

To complete these cases, we need to allow for the possibility that $\bchain_N$ contains index zero points which lead to unwanted side effect bigons when the crossover arrow is added. Since $\bchainhat_N$ is local system type by assumption, any index zero points must be local system intersection points. Clearly an unwanted bigon can only occur if the index zero points in question occur along $s_i$ or $s_j$; that is, we must have that $x_i$ or $x_j$ lie on non-primitive curves and that $s_i$ or $s_j$ pass through the crossing region of the relevant curve between $x_i$ or $x_j$ and $p$. For example, suppose $x_i$ is $x_{i_0}$ in a multiplicity 3 grouping of generators, and that the intersection between $s_{i_0}$ and $s_{i_2}$ is in $\bchain$ and falls between $x_{i_0}$ and $p$, as on the left side of Figure \ref{fig:case1-local-systems}. Adding the crossover arrow at $p$ to introduce a bigon from $x_{i_0}$ to $x_j$ also introduces the bigon from $x_{i_2}$ to $x_j$ shaded in the figure. As another example, suppose $x_j$ is $x_{j_1}$ in a multiplicity 3 grouping of generators, and that $\bchain$ contains the intersection between $s_{j_1}$ and $s_{j_0}$ as on the right side of Figure \ref{fig:case1-local-systems}; adding the point $p$ to $\bchain$ introduces the shaded bigon from $x_i$ to $x_{j_0}$ in addition to the desired bigon from $x_i$ to $x_{j_1}$. From these examples it is clear that fixing the coefficient $d^{\tracks_N}_{i,j}$ only affects other terms of the differential if $i = i_0$ and/or $j = j_\ell$ with $\ell \neq 0$, and in these cases the indices of the other terms affected are obtained from $(i,j)$ by replacing $i$ with $i_k$ for $k \neq 0$ and/or replacing $j$ with $j_0$. Any pair of this form is greater than $(i,j)$ with respect to the ordering $\lessdot$, so we can ignore these side effects.

\begin{figure}
\labellist
 \pinlabel {$x_i$} at -6 8
 \pinlabel {$x_j$} at -6 45
 \pinlabel {\tiny $cW^b$} at 38 45
   
\endlabellist
\includegraphics[scale=1]{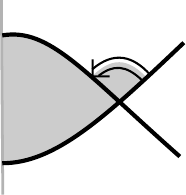}
\vspace{2 mm}
\caption{To correct $d^\tracks_{i,j}$ in cases (c), (e), (f), and (g) we add a multiple of the intersection point shown to $\bchain$; this modifies $\tracks$ by adding a crossover arrow in a neighborhood of this intersection point as pictured. This amounts to allowing left turns from $s_i$ to $s_j$ at this point.}
\label{fig:case1-crossing}
\end{figure}

\begin{figure}
\labellist
 \pinlabel {\tiny $i_0$} at -6 7
 \pinlabel {\tiny $i_1$} at -6 14
 \pinlabel {\tiny $i_2$} at -6 21
 \pinlabel {$j$} at -6 62
 \pinlabel {$p$} at 68 50
 
 \pinlabel {\tiny $j_0$} at 107 52
 \pinlabel {\tiny $j_1$} at 107 59
 \pinlabel {\tiny $j_2$} at 107 66
 \pinlabel {$i$} at 107 12
 \pinlabel {$p$} at 179 22
\endlabellist
\includegraphics[scale=1]{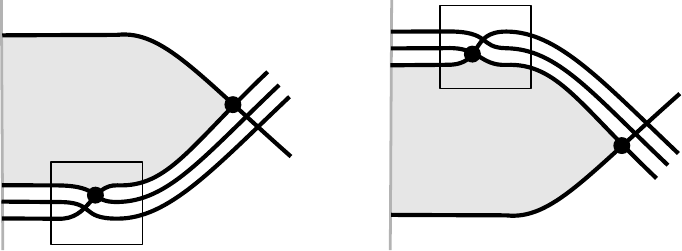}
\vspace{2 mm}
\caption{Possible side effect bigons in case (c), (e), (f), or (g)  involving weight zero points in $\bchain$.}
\label{fig:case1-local-systems}
\end{figure}
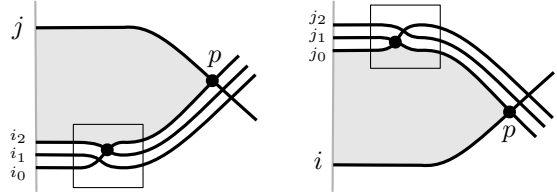

{\bf Cases (j) and (l):}  In these cases, we argue that $d^{\tracks_N}_{i,j}$ must already agree with $d_{i,j}$ by considering the coefficient of $x_{j'}$ in $(\partial^{\tracks_N})^2(x_i)$, which up to a power of $U$ and of $V$ is given by the left hand sum below. While $\partial^{\tracks_N}$ is not necessarily a differential, this term of $(\partial^{\tracks_N})^2$ is zero by Lemma \ref{lem:some-d-squared-zero-horizontal}(ii). The analogous $i,j'$-coefficient of $\partial^2$ is also zero, since $C$ is a chain complex, and so
\begin{equation}\label{eq:case-j}
\sum_{\ell = 1}^n d^{\tracks_N}_{i, \ell} d^{\tracks_N}_{\ell, j'} = 0 = \sum_{\ell = 1}^n d_{i, \ell} d_{\ell, j'}.
\end{equation}
We claim that every coefficient $d^{\tracks_N}_{i, \ell}$ or $d^{\tracks_N}_{\ell, j'}$ which appears in a nontrivial term of the left hand sum agrees with the corresponding coefficient $d_{i, \ell}$ or $d_{\ell, j'}$ except possibly $d^{\tracks_N}_{i,j}$. The sum can be taken only over indices $\ell$ for which $M(x_\ell)$ and $M(x_i)$ have opposite parity and for which $ 0 \le a_{i,\ell} \le a_{i,j'}$ and $0 \le b_{i,\ell} \le b_{i,j'} $. Note that because there is a horizontal arrow from $x_j$ to $x_j'$, $b_{i,j'} = b_{i,j} = b$. We have that $b_{i,\ell}$ and $b_{\ell, j'}$ are nonnegative and sum to $b_{i,j'}$, so both are less than $b$ unless one of them is zero. If both are less than $b$, then $d^{\tracks_N}_{i,\ell}$ agrees with $d_{i,\ell}$ and $d^{\tracks_N}_{\ell,j'}$ agrees with $d_{\ell,j'}$ since $C^{\tracks_N}$ agrees with $C$ mod $W^b$. If $b_{i,\ell} = 0$ then $d^{\tracks_N}_{i,\ell} = d_{i,\ell}$ must be zero, since there are no horizontal arrows starting at $x_i$. If $b_{\ell,j'} = 0$ and $d^{\tracks_N}_{\ell,j'} = d_{\ell,j'}$ is nonzero, then there is a horizontal arrow from $x_\ell$ to $x_{j'}$. This implies that $x_\ell$ is $x_j$ or possibly another generator in the same grouping as $x_j$, if $x_j$ lies on a non-primitive curve. In the later case, $x_{j'}$ must be the first generator in its grouping since only the first generator can have an incoming horizontal arrow, and thus $x_j$ must be the first generator in its grouping; it follows that $(i, \ell) \lessdot (i, j)$ so we can assume that $d^{\tracks_N}_{i,\ell} = d_{i,\ell}$. After deleting all terms which agree on both sides, Equation \eqref{eq:case-j} reduces to
$$ d^{\tracks_N}_{i, j} d^{\tracks_N}_{j, j'} = d_{i, j} d_{j', j'} $$
with $d^{\tracks_N}_{j, j'} = d_{j, j'} \neq 0$. Thus $d^{\tracks_N}_{i,j} = d_{i,j}$.

{\bf Case (s):} This is similar to cases (j) and (l) except that we consider the coefficient of the $x_{j}$ in $(\partial^{\tracks_N})^2(x_{i'})$, which is zero by Lemma \ref{lem:some-d-squared-zero-horizontal}(i). This leads to the equation 
\begin{equation}\label{eq:case-s}
\sum_{\ell = 1}^n d^{\tracks_N}_{i', \ell} d^{\tracks_N}_{\ell, j} = 0 = \sum_{\ell = 1}^n d_{i', \ell} d_{\ell, j}.
\end{equation}
As in the previous case it is clear that the $\ell$th term is equal on both sides unless either $b_{i',\ell}$ or $b_{\ell, j}$ are zero. If $b_{\ell, j} = 0$ then $d^{\tracks_N}_{\ell, j} = d_{\ell, j} = 0$, since $x_j$ has no incoming horizontal arrows. If $b_{i',\ell} = 0$ and $d^{\tracks_N}_{i',\ell} = d_{i',\ell} \neq 0$ then there is a horizontal arrow from $x_{i'}$ to $x_{\ell}$, which implies either $\ell = i$ or that $\ell = i_0$ where $x_i$ is part of a grouping of generators with indices $\{i_0, \ldots, i_r\}$. If $\ell = i_0 \neq i$, then $(\ell, j) \lessdot (i,j)$ so we can assume that $d^\tracks_{\ell,j} = d_{\ell, j}$. Equation \eqref{eq:case-s} then reduces to $ d^{\tracks_N}_{i', i} d^{\tracks_N}_{i, j} = d_{i', i} d_{i, j} $ and $d^{\tracks_N}_{i', i} = d_{i', i} \neq 0$, so $d^{\tracks_N}_{i,j} = d_{i,j}$. 

{\bf Cases (a) and (d):} As in the cases (c), (e), (f), and (g) above, the segments $s_i$ and $s_j$ cross and their intersection point $p$ has index $b$. As before we can modify the coefficient $d^{\tracks_N}_{i,j}$ as needed by adding a left turn crossover arrow at $p$. The difference is that now making this change also introduces another bigon from $x_{i'}$ to $x_{j'}$. If $(i,j) \lessdot (i',j')$ then we can ignore the effect of this bigon (as well as the effect of any other bigon from a generator in the grouping of $x_{i'}$ to a generator in the grouping of $x_{j'}$, when there are non-primitive curves involved) and the proof proceeds as before. If instead $(i',j') \lessdot (i,j)$ we argue that $d^{\tracks_N}_{i,j}$ must already agree with $d_{i,j}$.

In case (a), we consider the coefficient of $x_{j'}$ in $(\partial^{\tracks_N})^2(x_i)$, as in cases (j) and (l) above. The proof is identical to those cases, except that we must consider values of $\ell$ with $b_{i,\ell} = 0$ in addition to those with $b_{\ell, j'} = 0$, since $x_i$ has an outgoing horizontal arrow. However, the only horizontal arrow out of $x_i$ goes to $x_{i'}$, unless $x_i$ is in a grouping of generators with indices $\{i_0, \ldots, i_{r-1}\}$ and $i \neq i_0$, in which case there may also be a horizontal arrow from $x_i$ to $x_{i'_0}$. Thus if $b_{i,\ell} = 0$ and $d^{\tracks_N}_{i,\ell} = d_{i,\ell} \neq 0$ then $\ell$ is $i'$ or is $i'_0$, the first index in a grouping containing $i'$. We have assumed that $(i',j') \lessdot (i,j)$, and if $\ell = i'_0 \neq i'$ then $(\ell, j') \lessdot (i', j') \lessdot (i,j)$, so by the inductive assumption $d^{\tracks_N}_{\ell, j'} = d_{\ell, j'}$. As before, Equation \eqref{eq:case-j} implies $d^\tracks_{i,j} = d_{i,j}$. 

For case (d) we consider the coefficient of $x_j$ in $(\partial^{\tracks_N})^2(x_{i'})$ as in case (s) above. Since $x_j$ has an incoming horizontal arrow we must consider terms $d^{\tracks_N}_{i',\ell}d^{\tracks_N}_{\ell, j}$ where $b_{\ell,j} = 0$, but for such terms if $d^{\tracks_N}_{i',\ell} \neq 0$ then either $\ell = j'$ or $x_{j'}$ is the first in a grouping of generators and $\ell$ is the index of another generator in that grouping. In either case $(i', \ell) \lessdot (i',j')$, so we can assume $d^{\tracks_N}_{i',\ell} = d_{i',\ell}$. Then Equation \eqref{eq:case-s} implies $d^{\tracks_N}_{i,j} = d_{i,j}$. 

{\bf Cases (n) and (q):} These are similar to the previous few cases, with one additional consideration that we highlight here. For brevity we focus on case (n); the translation to case (q) is straightforward. Like in case (j), we consider the coefficient of $x_{j'}$ in $(d^{\tracks_N})^2(x_i)$; this is zero by Lemma \ref{lem:some-d-squared-zero-horizontal}(ii). As before the relevant terms in Equation \eqref{eq:case-j} are those for which either $\ell$ is $i'$ (or the first index in the grouping containing $i'$) or $\ell$ is $j$ (or another index in the grouping containing $j$, where $j$ is the first of that grouping). We again argue that for such $\ell$ the coefficients $d^{\tracks_N}_{i,\ell}$ and $d^{\tracks_N}_{\ell, j'}$ already agree with their counterparts on the other side of the equation except possibly $d^{\tracks_N}_{i,j}$. The key observation to make is that $(i',j') \lessdot (i,j)$, the other cases follow as before from the way pairs coming from groupings of indices are ordered. Note that since there are horizontal arrows from $i$ to $i'$ and from $j$ to $j'$ we have $b_{i',j'} = b_{i,j} = b$. We have that $a_{i,j} \ge b_{i,j}$ since $i < j$, while since $i' > j'$ we have $a_{i',j'} \le b_{i',j'}$. It follows that $\min(a_{i',j'}, b_{i',j'}) < \min(a_{i,j}, b_{i,j})$, and hence $(i',j') \lessdot (i,j)$, unless $A(x_i) = A(x_j)$ and $A(x_{i'}) = A(x_{j'})$. If this is the case, then the segments $s_i$ and $s_j$ are parallel, they neither cross nor diverge before returning to $\mu$, and they return to $\mu$ at the same height. It follows that the pair $(i,j)$ has complexity one higher than the pair $(i',j')$, and we have chosen the ordering $\lessdot$ so that this implies $(i',j') \lessdot (i,j)$.

{\bf Cases (b), (i), (m), (o), (p) and (r):} In each of these cases, $d^{\tracks_N}_{i,j}$ does not necessarily agree with $d_{i,j}$, so if it does not we must modify the train track $\tracks$ to adjust the coefficient $d^{\tracks_N}_{i,j}$. However, in all but case (b) there is no intersection point at which to add a crossover arrow as we did before, while in case (b) there is an intersection between the segments $s_i$ and $s_j$ but it has degree $2b > 0$ and thus a crossover arrow added there must be a right turn crossover arrow. Instead, we will modify $\tracks_N$ by adding a crossover arrow which is not at a self intersection point of $\Gamma$, as pictured in Figure \ref{fig:cases-with-sliding}. Each arrow lies just to the right of $\mu$ and connects $s_i$ to $s_j$ near an endpoint of each of those segments and is weighted by $\pm c W^b$, where $c = d_{i,j} - d^{\tracks_N}_{i,j}$. In cases (o) and (m) the arrow connects the $x_i$ end of $s_i$ to the $x_{j'}$ end of $s_j$ and is weighted by $c W^b$, in cases (p) and (r) the arrow connects the $x_{i'}$ end of $s_i$ to the $x_j$ end of $s_j$ and is weighted by $-c W^b$, and in cases (b) and (i) we can choose either of those two options. In each case there is an obvious bigon from $x_i$ to $x_j$ introduced by adding the crossover arrow. This bigon has weight $c (UV)^b U^{a-b}$, and so it contributes exactly the desired change to $d^{\tracks_N}_{i,j}$. We can check that there are no relevant side effect of this change exactly as in cases (c), (e), (f), and (g). First, we ignore bigons crossing $\mu$ or involving other crossover arrows at intersection points of strictly negative degree, since these will not contribute mod $W^{b+1}$. We then observe that the only other possible bigons involving the new crossover arrow would connect a generator in the grouping containing $x_i$ to a generator in the grouping containing $x_j$, and we can check that any affected pairs come after $(i,j)$ with respect to the ordering $\lessdot$. Thus we have modified $\tracks_N$ to obtain $\tracks_{N+1}$ such that the coefficients $d^{\tracks_{N+1}}$ and $d$ agree for all pairs up to and including $(i,j)$. That is, $\tracks_{N+1}$ represents $C$ correctly to $N+1$ entries. The only problem is that now the train track $\tracks_{N+1}$ does not have the form of a collection of immersed curves along with left-turn crossover arrows.

To resolve this problem we will slide the new crossover arrow, much as we slid arrows in the proof of Proposition \ref{prop:simple-curves-existence}, until either it reaches a negative degree intersection point, at which it becomes a left-turn arrow, or it can be removed. The argument is in fact much easier than the arrow sliding argument used to prove Proposition \ref{prop:simple-curves-existence} because we will not resolve any crossings, so the immersed curves do not change, and because most new composition arrows formed by sliding one arrow past another can be immediately ignored. 

Initially, the arrow is just to the right of $\mu$ and connects two horizontal segments; let $x$ denote the generator corresponding to the segment at the tail of the crossover arrow and $y$ denote the generator corresponding to the segment at the head (the pair $(x,y)$ is either $(x_i, x_{j'})$ or $(x_{i'},x_j)$, depending on the case). If $A(y)$ is strictly greater than $A(x)$ then the crossover arrow can be removed by performing a change of basis replacing $x$ with $x \pm U^{A(y)-A(x)} w y$ where $w$ is the weight on the crossover arrow and the sign depends on the orientation on $\Gamma$. Removing the arrow may modify some of the coefficients of $\partial^{\tracks_{N+1}}$ which we have already arranged to agree with the corresponding coefficients of $\partial$, but if we change our chosen basis for $C$ as above then the coefficients of $\partial$ will change as well, and Proposition \ref{prop:basis-change} ensures that the changes to $\partial^{\tracks_{N+1}}$ and $\partial$ are the same (modulo $W^{b+1}$). Thus, after removing the arrow and changing the basis, it is still the case that $d^{\tracks_{N+1}}$ and $d$ agree for all pairs up to and including $(i,j)$. If $A(y) = A(x)$, then we can slide the crossover arrow to the other side of $\mu$. Once again, this may change some coefficients of $\partial^{\tracks_{N+1}}$, but by Proposition \ref{prop:basis-change} the change is exactly matched in $\partial$ by replacing the basis element $x$ with $x \pm wy$ where $w$ is the weight on the crossover arrow and the sign depends on the orientation on $\Gamma$.

Once the arrow passes through $\mu$, one of four things can happen: (1) the segments connected by the arrow cross; (2) without crossing, the segments turn opposite directions (i.e., the segment from $y$ returns to $\mu$ above $y$ while the segment from $x$ returns to $\mu$ below $x$) or one or both segments do not return to $\mu$; (3) the segments turn the same direction but return to $\mu$ at different heights; or (4) the segments return to $\mu$ at the same height. Examples of these four cases are shown in Figure \ref{fig:arrow-sliding-cases}. In the first case, we can slide the crossover arrow until it reaches the crossing and observe that it is now a left-turn crossover arrow (in particular, this intersection point has negative degree $-2b$). In the second case, it is easy to see that the arrow can be removed without affecting the differential $\partial^{\tracks_{N+1}}$ modulo $W^{b+1}$; since there is a path from the left side of the crossover arrow to $\partial \strip$ disjoint from the two arcs connected by the arrow, any bigon involving the crossover arrow must either involve another crossover arrow at a negative degree intersection point or it must cross to the other side of $\mu$ and back, either of which would contribute up at least one additional factor of $W$. In the third case, we slide the crossover arrow to the other end of the two segments so that it points downward just to the left of $\mu$ connecting the other endpoints of the segments beginning at $x$ and $y$ respectively. Applying Proposition \ref{prop:basis-change} again, we can remove this arrow if we change the basis for $C$. Finally, in the fourth case we again slide the arrow so that it points downward from $x'$ to $y'$ but then we slide the arrow across $\mu$ while changing the basis as dictated by Proposition \ref{prop:basis-change}. We now have (a rotated version of) the same four cases to consider, and we repeat the argument. Clearly the arrow will eventually be removed or stop at a crossing unless the curves connected by the arrow are completely parallel. However, we have assumed that $\Gamma$ is in standard position and no two curves bound an immersed annulus, so this is not possible.

We have now corrected the coefficient $d_{i,j}^{\tracks_N}$, though to do so we were forced to add a crossover arrow which is not a left turn crossover arrow. We have also shown that, at the expense of choosing a new basis, we can always remove this arrow or replace it with a left turn crossover arrow. The only thing remaining to check is that we can also deal with any new arrows that might be introduced while sliding the arrow as described above. Recall from Section \ref{sec:curves-with-crossover-arrows} that if the head of this crossover arrow slides past the tail of another crossover arrow, or vice versa, we need to add a new crossover which is the composition of the two. However, the weight of the new arrow will be the product of the weights of the two arrows; since the arrow we are sliding has weight $cW^b$ and at this stage we only need to preserve $C(\tracks_{N+1})$ modulo $W^{b+1}$, we may immediately ignore any such compositions unless the arrow being passed is at a degree zero intersection point. By assumption, such arrows only occur in the crossing region of some non-primitive curve.  If the crossover arrow we wish to remove connects segments $s_1$ and $s_2$, then new arrows introduced will be the same except that they will connect a segment in the same grouping as $s_1$ to a segment in the same grouping as $s_2$. We can slide all of these arrows together as a group, and the segments they connect will always behave the same. Thus when we can remove the original arrow or move it to an intersection point, the same is true of all of the other arrows.

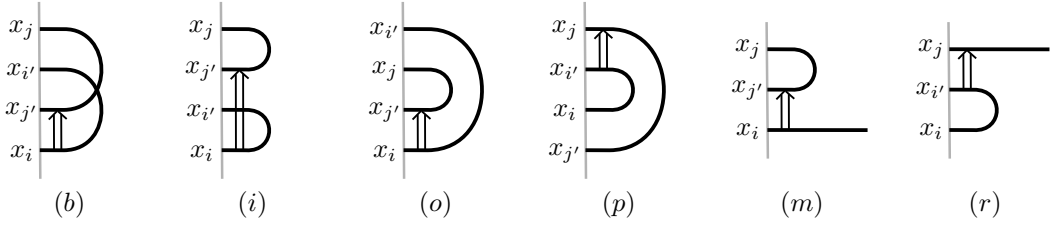
\begin{figure}
\labellist
  \pinlabel {$x_j$} at -5 47
  \pinlabel {$x_{i'}$} at -5 34
  \pinlabel {$x_{j'}$} at -5 21
  \pinlabel {$x_i$} at -5 8

  \pinlabel {$x_j$} at 52 47
  \pinlabel {$x_{j'}$} at 52 34
  \pinlabel {$x_{i'}$} at 52 21
  \pinlabel {$x_i$} at 52 8
  
  \pinlabel {$x_{i'}$} at 110 47
  \pinlabel {$x_j$} at 110 34
  \pinlabel {$x_{j'}$} at 110 21
  \pinlabel {$x_i$} at 110 8
  
  \pinlabel {$x_j$} at 167 47
  \pinlabel {$x_{i'}$} at 167 34
  \pinlabel {$x_i$} at 167 21
  \pinlabel {$x_{j'}$} at 167 8
  
  \pinlabel {$x_j$} at 225 41
  \pinlabel {$x_{j'}$} at 225 28
  \pinlabel {$x_i$} at 225 15
  
  \pinlabel {$x_j$} at 283 41
  \pinlabel {$x_{i'}$} at 283 28
  \pinlabel {$x_i$} at 283 15
  
  \pinlabel {$(b)$} at 10 -8
  \pinlabel {$(i)$} at 68 -8
  \pinlabel {$(o)$} at 126 -8
  \pinlabel {$(p)$} at 184 -8
  \pinlabel {$(m)$} at 242 -8
  \pinlabel {$(r)$} at 300 -8

\endlabellist
\includegraphics[scale=1.2]{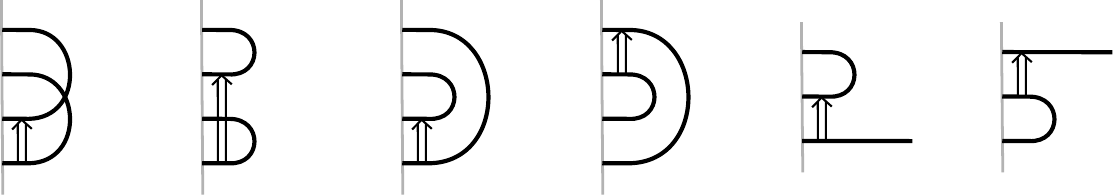}
\vspace{3 mm}
\caption{Crossover arrows added to modify $d^\tracks_{i,j}$; each crossover arrow is weighted by a multiple of $W^b$ and adding it introduces a new bigon from $x_i$ to $x_j$.}
\label{fig:cases-with-sliding}
\end{figure}

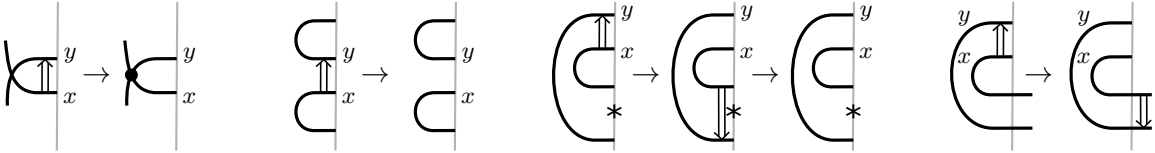
\begin{figure}
\labellist
\small
  \pinlabel {$x$} at 25 20
  \pinlabel {$y$} at 25 36
  \pinlabel {$x$} at 70 20
  \pinlabel {$y$} at 70 36

  \pinlabel {$x$} at 130 20
  \pinlabel {$y$} at 130 36
  \pinlabel {$x$} at 175 20
  \pinlabel {$y$} at 175 36
  
  \pinlabel {$x$} at 235 37
  \pinlabel {$y$} at 235 52
  \pinlabel {$x$} at 280 37
  \pinlabel {$y$} at 280 52
  \pinlabel {$x$} at 325 37
  \pinlabel {$y$} at 325 52
  
  \pinlabel {$x$} at 362 35
  \pinlabel {$y$} at 362 50
  \pinlabel {$x$} at 407 35
  \pinlabel {$y$} at 407 50
  
\large
  \pinlabel{$\to$} at 35 28
  \pinlabel{$\to$} at 140 28
  \pinlabel{$\to$} at 242 28
  \pinlabel{$\to$} at 287 28
  \pinlabel{$\to$} at 390 28

  \endlabellist
\includegraphics[scale=1]{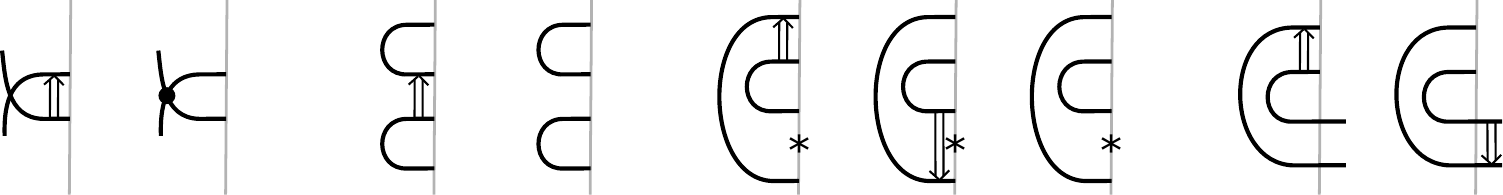}
\caption{Four possibilities after sliding a crossover arrow (to its left) across $\mu$. We can either slide the arrow to a crossing where it becomes a left turn arrow and can be incorporated into $\bchain$, remove the arrow (possibly after a change of basis) or we slide the arrow between parallel segments until it crosses $\mu$ again.}
\label{fig:arrow-sliding-cases}
\end{figure}

{\bf Cases (h) and (k):} These final two cases are similar to cases (b), (i), (o), (p), (m), and (r). Like in those cases, we add a crossover arrow with appropriate weight to the right of $\mu$ from $x_{i'}$ to $x_{j}$ in case (h) or from $x_{i}$ to $x_{j'}$ in case (k); this modifies the coefficient $d^{\tracks_N}_{i,j}$ to make it agree with $d_{i,j}$. The new crossover arrow is not at an intersection point, but we can slide it just as in the previous cases until it reaches an intersection point as a left-turn crossover arrow or it can be removed. The only difference is that in these cases, adding the crossover arrow has the additional effect of changing $d^{\tracks_N}_{i',j'}$, as it also creates a new bigon from $x_{i'}$ to $x_{j'}$. We proceed as in cases $(a)$ and $(d)$. If $(i,j) \lessdot (i',j')$, then we can ignore any side effect bigons connecting $x_{i'}$ to $x_{j'}$. If on the other hand $(i',j') \lessdot (i,j)$ then we argue that $d^{\tracks_N}_{i,j}$ must already agree with $d_{i,j}$ without adding the crossover arrow. By assumption, $d^{\tracks_N}_{i',j'}$ agrees with $d_{i',j'}$.  In case $(h)$, we consider the coefficient of $x_{j'}$ in $(\partial^{\tracks_N})^2(x_i)$ and proceed exactly as in case $(a)$, while in case $(k)$ we consider the coefficient of $x_j$ in $(\partial^{\tracks_N})^2(x_{i'})$ and proceed as in case $(d)$.
\end{proof}

We can now prove the existence part of Theorem \ref{thm:curve-invariant-for-complex} for complexes without flip maps.  By Proposition \ref{prop:simple-curves-existence}, there exists a decorated curve $(\Gamma, \bchainhat)$ in $\strip$ that represents $C$ over $\sRhat$ and for which $\bchainhat$ is of local system type. By perturbing $\Gamma$, we may assume that it is in almost simple position. Defining $\bchain_0$ to be $\bchainhat$ and inducting using Proposition \ref{prop:inductive-step}, we find $\bchain = \bchain_{n^2}$ such that $(\Gamma, \bchain)$ represents $C$ over $\sRminus$.

\subsection{Enhanced curves in $\cylinder$}
Let $C$ be a bigraded complex over $\sRminus$ and let $\Psi_*: H_* C^h \to H_* C^v$ be a of flip isomorphism. To represent this data with a decorated curve in the marked cylinder $\cylinder$, we first represent the corresponding $UV=0$ complex $\widehat C$ and flip isomorphism $\widehat\Psi_*$ as in Section \ref{sec:flip-maps}. Recall that we construct the decorated curve $(\Gamma, \bchainhat)$ in $\cylinder$ by gluing together a curve in $\strip$ representing the $C$ over $\sRhat$ and a curve in $\sF$ representing $\widehat\Psi_*$. If necessary we perform the arrow sliding algorithm to ensure that the decoration takes the form of a bounding chain $\bchainhat$ consisting of only local system intersection points. We take note of any basis changes required during the arrow sliding process, even if they do not affect the complex over $\sRhat$.

Once we have a nice representative for the $UV=0$ data in $\cylinder$, we apply a homotopy so that the curve restricted to the marked strip $\strip$ is in almost simple position; this entails pushing the curve up or down at each intersection with the boundaries of the strips to ensure these intersections are ordered correctly, as well as applying homotopies to remove any immersed annuli within $\strip$. We can now enhance the curves within $\strip$ so that the restriction to $\strip$ represents the complex $C$ over $\sRminus$, as in Section \ref{sec:enhanced-general-case}. Because of the ordering of the intersection points with the boundaries of $\strip$, the decorated curve in the whole cylinder $\cylinder$ also represents the complex $C$; that is, all bigons in $\cylinder$ contributing to the Floer homology of the curve with $\mu$ are contained in $\strip$.

Importantly, we build the enhanced curves in $\strip$ starting from the basis for $C$ obtained at the end of the arrow sliding algorithm in the construction of the $UV=0$ curves, which may be different from the original basis. We may perform further basis changes while enhancing the decorated curve in $\strip$ to represent $C$ over $\sRminus$, but these basis changes have no effect modulo $UV$. It follows that the decorated arcs in the $\sF$ portion of $\cylinder$ still correctly represent the simplified flip isomorphism $\widehat\Psi_*$. We now enhance these decorated arcs in $\sF$ to capture any missing information from $\Psi_*$. Let $\{x_1,\ldots, x_m\}$ denote the intersections of $\Gamma$ with $\partial_R \strip_i = \partial_L \sF_i$ and we identify these with a subset of the (new) basis for $C$ which also forms a basis for the horizontal homology. Similarly, we let $\{y_1, \ldots, y_m\}$ denote the intersections of $\Gamma$ with $\partial_L \strip_{i+1} = \partial_R \sF_i$ and also the corresponding basis of the vertical homology of $C$. With respect to this basis, the flip isomorphism takes the form
$$\Psi_*(x_i) = \sum_{k=1}^m c_{i,j} U^{\frac{\gr_w(y_j) - \gr_z(x_i)}{2}} V^{\frac{\gr_z(y_j)-\gr_w(x_i)}{2}} y_j,$$
where the coefficient $c_{i,j}$ is zero if $\gr_w(y_j) - \gr_z(x_i)$ is odd or negative. Restricting to terms for which the power of $U$ is zero gives the simplified flip map $\widehat\Psi_*$.

Recall that $\widehat\Psi_*$ restricts to an isomorphism at each grading level, using the grading $\gr_z$ on the source and $\gr_w$ on the target, and that the decorated curve in $\sF$ already contains a bundle of segments with a collection of turning points representing this isomorphism for each grading. Because the intersections of $\Gamma$ on $\partial_L \sF = \partial_R \strip$ are ordered so that $\gr_z$ is non-increasing moving upward and the intersections of $\Gamma$ with $\partial_R \sF = \partial_L \strip$ are ordered so that $\gr_w$ is non-decreasing moving upward, every bundle of arcs crosses every other bundle (see Figure \ref{fig:flip-map-strip}).  For any term of $\Psi_*$ not included in $\widehat\Psi_*$, the power of $U$ is positive and so $\gr_w(y_j) > \gr_z(x_i)$. In this case, we can add the intersection between the segment starting at $x_i$ and the segment ending at $y_j$ to the bounding chain, which introduces a polygonal path across $\sF$ from $x_i$ to $y_j$. This intersection point has degree $\gr_z(x_i)-\gr_w(y_j) < 0$ so adding it to $\bchain$ corresponds to adding a left turn crossover arrow weighted by an appropriate multiple of $W^a$, where $a$ is the power of $U$ in the relevant term of $\Psi_*$. It is straightforward to add intersection points to $\bchain$ in this way so that the decorated arcs across $\sF$ represent the map $\Psi_*$. When this is accomplished, the decorated curves in $\cylinder$ represent $C$ and $\Psi_*$ over $\sRminus$.

\begin{figure}
\includegraphics[scale = 1]{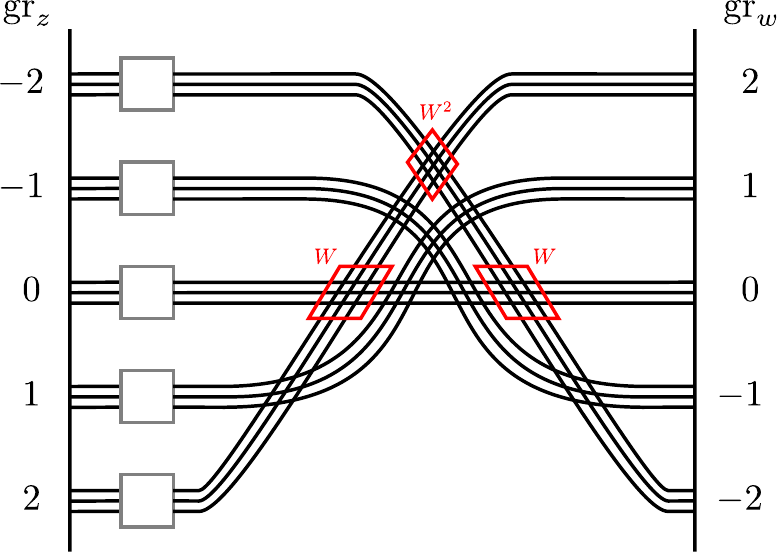}
\caption{The $\sF$ portion of an immersed curve in $\cylinder$ that has been perturbed so that the restriction to $\strip$ is in almost simple position. The bundles of strands with the same grading determine the map $\widehat\Psi_*$. We can represent $\Psi_*$ by adding turning points at the intersections between bundles oriented the same direction, weighted by a power of $W$ determined by the difference in grading between the bundles as indicated.}
\label{fig:flip-map-strip}
\end{figure}

The final step is to apply a homotopy to the curves so that they are in almost simple position as curves in $\cylinder$ (note that as of now the restriction of the curves to $\strip$ are in almost simple position, but there may be unnecessary intersection points when the arcs in $\strip$ are glued with arcs in $\sF$ to form closed curves). When performing this homotopy, we take care to apply the local moves in Figure \ref{fig:invariance-moves}; in particular, homotopies that add or remove intersection points require corresponding modifications to the bounding chain. As long as these rules are observed, the resulting decorated curves still represent the complex $C$ over $\sRminus$ and the flip isomrophism $\Psi_*$. Moreover, the restriction of the bounding chain to degree zero intersection points is of local system type, since this was true of the bounding chain from the $UV = 0$ construction and later steps only add intersection points with strictly negative degree.

\begin{remark}
When constructing enhanced curves in $\cylinder$ representing $C$ and $\Psi_i$ over $\sRminus$, we first construct curves in $\cylinder$ representing both the complex and the flip isomorphism over $\sRhat$ before adding minus information in $\strip$ and then in $\sF$ to get a representative over $\sRminus$. Another approach would be to first construct enhanced curves in $\strip$ representing $C$ over $\sRminus$, and only then pass to the cylinder by adding decorated arcs in $\sF$ representing the flip isomorphism $\Psi_*$. The problem with this approach is that representing $\Psi_i$ may then require adding crossover arrows at degree zero intersection points in $\sF$ that are not local system intersection points. To get a decorated curve of the desired form, we would need to slide these arrows to remove them. In practice it is usually clear how to do this, but there are significant difficulties in defining general arrow sliding rules in the minus setting. For this reason we want to do essentially all arrow sliding while working over $\sRhat$, and then observe that enhancing the curves does not introduce any crossover arrows at degree zero intersection points. We do use limited arrow sliding in the process of enhancing curves in $\strip$; this is mainly possible because the arrows being moved are weighted by $W^b$ for some $b$ and it is enough to work over $\sR_{b+1}$ at the time we slide the arrow.
\end{remark}

\begin{example}\label{ex:+1-surgery-LHT-enhanced-cylinder}
Consider the complex $C$ and and flip isomorphism $\Psi_*$ from Example \ref{ex:1-surgery-on-LHT} associated with $+1$-surgery on the left-handed trefoil. In Example \ref{ex:1-surgery-on-LHT-curves} we constructed a curve $\Gamma$ in $\cylinder$ (with trivial bounding chain) representing this data over $\sRhat$. Inspection reveals that the same curve also represents $C$ and $\Psi_*$ over $\sRminus$; nevertheless, it is instructive to step through the process of enhancing curves more carefully in this case.

We begin with the curve $\Gamma_\strip$ in $\strip$ representing $C$ over $\sRhat$, with respect to the basis $\{a,b,c,d,e\}$, which appears in Figure \ref{fig:1-surgery-on-LHT-curves}(a). Because $C$ has no diagonal arrows, it happens that this curve also represents $C$ over $\sRminus$; however, this is not the enhanced curve we wish to use in $\strip$. Instead, we should first add the flip map information and simplify the curve as a representative over $\sRhat$, as in Example \ref{ex:1-surgery-on-LHT-curves}. Recall that in this process we remove some arrows which corresponds to changing the basis of $C$, with the new basis given by 
$$a' = -a, \quad b' = b, \quad c' = -c + Ua - Ve, \quad d' = d, \quad \text{ and } \quad e' = e.$$
With respect to this basis, the differential on $C$ is
$$\partial(a') = -Vb', \quad \partial(b') = 0, \quad \partial(c') = UVb' - UVd', \quad \partial(d') = 0, \quad \text{ and } \quad \partial(e') = Ud'.$$
Since there are now diagonal arrows, enhancing the curves in $\strip$ to represent $C$ over $\sRminus$ with respect to this basis requires adding a nontrivial bounding chain. In particular, putting the curve from Figure \ref{fig:1-surgery-on-LHT-curves}(a) in almost simple position introduces two intersections, and we include both of these intersection points in $\bchain_\strip$ with appropriate weights to recover the two diagonal arrows in $C$, as shown in Figure \ref{fig:1-surgery-on-LHT-curves-minus}(a). We now add a collection of arcs $\Gamma_\sF$ in $\sF$ representing $\Psi_*$, which with respect to the new basis takes $a'$ to $V^{-1} c$, $b'$ to $d'$, and $c'$ to $Ue'$. Since $\Psi_*$ is simple with respect to this basis (it is a permutation of the generators) the bounding chain $\bchain_{\sF}$ in $\sF$ is trivial. Gluing the curves in $\strip$ and $\sF$ produces the decorated immersed curve shown in Figure \ref{fig:1-surgery-on-LHT-curves-minus}(b). Finally, we homotope the curve to have minimal self-intersection, noting that this is possible by move $(j)$ in Figure \ref{fig:invariance-moves}. The resulting curve in $\cylinder$ is shown in Figure \ref{fig:1-surgery-on-LHT-curves-minus}(c), which happens to be the same as the curve constructed to represent this data modulo $UV$.
\end{example}

\begin{figure}
\includegraphics[scale = .85]{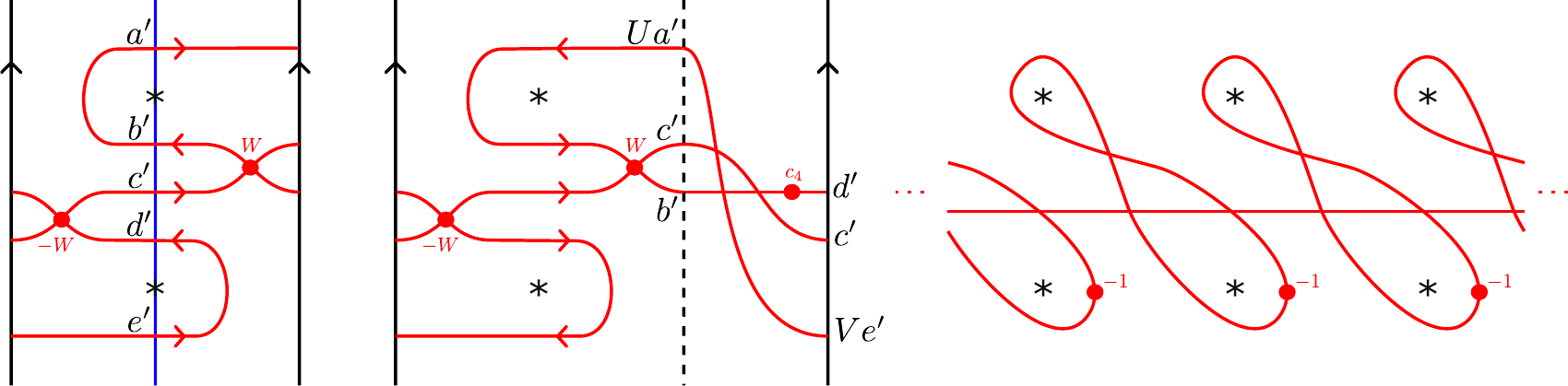}

(a) \hspace{35mm} (b) \hspace{50mm} (c) \hspace{15 mm}
\caption{(a) An immersed curve in $\strip$ representing the complex in Example \ref{ex:+1-surgery-LHT-enhanced-cylinder} with respect to the relevant basis; (b) The immersed curve in $\cylinder$ representing the complex and flip map in the example; (c) The curves pulled tight and lifted to the plane.}
\label{fig:1-surgery-on-LHT-curves-minus}
\end{figure}

\section{Morphisms and mapping cones}\label{sec:morphisms}% !TEX root = ../CFKcurvesZHS3s.tex
%morphisms.tex

In the previous three sections we constructed several versions of decorated immersed curves associated to bigraded complexes and flip maps. We will now identify the Floer homology of two such curves with algebraic operations. In particular, we relate Floer homology of curves in the marked strip to morphism spaces between two complexes, and we relate Floer homology of curves in the marked cylinder to mapping cones of certain maps defined using the flip maps. Using these observations, we can address the question of whether the immersed curves we have constructed are unique.

\subsection{Floer homology in $\strip$ as morphism spaces}

Given two bigraded complexes $C_1$ and $C_2$ over $\sRminus$, the space of morphisms from $C_1$ to $C_2$ is also a complex over $\sRminus$, which we denote $\Mor_{\sRminus}(C_1, C_2)$. Fixing homogeneous bases $\{x_1, \ldots, x_n\}$ for $C_1$ and $\{y_1, \ldots, y_m\}$ for $C_2$, let $(x_i \!\to\! y_j)$ denote the morphism $f:C_1 \to C_2$ with $f(x_i) = y_j$ and $f(x_k) = 0$ for $k \neq i$. The ring $\sRminus$ acts on morphisms in the obvious way: for $c U^a V^b \in \sR$, the morphism $cU^a V^b (x_i \!\to\! y_j)$ takes $x_i$ to $c U^a V^b y_j$ and takes $x_k$ to 0 for $k \neq i$. As a module over $\sRminus$, $\Mor_{\sR}(C_1, C_2)$ is generated by the morphisms $(x_i \!\to\! y_j)$ for $1\le i \le n$ and $1\le j \le m$. The bigrading $(\gr_w, \gr_z)$ on $\Mor_{\sRminus}(C_1, C_2)$ records the change in bigrading under a given morphism; in particular, $\gr_w((x_i \!\to\! y_j) ) = \gr_w(y_j) - \gr_w(x_i)$ and $\gr_z( (x_i \!\to\! y_j) ) = \gr_z(y_j) - \gr_z(x_i)$. Similarly, the Alexander grading $A = \tfrac 1 2(\gr_w - \gr_z)$ on $\Mor_{\sRminus}(C_1, C_2)$ is given by the change in Alexander grading under a given morphism. As usual, multiplication by $U$ and $V$ shift the bigrading by $(-2, 0)$ and $(0,-2)$, respectively. The differential on $\Mor_{\sRminus}(C_1, C_2)$ is defined by
$$(\partial f) (x) = \partial_{2}( f(x) ) - (-1)^{\gr_w(f)} f(\partial_{1} x).$$
Note that since the differential preserves the Alexander grading on $C_1$ and $C_2$, the same is true on $\Mor_{\sRminus}(C_1, C_2)$; we let $\Mor_{\sRminus}(C_1, C_2)_s$ denote the direct summand of $\Mor_{\sR}(C_1, C_2)$ in Alexander grading $s$ for each $s$ in $\Z$. $\Mor_{\sRminus}(C_1,C_2)_s$ is a module over $\F[W]$ generated by $U^{A(y_j) - A(x_i) - s} (x_i \!\to\! y_j )$ for $1\le i \le n$ and $1\le j \le m$ with $A(y_j) - A(x_i) \ge s$ and $V^{s + A(x_i) - A(y_j)} (x_i \!\to\! y_j )$ for $1\le i \le n$ and $1\le j \le m$ with $A(y_j) - A(x_i) < s$. A similar construction holds in the $UV = 0$ setting: letting $\widehat C_1$ and $\widehat C_2$ denote the $UV = 0$ quotient of $C_1$ and $C_2$, respectively, $\Mor_{\sRhat}(\widehat C_1, \widehat C_2)$  is a bigraded complex over $\sRhat$. In particular, $\Mor_{\sRhat}(\widehat C_1, \widehat C_2)$ is the $UV = 0$ quotient of $\Mor_{\sR}(C_1, C_2)$. For each $s$, $\Mor_{\sRhat}(\widehat C_1, \widehat C_2)_s$ is a vector space over $\F$.

Our aim is to relate the homology of the morphism complexes described above to an appropriate version of Floer homology of immersed curves. For $i \in \{1,2\}$, let $(\Gamma_i, \bchain_i)$ be the curves with bounding chains in the marked strip $\strip$ representing $C_i$, and let $\tracks_i$ denote the immersed train track in $\strip$ corresponding to this decorated curve. For each integer $s$, let $\tracks_i[s]$ denote the result of shifting $\tracks_i$ upward by $s$ units. We will relate $\Mor_{\sRminus}(C_1, C_2)$ to the \emph{wrapped Floer homology} $HF(p(\tracks_2), p(\tracks_1))$, where $p$ is the projection map from $\strip$ to $\mathcal{T} = \strip/(x,y)\sim(x,y+1)$.  By wrapped Floer homology we mean the Floer homology after we modify $p(\tracks_1)$ in a neighborhood of $\partial \mathcal{T}$ so that it spirals around the boundary (following the boundary orientation) infinitely many times as it approaches the boundary. The wrapped Floer homology $HF(p(\tracks_2), p(\tracks_1))$ in $\mathcal{T}$ can be understood by considering the lifts in $\strip$, with each Alexander grading summand coming from a different lift of $p(\tracks_1)$. In particular, the Alexander grading $s$ piece of $HF(p(\tracks_2), p(\tracks_1))$ can be identified with the wrapped Floer homology $HF(\tracks_2, \tracks_1[s])$ in $\strip$, which is defined to be the Floer homology after the endpoints of $\tracks_1[s]$ have been pushed upward on $\partial_R \strip$ and downwards on $\partial_L \strip$ past all endpoints of $\tracks_2$.

\begin{proposition}\label{prop:pairing-in-strip}
 For $i \in \{0,1\}$, let $\tracks_i = (\Gamma_i, \bchain_i)$ be an immersed multicurve with bounding chain in $\strip$ representing a bigraded complex $C_i$. For every $s \in \Z$, there is an isomorphism of graded $\F[W]$-modules
$$H_*( \Mor_{\sRminus}(C_1, C_2)_s ) \cong HW(\tracks_2, \tracks_1[s] ),$$
where the right hand side is wrapped Floer homology in the marked strip $\strip$.
\end{proposition}
\begin{proof}
We will homotope the immersed curves in $\tracks_1[s]$ and $\tracks_2$ into a particular form so that their wrapped Floer chain complex $CF(\tracks_2, \tracks_1[s])$ is isomorphic to $\Mor_{\sRminus}(C_1, C_2)_s$ a complex over $\F[W]$. We first homotope both curves so that their intersection with the strip $[-\tfrac 1 4, \tfrac 1 4]\times \R$ is a collection of horizontal segments; there is one horizontal segment for each intersection with $\mu$, and these correspond to generators of the respective complexes. Note that the segment corresponding to $y_j \in C_2$ appears at height $A(y_j)$ and the segment corresponding to $x_i \in C_1$ appears at height $A(x_i) + s$. For each generator $x_i$ of $C_1$, let $x_i^L$ and $x_i^R$ denote the left and right endpoints, respectively, of the horizontal segment in $\tracks_1[s]$ corresponding to $x_i$; note that $x_i^L$ lies on the line $\mu_{-\tfrac 1 4} = \{-\tfrac 1 4\}\times \R$ and $x_i^R$ lies on the line $\mu_{\tfrac 1 4} = \{\tfrac 1 4\}\times \R$. Similarly, for each generator $y_j$ of $C_2$ let $y_j^L$ and $y_j^R$ denote the endpoints of the appropriate horizontal segment in $\tracks_2$. We now choose some heights $h_{min}$ and $h_{max}$ so that all horizontal segments in both curve sets fall between these two heights, and we perturb $\tracks_1[s]$ by sliding the portion in $[\tfrac 1 4, \tfrac 1 2]\times \R$ upward until it is entirely above height $h_{max}$; note that the endpoints $x_i^R$ slide upward along $\mu_{\tfrac 1 4}$ as  well and we perturb the horizontal segments in the strip $[-\tfrac 1 4, \tfrac 1 4] \times \R$ accordingly. Similarly, we slide the portion of $\tracks_1[s]$ in $[-\tfrac 1 2, -\tfrac 1 4]\times \R$ downward until it is entirely below the height $h_{min}$ and perturb the horizontal segments accordingly. See Figure \ref{fig:morphism-complex} for an example.

\begin{figure}
\labellist
\footnotesize
  \pinlabel {$x_1$} at 96 159
  \pinlabel {$x_2$} at 96 114
  \pinlabel {$x_3$} at 96 69
  
  \pinlabel {$y_1$} at 96 144
  \pinlabel {$y_2$} at 96 99
  \pinlabel {$y_3$} at 96 54

\footnotesize
  \pinlabel {$y_1^R$} at 141 146
  \pinlabel {$y_2^R$} at 141 91
  \pinlabel {$y_3^R$} at 141 56
  \pinlabel {$y_1^L$} at 41 146
  \pinlabel {$y_2^L$} at 41 101
  \pinlabel {$y_3^L$} at 41 45
  
 \tiny
  \pinlabel {$x_1^R$} at 141 207
  \pinlabel {$x_2^R$} at 141 194
  \pinlabel {$x_3^R$} at 141 174
  \pinlabel {$x_1^L$} at 42 33
  \pinlabel {$x_2^L$} at 42 11
  \pinlabel {$x_3^L$} at 42 -1
\endlabellist
\includegraphics[scale=1.2]{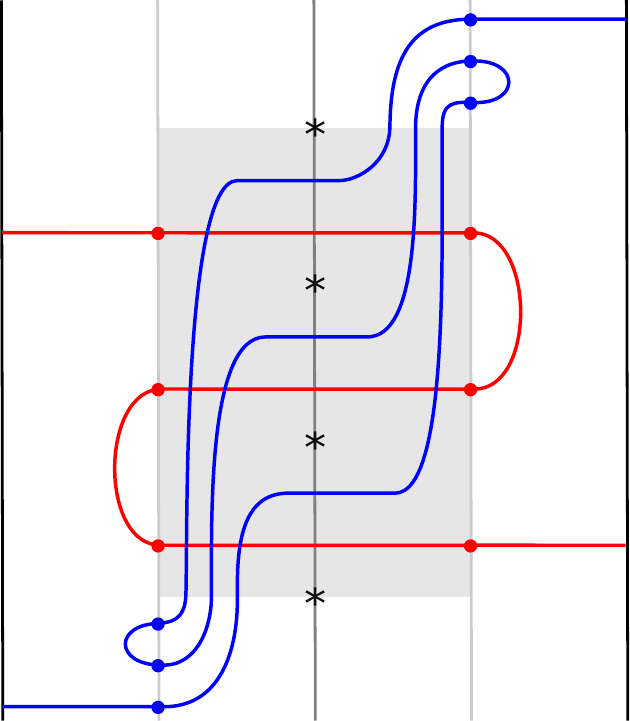}
\hspace{2 cm}
\includegraphics[scale=1.2]{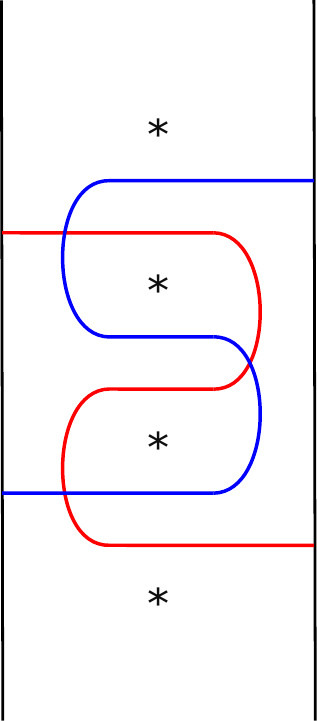}
\vspace{3 mm}
\caption{The curves $\tracks_1[0]$ (blue) and $\tracks_2$ (red), where the complexes $C_1$ and $C_2$ are the knot Floer homology of the left handed and right handed trefoil, respectively. The Floer homology of these curves computes the homology of $\Mor_{\sRminus}(C_1, C_2)_0$. On the left the curves are in the position described in the proof of Proposition \ref{prop:pairing-in-strip}, and on the right they are in minimal position.}
\label{fig:morphism-complex}
\end{figure}
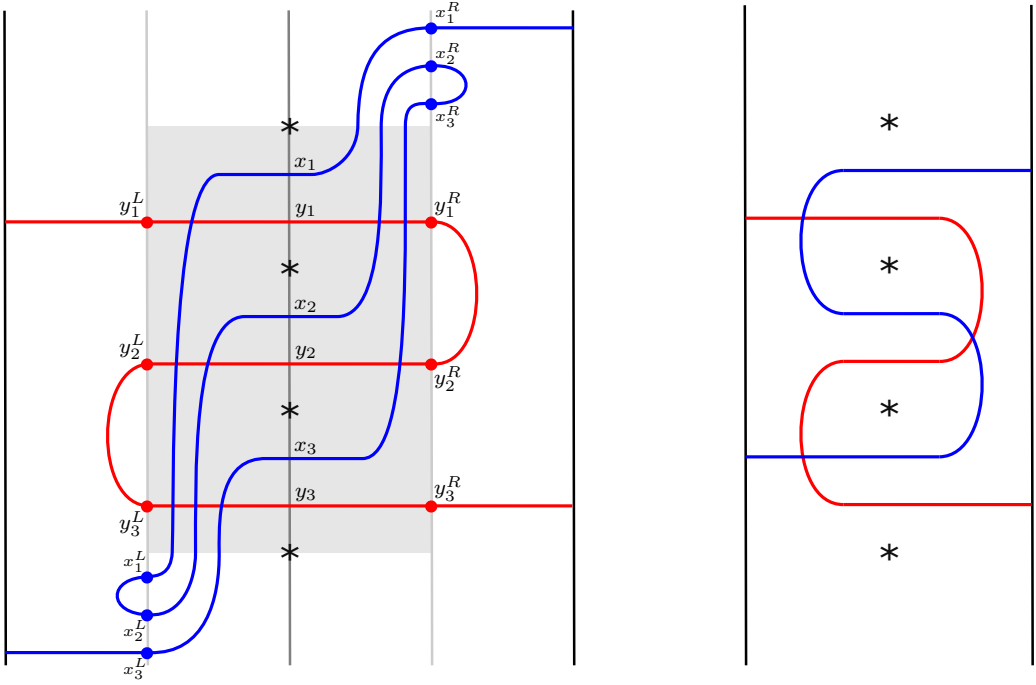

With the curves in the position described above, all intersection points lie in the rectangle $[-\tfrac 1 4, \tfrac 1 4]\times [h_{min},h_{max}]$ (the shaded rectangle in Figure \ref{fig:morphism-complex}). Each perturbed horizontal segment in $\tracks_1[s]$ (which now run from the bottom edge of the rectangle to the top edge) intersects each horizontal segment of $\tracks_2$ once, so generators of $CF( \tracks_1[s], \tracks_2 )$ are in bijection with pairs $(x_i, y_j)$ for generators $x_i$ of $C_1$ and $y_j$ of $C_2$; we will refer to intersection points by these ordered pairs throughout the proof. These in turn are in bijection with the generators of $\Mor_{\sRminus}(C_1, C_2)_s$, where the pair $(x_i, y_j)$ corresponds to the morphism $(x_i \to y_j)$ multiplied by the appropriate power of either $U$ or $V$ to give a morphism of Alexander grading $s$. The multiplier needed is $U^{A(y_j) - A(x_i) - s}$ if $A(y_j) - A(x_i) \ge s$ or $V^{s + A(x_i) - A(y_j)}$ if $A(y_j) - A(x_i) < s$. There is a convenient graphical interpretation of this multiplier: for each pair $(x_i, y_j)$ there is a triangle formed by the perturbed horizontal segment corresponding to $x_i$, the horizontal segment corresponding to $y_j$, and the vertical line $\mu$, and if we place $z$ and $w$ basepoints just to the left and right of each marked point as usual then the multiplier has one $U$ for each $w$ basepoint covered by the triangle or one $V$ for each $z$ basepoint covered. This identifies the generators of $CF( \tracks_1[s], \tracks_2 )$ and $\Mor_{\sRminus}(C_1, C_2)_s$, and it is a straightforward exercise to check that the bigradings agree.

We now identify the differential on $\Mor_{\sRminus}(C_1, C_2)_s$ with the Floer differential. The differential on $\Mor_{\sRminus}(C_1, C_2)$ comes from combining two contributions: (1) for each generator $x_i$ of $C_1$ and each arrow from $y_j$ to $cU^aV^b y_{j'}$ in $C_2$, there is an arrow from $(x_i \!\to\! y_j)$ to $cU^aV^b(x_i\!\to\! y_{j'})$, and (2) for each generator $y_j$ of $C_2$ and each arrow from $x_i$ to $cU^aV^b x_{i'}$ in $C_1$, there is an arrow from $(x_{i'}\!\to\! y_j)$ to $(-1)^\star cU^a V^b (x_i\!\to\! y_j)$ where $\star = 1 + \gr(x_i \to y_j)$. 
Consider a contribution of the first type. The arrow from $y_j$ to $c U^a V^b y_{j'}$ in $C_2$ corresponds to a bigon bounded by $\tracks_2$ and $\mu$ from $y_j$ to $y_{j'}$ covering the $w$ and $z$ basepoints $a$ and $b$ times, respectively. We will first assume that $y_j$ is above $y_{j'}$, so that this bigon begins to the left of $\mu$; note that in this case $b = a + A(y_j) - A(y_{j'}) \ge a$. Removing the immersed rectangle with corners $y_j$, $y_{j'}$, $y_{j'}^L$, and $y_j^L$ clearly gives rise to a bigon bounded by $\tracks_2$ and the line $x = -\tfrac 1 4$ from $y_j^L$ to $y_{j'}^L$. This new bigon covers both the $w$ and $z$ basepoints $a$ times, since the rectangle contains $A(y_j) - A(y_{j'}) = b - a$ copies of the $z$ basepoint. For each generator $x_i$ of $C_1$, the perturbed horizontal segment corresponding to $x_i$ crosses both the horizontal segments corresponding to $y_j$ and $y_{j'}$, forming a rectangle with these segments and the line $x = -\tfrac 1 4$; this rectangle covers both the $z$ and $w$ basepoints $\max(A(y_j), A(x_i)+s) -  \max(A(y_{j'}, A(x_i) + s)$ times. Adding this rectangle to the bigon described above gives a bigon from the intersection point $(x_i, y_j)$ to the intersection point $(x_i, y_{j'})$. This bigon covers the marked point $a$ times if $A(x_i) + s \ge A(y_j)$, $a + A(y_j) - A(x_i) - s$ times if $A(y_j) > A(x_i) + s \ge A(y_{j'})$, and $b$ times if $A(y_{j'}) > A(x_i) + s$.

If $A(x_i) + s \ge A(y_j)$, the contribution to the differential on $\Mor_{\sR}(C_1,C_2)_s$ is the arrow
$$V^{A(x_i) + s - A(y_j)} (x_i \!\to\! y_j) \to V^{A(x_i) + s - A(y_j)} [c U^a V^b  (x_i \!\to\! y_{j'})] = c (UV)^a V^{A(x_i) + s - A(y_{j'})} (x_i \!\to\! y_{j'}).$$
Since in this case the intersection points $(x_i, y_j)$ and $(x_i, y_{j'})$ correspond to the morphisms $V^{A(x_i) + s - A(y_j)} (x_i \!\to\! y_j)$ and $V^{A(x_i) + s - A(y_{j'})} (x_i \!\to\! y_{j'})$, respectively, this corresponds to an arrow
$$ (x_i, y_j) \to c (UV)^a (x_i, y_{j'})$$
in the differential of $CF( \tracks_2, \tracks_1[s] )$, which is precisely the arrow given by the bigon constructed above. If instead $A(y_j) > A(x_i) + s \ge A(y_{j'})$, the contribution to the differential on $\Mor_{\sR}(C_1,C_2)_s$ is the arrow
\begin{align*}
U^{A(y_j) - A(x_i) - s} (x_i \!\to\! y_j) &\to U^{A(y_j) - A(x_i) - s} [c U^a V^b  (x_i \!\to\! y_{j'})] \\
& = c (UV)^{a + A(y_j) - A(x_i) - s} V^{A(x_i) + s - A(y_{j'})} (x_i \!\to\! y_{j'}).
\end{align*}
In this case the intersection points $(x_i, y_j)$ and $(x_i, y_{j'})$ correspond to the morphisms $U^{A(y_j) - A(x_i) - s} (x_i \!\to\! y_j)$ and $V^{A(x_i) + s - A(y_{j'})} (x_i \!\to\! y_{j'})$, respectively, and the bigon constructed above contributes the corresponding arrow
$$ (x_i, y_j) \to c (UV)^{a + A(y_j) - A(x_i) - s} (x_i, y_{j'}) $$
to the differential. Finally, if $A(y_{j'}) > A(x_i) + s$ then the contribution to the differential on $\Mor_{\sR}(C_1,C_2)_s$ is the arrow
$$U^{A(y_j) - A(x_i) - s} (x_i \!\to\! y_j) \to U^{A(y_j) - A(x_i) - s} [c U^a V^b  (x_i \!\to\! y_{j'})] = c (UV)^b U^{A(y_{j'}) - A(x_i) - s} (x_i \!\to\! y_{j'}).$$
In this case the intersection points correspond to the morphisms $U^{A(y_j) - A(x_i) - s} (x_i \!\to\! y_j)$ and $U^{A(y_{j'}) - A(x_i) - s} (x_i \!\to\! y_{j'})$, respectively, and the bigon constructed above contributes the corresponding arrow
$$ (x_i, y_j) \to c (UV)^b (x_i, y_{j'}) .$$

The case that $y_j$ is below $y_{j'}$ is similar, except that the bigon representing the arrow from $y_j$ to $cU^a V^b$ lies to the right of $\mu$ near its $\mu$ boundary. Removing an appropriate rectangle gives a bigon bounded by $\tracks_2$ and the line $x = \tfrac 1 4$ from $y_j^R$ to $y_{j'}^R$, which covers the marked point $b$ times; adding back on a different rectangle gives a bigon connecting $(x_i, y_j)$ to $(x_i, y_{j'})$. To count the multiplicity with which this bigon covers the marked point, we can again consider cases depending on whether $A(x_i) + s$ is above $A(y_{j'})$, below $A(y_j)$, or in between them. In each case it is straightforward to check that the contribution of the bigon to the differential of $CF( \tracks_2, \tracks_1[s] )$ exactly matches the contribution to the differential of $\Mor_{\sRminus}(C_1, C_2)_s$.

Contributions to the differential of $\Mor_{\sRminus}(C_1, C_2)_s$ of the second type, coming from a generator $y_j$ of $C_2$ and an arrow from $x_i$ to $c U^a V^b x_{i'}$ in $C_1$, can be dealt with similarly. We will only describe the case that $x_i$ is above $x_{i'}$ as generators of $C_1$ and that $A(y_i) \le A(x_{i'}) + s$, leaving the remaining (similar) cases to the reader. Note that $b = a + A(x_i) - A(x_{i'}) \ge a$. The arrow from $x_i$ to $c U^a V^b x_{i'}$ corresponds to a bigon between $\tracks_1[s]$ and $\mu$ which lies locally to the left of $\mu$ near its $\mu$ boundary. Removing an appropriate rectangle gives a bigon bounded by $\tracks_1[s]$ and the line $x = -\tfrac 1 4$ from $x_i^L$ to $x_{i'}^L$, which covers the marked point $a$ times. Given a generator $y_j$, we can add a rectangle to produce a bigon connecting $(x_i, y_j)$ and $(x_{i'}, y_j)$. Note that in $CF(\tracks_2, \tracks_1[s] )$ we count bigons with the right boundary on $\tracks_2$, so this bigon connects $(x_{i'}, y_j)$ to $(x_i, y_j)$. Since we have assumed $A(y_i) < A(x_{i'}) + s$, the rectangle we added contained no marked points and so this bigon contributes
$$(x_{i'}, y_j) \to (-1)^\star (UV)^a (x_i, y_j)$$
to the differential. The sign term comes from the fact that the sign of the bigon depends on the orientation of the $\tracks_2$ portion of the boundary while the sign of the arrow in $C_1$ depends on the orientation of the $\tracks_1[s]$ portion of boundary, and these agree precisely when the intersection point $(x_i, y_j)$ has odd grading. Since $(x_i, y_j)$ and $(x_{i'}, y_j)$ correspond to the morphisms $V^{A(x_i) + s - A(y_j)} (x_i \!\to\! y_j)$ and $V^{A(x_i) + s - A(y_{j'})} (x_i \!\to\! y_{j'})$, respectively, this matches the contribution of the arrow
$$V^{A(x_{i'}) + s - A(y_j)} (x_i \!\to\! y_j) \to V^{A(x_{i'}) + s - A(y_j)} [c (-1)^\star U^a V^b  (x_i \!\to\! y_j)] $$
to the differential on $\Mor_{\sRminus}(C_1, C_2)$.

We have shown that for each term in the differential on $\Mor_{\sRminus}(C_1, C_2)_s$, there is a bigon producing the analogous term in the differential of $CF(\tracks_2,  \tracks_1[s] )$; it only remains to show that there are no other bigons contributing to the differential of $CF( \tracks_2, \tracks_1[s] )$. Suppose there is a bigon $B$ whose initial corner is the intersection point $(x_i, y_j)$. Starting from this initial point, we will follow the $\tracks_2$ part of $\partial B$, which initially follows the horizontal segment corresponding to $y_j$. Suppose first that we are moving leftward along this horizontal segment. One of two things can happen: the $\tracks_2$ part of the $\partial B$ can leave the strip or it can reach the final corner of the bigon within the strip. If it leaves the strip then it will return to the strip again without interacting with $\tracks_1[s]$ and then cross the perturbed segment of $\tracks_1[s]$ corresponding to $y_j$, which necessarily completes a bigon of the first type described above, and if it reaches the terminal corner of $B$ without leaving the strip then $B$ is clearly of the second type described above.\end{proof}

\begin{remark}
A version of Proposition \ref{prop:pairing-in-strip} in the $UV=0$ setting follows from the general pairing theorems for immersed curves representing type D structures, in particular Theorem 1.5 of \cite{KWZ}, and the proof here is similar even when extending to the minus setting. In particular, this pairing is considered for type D structures over $\sRhat$ in Theorem 3 of \cite{KWZ:mnemonic}, where the knot Floer complex associated with a connected sum, which comes from the morphism space of two complexes, is identified with the wrapped Floer homology in the doubly marked disk of corresponding immersed curves. Recall that, as mentioned in Section \ref{sec:intro-bordered}, the doubly punctured disk is analogous to the punctured cylinder (by switching punctures with boundary components), and the punctured infinite strip $\strip^*$ is a covering space of this.
\end{remark}

\subsection{Floer homology in $\cylinder$ as a mapping cone; an unshifted pairing}\label{sec:unshifted-pairing}
Consider two bigraded complexes $C_1$ and $C_2$ equipped a flip isomorphisms $\Psi_{1,*}$ and $\Psi_{2,*}$. Each set of data can be represented by a decorated immersed multicurves in the cylinder $\cylinder$, and we can consider the Floer homology of these two curves. We will now define an algebraic pairing on the two sets of data and show that this agrees with the Floer homology of the corresponding curves.

To define the algebraic pairing, we consider several maps between morphism spaces. In particular, consider the $\F[W]$-complexes 
$$A_s = \Mor_{\sRminus}(C_1, C_2)|_{A=s},$$
$$B^v =  \Mor_{\F[W]}( H_*C_1^v, H_*C_2^v), \text{ and }$$
$$B^h =  \Mor_{\F[W]}( H_*C_1^h, H_*C_2^h). $$
Just as setting $V = 1$ and $U=W$ gives inclusion maps $C|_{A=s} \into C^v$, we also get inclusions $A_s \into \Mor_{\F[W]}(C^v_1, C_2^v)$. Recall that the generators of $A_s = \Mor_{\sRminus}(C_1, C_2)|_{A=s}$ as an $\F[W]$-module are $(x_i\to y_j)$ for generators $x_i$ of $C_1$ and $y_j$ of $C_2$ multiplied by either $U^{A(y_j)-A(x_i)-s}$ if $A(y_j)-A(x_i) \ge s$ or by $V^{s+A(x_i) - A(y_j)}$ if $A(y_j)-A(x_i) \le s$, while the generators of $\Mor_{\F[W]}(C^v_1, C_2^v)$ are simply $(x_i\to y_j)$ for generators $x_i$ of $C_1$ and $y_j$ of $C_2$. It follows that the inclusion map $A_s \into \Mor_{\F[W]}(C^v_1, C_2^v)$ is given by multiplying each generator by $W^a$ where $a = \max({A(y_j)-A(x_i)-s}, 0)$. The homology functor induces a map from $\Mor_{\F[W]}( C_1^v, C_2^v )$ to $B^v$. We define the composition of these two maps to be $v_s: A_s \to B^v$. Similarly, setting $U=1$ and $V=W$ gives an inclusion map $A_s \into \Mor_{\F[W]}(C^h_0, C_2^h)$, which after taking homology gives a map $h_s: A_s \to B^h$. The flip isomorphisms  $\Psi_{1,*}$ and $\Psi_{2,*}$ induce a map $F_{\Psi_{1,*}, \Psi_{2,*}}:  B^h \to  B^v$
taking a morphism $f$ to $\Psi_{2,*} \circ f \circ (\Psi_{1,*})^{-1}$. We will define $h^\Psi_s: A_s \to B^v$ to be the composition $F_{\Psi_{1,*}, \Psi_{2,*}} \circ h_s$.

We now consider the map $D = v_0 + h^\Psi_0$ from $A_0$ to $B^v$, and we define $\mathbb{X} = \mathbb{X}(C_1, \Psi_{1,*}, C_2, \Psi_{2,*})$ to be the mapping cone of $D$. The homology of $\mathbb{X}$ is a graded differential module over $\F[W]$; this is what we take to be the algebraic pairing of $(C_1, \Psi_{1_*})$ with $(C_2, \Psi_{2,*})$.

\begin{proposition}\label{prop:unshifted-pairing}
For $i \in \{1,2\}$, let $\tracks_i = (\Gamma_i, \bchain_i)$ be a decorated curve in the marked cylinder $\cylinder$ representing the complex $C_i$ and the flip isomorphism $\Psi_{i,*}$, and let $\mathbb{X}$ be the complex defined above. The Floer complex $CF( \tracks_2, \tracks_1)$ is quasi-isomorphic to $\mathbb{X}$ as graded complexes over $\F[W]$.
\end{proposition}
\begin{proof}
Recall that the complex $\mathbb{X} = \text{Cone}(D)$ can be realized as $A_0 \otimes B^v$ with differential
$$\left( \begin{array}{cc}
\partial_{A_0} & 0 \\
\widetilde{D} & \partial_{B^v} \end{array} \right),$$
where throughout this proof for a map $f$ we use $\widetilde f$ to denote the function defined by
$$\widetilde{f}(x) = (-1)^{\gr_w(x)} f(x).$$
It is clear that the complex $\mathbb{X}$ is also quasi-isomorphic to the complex
$$ \begin{array}{ccc}
A_0 & \overset{\widetilde{h}_0}{\longrightarrow} & B^h \\
\hspace{3mm} \downarrow {}_{ \widetilde{v}_0} & &  \hspace{3mm} \uparrow {}_{-\widetilde{\Id}}  \\
B^v & \overset{-\widetilde{F}_{\Psi_{1,*}, \Psi_{2,*}}}{\longleftarrow} & B^h 
\end{array}$$
where $\Id$ is the identity map on $B^h$. We will perturb the curves $\tracks_1$ and $\tracks_2$ so that the Floer complex $CF(\tracks_2, \tracks_1)$ is isomorphic to this complex. An example is shown in Figure \ref{fig:simple-pairing-cylinder}. We divide the marked cylinder $\cylinder$ (identified with $(\R/\Z)\times\R$) into a marked strip $\strip_0 = [-\tfrac 1 8, \tfrac 1 8]\times \R$ and three unmarked strips $\sF_a = [a-\tfrac{1}{8}, a+\tfrac{1}{8}]\times \R$ for $a \in \{-\tfrac 1 4, \tfrac 1 4, \tfrac 1 2\}$, and we let $\sF$ denote $\sF_{\tfrac 1 4}\cup\sF_{\tfrac 1 2}\cup\sF_{-\tfrac 1 4}$. We first homotope each curve $\tracks_i$ as in the proof of Proposition \ref{prop:pairing-in-strip} so that the curves restricted to $\strip_0$ are in almost simple position and represent the complex $C_i$ and the curves restricted to $\sF$ are a collection of arcs from one side of the strip to the other, possibly with left-turn crossover arrows between arcs oriented the same direction, representing the flip isomorphism $\Psi_{i,*}$. We perturb the curves in $\strip_0$ as in the proof of Proposition  \ref{prop:pairing-in-strip}, so that the Floer complex of the curves restricted to $\strip_0$ agrees with the complex $A_0$ exactly. Note that on the right the endpoints of $\tracks_1$ are above the endpoints of $\tracks_2$, while the opposite is true on the left. We perturb each curve in $\sF$ so that $\tracks_1$ is below $\tracks_2$ on $\partial_R \sF_{\tfrac 1 4} = \partial_L \sF_{\tfrac 1 2}$ and above $\tracks_2$ on $\partial_R \sF_{\tfrac 1 2} = \partial_L \sF_{-\tfrac 1 4}$, and so that all crossings and crossover arrows between arcs in $\tracks_1$ or between arcs in $\tracks_2$ occur in $\sF_{\tfrac 1 2}$ to the right of all crossings between $\tracks_1$ and $\tracks_2$.

The restriction of the Floer complex to generators from $\sF_{\tfrac 1 2}$ can be identified with $B^h$, where each arc of $\tracks_i$ in $\sF_{\tfrac 1 2}$ corresponds to a generator of $H_* C^h_i$, and the intersection of arcs corresponding to generators $x$ and $y$ can be identified with the morphism $(x \to y)$.  The only bigons connecting two intersection points in $\sF_{\tfrac 1 2}$ extend through $\sF_{\tfrac 1 4}$ into $S_0$ on either $\tracks_1$ or $\tracks_2$ and correspond to a term in the differential of either $H_*C^h_1$ or $H_* C^h_2$ along with a fixed generator in the other complex. These bigons precisely recover the differential on $B^h$, so that the restriction of the Floer complex $CF(\tracks_2, \tracks_1)$ to generators in $\sF_{\tfrac 1 2}$ is the complex $B^h$. The same is true for the restriction of the Floer complex to generators in $\sF_{\tfrac 1 2}$, but we need to introduce signs in the identification between intersection points and generators of $B^h$ and a grading shift. The bigons connecting intersection points in $\sF_{\tfrac 1 4}$ are clearly the same as those connecting intersection points in $\sF_{\tfrac 1 2}$ after removing a small rectangular strip running between $\sF_{\tfrac 1 4}$ and $\sF_{\tfrac 1 2}$, but the sign of the bigons that extend into $S_0$ on the $\tracks_1$ side is flipped. To correct for this, we identify the intersection point of segments corresponding to $x$ and $y$ with $(-1)^{\gr_w(x)}$ times the morphism $(x\to y)$ and observe that once again the bigons recover the differential on $B^h$. We also note that there is a grading difference of 1 (for either grading) between intersection points in $\sF_{\tfrac 1 2}$ and the corresponding intersection points in $\sF_{\tfrac 1 4}$, so the Floer complex restricted to generators in $\sF_{\tfrac 1 4}$ is the complex $B^h[-1]$. Similarly, the arcs in $\tracks_i$ through $\sF_{-\tfrac 1 4}$ correspond to generators of $H_* C^v_i$ and the intersections in $\sF_{-\tfrac 1 4}$ can be identified up to sign with the generators of $B^v$. If we identify the intersection point between segments corresponding to $x$ in  $H_* C^v_1$ and $y$ in $H_* C^v_2$ with $(-1)^{\gr_w(x)+1}$ times the morphism $(x\to y)$, we observe that the Floer complex restricted to generators from $\sF_{-\tfrac 1 4}$ agrees with $B^v[-1]$.

The only bigons connecting intersection points from different strips connect intersection points in adjacent strips; more specifically they connect points in $\strip_0$ or $\sF_{\tfrac 1 2}$ to points in $\sF_{\tfrac 1 4}$ or $\sF_{-\tfrac 1 4}$. Counting bigons from $\sF_{\tfrac 1 2}$ to $\sF_{\tfrac 1 4}$ realizes the degree $-1$ map $-\widetilde{\Id}[-1]$ from $B^h$ to $B^h[-1]$. It is clear that there is exactly one bigon for each generator of $B^h$. To check the signs, note that the bigon corresponding to the pair $(x,y)$ contributes to the Floer complex with a minus sign if and only if $\gr_w(y)$ is even, and the intersection point on the $\sF_{\tfrac 1 4}$ end of the boundary represents the opposite of the generator $(x \to y)$ of $B^h[-1]$ if and only if $\gr_w(x)$ is odd; it follows that the map takes a generator $B^h$ to the corresponding generator of $B^h[-1]$ with a minus sign if and only if $\gr_w( (x\to y) )$ is even. By shifting the degree up by one, the map $-\widetilde\Id[-1]: B^h \to B^h[-1]$ is equivalent to the degree zero map $-\widetilde\Id: B^h \to B^h$.

We next observe that counting the bigons from $\strip_0$ to $\sF_{\tfrac 1 4}$ defines the map $\widetilde{h}_0[-1]$ from $A_0$ to $B^h[-1]$. Note that generators $(x, y)$ map to zero if either $x$ or $y$ is the end of a horizontal arrow in the corresponding complex, and if $x$ and $y$ survive in horizontal homology the generator $(x, y)$ maps to itself multiplied by $W^a$ where $a = \max(A(x)-A(y), 0)$. The corresponding bigons are counted with a minus sign if and only if $y$ has odd grading. Combining this sign with the minus sign on generators of $B^h[-1]$ for which $\gr(x)$ is odd gives a minus sign precisely when $(x \to y)$ has odd grading. Similarly, counting the bigons from $\strip_0$ to $\sF_{-\tfrac 1 4}$ realize the map $\widetilde{v}_0[-1]$  form $A_0$ to $B^v[-1]$. Finally, we check that counting bigons from $\sF_{\tfrac 1 2}$ to $\sF_{-\tfrac 1 4}$ realizes the degree $-1$ map $\widetilde{F}_{\Psi_{1,*}, \Psi_{2,*}}[-1]$ from $B^h$ to $B^v[-1]$. The identification is obvious up to sign; to check the signs, note that a bigon starting at $(x \to y)$ contributes with sign $(-1)^{\gr_w(y)}$ and the identification between the terminal intersection point of the bigon and a genertor of $B^v[-1]$ contributes the sign $(-1)^{\gr_w(x)+1}$.

Putting all these observations together, and shifting the grading of $B^h[-1]$ and $B^v[-1]$, we see that the Floer complex can be identified with the complex given at the beginning of the proof, which is quasi-isomorphic to $\mathbb{X}$.
\end{proof}

\begin{figure}
\includegraphics[scale = 1]{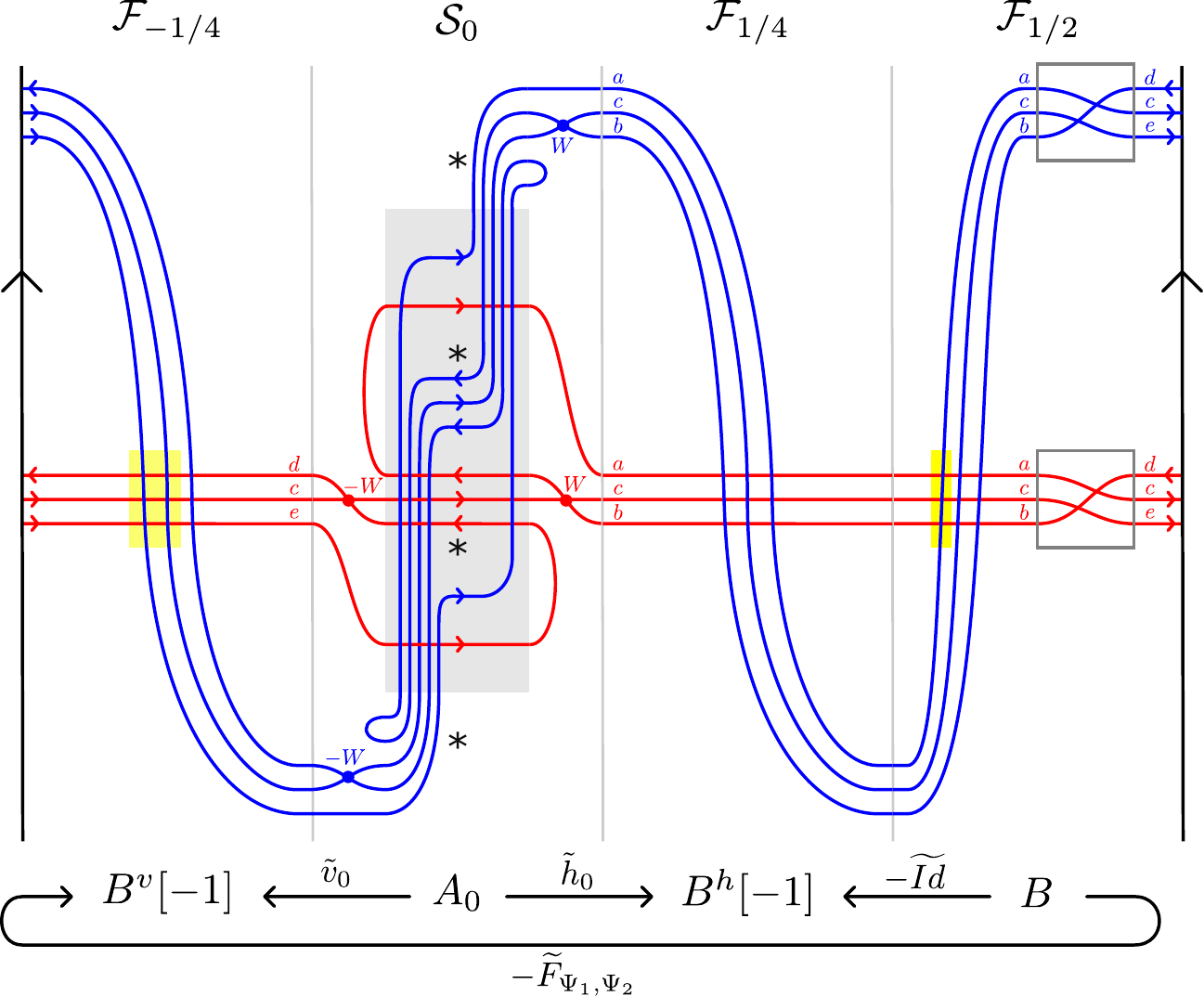}
\caption{The simple pairing of the curve invariant for $+1$-surgery on the left handed trefoil from Example \ref{ex:1-surgery-on-LHT} with itself, in the form described in he proof of Proposition \ref{prop:unshifted-pairing}. The highlighted intersection points have a minus sign when identified with a generator of the corresponding complex.}
\label{fig:simple-pairing-cylinder}
\end{figure}

An example of the unshifted pairing is shown in Figure \ref{fig:simple-pairing-cylinder}, where $(C_1, \Psi_{1,*})$ and $(C_2, \Psi_{2,*})$ are both the complex and flip isomorphism associated with the dual knot in $+1$-surgery on the left handed trefoil from Example \ref{ex:1-surgery-on-LHT}. The curves have been perturbed as in the proof of Proposition \ref{prop:unshifted-pairing} so that the Floer complex is identified with the complex quasi-isomorphic to $\mathbb{X}$ given in the proof.

We remark that in the proof of Proposition \ref{prop:unshifted-pairing}, it was not necessary to divide the unmarked strip $\sF$ into three strips and perturb $\tracks_1$ to intersect $\tracks_2$ in each of these. We could have instead assumed the collection of $\tracks_1$ arcs crossed the collection of $\tracks_2$ arcs once in $\sF$, corresponding to the intersection points in $\sF_{-\tfrac 1 4}$, and shown that counting the bigons moving rightward from $\strip$ to $\sF$ recovers the map $\widetilde{h}^\Psi_0$. Thus we chose to performed a finger move to add additional intersections and construct a larger chain complex. We feel this makes the argument more clear, since we can consider the maps $\widetilde{h}_0$ and $\widetilde{F}_{\Psi_{1,*},\Psi_{2,*}}$ separately, but we will adopt the simpler configuration in similar proofs moving forward.

\subsection{Uniqueness of curves}\label{sec:uniqueness}

The algebraicly defined pairing introduced above only depends on the chain complexes of $C_1$ and $C_2$ up to bigraded chain homotopy equivalence and the flip isomorphisms up to isomorphism. It follows from Proposition \ref{prop:unshifted-pairing} that the Floer homology of the corresponding decorated curves only depends on this data. In particular, any two decorated curves representing homotopy equivalent complexes equipped with equivalent flip isomorphism are indistinguishable in the context of Floer homology in $\cylinder$, and can thus be considered as equivalent objects of the Fukaya category of $\cylinder$.

\begin{definition}
Two objects in the Fukaya category of the marked cylinder $\cylinder$ are said to be \emph{equivalent} if they have the same Floer homology with any other object in the Fukaya category.
\end{definition}

We have thus shown the following:

\begin{proposition}\label{prop:dumb-uniqueness}
Any chain homotopy equivalence class of bigraded complexes over $\sRminus$ and any flip isomorphism from the horizontal homology of any of these complexes to the vertical homology of any of these complexes is represented by a \emph{unique} decorated curve $(\Gamma, \bchain)$ in $\cylinder$ (up to equivalence as objects in the Fukaya category of $\cylinder$).
\end{proposition}

While uniqueness up to equivalence in the Fukaya category is the right notion abstractly, in practice it can be unsatisfying since it can be difficult to check if two decorated curves are equivalent objects in the Fukaya category. It has already been noted that a complex and flip map can have many naive immersed curve representatives that may bear little resemblance to the simplified constructed in the previous sections, even though they are equivalent by Proposition \ref{prop:dumb-uniqueness}. Fortunately, for representatives in simple position with bounding chains of local system type, we can give a stronger uniqueness statement about the curves $\Gamma$ and the decoration $\bchainhat$ obtained from $\bchain$ by restricting to degree zero intersection points.

\begin{proposition}\label{prop:uniqueness}
Suppose $(\Gamma_1,\bchain_1)$ and $(\Gamma_2, \bchain_2)$ are equivalent objects in the Fukaya category of the marked cylinder $\cylinder$, and suppose the restriction $\bchainhat_i$ of $\bchain_i$ to degree zero intersection points is of local system type for $i \in \{1,2\}$. Then $\Gamma_1$ and $\Gamma_2$ are homotopic and $\bchainhat_1$ and $\bchainhat_2$ agree (under the natural identification between local system intersection points of $\Gamma_1$ and $\Gamma_2$).
\end{proposition}

\begin{proof}
This property, which holds more generally for compact objects in the Fukaya category of any marked surface, is an essential part of the uniqueness proof for the structure theorems for type D structures appearing in \cite{HRW} and other generalizations of that work. For the sake of keeping the present paper self-contained we briefly summarize the argument. The approach described here, which is simpler than the uniqueness proof in \cite{HRW}, is modeled on the proof of \cite[Proposition 4.46]{Zibrowius:peculiar-modules}. 

We need to show that if $\Gamma_1$ and $\Gamma_2$ are not homotopic or if $\bchainhat_1$ and $\bchainhat_2$ do not agree, then $(\Gamma_1, \bchainhat_1)$ and $(\Gamma_2, \bchainhat_2)$ are not equivalent as elements of the Fukaya category. This means that there is some test curve $(\Gamma_3, \bchainhat_3)$ that pairs differently with $(\Gamma_1, \bchainhat_1)$ and $(\Gamma_2, \bchainhat_2)$. It is sufficient to consider the hat version of pairing, that is Floer homology in the punctured cylinder $\pcylinder$. The key observation is that the dimension of the Floer homology of two connected decorated curves $(\gamma_1, \bchainhat_{\gamma_1})$ and $(\gamma_2, \bchainhat_{\gamma_2})$ is given by the minimal intersection number of $\gamma_1$ and $\gamma_2$ (since there are no bigons when the curves are in minimal position), and thus does not depend on the decorations, unless the curves are parallel, where by parallel we mean homotopic to multiples of the same primitive curve.  If the curves are parallel then admissibility forces us to perturb the curves from minimal position. In this case, assuming without loss of generality that the orientations agree, we can check that the dimension of Floer homology differs from the minimal intersection number by
$$2 k_1 k_2 - \dim( \ker( A_1 \otimes A_2^{-1} - \Id))$$
where $k_i$ is the multiplicity of the possibly non-primitive curve $\gamma_i$, $A_i$ is the $k_i$-dimensional local system determined by the decoration $\bchainhat_{\gamma_i}$, and $\Id$ is the identity map on $\F^{k_1 k_2}$.

If $\Gamma_1$ contains a component $\gamma$ that is not parallel to any component of $\Gamma_2$, we consider test curves with $\Gamma_3$ a mulitple of $\gamma$ and consider different decorations $\bchainhat_3$. As $\bchainhat_3$ varies the pairing of $(\Gamma_3, \bchainhat_3)$ with $(\Gamma_1, \bchainhat_1)$ will change but the pairing with $(\Gamma_2, \bchainhat_2)$ will not, so for some choice of $\bchainhat_3$ the test pairing distinguishes $(\Gamma_1, \bchainhat_1)$ from $(\Gamma_2, \bchainhat_2)$. A similar argument applies if there is some primitive curve $\gamma$ for which the collection of components of $\Gamma_1$ parallel to $\gamma$ is not homotopic to the collection of components of $\Gamma_2$ parallel to $\gamma$, or the corresponding decorations do not agree. Let $k_i$ be the total multiplicity of all components of $\Gamma_i$ parallel to $\gamma$, and let $A_i$ be the dimension $k_i$ local system determined by the decorations on these components. Note that the components of $\Gamma_i$ parallel to $\gamma$, and their decorations, are uniquely determined by the isomorphism type of $A_i$ (we can construct the decorated curves from a matrix in rational canonical form representing $A_i$). Thus if these collections are not equivalent we must have that $A_1$ is not isomorphic to $A_2$. In this case we consider the test curve $\Gamma_3 = \gamma$ and once again vary the decoration $\bchainhat_3$. The pairing with $(\Gamma_i, \bchainhat_i)$ depends on
$$2 k_i k_3 - \dim( \ker( A_i \otimes A_3^{-1} - \Id)),$$
where here $A_3$ is the local system of dimension $k_3$ determined by $\bchainhat_3$. A linear algebra exercise shows that if $A_1$ is not isomorphic to $A_2$ then there is some $A_3$ for which this quantity differs between $i=1$ and $i=2$, so $(\Gamma_1, \bchainhat_1)$ and $(\Gamma_2, \bchainhat_2)$ are distinguished by pairing.
\end{proof}

The proof above relies on the immersed curves having only closed components, since immersed arcs have no local system intersection points. However, we can apply this result to get uniqueness of curves in $\strip$ representing complexes by constructing closed curves in $\cylinder$ from these. This is similar to the doubling argument used to show uniqueness of non-compact immersed curves in \cite[Theorem 5.27]{KWZ}

\begin{proposition}\label{prop:uniqueness-strip}
Suppose $(\Gamma_1,\bchain_1)$ and $(\Gamma_2, \bchain_2)$ are equivalent objects in the Fukaya category of the marked strip $\strip$, and suppose $\Gamma_i$ is in almost simple position and the restriction $\bchainhat_i$ of $\bchain_i$ to degree zero intersection points is of local system type for $i \in \{1,2\}$. Then $\Gamma_1$ and $\Gamma_2$ are homotopic and $\bchainhat_1$ and $\bchainhat_2$ agree (under the natural identification between local system intersection points of $\Gamma_1$ and $\Gamma_2$).
\end{proposition}
\begin{proof}
For $i\in\{1,2\}$, let $(\Gamma^\dagger_i, \bchain^\dagger_i)$ denote the decorated immersed curve obtained by rotating $(\Gamma_i, \bchain_i)$ about the origin and interchanging the two gradings; note that the orientation on $\Gamma^\dagger_i$, which is determined by the parity of the gradings, is the opposite of the image of the orientation on $\Gamma_i$ under the rotation. For an integer $n$ let $(\Gamma_i, \bchain_i)[n]$ be the decorated immersed curve obtained from $(\Gamma_i, \bchain_i)$ by translating upward by $n$ and subtracting $2n$ from the grading function $\tilde\tau_w$, and let $(\Gamma^\dagger_i, \bchain^\dagger_i)[-n]$ denote the result of translating $(\Gamma^\dagger_i, \bchain^\dagger_i)$ down by $n$ and subtracting $2n$ from $\tilde\tau_z$. We now fix $n$ sufficiently large so that $(\Gamma_i, \bchain_i)[n]$ lies entirely above height zero and $(\Gamma^\dagger_i, \bchain^\dagger_i)[-n]$ lies entirely below height zero and we define $(\Gamma'_i, \bchain'_i)$ to be the decorated immersed curve in the cylinder $\cylinder$ (viewed as the union of a marked strip $\strip$ and an unmarked strip $\sF$) obtained from the union of the decorated curves $(\Gamma_i, \bchain_i)[n]$ and $(\Gamma^\dagger_i, \bchain^\dagger_i)[-n]$ in $\strip$ by adding arcs in $\sF$ connecting each right endpoint of $\Gamma_i[n]$ on $\partial_R \strip$ with the corresponding left endpoint of $\Gamma^\dagger_i[-n]$ on $\partial_L \strip$ and identifying each right endpoint of $\Gamma^\dagger_i[-n]$ on $\partial_R \strip$ with the corresponding left endpoint of $\Gamma_i[n]$ on $\partial_L \strip$.

 Let $C_i$, $C_i[n]$ and $C^\dagger_i[-n]$ denote the bigraded complexes represented by $(\Gamma_i, \bchain_i)$, $(\Gamma_i, \bchain_i)[n]$, and $(\Gamma^\dagger_i, \bchain^\dagger_i)[-n]$, respectively. Clearly $C_i[n]$ is obtained from $C_i$ by shifting the grading $\gr_w$ up by $2n$, and $C^\dagger_i[-n]$ is obtained from $C_i[n]$ by interchanging the two gradings, swapping the role of $U$ and $V$ and multiplying the differential $\partial$ by $-1$. Because the horizontal complex of $C^\dagger_i[-n]$ is isomorphic by construction to the vertical complex of $C_i[n]$ and vice versa, there is an obvious flip map $\Psi_i'$ on $C_i[n] \oplus C^\dagger_i[-n]$ that takes each each generator in $(C_i[n] \oplus C^\dagger_i[-n])^h$ to the corresponding generator of $(C_i[n] \oplus C^\dagger_i[-n])^v$. It is easy to see that $(\Gamma'_i, \bchain'_i)$ represents the pair $(C_i[n] \oplus C^\dagger_i[-n], \Psi_i')$.
 
Since $(\Gamma_1, \bchain_1)$ and $(\Gamma_2, \bchain_2)$ are equivalent objects, we have that $C_1$ is chain homotopic to $C_2$, and it is clear that the same is true for $C_1[n]$ and $C_2[n]$ and for $C_1^\dagger[-n]$ and $C_2^\dagger[-n]$, and that the pairs $(C_1[n] \oplus C^\dagger_1[-n], \Psi_1')$ and $(C_1[n] \oplus C^\dagger_1[-n], \Psi_1')$ are homotopy equivalent. It follows from Proposition \ref{prop:uniqueness} that $\Gamma'_1$ and $\Gamma'_2$ are homotopic and $\bchainhat'_1$ and $\bchainhat'_2$ agree. It follows that the union of decorated curves $(\Gamma_i, \bchain_i)[n] \cup (\Gamma^\dagger_i, \bchain^\dagger_i)[-n]$, which is the restriction to $\strip$ of $(\Gamma_i, \bchain_i)$ in $\cylinder = \strip \cup \sF$, is the same up to homotopy for $i = 1,2$. We can uniquely recover the decomposition since $(\Gamma_i, \bchain_i)[n]$ consists of precisely the components of $(\Gamma_i, \bchain_i)[n] \cup (\Gamma^\dagger_i, \bchain^\dagger_i)[-n]$ above height zero. Finally, by translating down by $n$ we see that $\Gamma_1$ is homotopic to $\Gamma_2$ and $\bchainhat_1$ agrees with $\bchainhat_2$.
\end{proof}

Propositions \ref{prop:uniqueness} and \ref{prop:uniqueness-strip}  lead to an obvious question: is a similar uniqueness statement encorporating the whole bounding chain possible? That is, is there some normal form for decorated curves $(\Gamma, \bchain)$ representing pairs $(C, \Psi)$ over $\sRminus$ such that every $(C, \Psi)$ is represented by a curve of this form and such that $\bchain$ for such a representative is unique as a subset of the self intersection points of $\Gamma$? We suspect that this is possible, but we do not undertake the task of proving it in the present paper. A possible strategy is to define an arrow sliding algorithm for the left turn crossover arrows appearing at negative degree self intersection points in $\bchain$ and slide arrows until they are either removed if possible or in some preferred position if they can not be removed. Unfortunately, as has been noted already, there are some technical difficulties defining general arrow sliding moves in the minus setting so that this approach requires more work; we hope to explore this in future work. In the meantime, we remark that in practice the uniqueness of the immersed curve is the most important part of the uniqueness result, since once an immersed curve $\Gamma$ is fixed there are finitely many possible collections of turning points $\bchain$. Usually a small number of these are valid bounding chains, and in practice it is not difficult to check when two different collections of turning points $\bchain$ on $\Gamma$ are equivalent and find a unique simplest representative (see for example Corollary \ref{cor:embedded-curve-with-fig8s}).

\subsection{A shifted pairing in $\cylinder$}\label{sec:shifted-pairing}

In the next section, we will need to consider more general ways of pairing complexes and their corresponding curves. Fixing complexes $C_1$ and $C_2$ equipped with flip isomorphisms $\Psi_{1,*}$ and $\Psi_{2,*}$, for each $p/q\in \Q$ and for each $i \in \Z/p\Z$ we will define an algebraic pairing by constructing a chain complex $\mathbb{X}_{i; p/q} = \mathbb{X}_{i; p/q}(C_1, \Psi_{1,*}, C_2, \Psi_{2,*})$ over $\F[W]$ and taking its homology. More precisely, we will define a finitely generated complex $\mathbb{X}^N_{i; p/q} = \mathbb{X}^N_{i; p/q}(C_1, \Psi_{1,*}, C_2, \Psi_{2,*})$ for each sufficiently large $N$, the homology of which does not depend on $N$; it is possible to define a single complex $\mathbb{X}_{i; p/q}$ without fixing an $N$, but this complex is infinitely generated and we prefer to avoid it for technical reasons.

The complex is the simplest when $p/q = 0$, and in this case the complex is independent of $N$. The complex $\mathbb{X}_{0; 0}$ is precisely the complex $\mathbb{X}$ defined in Section \ref{sec:unshifted-pairing}, and for any $s \in \Z$, the complex $\mathbb{X}_{s; 0}$ is defined the same way as $\mathbb{X}_{0;0}$ with $A_0$ replaced with $A_s$ and the maps $v_0$ and $h^\Psi_0$ replaced with $v_s$ and $h^\Psi_s$. For $p/q \neq 0$, we need to choose $N \ge g_1 + g_2$, where $g_i$ is any integer such that $|A(x)| \le g_i$ for all generators $x$ of $C_i$ (if $C_i$ is the knot Floer complex of a knot, we may take $g_i$ to be the genus of the knot). If $p/q >0$ we then define, for each $i$ in $\Z/p\Z$,
\begin{equation}\label{eq:AandB-truncated}
\mathbb{A}^N_{i;p/q} = \bigoplus_{n = n_{min}}^{n_{max}} A_{\left\lfloor \frac{i+np}{q} \right\rfloor} \qquad \text{ and } \qquad \mathbb{B}^N_{i;p/q} = \bigoplus_{n = n_{min}+1}^{n_{max}} B^v,
\end{equation} 
where $n_{min}$ is the smallest integer $n$ for which $\left\lfloor \frac{i+np}{q} \right\rfloor > -N$, and $n_{max}$ is the largest integer $n$ such that $\left\lfloor \frac{i+np}{q} \right\rfloor < N$. We define $D^N_{i;p/q}: \mathbb{A}^N_{i;p/q} \to \mathbb{B}^N_{i;p/q}$ to be the map
\begin{equation}\label{eq:D-truncated}
D^N_{i;p/q} = \left( \bigoplus_{n = n_{min}}^{n_{max}-1} h^\Psi_{\left\lfloor \frac{i+np}{q} \right\rfloor} \right) \oplus \left( \bigoplus_{n = n_{min}+1}^{n_{max}} v_{\left\lfloor \frac{i+np}{q} \right\rfloor} \right),
\end{equation} 
where we understand $v_{\left\lfloor \frac{i+np}{q} \right\rfloor}$ as taking the summand of $\mathbb{A}^N_{i;p/q}$ corresponding to the index $n$ to the summand of $\mathbb{B}^N_{i;p/q}$ corresponding to the index $n$ and $h_{\left\lfloor \frac{i+np}{q} \right\rfloor}$ as taking the summand of $\mathbb{A}^N_{i;p/q}$ corresponding to the index $n$ to the summand of $\mathbb{B}^N_{i;p/q}$ corresponding to $n+1$. We define $\mathbb{X}^N_{i;p/q}$ to be the mapping cone of $D^N_{i;p/q}$. When $p/q < 0$ the definition is similar, with slightly different ranges for the indices. In this case we define 
\begin{equation}\label{eq:AandB-truncated-negative-slope}
\mathbb{A}^N_{i;p/q} = \bigoplus_{n = n_{min}}^{n_{max}} A_{\left\lfloor \frac{i+np}{q} \right\rfloor} \qquad \text{ and } \qquad \mathbb{B}^N_{i;p/q} = \bigoplus_{n = n_{min}}^{n_{max}+1} B^v,
\end{equation} 
where $n_{min}$ is the smallest integer $n$ for which $\left\lfloor \frac{i+np}{q} \right\rfloor < N$, and $n_{max}$ is the largest integer $n$ such that $\left\lfloor \frac{i+np}{q} \right\rfloor > -N$, and we define $D^N_{i;p/q}$ to be the map
\begin{equation}\label{eq:D-truncated-negative-slope}
D^N_{i;p/q} = \left( \bigoplus_{n = n_{min}}^{n_{max}} h^\Psi_{\left\lfloor \frac{i+np}{q} \right\rfloor} \right) \oplus \left( \bigoplus_{n = n_{min}}^{n_{max}} v_{\left\lfloor \frac{i+np}{q} \right\rfloor} \right).
\end{equation}

Note that up to quasi-isomorphism the choice of $N$ does not matter (provided $N$ is at least $g_1 + g_2$). A larger choice of $N$ gives a bigger complex, with more copies of $A_s$ with $|s|\ge N$ and more corresponding copies of $B^v$, but the homology is the same. This follows from the fact that $h_s$ is an isomorphism for $s \le -N$ and $v_s$ is a in isomorphism for $s \ge N$. It is also possible to define an infinitely generated complex over $\F[W]$ by allowing $n$ to range over all integers in Equations \eqref{eq:AandB-truncated}-\eqref{eq:D-truncated-negative-slope}, but to make sense of these infinitely generated modules we need to work with completions with respect to the variable $W$ and replace direct sums with direct products. This infinitely generated complex will be denoted and $\mathbb{X}_{i;p/q}$, though we will generally work with the truncated complexes $\mathbb{X}^N_{i;p/q}$. We do not need to truncate $\mathbb{X}_{s;0}$, as it is already finitely generated; for any $N$, we will understand $\mathbb{X}^N_{s;0}$ to mean $\mathbb{X}_{s;0}$.

The homology of $\mathbb{X}^N_{i;p/q}$ (for any sufficiently large $N$) gives an algebraic pairing of $(C_1, \Psi^*_1)$ with $(C_2, \Psi^*_2)$ associated with $i$ and $p/q$. We will show that this agrees with the Floer homology of certain curves in the cylinder $\cylinder$. For $i$ in $\{0,1\}$, let $\tracks_i = \tracks(\Gamma_i, \bchain_i)$ be a train track in $\cylinder$ representing $(C_i, \Psi^*_i)$. We will define a shifted version $\tracks_1[i;p/q]$ of $\tracks_1$. If $p/q =0$, then for any $i \in \Z$ the shifted $\tracks_1[i;0]$ is simply the curve $\tracks_1$ shifted upward by $i$. For $p/q \neq 0$, we construct a non-compact curve by cutting $\tracks_1$ along the line $\{\tfrac 1 2\}\times \R$ and gluing together infinitely many shifted copies of this cut open curve. Let $\tracks_1^{cut}$ denote the decorated curve in $\strip$ obtained by cutting $\tracks_1$, so that $\tracks_1$ in $\cylinder$ is recovered by gluing the opposite sides of $\strip$ and identifying endpoints of $\tracks_1^{cut}$. For any integer $s$, let $\tracks_1^{cut}[s]$ denote the curve $\tracks_1^{cut}$ shifted upward by $s$ units. The curve $\tracks_1[i;p/q]$ in $\cylinder$ is constructed from
$$\bigcup_{n \in \Z} \tracks_1^{cut} \left[ \left\lfloor \frac{i+np}{q} \right\rfloor \right]$$
by identifying the right endpoints of the copy of $\tracks_1^{cut}$ corresponding to the index $n$ with the left endpoints of the copy of $\tracks_1^{cut}$ corresponding to the index $n+1$. 

\begin{proposition}\label{prop:shifted-pairing}
For any $p/q \in \Q$ and any $i \in \Z/p\Z$, the Floer complex $CF(\tracks_2, \tracks_1[i;p/q])$ is quasi-isomorphic to the complex $\mathbb{X}^N_{i;p/q}$ for any sufficiently large $N$.
\end{proposition}
\begin{proof}
When $p/q = 0$, the proof is exactly the same as Proposition \ref{prop:unshifted-pairing} except that $\tracks_1$ is shifted upward by $i$. For other values of $p/q$ the proof is similar. We view the cylinder $\cylinder$ as $\strip \cup \sF$, with $\sF$ small enough that $\tracks_1$ and $\tracks_2$ both consist of parallel arcs when restricted to $\sF$. For each $n$ from $n_{min}$ to $n_{max}$, letting $s = \left\lfloor \frac{i+np}{q} \right\rfloor$, we perturb the corresponding shifted train track $\tracks_1^{cut}[s]$ in $\strip$ as in the proof of Proposition \ref{prop:pairing-in-strip} so that the Floer chain complex of $\tracks_2$ with this train track in $\strip$ is precisely $A_s$. We may assume that the endpoints of $\tracks_1^{cut}[s]$ occur above height $N$ on $\partial_R \strip$ and below height $-N$ on $\partial_L \strip$. We do not perturb the copies of $\tracks_1^{cut}$ corresponding to indices $n > n_{max}$ or $n < n_{min}$; these train tracks lie entirely above height $N$ or below height $-N$ and are thus disjoint from $\tracks_2$. Note that for any adjacent indices $n$ and $n+1$, the arcs in $\sF$ connecting the endpoints of the two corresponding shifted copies of $\tracks_1$ intersect the arcs of $\tracks_2$ in $\sF$ if $n_{min} \le n < n_{max}$ or if $p/q < 0$ and $n = n_{min}-1$ or $n = n_{max}$; we identify these intersection points with a copy of $B^v$ indexed by $n+1$, so that the Floer complex is identified with $X^N_{i;p/q}$ as a vector space.

With the curves perturbed as above, counting bigons exactly recovers the map $D^N_{i;p/q}$. The proof is essentially the same as the proof of Proposition \ref{prop:unshifted-pairing} (though note that we have not introduced the additional intersection points in $\sF$ corresponding to the two copies of $B^h$ connected by $\widetilde{\Id}$). Bigons that do not contribute to the internal differnetial on one of the summands of $\mathbb{X}^N_{i;p/q}$ can only start at intersection points on a perturbed copy of $\tracks_i[s]$ corresponding to some index $n$, and all such bigons end at an intersection points in $\sF$ corresponding to the copies of $B^v$ with index either $n$ or $n+1$. Similar to the proof of Proposition \ref{prop:unshifted-pairing}, we can check that counting the bigons of these two types recovers the maps  $\widetilde{v}_s$ or $\widetilde{h}^\Psi_s$, respectively.
\end{proof}

An example of a shifted pairing is shown in Figure \ref{fig:shifted-pairing}, where $(C_1, \Psi_{1,*})$ is the knot Floer invariant of the right handed trefoil, $(C_2, \Psi_{2,*})$ is the knot Floer invariant of the dual knot in $+1$-surgery on the left handed trefoil, $p/q = -1$, $i = 0$. The Figure shows the curves $\tracks_2$ and $\tracks_1[i;p/q]$, lifted to the covering space $\widetilde T_M$ for clarity, with $\tracks_1[i;p/q]$ perturbed as in the proof of Proposition \ref{prop:shifted-pairing} so that the Floer chain complex agrees exactly with $\mathbb{X}^N_{i;p/q}$ with $N = 2$.

\begin{figure}
\includegraphics[scale = .8]{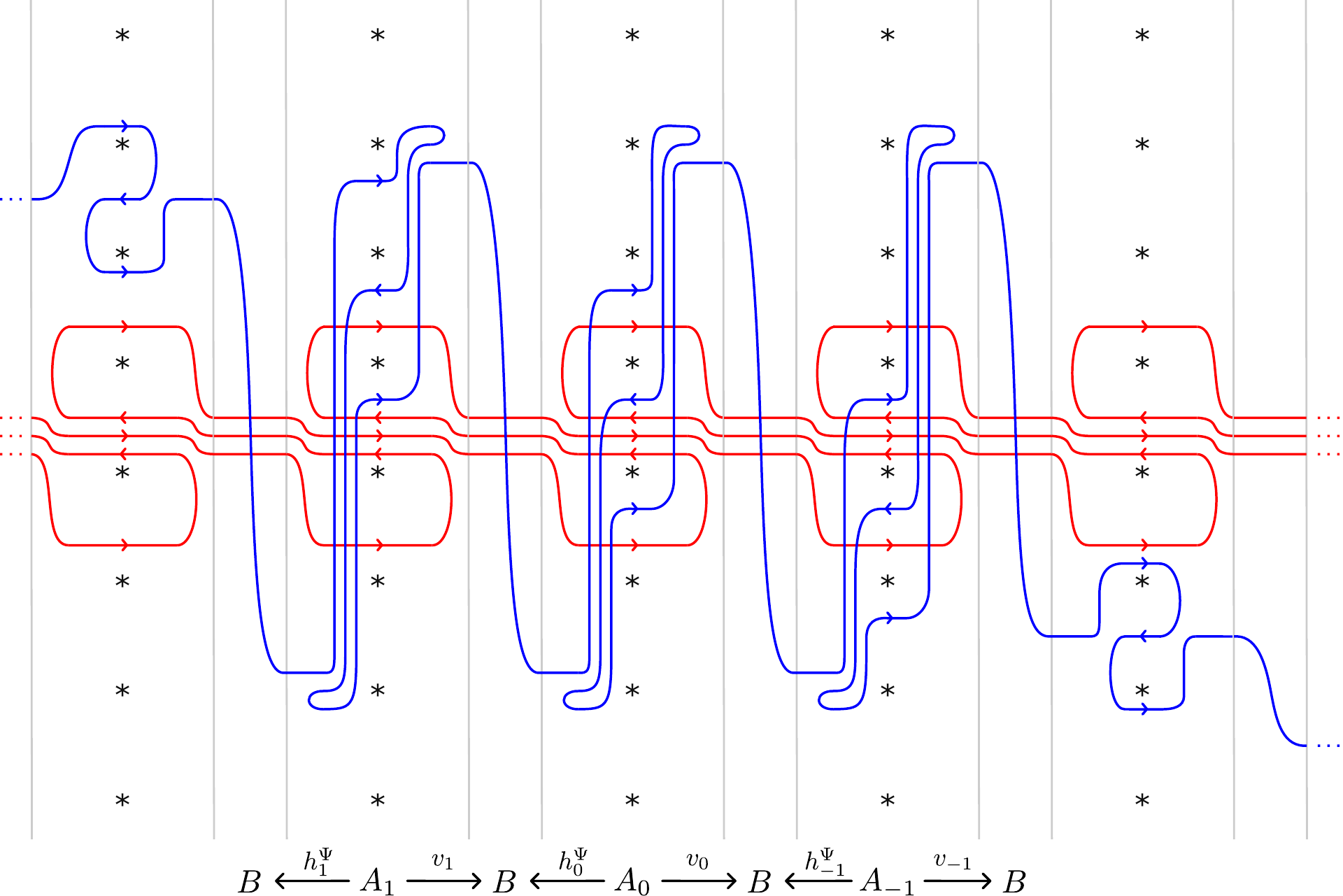}
\caption{A shifted pairing of the knot Floer invariants of the right handed trefoil and the dual knot in $+1$ surgery on the left handed trefoil.}
\label{fig:shifted-pairing}
\end{figure}

\section{Surgery formulas}\label{sec:surgery}% !TEX root = ../CFKcurvesZHS3s.tex
%surgery.tex

In the previous section we related algebraic pairings of complexes and flip maps to geometric pairings of the corresponding curves, but we did not ascribe topological significance to either of these pairings. In this section we will show that these pairings compute the Heegaard Floer homology of Dehn surgeries.

\subsection{Rational surgery formula}

Recall that Ozsv{\'a}th and Szab{\'o} define a rational surgery formula for Heegaard Floer homology in \cite{OzSz:rational-surgeries}. This surgery formula realizes the Heegaard Floer homology of rational surgery on a knot as the homology of a mapping cone complex constructed from certain subcomplexes of the knot Floer complex of $K$. In fact, this mapping cone complex is a special case of the complex $\mathbb{X}_{i;p/q}$ defined in the previous section. The surgery formula was originally stated for the plus version of Heegaard Floer homology, but an analogous formula holds for the minus version. In the minus version, for technical reasons we need to work with completions of the various modules involved and replace direct sums with direct products as described in \cite{ManolescuOzsvath}, but this subtlety can be avoided by working with truncated mapping cone complexes (which are finitely generated), as discussed in the previous section.

Fix a null-homologous knot $K$ in a 3-sphere $Y$, let $C$ be the complex $CFK_{\sRminus}(Y,K)$, and let $\Psi_*$ be the flip isomorphism associated with $K$. We also consider the complex $C_{triv}$ that has a single generator in bigrading $(0,0)$ equipped with the identity flip isomorphism $\Psi_{triv,*}$; note that $(C_{triv}, \Psi_{triv,*})$ is the knot Floer invariant associated to the unknot. The proof of the following proposition relies on observing that $\mathbb{X}^N_{i;p/q}(C_{triv}, \Psi_{triv}, C, \Psi)$ is quasi-isomorphic to the mapping cone complex $\mathbb{X}_{i;p/q}$ defined in \cite{OzSz:rational-surgeries}.

\begin{proposition}\label{prop:surgery-formula}
Let $\tracks$ be the decorated curve in $\cylinder$ associated with the knot $K$. For any nonzero $p/q \in \Q$ and any $i \in \Z/p\Z$, let $\ell_{i;p/q}$ be a line in $\cylinder$ of slope $p/q$ that passes intersects $\mu$ just above height $-\tfrac 1 2 + \tfrac i q$. There is a relatively graded isomorphism of $\F[W]$ modules
$$HF^-(Y_{p/q}(K), i) \cong HF(\tracks, \ell_{i;p/q}),$$
where the right side is Floer homology in the marked cylinder $\cylinder$.
\end{proposition}
\begin{proof}
It is easy to check that $\ell_{i;p/q}$ is homotopic to $\tracks_{triv}[i;p/q]$, where $\tracks_{triv}$ is the decorated curve representing the pair $(C_{triv}, \Psi_{triv})$---$\tracks_{triv}$ is simply the horizontal simple closed curve in $\cylinder$ at height zero. Thus by Proposition \ref{prop:shifted-pairing} the Floer complex on the right side is quasi-isomorphic to $\mathbb{X}^N_{i;p/q}(C_{triv}, \Psi_{triv}, C, \Psi)$ for sufficiently large $N$. We just need to show that this latter complex computes $HF^-(Y_{p/q}(K), i)$; we do this by showing that $\mathbb{X}^N_{i;p/q}(C_{triv}, \Psi_{triv}, C, \Psi)$ is quasi-isomorphic to the (truncated) mapping cone complex defined by Ozsv{\'a}th and Szab{\'o}.

Observe that in the construction of $\mathbb{X}^N_{i;p/q}(C_{triv}, \Psi_{triv}, C, \Psi)$, $A_s = \Mor(C_{triv}, C) |_{A=s}$ is isomorphic to $C |_{A=s}$, since every morphism is determined by where it takes the generator of $C_{triv}$. Similarly, $B^v = \Mor(H_* C^v_{triv}, H_* C^v)$ is simply $H_* C^v$.  We next note that the complex $A_s = C|_{A=s}$ is isomorphic to the minus analog of the complex $A^+$ in \cite{OzSz:rational-surgeries}. Recall that in the notation of \cite{OzSz:rational-surgeries}, we view $\CFKinfty(Y,K)$ as being generated over $\F$ by triples $[x,i,j]$ with $j - i = A(x)$ and $A^+_s$ is defined to be quotient complex generated by triples with $\max(i, j-s) \ge 0$. The analogous $A^-_s$ is the subcomplex generated over $\F$ by triples with $\max(i, j-s) \le 0$. In our notation, this is isomorphic to the subcomplex of $CFK_{\sRminus}(Y,K)\otimes \F[V, V^{-1}]$ with Alexander grading zero generated by terms of the form $U^{A(x) - s} V^{-s} x$ if $A(x) \ge s$ or $V^{-A(x)} x$ if $A(x) < s$; multiplying by $V^s$ gives an isomorphism between this and $C |_{A=s}$. Similarly, the minus analog $B^-_s$ of the $B^+_s$ modules appearing in \cite{OzSz:rational-surgeries} can be identified with the Alexander grading zero summand of $CFK_{\sRminus}(Y,K)\otimes \F[V, V^{-1}]$, which by setting $V=1$ is equivalent to $C^v$ as a module over $\F[W]$.

The maps $v_s$ and $h_s$ in the construction of $\mathbb{X}^N_{i;p/q}(C_{triv}, \Psi_{triv}, C, \Psi)$ are simply the inclusion maps $C|_{A=s}\into C^v$ and $C|_{A=s}\into C^h$ obtained by setting $V=1$ and $U =1$, respectively, followed by taking homology of $C^v$ or $C^h$. The map $h^\Psi_s$ is the composition of $h_s$ with the map $F_{\Psi_{triv,*}, \Psi_*}$, which can be identified with $\Psi_*$ since $\Psi_{triv,*}$ is the identity map and the source $B^h$ and target $B^v$ are identified with $H_* C^h$ and $H_* C^v$, respectively. The map $v^-_s: A^-_s \to B^-$ in the minus analog of the mapping cone construction from \cite{OzSz:rational-surgeries} is also the inclusion map $C |_{A=s} \into C^v$, and the map $h^-_s: A^-_s \to B^-$ is the inclusion map $C|_{A=s}\into C^h$ followed by the flip map $\Psi$. 

With these observations in place, we see that $\mathbb{X}^N_{i;p/q}(C_{triv}, \Psi_{triv}, C, \Psi)$ is closely related to the complex $\mathbb{X}^-_{i;p/q}$ from the minus analog of the construction in \cite{OzSz:rational-surgeries}, with two differences. The first difference is that $\mathbb{X}^N_{i;p/q}(C_{triv}, \Psi_{triv}, C, \Psi)$ is truncated; that this does not affect the complex up to quasi-isomrphism follows easily from the fact that $v_s$ is an isomorphism for $s \ge N$ and $h_s$ is an isomorphism for $s \le -N$. The other difference is that in $\mathbb{X}^N_{i;p/q}(C_{triv}, \Psi_{triv}, C, \Psi)$ we have already taken the homology of each copy of $C^v$; this does not affect the homology of the complex. The result then follows from Theorem 1.1 of \cite{OzSz:rational-surgeries}.
\end{proof}

If we forget the spin$^c$ decomposition and project to the marked torus, we recover Theorem \ref{thm:surgery-formula} from the introduction.

\begin{proof}[Proof of Theorem \ref{thm:surgery-formula}]
Recall that the marked cylinder $\cylinder$ can be identified with $\overline{T}_M$, where $M$ is the knot complement of $K \subset Y$ and we identify the vertical direction with $\mu$ and the horizontal direction with $\lambda$. If we do not care about the spin$^c$ decomposition on $Y_{p/q}(K)$ we can project to the marked torus $T_M$. The curves $\ell_{i;p/q}$ project to single simple closed curve $\ell_{p/q}$ of slope $p/q$, and the Floer complex of $p(\tracks)$ with $\ell_{p/q}$ is the direct sum over $i$ of the Floer complexes of $\tracks$ with $\ell_{i;p/q}$ in $\cylinder$. The result then follows from Proposition \ref{prop:surgery-formula}.
\end{proof} 

\subsection{Surgery formula for dual knots}\label{sec:dual-surgery-formula}

When performing $p/q$ surgery on a knot $K$ in $Y$, the core of the filling torus defines a dual knot $K^* \subset Y_{p/q}(K)$. In \cite{HeddenLevine:surgery}, Hedden and Levine enhanced the surgery formula of Ozsv{\'a}th and Szab{\'o} for nonzero integer surgeries to give a surgery formula for the knot Floer complex of the dual knot in a surgery. This enhancement also has a nice description in terms of Floer homology of curves, which we now describe.

Recall that for an integral surgery $n$ on a nullhomologous knot the mapping cone complex $\mathbb{X}$ is the mapping cone of the map
$$\bigoplus_{s \in \Z} v^-_s + h^-_s : \bigoplus_{s\in\Z} A^-_s \to \bigoplus_{s\in\Z} B^-_s,$$
where $v^-_s$ maps $A^-_s$ to $B^-_s$ and $h^-_s$ maps $A^-_s$ to $B^-_{s+n}$. This complex splits into subcomplexes $\mathbb{X}_{i;n}$ containing the $A^-_s$ and $B^-_s$ with $s$ congruent to $i$ mod $n$. For sufficiently large $N$, each of these complexes can be truncated to include only $A_s$'s with $|s| < N$ and only $B_s$'s with $-N+n \le s \le N$ if $n > 0$ or only the $B_s$'s with $-N \le s \le N - n$ if $n<0$. We will view the mapping cone complex as a module over $\F[W]$.

The surgery formula in \cite{HeddenLevine:surgery} adds a new rational Alexander filtration $\mathcal{J}$ to the (truncated) mapping cone complex so that it is filtered chain homotopy equivalent to $\CFKminus(Y_n(K), K^*)$. Note that there is already an integer filtration $\mathcal{I}$ given by negative powers of $W$. To describe the $\mathcal{J}$ filtration, recall that we identify the complex $A^-_s$ with the complex $C_{A=s}$, which has generators of the form $V^{s-A(x)}x$ or $U^{A(x)-s}x$ for generators $x$ of $C$. Each generator $U^{A(x)-s}x$ of $A^-_s$ with $A(x) \ge s$ has $\mathcal{J}$ filtration level $\tfrac{2s + n - 1}{2n}$ while every generator $V^{s-A(x)}x$ of $A^-_s$ with $A(x) < s$ and each generator of $B^-_s$ has $\mathcal{J}$ filtration level $\tfrac{2s + n - 1}{2n} - 1$. Although $\mathcal{J}$ is a rational filtration, we are primarily interested in the relative integral filtration on each summand $\mathcal{X}_{i;n}$. For each $i$, we fix a rational shift $s_i$ so that $\mathcal{J}$ takes values in $\Z + s_i$ on $\mathcal{X}_{i;n}$ and $\mathcal{J} - s_i$ is an integer filtration. A key observation is that when moving from index $s$ to index $s+n$ the filtration levels of generators increases by 1, and that generators $U^{A(x)-s}x$ of $A^-_s$ with $A(x) \ge s$ are in the same filtration level as the generators of $B^-_{s+n}$ and the generators $V^{s+n-A(x)}x$ of $A^-_{s+n}$ with $A(x) <  s + n$. 

Though \cite{HeddenLevine:surgery} does not use the notation of bigraded complexes used in this paper, we can pass to a bigraded complex over $\sRminus$ by replacing the formal variable $W$ with the pair of variables $U$ and $V$ and defining a bigrading $(\gr^*_w, \gr^*_z)$ so that $\gr^*_w$ is the Maslov grading on the mapping cone complex and $\gr^*_z$ is defined so that $A^* = \tfrac{\gr^*_w - \gr^*_z}{2}$ gives the filtration level $\mathcal{J} - s_i$. To get the differential on the bigraded complex from the differential on the complex over $\F[W]$ we replace $W$ with the product $UV$ and then add additional factors of $V$ as needed to be consistent with the bigrading; note that forgetting the new filtration by setting $V=1$ recovers the original complex.

We will realize this bigraded complex as the Floer complex of curves in the doubly marked cylinder. Let $\tracks$ be the decorated curve in $\cylinder$ representing the knot Floer homology of $K$. For each $i \in \Z/n\Z$, let $\ell^*_{i;n}$ be a curve of slope $n$ in $\cylinder$ that passes through the marked points at height $i + kn - \tfrac 1 2$ for integers $k$ (note that in the doubly marked cylinder $\dmcylinder$ the curve $\ell^*_{i;n}$ passes between the pair of marked points at these heights). Recall that $Y_n(K)$ has $n$ spin$^c$ structures, which can be canonically identified with $\Z/n\Z$.

\begin{proposition}\label{prop:dual-surgery-formula}
The bigraded complex $CFK_{\sRminus}(Y_n(K), K^*; i)$ is given by the Floer complex of $\tracks$ with $\ell^*_{i;n}$ in the doubly marked cylinder $\dmcylinder$.
\end{proposition}

\begin{proof}
We identify the truncated complex $\mathbb{X}^N_{i;n}$ with the Floer complex of $\tracks$ with $\ell^*_{i;n}$ by perturbing $\ell^*_{i;n} = \tracks_{triv}[i;n]$ as in the proof of Proposition \ref{prop:shifted-pairing}. Recall that, cutting the cylinder into marked and unmarked strips $\strip$ and $\sF$, each copy of $A_s$ in the mapping cone complex corresponds to the Floer complex of $\tracks$ with a connected component of $\tracks_{triv}[i;n]$ restricted to $\strip$. The piece of $\tracks_{triv}[i;n]$ in question crosses $\mu$ between the marked points at heights $s-\tfrac 1 2$ and $s+\tfrac 1 2$, lying to the left of $\mu$ below this point and to the right of $\mu$ after this point. If we assume that $\tracks_{triv}[i;n]$ crosses $\mu$ at height $s-\tfrac 1 2 + \epsilon$ for a sufficiently small $\epsilon$, then the generators of $A_s$ of the form $U^{A(x)-s} x$ for $A(x) \ge s$ are precisely those corresponding to intersection points on the right side of $\mu$. Combining this with the observation above we note that, when moving along $\tracks_{triv}[i;n]$, the $\mathcal{J}$ filtration level of generators of the Floer complex increases by one each each time $\mu$ is crossed moving rightward.

We claim that this behavior is reproduced in the Floer complex if we add a new $z$ marked point on $\mu$ at height $2\epsilon$ above each existing $w$ marked point. Now any bigon contributing to the Floer complex covers the same number of $z$ and $w$ marked points except that it covers one extra $z$ marked point for each time the $\ell^*_{i;n}$ part of the boundary crosses $\mu$ moving rightward, and one extra $w$ marked point for each time the $\ell^*_{i;n}$ part of the boundary crosses $\mu$ moving leftward. Since the differential preserves the new Alexander grading $A^*$, it follows that if there is a bigon from $x$ to $y$ with weight $U^a V^b$ then $A^*(y) - A^*(x) = a - b$. We can of course shift the marked points and the intersection of $\ell^*_{i;n}$ with $\mu$ down by $\epsilon$ without any effect on the complex, so that the $z$ and $w$ marked point occur a distance of $\epsilon$ above and below the original marked point, and $\ell^*_{i;n}$ passes through the original marked point. 

Having identified the knot Floer complex with the Floer homology of curves in their perturbed form, we know that homotopic curves will represent chain homotopic complexes. In particular, we can pull the curve $\ell^*_{i;n}$ to be a straight line of slope $n$ (that is, a curve in the cylinder that lifts to a straight line of slope $n$ in the universal cover). We remark that there is a slight subtley coming from the fact that in this proof we place the $z$ and $w$ marked points above and below the original marked point, whereas we usually think of $z$ and $w$ in $\dmcylinder$ lying to the left and right of the marked points in $\cylinder$. If $n > 0$ it is clear that when $\ell^*_{i;n}$ is pulled tight we can just as well place $z$ and $w$ to the left and right of the marked point. If $n < 0$ then pulling 
$\ell^*_{i;n}$ tight results in $z$ being to the right of $\ell^*_{i;n}$ and $w$ being to the left of $\ell^*_{i;n}$. This is seemingly a problem, but in fact it has no effect on the complex because of the symmetry of knot Floer homology: the curves are symmetric under 180 degree rotation, and this rotation interchanges the roles of $z$ and $w$.
\end{proof}

Theorem \ref{thm:surgery-formula-dual-knots} in the introduction follows immediately from Proposition \ref{prop:dual-surgery-formula} by projecting 
from the marked cylinder to the marked torus.

\begin{proof}[Proof of Theorem \ref{thm:surgery-formula-dual-knots}]
We identify the marked cylinder $\cylinder$ with $\overline{T}_M$ in the usual way and then project to the marked torus $T_M$. The lines $\ell^*_{i;n}$ all project to a single curve $\ell^*_n$ of slope $n$ through the marked point, and the Floer complex of the projection of $\tracks$ with $\ell^*_n$ in the doubly marked torus $T_M^{z,w}$ is the direct sum of the Floer complexes in $\cylinder$ of $\tracks$ with $\ell^*_{i;n}$, which by Proposition \ref{prop:dual-surgery-formula} are the spin$^c$ summands of $CFK_{\sRminus}(Y_n(K), K^*)$.
\end{proof}

\begin{example}
An example of the dual surgery formula is shown in Figure \ref{fig:dual-knot-surgery-formula}, where we consider $+1$ surgery on the left handed trefoil. On the left of the figure we show the curves perturbed as in the proof of Proposition \ref{prop:shifted-pairing} so that the Floer complex realizes the mapping cone complex, and the new filtration is encoded by adding marked points $z$ above each existing marked point. To more easily compute the knot Floer complex of the dual knot, we pull the curves tight as in the top right part of the figure. We see that there are five generators, $a$, $b$, $c$, $d$, and $e$, with Alexander gradings $1$, $0$, $0$, $0$, and $-1$. The differential is given by
$$\partial(a) = Ub, \quad \partial(b) = 0, \quad \partial(c) = -UVb + UVd, \quad \partial(d) = 0, \quad \text{ and } \quad \partial(e) = -Vd,$$
as stated in Section \ref{sec:intro-example}. We can find the immersed curve in the strip representing this new complex from the immersed curve for the left handed trefoil by applying the reparametrization that takes the line of slope $+1$ to the vertical direction while fixing the horizontal direction, as shown in the bottom right of the Figure \ref{fig:dual-knot-surgery-formula}. Pairing the resulting immersed curve paired with the vertical line $\mu$ gives the same complex as pairing trefoil curve with the line of slope 1, since applying an ambient diffeomorphism to the surface does not affect the Floer complex, so by definition this immersed curve represents the knot Floer complex of the dual knot. This is consistent with Conjecture \ref{conj:invariant-of-complement}, since the immersed curves for both the knot and the dual knot are the same when viewed as curves in the boundary of the knot complement and appear different in strip only because they are expressed in terms of different parametrizations arising from different choices of meridian.
\end{example}

\begin{figure}
\includegraphics[scale = .8]{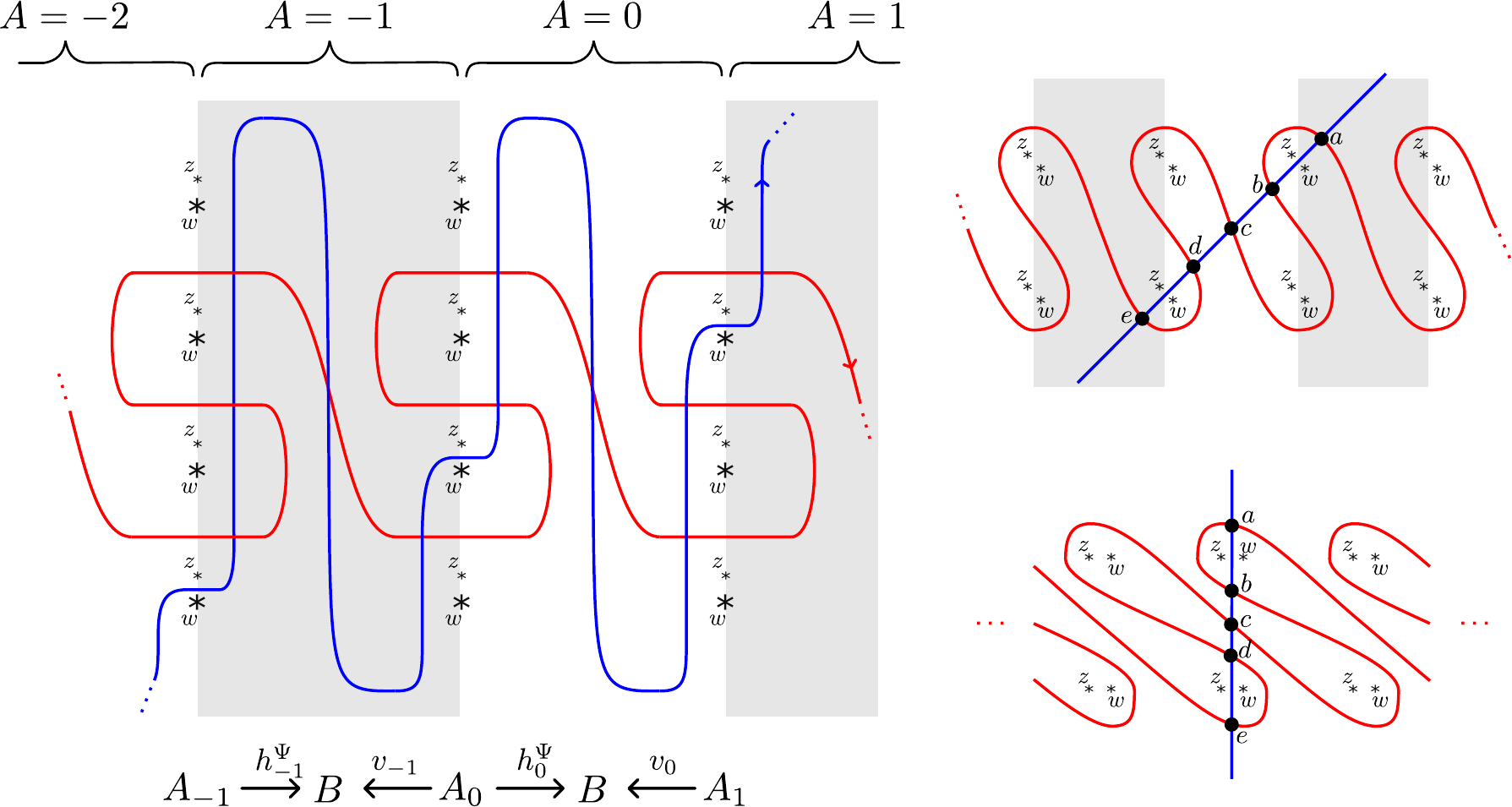}
\caption{Computing $\CFKminus$ of the dual knot in $+1$ surgery on the left handed trefoil.}
\label{fig:dual-knot-surgery-formula}
\end{figure}

\section{Examples and further considerations}\label{sec:examples}% !TEX root = ../CFKcurvesZHS3s.tex
%examples.tex

In the final section we will explore some concrete examples and also discuss a current limitation of the invariants $(\Gamma, \bchain)$ we have described, namely that although $\Gamma$ is well defined up to homotopy and the restriction to $\bchainhat$ of $\bchain$ to degree zero self-intersection points is well-defined, the full bounding chain $\bchain$ is not uniquely determined as a subset of the self-intersection points. In other words, we do not have a satisfactory normal form for an equivalence class of bounding chains. That said, as we will see in practice there is often an obvious simplest choice of $\bchain$. In particular, for a large class of complexes arising frequently for knots in $S^3$, all choices of bounding chain on the immersed curves are equivalent and the trivial bounding chain is a valid choice of bounding chain, so for these complexes we can take $\bchain$ to be trivial. We discuss some of the challenges in defining a normal form more generally, and some potential solutions.

\subsection{Simplifying bounding chains}

The non-uniqueness of the bounding chain in our construction as described so far is evident in the following example. For simplicity, we will consider this example with coefficients in $\Z/2\Z$ and we will not specify the bigrading, which is not relevant to the present discussion.

\begin{example}\label{ex:fig8-example}
Let $C$ be the chain complex over $\sRminus$ pictured on the left of Figure \ref{fig:fig8-example}. It has generators $a, b, c, d$, and $e$, differential
$$\partial(a) = Vb, \quad \partial(b) = 0, \quad \partial(c) = Ua + Ve + UVd, \quad, \partial(d) = 0, \quad  \partial(e) = Ub,$$
and gradings as given in the table below:
$$\def\arraystretch{1.5} \begin{array}{c|ccccc}
  & a & b & c & d &e \\ \hline
(\gr_w, \gr_z) & (0,-2) & (-1,-1) & (-1,-1) & (0,0) & (-2,0) \\
A & 1 & 0 & 0 & 0 & -1
\end{array} $$
Both the horizontal and vertical homology are generated by $e$; let $\Psi$ be the flip isomorphism taking $e$ to itself. To compute the decorated immersed multicurve $(\Gamma, \bchain)$ representing this data we first ignore the diagonal arrow and find $(\Gamma, \bchainhat)$ representing the complex over $\sRhat$. Since the given basis is both horizontally and vertically simplified, we quickly arrive at the immersed curve in the middle of Figure \ref{fig:fig8-example} with no crossover arrows, so we do not need to remove arrows and $\bchainhat$ is trivial. We then enhance the decorated curve to capture the diagonal arrow that was ignored; following the algorithm in Section \ref{sec:enhanced-curves}, we see that this is accomplished by decorating $\Gamma$ with the bounding chain $\bchain$ consisting of one intersection point with coefficient $W$ as pictured on the right of Figure \ref{fig:fig8-example}.
\end{example}

\begin{figure}
\includegraphics[scale = 1]{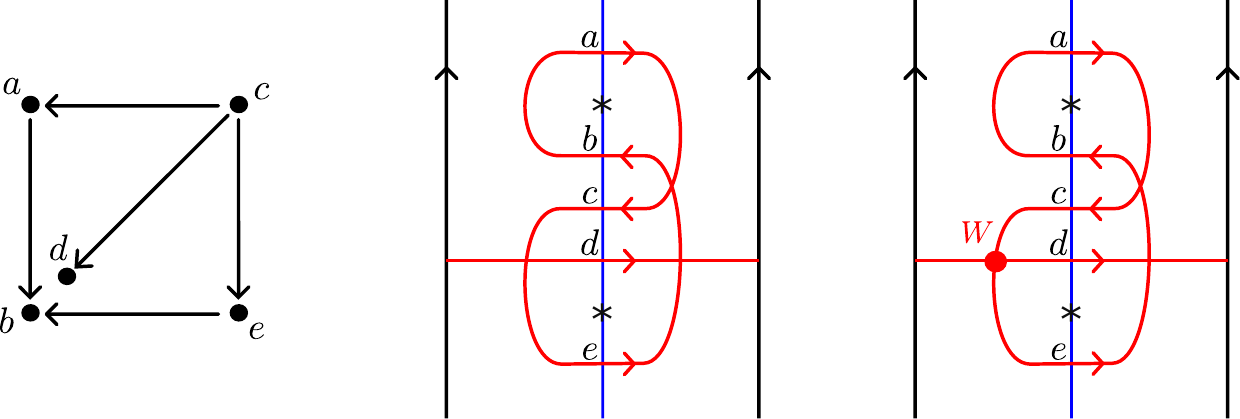}
\caption{An immersed curve representative for the knot Floer complex of the figure-eight knot. The curve with trivial bounding chain (middle) represents the complex over $\sRhat$, and adding the indicated intersection point to the bounding chain (right) encodes the complex over $\sRminus$.}
\label{fig:fig8-example}
\end{figure}

In this example, the procedure we have presented for computing the curve representative stops with the decorated curve $(\Gamma, \bchain)$ since $\Gamma$ is in almost simple position and $\bchainhat$ contains only local system intersection points (in fact, $\bchainhat$ is trivial). However, it is not hard to see that there is a more convenient representative for this chain homotopy equivalence class of curves. The change of basis replacing $e$ with $e + Ud$ has the effect of removing the diagonal arrow from the complex in Example \ref{ex:fig8-example}. Clearly $\Gamma$ represents this complex over $\sRminus$ with no need to decorate with a bounding cochain. In other words, there are two different bounding chains $\bchain$ (consisting of the one intersection point in Figure \ref{fig:fig8-example}) and $\bchain'$ (which is trivial) on $\Gamma$ such that $(\Gamma, \bchain)$ and $(\Gamma, \bchain')$ both represent the same homotopy equivalence class of chain complexes.

These decorated curves are $(\Gamma, \bchain)$ and $(\Gamma, \bchain')$ are equivalent as elements of the Fukaya category, so this disparity does not violate the uniqueness statement of Theorem \ref{thm:curve-invariant-for-complex}, but we would hope to have a preferred representative for any equivalence class of decorated curves. In the above example, there is an obvious choice for the preferred representative: it makes sense to remove the decoration entirely if possible and choose $(\Gamma, \bchain')$ as the representative of this equivalence class. In general we'd like to choose the simplest option for the bounding chain representing a given equivalence class, but it is not always clear what ``simplest" means when the trivial bounding chain is not an option. We ask the following:

\begin{question}
Is there normal form for bounding chains on immersed curves such that every complex over $\sRminus$ is represented by $(\Gamma, \bchain)$ and for any $\bchain$ and $\bchain'$ of this form if $(\Gamma, \bchain)$ and $(\Gamma, \bchain')$ are equivalent objects then $\bchain$ and $\bchain'$ agree as linear combinations of self-intersection points of $\Gamma$?
\end{question}

One way of describing such a normal form would be to extend the arrow sliding algorithm used to simplify $\bchainhat$. Recall that we systematically remove all arrows that can be removed, leaving only left-turn arrows at local system intersection points. In the same way, we could start with a decorated curve $(\Gamma, \bchain)$, consider the corresponding train track in which $\bchain$ is represented by a collection of left-turn crossover arrows, and then systematically slide crossover arrows to remove them when possible or if not put them in a preferred position. This strategy works for Example \ref{ex:fig8-example}, as shown in Figure \ref{fig:fig8-example2}. The bounding chain $\bchain$ can be interpreted as a single crossover arrow with weight $W$. We can slide this arrow rightward until it is parallel with $\mu$, pointing from the segment containing $e$ to the segment containing $e$. We can then slide it past $\mu$ (and over a marked point) at the expense of changing the weight from $W$ to $1$; an argument similar to the proof of Proposition \ref{prop:basis-change} shows this move has the effect of the change of basis replacing $e$ with $e + Ud$. Finally, we continue sliding the arrow until it is removable.

\begin{figure}
\includegraphics[scale = .8]{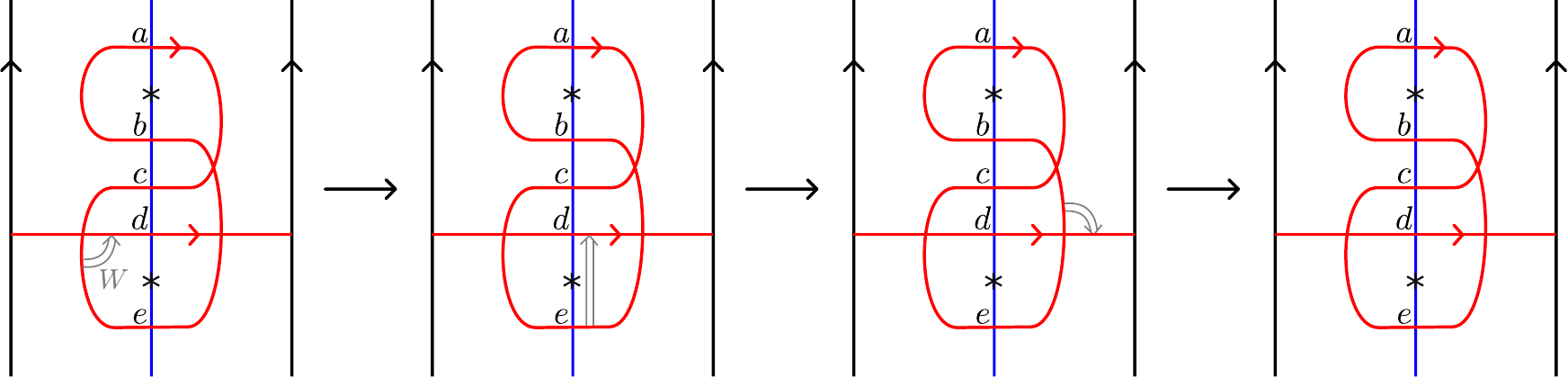}
\caption{Simplifying the bounding chain from Example \ref{ex:fig8-example} by arrow sliding.}
\label{fig:fig8-example2}
\end{figure}

We suspect that this strategy will work in general, but it turns out that sliding arrows is fairly subtle in the minus setting. For example, in Lemma \ref{prop:basis-change} we proved the fact that sliding an arrow across $\mu$ corresponds to a change of basis, but we did this in the $UV=0$ setting. The natural generalization of the statement to the minus setting does not always hold (if only holds if the arrow in question is unobstructed). Indeed, the main reason is that arrows sliding in the minus setting is more difficult is that monogons are plentiful and sliding arrows often gives rise to arrows that are not unobstructed. We avoided these subtleties in the constructions earlier in this paper by doing the majority of the arrow simplification in the $UV=0$ setting, and when constructing the enhanced curves over $\sRminus$ we really worked over a quotient $\sR_{k+1} = \sRminus/W^{k+1}$ at each step and only manipulated arrows weighted by the maximal power $W^{k}$.

Another requirement for defining a normal form by arrow sliding is giving a clear description of when the process should stop; that is, we need to understand which arrows are removable. The next example is a variation on the previous one but the different choices of bounding chain are not equivalent, and in particular these crossover arrows are not removable.

\begin{example}\label{ex:example-multiple-bounding-chains}
Consider the complex with five generators and differential
$$\partial(a) = V^2 b, \quad \partial(b) = 0, \quad \partial(c) = U^2 a + V^2 e + k_1 UVd, \quad \partial(d) = k_2 , \quad  \partial(e) = U^2b,$$
where $k_1$ and $k_2$ are either 0 or 1 and they are not both 1. The decorated curves representing the three complexes arising from the choice of $k_1$ and $k_2$ are shown in Figure \ref{fig:example-multiple-bounding-chains}. The complexes agree over $\sRhat$, so the underlying curve $\Gamma$ is the same in each case and we have three different choices of bounding chain on $\Gamma$. These three complexes are not homotopy equivalent to each other, so these three decorated curves are not equivalent objects in the Fukaya category of the marked cylinder. In particular, there is no way to simplify either of the nontrivial bounding chains to obtain the trivial bounding chain so each of the crossover arrows shown are not removable. We see that the arrow slide used in the previous example fails because it requires sliding the crossover arrow over two marked points, but the arrow is only weighted by $W$ and the power of $W$ must decrease by one for each marked point the arrow crosses. If there is a normal form for bounding chains on an immersed multicurve, all three of these bounding chains must satisfy the conditions of being in normal form.
\end{example}

\begin{figure}
\includegraphics[scale = .7]{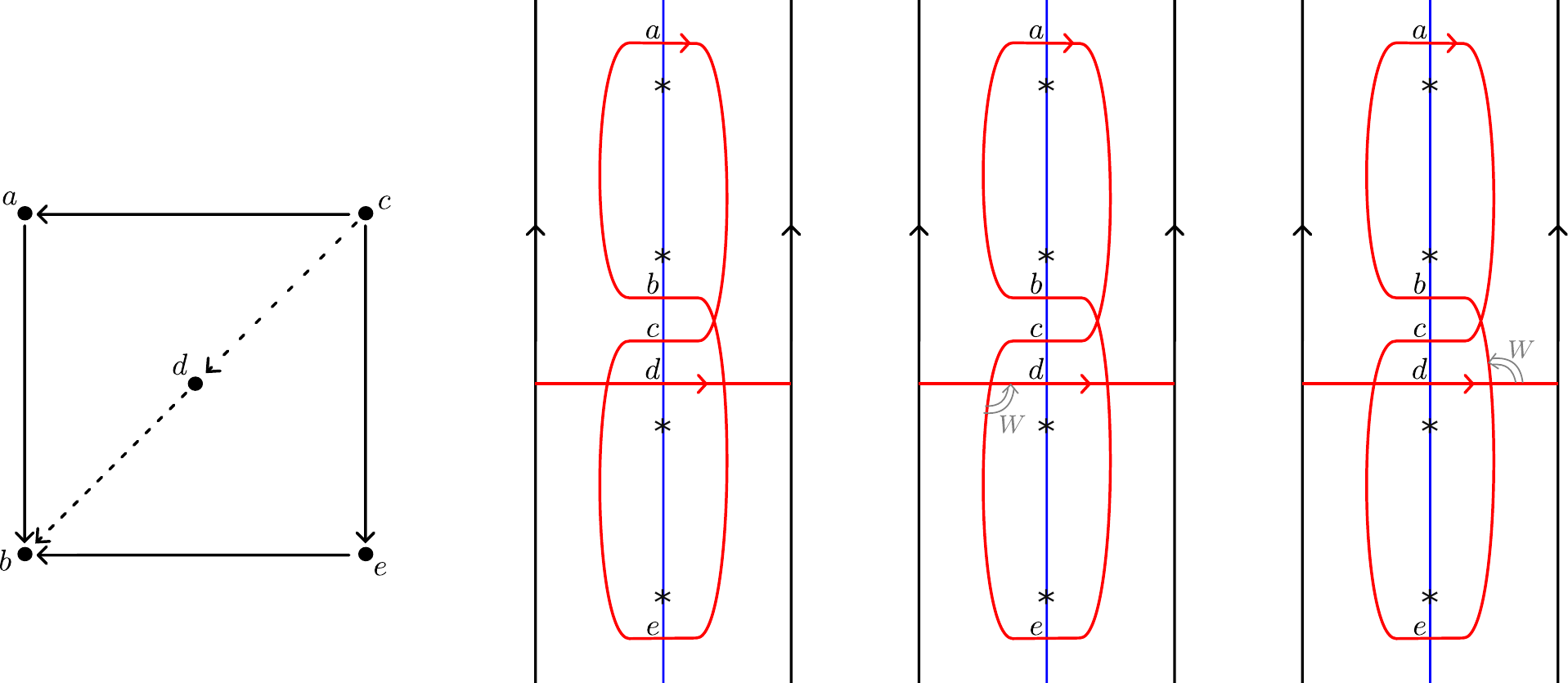}
\caption{Decorated immersed curves representing the complex shown at the left with no diagonal arrows or with either one of the dotted diagonal arrows included.}
\label{fig:example-multiple-bounding-chains}
\end{figure}

\subsection{Some curves with only trivial bounding chains}\label{sec:trivial-bounding-chains}
Having discussed some of the challenges with simplifying the bounding chain decoration to a normal form in general, we now point out that in practice it is often much easier to find a simplest representative. In fact, we will describe a family of immersed multi-curves for which every possible bounding chain is equivalent to the trivial one. For these curves we can always take the bounding chain to be trivial and a complex which is represented over $\sRhat$ by some curve $\Gamma$ of this form is also represented over $\sRminus$ by $\Gamma$.
 
The multicurves in question contain components which are figure eight shaped curves enclosing two adjacent marked points like the closed curves in Figure \ref{fig:fig8-example}; we call curves of this form \emph{simple figure eight curves}. A curve of this form represents a simple complex which we will denote $C_{box}$; this complex has generators $x_1$, $x_2$, $x_3$, and $x_4$ and differential
 $$\partial(x_1) = Ux_2 + Vx_3, \quad \partial(x_2) = V x_4, \quad \partial(x_3) = -Ux_4, \quad \partial(x_4) = 0.$$
 We will show that for a multicurves with simple figure eight components, any bounding chain can me simplified to avoid these components. To avoid the complications with sliding arrows in the minus setting described in the previous section, we will deduce this from an algebraic fact about bigraded chain complexes (though we remark that an arrow sliding proof would be desirable, since it would likely generalize to other families of curves for which proving the simplification algebraically would be difficult).
 
 \begin{lemma}\label{lem:splitt-off-box-summand}
 Suppose a reduced bigraded complex $C$ over $\sRhat$ has a basis containing four generators $x_1$, $x_2$, $x_3$, and $x_4$ such that there are length one horizontal arrows from $x_1$ to $x_2$ and from $x_3$ to $x_4$, length one vertical arrows from $x_1$ to $x_2$ and from $x_3$ to $x_4$, and no other horizontal or vertical arrows into or out of these generators. Then possibly after a change of basis $C$ splits as a direct sum $C' \oplus C_{box}$.

 \end{lemma}
 
 \begin{proof}
By scaling the generators $x_1$, $x_2$, $x_3$, and $x_4$ by constants we can arrange that the horizontal arrow from $x_1$ has weight $U$ and the vertical arrows from $x_1$ to $x_3$ and from $x_2$ to $x_4$ have weight $V$; it then follows from $\partial^2 = 0$ that the weight of the horizontal arrow from $x_3$ to $x_4$ is $-U$. These four generators and the arrows between them thus form a copy of $C_{box}$, so we just need to remove any other arrows in or out of these four generators. Any such arrows are diagonal by assumption.

We will first arrange that there are no other arrows into the generator $x_4$. Suppose there is another generator $y$ with an arrow of weight $cU^a V^b$ from $y$ to $d$; this is a diagonal arrow so $a$ and $b$ are both positive. We perform the change of basis replacing $y$ with $y' = y - w U^a V^{b-1} x_2$ and note that the coefficient of $x_4$ in $\partial(y')$ is zero. The only other changes involving the generators in the box complex is that for any arrow $z$ into $y$ there is a new arrow from $z$ to $x_2$. Note that if $b = 1$ and the arrow from $z$ to $y$ is a horizontal arrow, then the new arrow from $z$ to $x_2$ will be horizontal so the complex may no longer be horizontally simplified. However, we can say that if the new arrow from $z$ to $x_2$ is horizontal it must have length at least two. If we repeat this for all generators $y$ other than $x_2$ and $x_3$ with arrows into $x_4$, we arrive at a complex with no unwanted arrows into $x_4$.

We next eliminate any unwanted arrows into $x_2$ in a similar way. Suppose there is a generator $y$ other than $x_1$ with an arrow weighted by $w U^a V^b$ from $y$ to $x_2$. We must have that $a > 0$ since the basis is still vertically simplified, and either $b > 0$ or $a > 1$ by the observation in the previous paragraph. We perform the change of basis replacing $y$ with $y' = y - w U^{a-1} V^b x_1$, and note that the coefficient of $x_2$ in $\partial(y')$ is zero. The basis change may also introduce new arrows into $x_1$, but otherwise arrows in or out of $x_1$, $x_2$, $x_3$, and $x_4$ are unaffected. Repeating this for all generators $y$ other than $x_1$ with arrows into $x_2$ results in a complex with no unwanted arrows into $x_2$.

We can now deduce that there are no arrows into $x_1$ and no arrows into $x_3$ except the vertical arrow from $x_1$. For the first claim, note that for any $y$ the coefficient of $x_2$ of $\partial^2(y)$, which must be zero, is simply $U$ times the coefficient of $x_1$ in $\partial(y)$ since the only arrow into $x_2$ is from $x_1$. For the second claim, note that for any $y$ other than $x_1$ the coefficient of $x_4$ in $\partial^2(y)$ is $U$ times the coefficient of $x_3$ in $y$ since the only arrows into $x_4$ are from $x_2$ and $x_3$ and the only arrow into $x_2$ is from $x_1$. Thus there are now no extraneous arrows into the four generators $x_1$, $x_2$, $x_3$, and $x_4$.

Removing extra arrows out of these four generators is similar. We first remove unwanted arrows that start at $x_1$: if there is a generator $y$ with an arrow from $x_1$ to $y$ we perform a basis change that adds an appropriate multiple of $y$ to $x_2$. Each such basis change might add a new arrow out of $x_2$, which must be either diagonal or vertical with length at least two. We then remove arrows out of $x_2$ other than the vertical arrow to $x_4$ by adding an appropriate multiple of $y$ to $x_4$ for any $y \neq x_4$ with an arrow from $x_2$ to $y$. Finally, we use $\partial^2 = 0$ to deduce that there are no unwanted arrows out of $x_3$ or $x_4$.
 \end{proof}
 
If an immersed multicurve $\Gamma$ contains a simple figure eight curve then for any bounding chain $\bchain$ on $\Gamma$ the corresponding complex $C(\Gamma, \bchain)$ over $\sRhat$ contains a copy of $C_{box}$. By splitting off $C_{box}$ summands from the corresponding complex over $\sRminus$, we can show that a bounding chain can be chosen with no intersection points on the simple figure eight components.

\begin{proposition}\label{prop:split-off-fig8s}
Let $\Gamma$ be an immersed multicurve in the marked strip $\strip$ or $\cylinder$ in almost simple position. Any bounding chain on $\Gamma$ is equivalent to one that contains no self-intersection points on any simple figure eight component of $\Gamma$.
\end{proposition}
\begin{proof}
We may assume that the initial bounding chain $\bchain$ has $\bchainhat$ of local system type, since we have shown that every bounding chain is equivalent to one of this form. We now consider the complex $C = C(\Gamma, \bchain)$ and, fixing a simple figure eight component of $\Gamma$, we consider the four generators of $C$ arising from this component. Because $\bchainhat$ has local system type and there are no local system self-intersection points on a simple figure eight curve, the connected component of the complex $\hat C = C|_{UV = 0}$ over $\sRhat$ does not depend on the bounding chain. It follows that if we ignore diagonal arrows, the four generators of $C$ coming from the specified simple figure eight component generate a copy of $C_{box}$ with no other horizontal or vertical arrows in or out of these generators. By Lemma \ref{lem:splitt-off-box-summand} we can change basis to remove any diagonal arrows in or out of these four generators, realizing $C$ as a direct sum $C' \oplus C_{box}$ for some smaller complex $C'$. It is clear from the proof of Lemma \ref{lem:splitt-off-box-summand} that the gradings of the generators on the $C_{box}$ summand are the same as the gradings of the four relevant generators of $C$. We now consider construct and immersed decorated multicurves $(\Gamma', \bchain')$ and $(\Gamma_{box}, \bchain_{box})$ representing $C'$ and $C_{box}$, respectively, and note that $(\Gamma'\sqcup\Gamma_{box}, \bchain' + \bchain_{box})$ represents $C = C' \oplus C_{box}$ (with respect to the new basis). It is clear that $\Gamma_{box}$ is a simple figure eight curve with the same bigradings as to the simple figure eight component of $\Gamma$ we singled out, and $\bchain_{box}$ is trivial. We have that $\Gamma' \sqcup \Gamma_{box}$ is homotopic to $\Gamma$, by the uniqueness of the immersed curve representing $C$, so $\bchain'$ may be viewed as a bounding chain on $\Gamma$ that does not include self-intersection points on the specified simple figure eight component. We can repeating this argument for all simple figure eight components, modifying $\bchain$ to avoid self-intersection points on any of them.
\end{proof} 

\begin{corollary}\label{cor:embedded-curve-with-fig8s}
If $\Gamma$ is an immersed curve in the strip $\strip$ or cylinder $\cylinder$ containing one embedded component and some number of simple figure eight components, then any bounding cochain on $\Gamma$ is equivalent to the trivial bounding cochain.
\end{corollary}
\begin{proof}
By Proposition \ref{prop:split-off-fig8s} any bounding cochain is equivalent to a linear combination of the self-intersection points of $\Gamma$ that are not on a simple figure eight component, but if the only component that is not a simple figure eight is embedded then there are no such self-intersection points.
\end{proof}

In particular, if $\Gamma$ contains one embedded component and some number of simple figure eight components then the minus invariant $(\Gamma, \bchain)$ is determined by the weaker hat version $(\Gamma, \bchainhat)$; equivalently, the full complex over $\sRminus$ is uniquely determined up to homotopy equivalance by the $UV = 0$ complex. This is relevant because computing the full knot Floer complex of a knot is hard but we have powerful computational tools for computing the $UV = 0$ complex. Knots whose knot Floer complexes satisfy the conditions of Corollary \ref{cor:embedded-curve-with-fig8s} are quite common in practice: using computer computations the author has found the immersed curves associated with all prime knots in $S^3$ up to 15 crossings, and of these 313,230 knots all but one have knot Floer complexes with this property. Thus, although the computational techniques used only compute the complex $CFK_{\sRhat}$, we can in fact say that we have computed $CFK_{\sRminus}$ of these knots.

Even when complexes do not take the form of an embedded curve along with simple figure eight components, it is common in for small examples that the the bounding chain $\bchain$ is determined up to equivalence by $\bchainhat$ (though the trivial bounding chain may not be an option). Another example is the complex in Figure \ref{fig:example-with-crossing} discussed in Section \ref{sec:two-examples}. There the bounding chain decoration is essential, since this curve without decoration is obstructed (in particular, the trivial linear combination of self-intersection points is not a valid bounding chain). However, there is only one linear combination of self-intersection points for which $(\Gamma, \bchain)$ satisfies the Maurer-Cartan equations, so in this case again $(\Gamma, \bchain)$ is determined by $\Gamma$. This is a geometric version of the statement that, up to homotopy, there is a unique way to add diagonal arrows to the $UV=0$ complex such that $\partial^2 = 0$. Although as Example \ref{ex:example-multiple-bounding-chains} demonstrates it is not hard to construct complexes over $\sRminus$ that are not determined by their quotients over $\sRhat$, these do not seem to arise often in practice. In fact, at that time of writing the author has not yet found an example of a knot for which $CFK_{\sRminus}(K)$ is not determined by $CFK_{\sRhat}(K)$.

\subsection{An example with nontrivial $\bchainhat$}

All of the examples presented so far have not required a bounding chain decoration to represent the $UV=0$ complex; in other words, $\bchainhat$ has been trivial. This is very common in practice, and in fact it is an open question whether the decoration $\bchainhat$ is needed to represent the complex $CFK_{\sRhat}(Y, K)$ for any knot $K \subset Y$. But it is not difficult to construct a complex for which the $\bchainhat$ is nontrivial, such as the example below.

\begin{example}\label{ex:nontrivial-local-system}
The bigraded complex on the left of Figure \ref{fig:nontrivial-local-system} is represented by the decorated immersed curve in the infinite strip $\strip$ shown on the right side of the figure. For simplicity we use $\Z/2\Z$ coefficients for this example. Note that the boundaing chain has nonzero coefficients for six self intersection points, five of these are weighted by $W$ and are ignored when representing the complex over $\sRhat$ but one intersection point is a local system intersection point of the non-primitive closed component and  included in $\bchainhat$ as well.
\end{example}

\begin{figure}
\includegraphics[scale = .9]{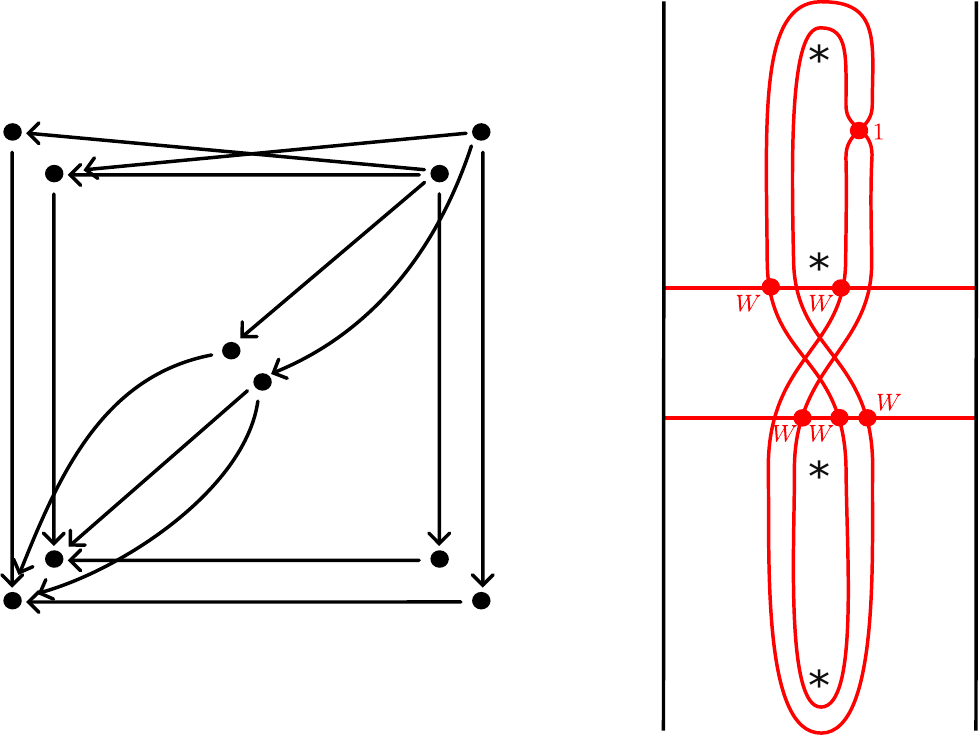}
\caption{The complex on the left is represented by the decorated curve on the right. Even as a complex over $\sRhat$ (ignoring diagonal arrows), a nontrivial bounding chain is required.}
\label{fig:nontrivial-local-system}
\end{figure}

This chain complex can not be $CFK_{\sRminus}$ for a knot in $S^3$, since $\infty$-filling has two dimensional $\HFhat$ but it satisfies all known constraints of being the knot Floer complex for a knot in some $Y$. In particular, the decorated immersed curve is symmetric under the action of the elliptic involution, although this is not obvious since the symmetry holds only up to homotopy and the non-uniqueness of the bounding chain is relevant here. Applying the elliptic involution gives the decorated immersed curve in the middle of Figure \ref{fig:nontrivial-local-system-symmetry} and homotoping the curves to their original position results in a different bounding chain as shown on the right of Figure \ref{fig:nontrivial-local-system-symmetry}. We leave it as an exercise to the motivated reader to check that this bounding chain is equivalent to the original one by adding a pair of crossover arrows from the lower horizontal arc to the higher horizontal arc and sliding them to opposite boundaries of the strip. This example speaks to the difficulty of defining a normal form for bounding chains as a subset of self-intersection points in the minus theory, as we have two valid choices of $\bchain$ for the same homotopy class of complexes and it is not obvious which should be considered the preferred representative.

 \begin{figure}
\includegraphics[scale = .9]{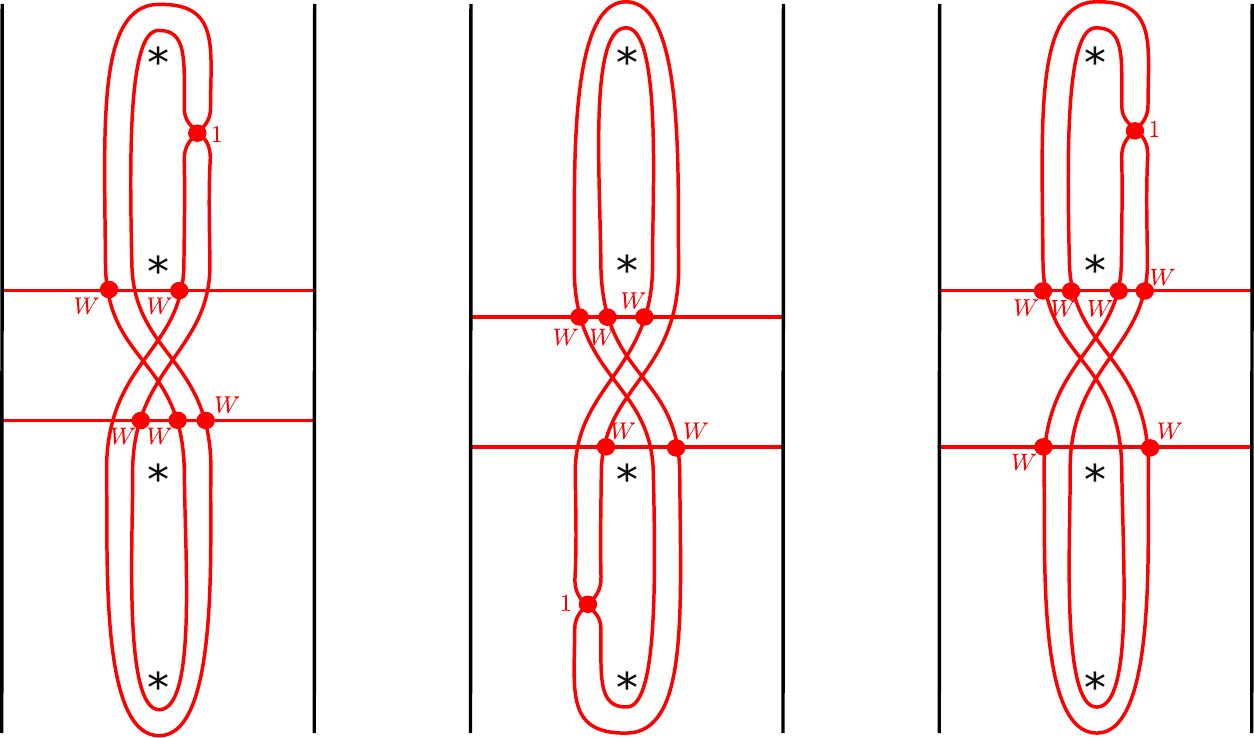}
\caption{Realizing the symmetry on the complex from Example \ref{ex:nontrivial-local-system}. On the left is the decorated curve from Figure \ref{ex:nontrivial-local-system}, in the middle is the result of rotating this decorated curve by half a rotation, and on the right is the result of homotoping the curve to agree with the curve on the left, keeping track of resulting changes to the bounding chains following the usual local moves. The bounding chains on the left and right are different as subsets of the self intersection points but the decorated curves are equivalent as elements of the Fukaya category of the marked strip.}
\label{fig:nontrivial-local-system-symmetry}
\end{figure}
 
 \subsection{Remarks on $\Z$-coefficients}\label{sec:Z-coefficients}
 
 Throughout this paper we have assumed field coefficients, and this is essential in some areas, but many arguments work with $\Z$-coefficients as well. We end with a few words about what works and what still needs to be done to have useful immersed curve invariants over $\Z$. We first observe that the definition of Lagrangian Floer homology of decorated immersed curves is the same using $\Z$ coefficients, except that to ensure invariance of Floer homology under homotopies we must require the weights associated with basepoints on curves to be $\pm 1$. This is because sliding one curve past a basepoint with weight $c$ on the other curves (as in moves (c) or (d) in Figure \ref{fig:invariance-moves} changes the Floer complex by replacing a generator $x$ with $cx$ and this is only a change of basis if $c$ is a unit. The rest of the proof of homotopy invariance is unaffected, and self intersection points can be decorated with arbitrary elements of $\Z[W]$ (or $\Z[U,V]$ in the doubly marked case). By taking Floer homology with $\mu$ in the doubly marked cylinder $\dmstrip$, any decorated curve in $\strip$ determines a bigraded complex over $\Z[U,V]$. Similarly, decorated curve in the infinite marked cylinder $\cylinder$ determines a a bigraded complex over $\Z[U,V]$ along with a flip isomorphism. Given two decorated immersed curve in $\cylinder$ representing  complexes over $\Z[U,V]$, the shifted pairing of curves still agrees up to homotopy with the mapping cone of morphism complexes as described in Section \ref{sec:morphisms}; in particular, if a decorated immersed curve with represents the knot Floer complex of a knot $K \subset Y$ with $\Z$ coefficients along with its flip isomorphism then taking Floer homology with a line of slope $\frac p q$ computes $\HFminus$ of $\frac p q$-surgery on $K$ with $\Z$ coefficients. 
 
Moreover, it is still true that any bigraded complex over $\Z[U,V]$ can be represented by some decorated immersed curve in $\strip$ (and similarly that any complex with a flip isomorphism can be represented by some decorated curve in $\cylinder$), namely the naive curve representative described in Section \ref{sec:naive-curves}. The main difference in the case of $\Z$ coefficients is that these representatives can not be simplified as fully. The crucial failing of the arrow sliding algorithm when working with $\Z$ coefficients is the local move at the bottom of Figure \ref{fig:local-moves}, in which a self intersection point next to crossover arrow is resolved and the crossing is replaced by a new crossover arrow. Because this move introduces an arrow and a basepoint whose weight is the inverse of the weight of the original arrow, this move is only possible when the original arrow is weighted by $\pm 1$. Having said this, we observe that all the other $n$-strand arrow configuration replacements in Figure \ref{fig:local-moves} are still valid with $\Z$ coefficients. 

Without the crossing resolving move, we can not run the arrow sliding algorithm in full. However, we can still follow a modified version of the algorithm, keeping the overall strategy of systematically removing each crossover arrow except for those that are not removable. Recall that the strategy for removing a single arrow with field coefficients is to slide it one direction until the curve segments it connects diverge, and then if the arrow points from the left segment to the right segment it can be removed. If it points from the right segment to the left segment we slide the arrow the other directions until the strands diverge and remove it if possible. If the arrow is not removable on either end, then the curve segments must cross; in this case we slide the arrow to the crossing, resolve the crossing, and then the two resulting arrows will be removable when pushed as far as possible in opposite directions. In this way we can remove all crossover arrows except those for which the curve segments they connect never diverge, and such arrows can be moved to be left-turn crossover arrows at local-system intersection points of non-primitive curve components. With $\Z$ coefficients we can follow the same strategy, except that when an arrow is not removable at both ends we can not resolve the crossing if the weight on the crossover arrow is not $\pm 1$. In this case, we still slide the crossover arrow to the crossing, where it will necessarily be a left-turn crossover arrow. This arrow is not removable, so we will simply include this self-intersection point with the appropriate coefficient in the bounding chain $\bchainhat$. In other words, a more complicated bounding chain, using more than just the local system intersection points, is required even setting $UV = 0$. Note that in the original curve sliding algorithm, some work was required to show that this strategy for removing a single arrow could be performed repeatedly to remove many arrows such that the total process would eventually terminate; this is slightly more difficult in the $\Z$ coefficient setting because there are more unremovable arrows present to interact with the arrow being removed, but the modifications to the algorithm are fairly routine.

\begin{example}\label{ex:z-coefficients-example}
Consider the chain complex over $\Z[U,V]$ shown in Figure \ref{fig:z-coefficients-example}. One can check that the first immersed curve in the figure represents this complex (with respect to the given basis). Note that even if we restrict to $UV=0$ we still require a non-trivial bounding chain in order to capture the arrows labeled by $3U$. We can pass to finite field coefficients in multiple ways by taking different quotients of $\Z$ and we observe that for this complex different finite field coefficients result in different immersed curves. For $\Z/3\Z$ coefficients we simply ignore the arrows labelled by $3U$ and remove the bounding chain decoration from the curve representing the complex with $\Z$ coefficients, but the underlying curve is unchanged. For $\Z/2\Z$ coefficients, on the other hand, the two marked intersection points remain in the bounding chain with coefficient 1. We interpret the bounding chain as giving left-turn crossover arrows at these intersection points and run the arrow sliding algorithm; we leave it as an exercise to see that this gives the rightmost curve in the figure. Note that if we keep track of the basis changes corresponding to the arrow slides, these introduce diagonal arrows. Following the algorithm for enhancing curves to capture diagonal arrows, we include the indicated intersection points with coefficient $W$ in $\bchain$; alternatively, once we find the curve representing the $UV=0$ complex we can observe that the bounding chain decoration $\bchain$ is forced and these intersection points must appear in $\bchain$ in order for the monogons enclosing two punctures to cancel with something.
\end{example}

 \begin{figure}
\includegraphics[scale = .65]{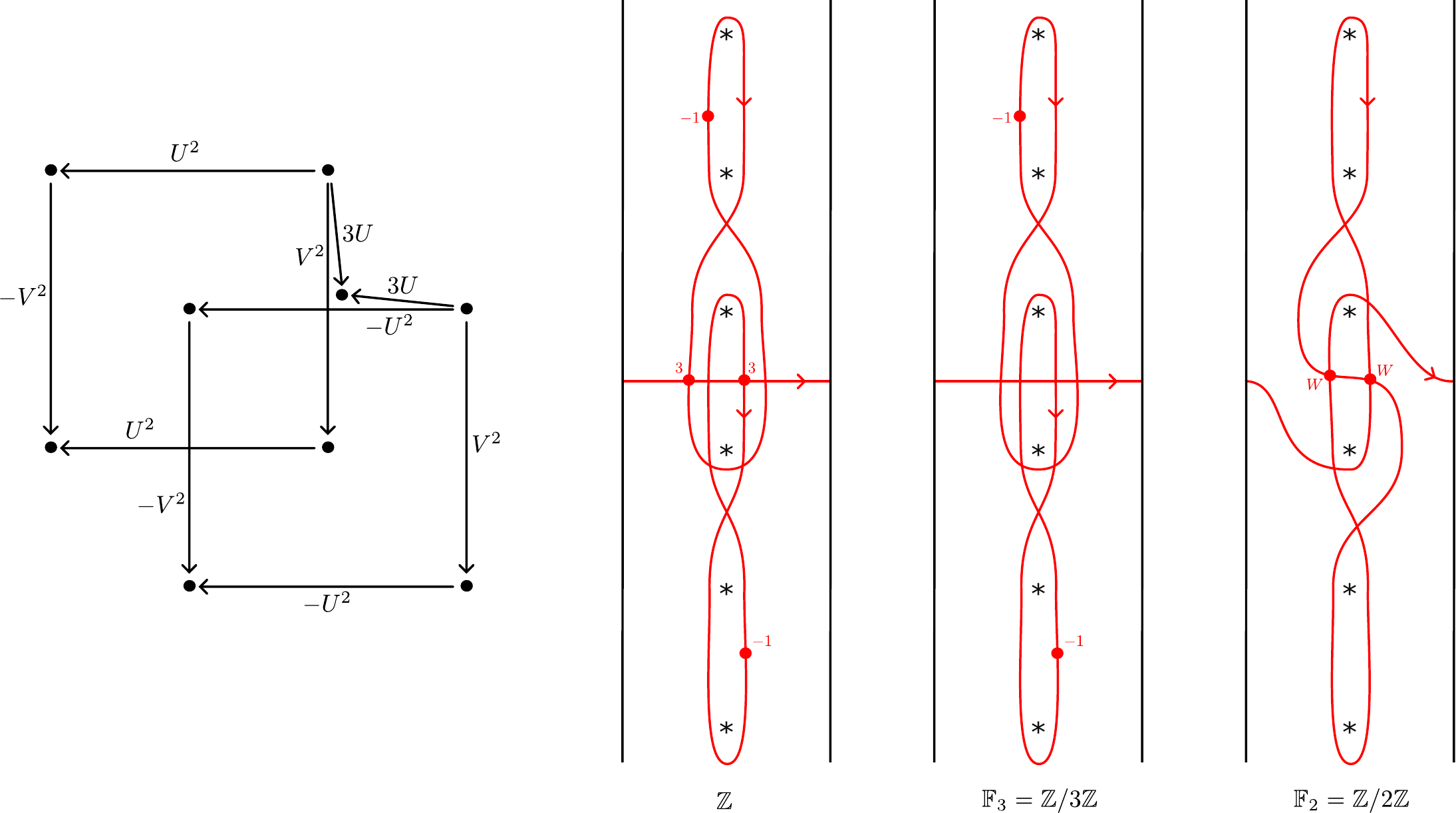}
\caption{A chain complex over $\Z[U,V]$ and a decorated immersed curve in $\strip$ representing it, along with decorated curves representing the same complex with coefficients taken modulo 3 and modulo 2, as indicated.}
\label{fig:z-coefficients-example}
\end{figure}

By following a modified arrow sliding algorithm, we can represent any bigraded complex over $\Z[U,V]/(UV=0)$ by a graded immersed curve $\Gamma$ decorated with a bounding chain $\bchainhat$ which is a linear combination of degree zero self intersection points of $\bchainhat$ such that the coefficient of any intersection point that is not a local system intersection point is not $\pm 1$. Moreover, by construction we have made some effort to remove the crossover arrows so it is reasonable to hope that this algorithm produces a representative that is as simple as possible in some sense. However, the notion of ``as simple as possible" is not clear in this case, since non-removable crossover arrows can be moved between different self-intersection points and there may be clever basis changes not suggested by the arrow sliding algorithm to replace a collection of crossover arrows with a simpler one. In the field coefficient case the fact that $\bchainhat$ can be taken to be trivial except on local system intersection points is powerful because for immersed curves decorated with such a bounding chain the dimension of Lagrangian Floer homology is simply the minimal intersection point of the curves except when the curves being paired have parallel components (in which case there is an additional term that is not hard to understand). This fact enabled us to prove that the decorated immersed curves are unique by observing that two different curves would have different pairing with some third curve. When $\bchainhat$ is more complicated the Lagrangian Floer homology no longer reduces to the minimal intersection number, so arguments of this form are more difficult.

To summarize, bigraded complexes over $\Z[U,V]$ can be represented by decorated immersed curves in $\strip$, and in practice we can choose a fairly nice representative, but even when considering the simpler $UV=0$ complex these objects are more complicated than immersed curves with local systems. Unlike the case of field coefficients, it is unclear to what extend the decorated curve representing the $UV=0$ complex is unique, and in fact the $UV=0$ case with $\Z$ coefficients exhibits the same subtleties concerning uniqueness of the bounding chain decoration that arise in the case of minus invariants with field coefficients. Despite this added complexity, we expect that immersed curves will be a valuable tool for studying knot Floer homology with $\Z$ coefficients and this is a topic for further exploration.

\bibliographystyle{alpha}
\bibliography{bibliography}

\end{document}